\documentclass[1pt,a4]{report}

\usepackage{amsmath}
\usepackage{epsfig,amsthm}
\usepackage{latexsym}
\usepackage{amsfonts}
\usepackage{amssymb}
\usepackage{amscd}
\usepackage{mathrsfs}
\usepackage[all,cmtip]{xy}
\usepackage{enumerate}
\usepackage{makeidx}

\numberwithin{equation}{section} 

\newtheorem{thm}{Theorem}[chapter]
\newtheorem{cor}[thm]{Corollary}
\newtheorem{lemma}[thm]{Lemma}
\newtheorem{prop}[thm]{Proposition}

\theoremstyle{definition}

\newtheorem{rmk}[thm]{Remark}

\makeindex

\begin{document}
\title{Transfer relations in essentially tame local Langlands correspondence}
\author{Geo Kam-Fai Tam}
\date{July 2012}
\maketitle

\pagenumbering{arabic}
\begin{abstract}
Let $F$ be a non-Archimedean local field and $G$ be the general linear group $\mathrm{GL}_n$ over $F$. Bushnell and Henniart described the essentially tame local Langlands correspondence of $G(F)$ using rectifiers, which are certain characters defined on tamely ramified elliptic maximal tori of $G(F)$. They obtained such result by studying the automorphic induction character identity. We relate this formula to the spectral transfer character identity, based on the theory of twisted endoscopy of Kottwitz, Langlands, and Shelstad. In this article, we establish the following two main results.
\begin{enumerate}[(i)]
  \item To show that the automorphic induction character identity is equal to the spectral transfer character identity, when both are normalized by the same Whittaker data.
      \item To express the essentially tame local Langlands correspondence using admissible embeddings constructed by Langlands-Shelstad $\chi$-data, and to relate Bushnell-Henniart's rectifiers to certain transfer factors.
\end{enumerate}
\end{abstract}

\newpage

\tableofcontents

\chapter{Introduction}\label{chapter intro}

\section{Historical background}\label{section history}

Let $F$ be a non-Archimedean local field. We assume that its residue field $\mathbf{k}_F$ has $q$ elements and of characteristic $p$. Let $G$ be the general linear group $\mathrm{GL}_n$ over $F$. We know that the supercuspidal spectrum of $G(F)$ is bijective to the collection of irreducible $n$-dimensional smooth complex representations of the Weil group $W_F$ of $F$. This is a main result of the local Langlands correspondence for $\mathrm{GL}_n$ proved by Harris and Taylor \cite{HT} and by Henniart \cite{Hen-simple} independently in the characteristic 0 case and by Laumon, Rapoport, and Stuhler \cite{LRS} in the positive characteristic case.

Bushnell and Henniart described in \cite{BH-ET1} the restriction of such correspondence to the essentially tame sub-collections. We briefly summarize their result as follows. Let $\mathcal{G}^\mathrm{et}_n(F)$ be the set of equivalence classes of essentially tame irreducible representations of $W_F$ of degree $n$ and $\mathcal{A}^\mathrm{et}_n(F)$ be the set of isomorphism classes of essentially tame irreducible supercuspidal representations. We will recall the two notions of essential-tameness in section \ref{section The correspondence}. Bushnell and Henniart proved that there exists a unique bijection
\begin{equation}\label{intro langlands corresp}
\mathcal{L}_n={}_F\mathcal{L}^\mathrm{et}_n:\mathcal{G}^\mathrm{et}_n(F)\rightarrow\mathcal{A}^\mathrm{et}_n(F), \sigma\mapsto \pi,
\end{equation}
 characterized by certain canonical conditions together with, roughly speaking, the compatibilities of automorphic induction \cite{Hen-Herb}, \cite{HL2010}, \cite{HL2011} and base change \cite{AC}, \cite{HL2011} (see also \cite{BH-LTL3}, \cite{BH-LTL4}) on $\mathcal{A}^\mathrm{et}_n(F)$ with induction and restriction on $\mathcal{G}^\mathrm{et}_n(F)$ (see the precise statement in Proposition 3.2 of \cite{BH-ET1}). We call the map $\mathcal{L}_n$ the {essentially tame local Langlands correspondence}. In most literature, we call $\sigma$ the Langlands parameter of $\pi$.

To describe $\mathcal{L}_n$, we continue to follow the idea in \cite{BH-ET1}. There they introduced the third set $P_n(F)$ of equivalence classes $(E/F,\xi)$ of admissible characters $\xi$ of $E^\times$ over $F$, where $E$ goes through tamely ramified extensions over $F$ of degree $n$. We will recall the notion of admissible character in section \ref{section admissible characters} and describe its properties in section \ref{section admissible characters revisited}. The collection $P_n(F)$ bijectively parameterizes $\mathcal{G}^\mathrm{et}_n(F)$ and $\mathcal{A}^\mathrm{et}_n(F)$ simultaneously. We denote these bijections by
\begin{equation}\label{bijection of P and G}
\Sigma_n:P_n(F) \rightarrow \mathcal{G}^\mathrm{et}_n(F),\,(E/F, \xi) \mapsto \sigma_\xi,
 \end{equation}
and
\begin{equation}\label{bijection of P and A}
\Pi_n: P_n(F) \rightarrow \mathcal{A}^\mathrm{et}_n(F),\,(E/F,\xi)\mapsto \pi_\xi.
 \end{equation}
 We describe these bijections in simple words.
 \begin{enumerate}[(i)]
   \item The correspondence $\Sigma_n$ is simply the induction of representation $\sigma_\xi=\mathrm{Ind}_{W_E}^{W_F}\xi$ if we regard $\xi$ as a character of $W_E$ by local class field theory \cite{Tate-NTB}. According to the compatibility condition characterizing $\mathcal{L}_n$, the bijection $\mathcal{L}_n\circ \Sigma_n:P_n(F)\rightarrow \mathcal{A}^\mathrm{et}_n(F)$ is a composition of successive lifts of representations by automorphic induction. We will review this notion in section \ref{section autom ind}. Here we remark that the existence of automorphic induction depends on some global arguments, as in section 7 and 8 of \cite{Hen-Herb}.
       \item The correspondence $\Pi_n$ is a compact induction
       \begin{equation}\label{intro compact ind}
         \pi_\xi=\mathrm{cInd}_{{\mathbf{J}_\xi}}^{G(F)}\Lambda_\xi
       \end{equation} of certain representation $\Lambda_\xi$ of a compact-mod-center subgroup ${\mathbf{J}_\xi}$ of $G(F)$. The representation $({\mathbf{J}_\xi},\Lambda_\xi)$ is called an extended maximal type in section 2.1 of \cite{BH-ET1}, based on the theory of simple types in \cite{BK}. We will provide a summary of the construction from $(E/F,\xi)$ to $({\mathbf{J}_\xi},\Lambda_\xi)$ in section \ref{Supercuspidal Representations}. We remark that such construction is completely local and representation theoretic.
 \end{enumerate}
In the case when $p \nmid n$, traditionally known as the tame case, all irreducible supercuspidal representations of $G(F)$ and all $n$-dimensional irreducible complex representations of $W_F$ are essentially tame, and $P_n(F)$ consists of $(E/F, \xi)$ for $E$ goes through all separable extensions over $F$ of degree $n$. The correspondence (\ref{bijection of P and A}) was constructed by Howe \cite{Howe} and, based on such construction, the description of $\mathcal{L}_n$ using $P_n(F)$ as parameters was studied by Moy \cite{Moy} and Reimann \cite{Rei}. In particular when $n=2$, the book \cite{BH-GL2} contains an extensive treatment in this theory.

The composition of the bijections  (\ref{bijection of P and G}), (\ref{intro langlands corresp}), and the inverse of (\ref{bijection of P and A}),
\begin{equation*}
\mu: P_n(F) \xrightarrow{\Sigma_n} \mathcal{G}^\mathrm{et}_n(F) \xrightarrow{\mathcal{L}_n} \mathcal{A}^\mathrm{et}_n(F) \xrightarrow{\Pi_n^{-1}} P_n(F),
\end{equation*}
is not the identity map on $P_n(F)$ in general. Bushnell and Henniart proved in \cite{BH-ET3} that, for each admissible character $\xi$ of $E^\times$, there is a unique tamely ramified character ${}_F\mu_\xi$ of $E^\times$, depending on the restriction $\xi|_{U^1_E}$ of $\xi$ on the 1-unit group ${U^1_E}$ of $E^\times$, such that ${}_F\mu_\xi\cdot \xi$ is also admissible and
\begin{equation*}
\mu(E/F, \xi) = (E/F, {}_F\mu_\xi\cdot \xi).
\end{equation*}
We call ${}_F\mu_\xi$ the {rectifier} of $\xi$. In the series \cite{BH-ET1}, \cite{BH-ET2}, \cite{BH-ET3}, they express explicitly the rectifier ${}_F\mu_\xi$ and also the correspondence $\mathcal{L}_n$.

We briefly explain how to deduce the values of ${}_F\mu_\xi$, following the series of papers just mentioned. We first construct a sequence of subfields
\begin{equation*}
F\subseteq K_0\subseteq K_1\subseteq \cdots\subseteq K_l\subseteq E,
\end{equation*}
 which satisfies the following conditions.
 \begin{enumerate}[(I)]
\item $K_0/F$ is the maximal unramified sub-extension of $E/F$. \label{intro unram}
\item $K_l/K_{l-1},\dots,K_1/K_0$ are quadratic totally ramified. \label{intro quad totally ram}
\item $E/K_l$ is totally ramified of odd degree. \label{intro odd degree totally ram}
\end{enumerate}
Notice that the extensions in (\ref{intro unram}) and (\ref{intro quad totally ram}) are cyclic. The one in (\ref{intro odd degree totally ram}) is cyclic if we adjoin to the base field $F$ sufficient number of roots of unity. Each rectifier ${}_F\mu_\xi$ then admits a factorization
\begin{equation}\label{coarse factorization of rectifier}
{}_{F}\mu_\xi=({}_{K_0/F}\mu_\xi)({}_{K_1/K_0}\mu_\xi)\cdots({}_{K_l/K_{l-1}}\mu_\xi)({}_{K_l}\mu_\xi),
\end{equation}
such that each factor is tamely ramified. The approach in the series \cite{BH-ET1}, \cite{BH-ET2}, \cite{BH-ET3} is to deduce each factor on the right side of (\ref{coarse factorization of rectifier}) through an inductive process as follows. Assume at this moment that the field extension in (\ref{intro odd degree totally ram}) above is cyclic. We reduce to the case when $K/F$ is a cyclic sub-extension of $E/F$ and assume that ${}_K\mu_\xi$ is known when $\xi$ is regarded as an admissible character of $E^\times$ over $K$. We can deduce a new character ${}_{K/F}\mu_\xi$ of $E^\times$, called a $\nu$-rectifier of $\xi$ in this thesis, by comparing the automorphic induction character \cite{Hen-Herb} with the twisted Mackey induction character (in (1.4.1) of \cite{BH-ET3}) of the compact induction (\ref{intro compact ind}). We then define $${}_{F}\mu_\xi=({}_{K/F}\mu_\xi)({}_{K}\mu_\xi).$$ In general we can apply a base change to assume that the extension in (\ref{intro odd degree totally ram}) is always cyclic, so that the rectifier ${}_{K_l}\mu_\xi$ is defined (see the proof of Theorem 3.5 and section 4.6-4.9 of \cite{BH-ET1}). We will recall the values of each factor in (\ref{coarse factorization of rectifier}) in section \ref{section explicit rectifiers}.

The theory of automorphic induction is subsumed under the transfer principle of Langlands and Shelstad \cite{LS} or more precisely under the theory of twisted endoscopy of Kottwitz and Shelstad \cite{KS}. We provide a brief idea of these theories as follows. We call a connected $F$-quasi-split reductive group $H$ a twisted endoscopic group of $G$ if the complex dual group of $H$ is a twisted centralizer of a semi-simple element in the complex dual of $G$. This applies to the pair $(G,H)=(\mathrm{GL}_n,\mathrm{Res}_{K/F}\mathrm{GL}_{n/|K/F|})$ when $K/F$ is cyclic. Under such setup, we can transfer the semi-simple conjugacy classes of $H$ to the twisted semi-simple conjugacy classes in $G$ and obtain equalities of weighted sums of orbital integrals of corresponding conjugacy classes. This is the main idea of the geometric transfer principle. The weights in the equalities are given by a function on the cartesian product of the conjugacy classes of $G(F)$ and $H(F)$ called the transfer factor, whose values are non-zero only at the pairs of conjugacy classes related by transfer. In chapter \ref{endoscopy}, we will study the transfer factor associated to the pair $(G,H)$, which admits a simple form in this case. We remark that the transfer principle is conditional on the Fundamental Lemma, which is proved by Waldspurger \cite{Walds-GLn} for the cases we concern in this thesis and by Ng\^o \cite{Ngo} for the ordinary endoscopy in general. In this thesis, we sometimes replace the word `twisted endoscopy' by simply `endoscopy' for convenience.

Of course, there should be a spectral counterpart of the geometric transfer principle. The spectral transfer principle conjectures that, for each irreducible tempered representation $\rho$ of $H(F)$, there is an irreducible tempered representation $\pi$ of $G(F)$ determined by certain character identity in terms of $\rho$. In the case when $(G,H)=(\mathrm{GL}_n,\mathrm{Res}_{K/F}\mathrm{GL}_{n/|K/F|})$ as in the previous paragraph, such transfer of representations is established in \cite{Hen-Herb}. The character relation between $\pi$ and $\rho$, which we shall call the spectral transfer character identity, is known to be the same as the automorphic induction character identity up to a constant. In particular, if $\rho$ is essentially tame supercuspidal and satisfies the regular condition by the action of the Galois group $\Gamma_{K/F}$, then $\pi$ is also essentially tame supercuspidal by \cite{BH-ET1}. We remark that Shelstad has established the transfer principle for tempered representations of real groups in \cite{Shelstad-temp-endo-for-real-group-1}, \cite{Shelstad-temp-endo-for-real-group-2}.

\section{Main results of the author}\label{section main results of author}

The two main results in this thesis and their proofs constitute chapter \ref{chapter comparing character} and \ref{chapter rectifier and transfer} respectively. We introduce the main idea of the statements in the subsections. For the readers who know the background theory well, we encourage them to take a look at the interlude in chapter \ref{chapter interlude}. There we briefly describe the idea in proving the two main results in more detail, since then we have enough information to clarify the notions and terminologies.

\subsection{The first result}

The first main result, Theorem \ref{auto-ind-char-HH-LS-equal}, is to clarify the `up to a constant' relation between the automorphic induction character identity and the spectral transfer character identity stated in the last paragraph of section \ref{section history}.

 Suppose that $\pi$ is an essentially tame supercuspidal representation of $G(F)$ automorphically induced from an essentially tame supercuspidal representation $\rho$ of $H(F)$. A necessary condition is that $\pi$ is isomorphic to its twist by the character $\kappa$ of $F^\times$ corresponding to the cyclic extension $K/F$ via local class field theory. Then there is a non-trivial $G(F)$-intertwining operator $$\Psi:\kappa\pi:=(\kappa\circ\det)\otimes \pi\rightarrow \pi$$ defined up to a constant. Let $\Theta_\rho$ be the character of $\rho$ and $ \Theta^\kappa_\pi$ be the twisted character of $\pi$ depending on the intertwining operator $\Psi$, both regarded as locally constant functions on the regular semi-simple parts of $H(F)$ and $G(F)$ respectively by \cite{HC-AID} and 3.9 Corollary of \cite{Hen-Herb}. Let $\Gamma_{K/F}=\Gamma_F/\Gamma_K$ be the Galois group of $K/F$. A very rough form of the automorphic induction character identity looks like
\begin{equation}\label{rough introduce automorphic ind char}
  \Theta^\kappa_\pi(\gamma)=\Delta_\mathrm{HH}(\gamma)\sum_{g\in \Gamma_{K/F}}\Theta_{\rho^g}(\gamma),
\end{equation}
for all $\gamma$ lying in the subset of elliptic semi-simple elements of $H(F)$ and being regular in $G(F)$. Here $\Delta_\mathrm{HH}$ is a transfer factor defined by Henniart and Herb \cite{Hen-Herb}. It depends on the choices of several auxiliary objects, which are hidden in the background at this moment. Varying these choices only changes the transfer factor $\Delta_\mathrm{HH}$, and so the right side of (\ref{rough introduce automorphic ind char}), only by a constant independent of $\gamma$. The point is that, in order to compare (\ref{rough introduce automorphic ind char}) with the twisted Mackey induction character from (\ref{intro compact ind}), we have to choose the correct normalization of $\Psi$ depending on the underlying representations.

On the other hand, Langlands and Shelstad defined a transfer factor $\Delta_\mathrm{LS}$ based on the theory of endoscopy. In the case when $G$ and $H$ are the groups defined above, we can show that $\Delta_\mathrm{LS}$ is equal to $\Delta_\mathrm{HH}$ again up to a constant. To get a canonical normalization of $\Delta_\mathrm{LS}$, we make use of the quasi-split property of $G$. We fix a Borel subgroup $B$ of $G$ defined over $F$ and call it the standard one for instance. We then normalize $\Delta_\mathrm{LS}$ as in section 5.3 of \cite{KS} in the sense that it depends only on the datum, called a standard Whittaker datum, related to the standard Borel $B$ and does not depend on other auxiliary data. The (rough form of the) sum
\begin{equation}\label{rough spectral transfer char}
 \Delta_\mathrm{LS}(\gamma)\sum_{g\in \Gamma_{K/F}}\Theta_{\rho^g}(\gamma)
\end{equation}
should also determine $\pi$ and differs from (\ref{rough introduce automorphic ind char}) by a constant.

We give some rough sketches of the two transfer factors. The Henniart-Herb transfer factor $\Delta_\mathrm{HH}$ depends on the discriminant of the element $\gamma$. More precisely, if $\gamma$ lies in $E^\times$ as an elliptic torus of $H(F)$ and is regular in $G(F)$, then roughly speaking $\Delta_\mathrm{HH}(\gamma)$ depends on the discriminant
\begin{equation}\label{intro discriminant}
  \prod_{\begin{smallmatrix}
  \Gamma_F/\Gamma_E=\{g_1,\dots,g_n\}\\ i<j
\end{smallmatrix}}({}^{g_i}\gamma-{}^{g_j}\gamma).
\end{equation}
We notice the symmetry imposed by the ordering of the cosets in the quotient $\Gamma_F/\Gamma_E$ (see the precise definition of such symmetry in section \ref{section root system}). The Langlands-Shelstad transfer factor $\Delta_\mathrm{LS}$ is a product of several factors, some of which depend on the root system $\Phi=\Phi(G,T)$ in $G$ of the elliptic torus $T$ whose $F$-points is $E^\times$. Each root in $\Phi$ is of the form \begin{equation}\label{intro root}
\left[\begin{smallmatrix}
  g_i \\ g_j
\end{smallmatrix}\right]:\gamma\mapsto {}^{g_i}\gamma({}^{g_j}\gamma)^{-1},\, g_i,g_j\in \Gamma_F/\Gamma_E,\, g_i\neq g_j.
\end{equation}
The similarity between (\ref{intro discriminant}) and (\ref{intro root}) gives the relation between the two transfer factors, see Proposition \ref{delta-2-equals-delta-II-III} for instance. We would make use of the interplay between the Galois groups and the root system: there is an easy bijection in Proposition \ref{orbit of roots as double coset} between the Galois orbits of the root system and the double cosets of the Galois group.

To compare (\ref{rough introduce automorphic ind char}) and (\ref{rough spectral transfer char}) we apply the theory of Whittaker model. The intertwining operator $\Psi$ is actually normalized by a Whittaker datum depending on the internal structure of the supercuspidal $\pi$. If $\pi$ comes from an admissible character $\xi$ via the bijection $\Pi_n$ in (\ref{bijection of P and A}), then the internal structure of $\pi$ comes from the corresponding internal structure, called the jump data, of $\xi$. We now recall that any two Whittaker data are conjugate under $G(F)$. Let $x\in G(F)$ be the element that conjugates the standard Whittaker datum to the one related on $\xi$ and denote $\kappa(x)=\kappa\circ\det(x)$. We can then show, using Proposition \ref{main prop on constants}, that \begin{equation}\label{rough normalized automorphic ind char}
  \kappa(x)\Delta_\mathrm{HH}(\gamma)\sum_{g\in \Gamma_{K/F}}\Theta_{\rho^g}(\gamma)
\end{equation}
is equal to the twisted character $\Theta^\kappa_\pi$ when the intertwining operator $\Psi$ is normalized by the standard Whittaker datum. The factor $\kappa(x)\Delta_\mathrm{HH}(\gamma)$ is then independent of the representation $\pi$ (or $\rho$). The main task of the first result, which consists of section \ref{label first cases}-\ref{section case unram}, is to compute $\kappa(x)$ in the case when $K/F$ is one of the extensions in (\ref{intro unram})-(\ref{intro odd degree totally ram}) and is moreover cyclic. With the values of $\kappa(x)$ we can show that
 $$\Delta_\mathrm{LS}=\kappa(x)\Delta_\mathrm{HH}$$
 when $K/F$ is specified as above. This implies our main result, Theorem \ref{auto-ind-char-HH-LS-equal}, that the automorphic induction character (\ref{rough normalized automorphic ind char}) and the spectral transfer character (\ref{rough spectral transfer char}) are equal when they are both normalized by the same Whittaker datum.

We remark that the result in the preprint of Hiraga and Ichino \cite{HI} implies that (\ref{rough normalized automorphic ind char}) and (\ref{rough spectral transfer char}) are equal for arbitrary cyclic extension $K/F$ not necessarily tamely ramified. Their method is to make use of a global argument to reduce to the case when $H$ is an elliptic torus and splits over a tamely ramified cyclic extension of $F$. Our method, which is to directly compute the transfer factors and the normalization constants, is completely local if we take the existence of automorphic induction for granted.

\subsection{The second result}
 The second main result, Theorem \ref{chi-data factor of BH-rectifier}, is to express the essentially tame local Langlands correspondence and Bushnell-Henniart rectifiers in terms of Langlands-Shelstad transfer principle. Recall in section (2.5)-(2.6) of \cite{LS} that they introduced a collection of characters, called $\chi$-data, to construct admissible embeddings of a maximal L-torus into the L-group of a connected reductive group over $F$. The construction applies to the case when the group and its maximal torus are $(G,T)=(\mathrm{GL}_n,\mathrm{Res}_{E/F}\mathbb{G}_m)$. Here $T$ is assumed to be contained in $G$ by a chosen embedding. Write $ \Phi= \Phi(G,T)$ the root system of $T$ in $G$. Then the set of $\chi$-data consists of characters $$\{\chi_{\lambda}\}_{\lambda\in W_F\backslash \Phi}.$$ Here $\lambda$ runs through a suitable subset, denoted by $ W_F\backslash \Phi$ at this moment, of representatives of the $W_F$-orbits of $\Phi$ such that the character $\chi_\lambda$ is defined on the multiplicative group of a field extension $E_\lambda$ containing $E$ for each $\lambda\in W_F\backslash \Phi$. We will recall briefly, in section \ref{section langlands shelstad chi data}, how to construct an admissible embedding $$\{\chi_{\lambda}\}=\{\chi_{\lambda}\}_{\lambda\in W_F\backslash \Phi}\mapsto(\chi_{\{\chi_{\lambda}\}}:{}^LT\rightarrow {}^LG)$$ from a set of $\chi$-data.

 The main result is to choose a canonical collection of tamely ramified $\chi$-data, $$\{\chi_{\lambda,\xi}\}_{\lambda\in W_F\backslash \Phi},$$ for each admissible character $\xi$ of $T(F)=E^\times$, such that the product of their restrictions on $E^\times$ is the rectifier of $\xi$. More precisely, we will prove in Theorem \ref{chi-data factor of BH-rectifier} that the rectifier ${}_F\mu_\xi$ of $\xi$ has a factorization of the form $${}_F\mu_\xi=\prod_{\lambda\in W_F\backslash \Phi}\chi_{\lambda,\xi}|_{E^\times}.$$
 We can compare this product to the original factorization (\ref{coarse factorization of rectifier}) provided by Bushnell and Henniart. The new factorization in Theorem \ref{chi-data factor of BH-rectifier} is a finer one, in the sense that we have a factorization for the $\nu$-rectifier
 $${}_{L/K}\mu_\xi=\prod_{\begin{smallmatrix}
  \lambda\in W_F\backslash \Phi\\ \lambda|_K\equiv1,\,\lambda|_L\neq 1
\end{smallmatrix}}\chi_{\lambda,\xi}|_{E^\times}$$
 for every intermediate field extension $L/K$ appearing on the right side of (\ref{coarse factorization of rectifier}). This leads to expressing a $\nu$-rectifier by certain transfer factor $\Delta_\mathrm{III_2}$, a fact stated as Corollary \ref{rectifier as transfer factor}. 

We can therefore interpret the essentially tame local Langlands correspondence using admissible embeddings as follows. We first recall the local Langlands correspondence for the torus $T$. This is the natural isomorphism $$\mathrm{Hom}(E^\times,\mathbb{C}^\times)\cong H^1(W_F,\hat{T}).$$ Let $\tilde{\xi}:W_F\rightarrow{}^LT $ be a 1-cocycle whose class corresponds to the admissible character $\xi:E^\times\rightarrow \mathbb{C}^\times$. Moreover, given the $\chi$-data $\{\chi_{\lambda,\xi}\}_{\lambda\in W_F\backslash \Phi}$ as in Theorem \ref{chi-data factor of BH-rectifier}, we consider its `inverse' collection $\{\chi^{-1}_{\lambda,\xi}\}=\{\chi^{-1}_{\lambda,\xi}\}_{\lambda\in W_F\backslash \Phi}$, which are also $\chi$-data by definition.

\begin{thm}\label{introduce main result ETLLC as adm-emb}
 The natural projection of
 $$\chi_{\{\chi^{-1}_{\lambda,\xi}\}}\circ\tilde{\xi}:W_F \rightarrow {}^LT\rightarrow {}^LG$$
 onto $\mathrm{GL}_n(\mathbb{C})$ is isomorphic to $\sigma_{{}_F\mu_\xi^{-1}\xi}=\mathrm{Ind}_{W_E}^{W_F}({}_F\mu_\xi^{-1}\xi)$ as a representation of $W_F$. This is the Langlands parameter of the supercuspidal $\pi_\xi$.
 \qed\end{thm}

The author would like to remark that the results of this sub-section are motivated from the original thesis problem assigned by James Arthur, based on his communications with Robert Kottwitz some years ago.

\section{Outline of the article}

The first three chapters contain mainly classical results. Chapter \ref{chapter basic} provides the rudiment objects for the whole theory, although we immediately bring out a few technical facts which will be used frequently in later chapters. In chapter \ref{chapter ETLLC}, we state the main statement of essentially tame local Langlands correspondence and give the notion of automorphic induction. The rectifiers of Bushnell and Henniart then appear naturally. In chapter \ref{endoscopy}, we recall the theory of endoscopy by Kottwitz, Langlands, and Shelstad in brief detail. We would compute explicitly those transfer factors of our concern and compare them to those provided by Henniart and Herb.

Before going into non-classical matters, we give in chapter \ref{chapter interlude} an interlude which outlines the main results and sketches some technical details of the proofs. Then in chapter \ref{chapter finite cymplectic module} and \ref{chapter Tame supercuspidal representations} we construct the essentially tame supercuspidal representations from admissible characters and study certain finite modules arising from the constructions. At the end of chapter \ref{chapter Tame supercuspidal representations}, we give the explicit values of the rectifiers in terms of certain invariants, called the t-factors, of the finite modules. In the last two chapters \ref{chapter comparing character} and \ref{chapter rectifier and transfer}, we state and prove the main results described in section \ref{section main results of author}.

\chapter{Basic setup}\label{chapter basic}

In this chapter we setup the basic objects. Throughout we let
\begin{enumerate}[(i)]
  \item $F$ be a non-Archimedean local field,
  \item $\bar{F}$ be an algebraic closure of $F$,
  \item $\mathbf{k}_F$ be the residue field of $F$, with $q$ elements and of characteristic $p$,
  \item $\Gamma_F$ be the Galois group of $F$, and
  \item $G$ be the general linear group $\mathrm{GL}_n$ as a reductive group over $F$.
\end{enumerate}
In section \ref{section root system}, we describe elliptic tori in $G(F)$, their root systems, and the actions of the Galois group on them. For explicit computations we need to identify the Galois orbits of the root system with the non-trivial double cosets of Galois groups. We then state in section \ref{section Galois groups} certain parity results concerning the number of double cosets of Galois groups, which are essential to the main results of this article. In section \ref{section L-groups}, we recall the dual group and the L-group of a reductive group, which are the main objects of the `Galois side' of the local Langlands correspondence.

\section{Root systems}\label{section root system}

 Given a field extension $E/F$ of degree $n$, we let $T$ be the induced torus $$T=\text{Res}_{E/F} \mathbb{G}_m.$$ By identifying $E$ with an $n$-dimensional $F$-vector space, we can embed $T$ into $G$ as an elliptic maximal torus. We can express the Galois action on $T$ as follows. By identifying our maximal torus $T$ as the group of functions on $\Gamma_F/\Gamma_E$ with values in $\mathbb{G}_m$, we can write the $\Gamma_F$-action on $T$ as
 $$({}^gf)(x\Gamma_E) = {}^g(f(g^{-1}x\Gamma_E))\text{, for all }f
\in  T,\,g,x\in \Gamma_F.$$
Hence the $F$-point of $T$ is $$T^{\Gamma_F} = \{ f : \Gamma_F/\Gamma_E \rightarrow \bar{F}^\times | {}^g(f(g^{-1}x\Gamma_E)) = f(x\Gamma_E)\text{, for all } g, x \in \Gamma_F\}$$
and is isomorphic to $E^\times$ by evaluating at the trivial coset
\begin{equation}\label{evaluation at the trivial coset}
T^{\Gamma_F} \rightarrow E^\times,\,f\mapsto f(\Gamma_E).
\end{equation}
Notice that constant functions correspond to elements in $F^\times$.

Let $\Phi=\Phi(G,T)$ be the root system of $T$ in $G$. It generates a free abelian group $X^*(T)$ of rank $n$, called the root lattice, which can be expressed as
\begin{equation*}
\begin{split}
&X^*(T) = \text{Hom}_{F\text{-reg}}(T, \mathbb{G}_m)
\\
&= \{ \phi: T \rightarrow \mathbb{G}_m\text{ regular over }F,\, \phi(fh) = \phi(f) \phi(h)\text{ for all }f,h\in T\}.
\end{split}
\end{equation*}
 We are going to assign a basis for $X^*(T)$ and express the $\Gamma_F$-action on $\Phi$ and $X^*(T)$ explicitly. Take a set of coset representatives $\{g_1,\dots,g_n\}$ of $\Gamma_F/ \Gamma_E$ with $g_1=1$. We define a map $g(\cdot,\cdot):\Gamma_F\times \{g_1,\dots,g_n\}\rightarrow \{g_1,\dots,g_n\}$ by requiring that
\begin{equation*}
gg_i\Gamma_E=g(g,g_i)\Gamma_E.
\end{equation*}
The characters $\{\phi_{g_1},\dots,\phi_{g_n}\}$ where $$\phi_{g_i}\in X^*(T),\,\phi_{g_i}(f)=f(g_i\Gamma_E)$$ form a basis for $X^*(T)$. Hence we can express a character in $X^*(T)$ as $$\phi_{g_1}^{m_1}\cdots\phi_{g_n}^{m_n}(f) = f(g_1\Gamma_E)^{m_1} \cdots f(g_n\Gamma_E)^{m_n}$$ and so the $\Gamma_F$-action on $X^*(T)$ as
\begin{equation*}
\begin{split}
({}^g (\phi_{g_1}^{m_1}\cdots\phi_{g_n}^{m_n}))(f) &= {}^g(\phi_{g_1}^{m_1}\cdots\phi_{g_n}^{m_n}({}^{g^{-1}}f))
\\
&= {}^g \left[ ({}^{g^{-1}}f)(g_1\Gamma_E)^{m_1} \cdots ({}^{g^{-1}}f)(g_n\Gamma_E)^{m_n} \right]
\\
&= f(g(g,g_1)\Gamma_E)^{m_1} \cdots f(g(g,g_n)\Gamma_E)^{m_n}.
\end{split}
\end{equation*}
We can define an action of $\Gamma_F$ on the dual torus $\hat{T}=X^*(T) \otimes_{\mathbb{Z}} \mathbb{C}^\times$ by extending the action on $\mathbb{C}^\times$ trivially. In terms of coordinates, we can write each element in $\hat{T}$ as $$\sum^n_{i=1} \phi_{g_i} \otimes z_{g_i}\text{ for some }z_{g_i}\in \mathbb{C}^\times.$$ Therefore we have the $\Gamma_F$-action
\begin{equation}\label{weil group action on dual torus}
{}^g \left( \sum^n_{i=1} \phi_{g_i} \otimes z_{g_i}\right)=\sum^n_{i=1} \phi_{g(g,g_i)} \otimes z_{g_i} = \sum^n_{j=1} \phi_{g_j} \otimes z_{g(g^{-1},g_i)}.
\end{equation}

We denote the roots in $\Phi$ by
$$\left[ \begin{smallmatrix}
g_i \\
g_j
\end{smallmatrix} \right]=
\phi_{g_i}\phi_{g_j}^{-1}\text{ for }g_i\neq g_j.$$
By the evaluation isomorphism \eqref{evaluation at the trivial coset}, we have that
$$\left[ \begin{smallmatrix}
g_i \\
g_j
\end{smallmatrix} \right](t)={}^{g_i}t({}^{g_j}t)^{-1}\text{, for all }t
\in E^\times.$$The $\Gamma_F$-action on $\Phi$ is therefore given by
\begin{equation*}
\left( g \cdot \left[ \begin{smallmatrix}
 g_i\\
 g_j
\end{smallmatrix}
\right] \right)(f) = f(gg_i\Gamma_E )f(gg_j\Gamma_E)^{-1} = \left[ \begin{smallmatrix}
g(g,g_i)\\
g(g,g_j)
\end{smallmatrix}\right](f).
\end{equation*}
Notice that such action factors through the action of the Weyl group $\Omega(G,T)$ of $T$. It is clear that the $\Gamma_F$-orbit of a root contains an element of the form
\begin{equation*}
\left[\begin{smallmatrix}1\\g\end{smallmatrix}\right]\text{, for some }g\in \{g_2,\dots,g_n\}.
\end{equation*}

For each root $\lambda \in \Phi$, we denote the stabilizers $\{ g \in \Gamma_F | g \lambda = \lambda\}$ and $\{ g \in \Gamma_F | g\lambda = \pm \lambda \}$ by $\Gamma_{\lambda} $ and $\Gamma_{\pm \lambda}$ respectively and the fixed fields $\bar{F}^{\Gamma_{\lambda}}$ and $ \bar{F}^{\Gamma_{\pm \lambda}}$ by $E_{\lambda}$ and $ E_{\pm \lambda}$ respectively. In general, $E_{\lambda}$ is a field extension of some conjugate of $E$. We call a root  $\lambda$ \emph{symmetric} if $|E_{\lambda}/E_{\pm\lambda}| = 2$, and \emph{asymmetric} otherwise. By definition this symmetry is preserved by the $\Gamma_F$-action. Write $[\lambda]$ be the $\Gamma_F$-orbit of $\lambda$. Let
\begin{enumerate}[(i)]
\item $ \Gamma_F\backslash\Phi_{\mathrm{sym}}$ be the set of $\Gamma_F$-orbits of symmetric roots,
\item $ \Gamma_F\backslash\Phi_\mathrm{asym}$ be the set of $\Gamma_F$-orbits of asymmetric roots, and
\item $ \Gamma_F\backslash\Phi_\mathrm{asym/\pm}$ be the set of equivalence classes of asymmetric $\Gamma_F$-orbits by identifying $[\lambda]$ and $[-\lambda]$.
\end{enumerate}
We denote by $\mathcal{R}_\mathrm{sym}, \mathcal{R}_\mathrm{asym}$ and $ \mathcal{R}_{\mathrm{asym}/ \pm}$ certain choices of sets of representatives in $\Phi$ of the above orbits and their equivalence classes respectively.

\begin{prop}\label{orbit of roots as double coset}
The set $\Gamma_F\backslash\Phi$ of $\Gamma_F$-orbits of the root system $\Phi$ is bijective to the collection of non-trivial double cosets in $\Gamma_E\backslash \Gamma_F/\Gamma_E$, by $$\Gamma_F\backslash\Phi\rightarrow (\Gamma_E\backslash \Gamma_F/\Gamma_E)-\{\Gamma_E\},\,[\lambda]=\left[\left[\begin{smallmatrix}1\\g\end{smallmatrix}\right]\right]\mapsto \Gamma_Eg\Gamma_E.$$
\end{prop}
\begin{proof}
The set of roots $\Phi$ can be identified with the subset of off-diagonal elements in $\Gamma_F/\Gamma_E \times \Gamma_F/\Gamma_E$ with $\Gamma_F$-action by ${}^g (g_1\Gamma_E, g_2\Gamma_E) = (gg_1\Gamma_E, gg_2\Gamma_E)$. By elementary group theory, we know that the orbits are bijective to the non-trivial double cosets in $\Gamma_E\backslash \Gamma_F/\Gamma_E$.
\end{proof}

We denote by $(\Gamma_E\backslash \Gamma_F/\Gamma_E)'$ the collection of non-trivial double cosets, and $[g]$ the double coset $\Gamma_Eg\Gamma_E$. We call $g\in \Gamma_F$ \emph{symmetric} if $[g]= [g^{-1}]$, and \emph{asymmetric} otherwise. Clearly such symmetry descends to an analogous property on $(\Gamma_E\backslash \Gamma_F/\Gamma_E)'$. By Proposition \ref{orbit of roots as double coset}, the symmetry of $(\Gamma_E\backslash \Gamma_F/\Gamma_E)'$ is equivalent to the symmetry of $\Gamma_F\backslash \Phi$. Let \begin{enumerate}[(i)]
\item $(\Gamma_E\backslash \Gamma_F/\Gamma_E)_\mathrm{sym}$ be the set of symmetric non-trivial double cosets,
\item $(\Gamma_E\backslash \Gamma_F/\Gamma_E)_\mathrm{asym}$ be the set of asymmetric non-trivial double cosets, and
\item $(\Gamma_E\backslash \Gamma_F/\Gamma_E)_{\mathrm{asym}/\pm}$ be the set of equivalence classes of $(\Gamma_E\backslash \Gamma_F/\Gamma_E)_\mathrm{asym}$ by identifying $[g]$ with $[g^{-1}]$.
\end{enumerate}
We denote by $\mathcal{D}_\mathrm{sym}, \mathcal{D}_\mathrm{asym}$ and $ \mathcal{D}_{\mathrm{asym}/ \pm}$ certain choices of sets of representatives in $\Gamma_F/\Gamma_E$ of the above subsets in $(\Gamma_E\backslash \Gamma_F/\Gamma_E)'$ and their equivalence classes respectively.

We hence observe, by the identification in Proposition \ref{orbit of roots as double coset}, that we can choose a collection $\mathcal{R} = \mathcal{R}_\mathrm{sym} \bigsqcup \mathcal{R}_{\mathrm{asym}}$ such that the field $E_\lambda$ is an extension of $E$ for all $\lambda \in \mathcal{R}$. More precisely, if $\lambda$ corresponds to $[g]$, then $E_\lambda = {}^g E E$.

\section{Galois groups}\label{section Galois groups}

Given a non-Archimedean local field $F$, its multiplicative group $F^\times$ decomposes into product of subgroups $$\left<\varpi_F\right>\times\mu_F\times U_F^1 .$$ They are namely the group generated by a prime element, the group of roots of unity, and the 1-unit group. We may identify $\mu_F$ with $\mathbf{k}_F^\times$ in the canonical way. Let $E/F$ be a field extension of degree $n$, and denote its ramification index by $e$ and its residue degree by $f$. In most of the article, we assume that $E/F$ is tamely ramified, which means that $p\nmid e$. By \cite{Lang-ANT} II.\S5, we can always assume that our choices of $\varpi_E$ and $\varpi_F$ satisfy
 \begin{equation}\label{e-th power of prime is also prime}
 \varpi_E^e=\zeta_{E/F}\varpi_F\text{, for some }\zeta_{E/F}\in \mu_E.
 \end{equation}

Let $L$ be the Galois closure of $E/F$. Hence $L/E$ is unramified and $L/F$ is a tamely ramified extension. With the choices of $\varpi_F$ and $\varpi_E$ as in (\ref{e-th power of prime is also prime}), we define the following $F$-operators on $L$.
\begin{enumerate}[(i)]
\item $\phi: \zeta \mapsto \zeta^q \text{, for all }\zeta \in \mu_L\text{, and }\phi:\varpi_E \mapsto \zeta_\phi\varpi_E$.
\item $\sigma : \zeta \mapsto \zeta\text{, for all }\zeta \in \mu_L\text{, and }\sigma:\varpi_E \mapsto \zeta_e \varpi_E.$
\end{enumerate}
Here $\zeta_\phi$ lies in $\mu_E$ satisfying $(\zeta_\phi\varpi_E)^e=\zeta_{E/F}^q\varpi_F$ and $\zeta_e$ is a choice of a primitive $e$th root of unity in $\bar{F}^\times$. More generally, we write $^{\phi^i}\varpi_E = \zeta_{\phi^i}\varpi_E$ such that $\zeta_{\phi^i}=\zeta_\phi^{1+q+\cdots+q^{i-1}}$ is an $e$th root of $\zeta_{E/F}^{q^i-1}$. Notice that we have an action of $\phi$ on $\sigma$ by
$${}^\phi \sigma = \phi \circ \sigma \circ \phi^{-1}= \sigma^q.$$ Therefore we can write our Galois groups as
\begin{equation}\label{galois groups}
\Gamma_{L/F} = \langle \sigma \rangle \rtimes \langle \phi \rangle \text{ and } \Gamma_{L/E} = \langle \phi^f \rangle \subseteq \langle \phi \rangle.
\end{equation}

\begin{prop} \label{explicit expression of double coset}
\begin{enumerate}[(i)]
\item We can choose
$\{\sigma^k\phi^i|k=0,\dots,e-1,\,i=0,\dots,f-1\}$
as coset representatives for the quotient $\Gamma_{E/F}= \Gamma_F/\Gamma_E$. \label{explicit expression of double coset galois set}
\item Let $\phi^f \backslash \langle \sigma \rangle$ be the set of orbits of $\langle \sigma \rangle$ under the action ${}^{\phi^f}\sigma = \sigma^{q^f}$, then the set of double cosets $\Gamma_E\backslash \Gamma_F/\Gamma_E$ is bijective to the set $(\phi^f \backslash \langle \sigma \rangle) \times \langle \phi \rangle.$ \label{explicit expression of double coset double coset}
\end{enumerate}
\end{prop}
\begin{proof}
In general, if we have an abelian group $B$ acting on a group $A$ as automorphisms and a subgroup $C$ of $B$, then the canonical maps
\begin{equation*}
 (A \rtimes B) /(1 \times C) \rightarrow A \times (B/C),\, (a,b)(1 \times C)\mapsto (a,bC),
\end{equation*}and \begin{equation*}
(1 \times C) \backslash (A \rtimes B) / (1 \times C) \rightarrow (C \backslash A) \times (B/C),\,(1 \times C) (a,b) (1 \times C)\mapsto (Ca,bC),
\end{equation*}
are bijective. We take $A \rtimes B=\Gamma_{L/F}$ and $C=\Gamma_{L/E}$ as in (\ref{galois groups}).
\end{proof}

 With such identification we can write down the symmetric double cosets explicitly.
\begin{prop}\label{properties of symmetric [g]}
  The double coset $[g]=[\sigma^{k}\phi^i] $ is symmetric if and only if
\begin{enumerate}[(i)]
  \item $i=0$ or, when $f$ is even, $i=f/2$, and
\item $e$ divides $(q^{ft}+1)k$ when $i=0$ and divides $(q^{f(2t+1)/2}+1)k$ when $i=f/2$, for some $t=0, \dots, |L/E|-1.$
\end{enumerate}
\end{prop}
\begin{proof}
  Since $\phi$ acts as $\sigma \mapsto \sigma^q$, we can show that the inverse of $\sigma^k\phi^i$ in $\Gamma_{L/F}$ is $\sigma^{-k\bar{q}^i} \phi^{-i}$ where $\bar{q}$ is the multiplicative inverse of $q$ in $(\mathbb{Z} / e)^\times$. We then check that the double cosets $[ \sigma^k \phi^i ]$ and $[(\sigma^k \phi^i)^{-1} ]$ are equal if and only if $i \equiv -i \mod f$  and  $q^{ft}k \equiv -k\bar{q}^i\mod e$ for some $t$. This implies the assertion by simple calculation.
\end{proof}
 By slightly abusing notation, we call those symmetric $[\sigma^k]$ ramified and those symmetric $[\sigma^k \phi^{f/2}]$ unramified.

Our main results rely on knowing the parity of the number of double cosets $\Gamma_E\backslash \Gamma_F/  \Gamma_E$ when $E/F$ is totally ramified. Suppose that $|E/F|=e$ and $\#\mathbf{k}_F=q$. By Proposition \ref{explicit expression of double coset}, this parity is the same as the parity of the number of $q$-orbits of $\mathbb{Z}/e$ under the multiplicative action $x\mapsto qx$. If we partition the set $\mathbb{Z}/e$ according to multiples as $$\mathbb{Z}/e = \bigsqcup_{d\mid e}(e/d)\left( \mathbb{Z}/d \right)^\times,$$ then, because $\gcd(e,q) = 1$, each subset consists of union of $q$-orbits. For each divisor $d$ of $e$, let $\mathrm{ord}(q,d)$ be the multiplicative order of $q$ in $(\mathbb{Z}/d)^\times$, and $\phi(d)$ be the order of the group $(\mathbb{Z}/d)^\times$.
\begin{lemma}  \label{number of double cosets}
\begin{enumerate}[(i)]
\item The number of double cosets in $\Gamma_E\backslash \Gamma_F/  \Gamma_E$ equals $$\sum_{d|e}\phi(d)/\mathrm{ord}(q,d).$$ \label{number of double cosets general}
\item If $d=1$ or 2, then $\phi(d)/\mathrm{ord}(q,d)=1$. If $d\geq 3$ and $q$ is a square, then $\phi(d)/\mathrm{ord}(q,d)$ is even. \label{number of double cosets cases}
\end{enumerate}
\end{lemma}
\begin{proof}
  Assertion (\ref{number of double cosets general}) is clear. For $d=1 \text{ or } 2$, (\ref{number of double cosets cases}) is also clear. For $d \geq 3$, let $\langle q \rangle$ be the cyclic subgroup in $\left( \mathbb{Z}/d \right)^\times$ generated by $q$. Suppose that $q=r^2$. If $\langle q \rangle$ is a proper subgroup of $\langle r \rangle$, then $\phi(d)/\mathrm{ord}(q,d)$ is a multiple of the index $|\langle r \rangle/ \langle q \rangle| = 2$. If $\langle q \rangle = \langle r \rangle$, then $q^k \equiv r \mod d$ for some $k$. The order $\mathrm{ord}(r,d)$ is then odd, and so is $\mathrm{ord}(q,d)$. But $\phi(d)$ is always even. This proves assertion (\ref{number of double cosets cases}).
\end{proof}

 \begin{prop}\label{parity of double coset with i=f/2}
The parity of the cardinality of the set $$(\Gamma_E\backslash \Gamma_F/  \Gamma_E)_{\mathrm{sym,unram}}:=\{[\sigma^k\phi^{f/2}]\in (\Gamma_E\backslash \Gamma_F/  \Gamma_E)_\mathrm{sym}\}$$ is equal to that of ${e(f-1)}$.
\end{prop}
\begin{proof}If $f$ is odd, then the set is empty and so the statement is true. Now we assume that $f$ is even. Since $\phi^{f/2}$ normalizes $\Gamma_E$, we have indeed a bijection
\begin{equation*}
\{[\sigma^k\phi^{f/2}] \in \Gamma_E\backslash \Gamma_F/  \Gamma_E \} \rightarrow  \Gamma_E\backslash \Gamma_K /\Gamma_E,\,[\sigma^k\phi^{f/2}]\mapsto [\sigma^k].
\end{equation*}
Since those asymmetric double cosets pair up, it suffices to show that the parity of $\#( \Gamma_E\backslash \Gamma_K /\Gamma_E )$ is the same as the parity of $e$. To this end we apply Lemma \ref{number of double cosets}. Here we have $q^f$, which is a square, in place of $q$ in Lemma \ref{number of double cosets}. If $e$ is odd, then the parity is $$\phi(1)/\mathrm{ord}(q,1)+\sum_{d|e,\,d\geq 3}\phi(d)/\mathrm{ord}(q,d)=1+\sum\text{ even }=\text{ odd}.$$
If $e$ is even, then the parity is $$\phi(1)/\mathrm{ord}(q,1)+\phi(2)/\mathrm{ord}(q,2)+\sum_{d|e,\,d\geq 3}\phi(d)/\mathrm{ord}(q,d)=1+1+\sum\text{ even }=\text{ even}.$$ Therefore we have proved the statement.
\end{proof}

\section{L-groups}\label{section L-groups}

We first recall some basic facts in local class field theory. They can be found in section 1 of \cite{Tate-NTB}. Given a non-Archimedean local field $F$, we denote the Galois group of the maximal unramified extension of $F$ in $\bar{F}$ by $I_F$ and call it the tame inertia group of $F$. We have an exact sequence $$1\rightarrow I_F\rightarrow\Gamma_F\rightarrow \Gamma_{\mathbf{k}_F}\rightarrow1.$$
The Galois group $\Gamma_{\mathbf{k}_F}$ is isomorphic to $$\hat{\mathbb{Z}}:=\lim_{\substack{\longleftarrow\\n}}\mathbb{Z}/n,$$ with the Frobenius element identified with a topological generator of $\hat{\mathbb{Z}}$. The Weil group $W_F$ of $F$ is defined to be the pre-image of the subgroup $\mathbb{Z}$ of $\hat{\mathbb{Z}}$ in $\Gamma_F$. One of the major results in local class field theory is the isomorphism of (topological) groups
$$W_F^\mathrm{ab}:=W_F/[W_F,W_F]\cong F^\times.$$
Hence the characters of $W_F$ and those of $F^\times$ are bijectively correspond.

We then provide the dual groups and the L-groups of our concern, and refer the general definition to section 2 of \cite{Borel-autom-L}. For $G=\mathrm{GL}_n$, we denote
\begin{equation*}
\hat{G} = \mathrm{GL}_n(\mathbb{C})\text{ and } {}^L \!G = \mathrm{GL}_n(\mathbb{C}) \times W_F,
\end{equation*}
namely the dual-group and the $L$-group of $G$. If $T$ is the induced torus $\mathrm{Res}_{E/F}\mathbb{G}_m$, then we write
\begin{equation*}
\hat{T} =  \mathrm{Ind}_{E/F}\mathbb{C}^\times\text{ and }{}^L \!T = \hat{T} \rtimes W_F
\end{equation*}
as its dual-group and $L$-group. Here the action of $W_F$ on $\hat{T}$ is the one factoring through $\Gamma_F$, which is given by (\ref{weil group action on dual torus}).

One remark is that there is no difference between Galois groups and Weil groups when we are dealing with finite extensions. For example, if $E/F$ is a finite extension, then $W_F/W_E\cong \Gamma_F/\Gamma_E$ and so $W_E\backslash W_F/W_E\cong \Gamma_E\backslash \Gamma_F/\Gamma_E$. Also most of our actions of Galois groups factor through actions of subgroups of finite extensions. In such cases we shall replace the actions of the Galois groups by those of the Weil groups without further notice.

\chapter{Essentially tame local Langlands correspondence}\label{chapter ETLLC}

As part of the local Langlands correspondence for $\mathrm{GL}_n$ proved in \cite{HT}, \cite{Hen-simple}, \cite{LRS}, we know that the supercuspidal spectrum $\mathcal{A}^0_n(F)$ of $G(F)$ is bijective to the collection $\mathcal{G}^0_n(F)$ of complex smooth $n$-dimensional irreducible representations of the Weil group $W_F$ of $F$. Such bijection is characterized by a number of canonical properties governed by both representation theory and number theory.

In this chapter, we recall the main results in \cite{BH-ET1}, \cite{BH-ET2}, \cite{BH-ET3} to describe the above correspondence when restricted to the essentially tame sub-collections. We first recall a collection $P_n(F)$ consisting of the equivalence classes of admissible characters in section \ref{section admissible characters}. We then provide the definitions of the two essential-tameness and describe the properties of this restricted bijection, the essentially tame local Langlands correspondence, in section \ref{section The correspondence}. One way to obtain from an admissible character a supercuspidal representation is the operation of automorphic induction \cite{Hen-Herb}, \cite{HL2010}, \cite{HL2011}. This operation corresponds to the usual induction of representations of Weil groups on the `Galois side' according to the functoriality principle. We study this topic in section \ref{section autom ind}. Finally, in section \ref{section rectifiers}, we can reduce the description of the essentially tame correspondence to expressing certain characters called rectifiers, which are the main objects we would study in this thesis.

The last section \ref{the Archi case} serves as an appendix. We recall briefly the Langlands correspondence for relative discrete series in the Archimedean case. If $F=\mathbb{R}$ and $G=\mathrm{GL}_n$, then $G(F)$ possesses relative discrete series representations only if $n=1$ or 2. The case $n=1$ is just local class field theory. We will consider the case when $n=2$.

\section{Admissible characters}\label{section admissible characters}
The notion of admissible characters is described in \cite{Howe}, \cite{Moy}. We recall the definition of these characters in this section, and provide more properties of them later in section \ref{section admissible characters revisited}.

 Let $E/F$ be a tamely ramified finite extension, and let $\xi$ be a character of $E^\times$. We call $\xi$ \emph{admissible} \index{admissible character} over $F$ if it satisfies the following conditions. For every field $K$ between $E/F$,
\begin{enumerate}[(i)]
\item if $\xi = \eta \circ N_{E/K}$ for some character $\eta$ of $K^\times$, then $E = K$;
\item if $\xi|_{U^1_E} = \eta \circ N_{E/K}$ for some character $\eta$ of $U^1_K$, then $E/K$ is unramified.
\end{enumerate}
If $\xi$ is admissible over $F$, then by definition it is admissible over $K$ for every field $K$ between $E/F$.

Let $P(E/F)$ be the set of admissible characters of $E^\times$ over $F$. Two admissible characters $\xi\in P(E/F)$ and $\xi'\in P(E'/F)$ are called $F$-equivalent if there is $g\in \Gamma_F$ such that ${}^gE=E'$ and ${}^g\xi=\xi'$. We denote the $F$-equivalence class of $\xi$ by $(E/F,\xi)$, which is called an admissible pair in \cite{BH-ET1}. Let $P_n(F)$ be the set of admissible pairs $F$. More precisely,
\begin{equation*}
P_n(F)=F\text{-equivalence}\backslash \left(\bigsqcup_EP(E/F)\right),
\end{equation*}
where $E$ goes through tamely ramified extensions over $F$ of degree $n$.

\section{The correspondence}\label{section The correspondence}

 For each Langlands parameter $\sigma\in \mathcal{G}^0_n(F)$, let $f(\sigma)$ be the number of unramified characters $\chi$ of $W_F$ such that $\chi\otimes\sigma\cong\sigma$. We call $\sigma$ \emph{essentially tame} if $p$ does not divide $n/f(\sigma)$. Let $\mathcal{G}^\mathrm{et}_n(F)$ be the set of isomorphism classes of essentially tame irreducible representations of degree $n$. Similarly, for each irreducible supercuspidal $\pi\in \mathcal{A}^0_n(F)$, let $f(\pi)$ be the number of unramified characters $\chi$ of $F^\times$ that $\chi\otimes\pi\cong\pi$. Here $\chi$ is regarded as a representation of $G(F)$ by composing with the determinant map. We call $\pi$ \emph{essentially tame} if $p$ does not divide $n/f(\pi)$. Let $\mathcal{A}^\mathrm{et}_n(F)$ be the set of isomorphism classes of essentially tame irreducible supercuspidal representations. We recall the following theorem in \cite{BH-ET1}.

 \begin{thm}[Essentially tame local Langlands correspondence for $\mathrm{GL}_n$]\label{langlands corresp}
For each positive integer $n$ there exists a unique bijection
\begin{equation*}
\mathcal{L}_n={}_F\mathcal{L}^\mathrm{et}_n:\mathcal{G}^\mathrm{et}_n(F)\rightarrow\mathcal{A}^\mathrm{et}_n(F),\,\sigma\mapsto \pi,
\end{equation*}
which is uniquely characterized by the following properties:
\begin{enumerate}[(i)]
\item $\mathcal{L}_1:\mathcal{G}_1(F)\rightarrow\mathcal{A}_1(F)$ is the local class field theory; \label{local class field theory}
\item $\mathcal{L}_n$ is compatible with induced actions of the automorphisms of the base field $F$ on both sides;
\item $\mathcal{L}_n$ is compatible with contragredients; in other words, $\mathcal{L}_n(\sigma^\vee)=\pi^\vee$;
\item $\mathcal{L}_n(\chi\otimes\sigma)=\mathcal{L}_1(\chi)\otimes \pi$ for $\chi\in \mathcal{G}_1(F)$;
\item $\det(\sigma)=\omega_\pi$ the central character of $\pi$; \label{ETLLC det is central char}
\item the operations of automorphic induction and base change on $\mathcal{A}^\mathrm{et}_n(F)$ correspond to the operations of induction and restriction on $\mathcal{G}^\mathrm{et}_n(F)$ via $\mathcal{L}_n$.\label{point AI and BC}
\end{enumerate}
 \end{thm}
 \begin{proof}
  See Proposition 3.2 and section 3.5 of \cite{BH-ET1}.
 \end{proof}
We remark that property (\ref{point AI and BC}) described above is a very rough version of the precise statements in Proposition 3.2 of \cite{BH-ET1}. Since our study depends heavily on automorphic induction, we will be more specific about this in section \ref{section autom ind}.

Recall that, via local class field theory (i.e., the bijection $\mathcal{L}_1$ in Theorem \ref{langlands corresp}.(\ref{local class field theory}), with $E$ in place of $F$), each character $\xi$ of $E^\times$ corresponds to a character $\mathcal{L}_1(\xi)$ of the Weil group $W_E$. For simplicity we write $\xi$ for $\mathcal{L}_1(\xi)$ throughout the article.
\begin{prop}\label{prop bijection of P and G} The map $$\Sigma_n:P_n(F)\rightarrow\mathcal{G}^\mathrm{et}_n(F),\,(E/F,\xi)\mapsto\sigma_\xi:=\mathrm{Ind}_{W_E}^{W_F}\xi,$$
is a bijection, with $f(\sigma_\xi)=f(E/F)$.
\end{prop}
\begin{proof}
The proof is in the Appendix of \cite{BH-ET1}.
\end{proof}
We usually denote the induced representation $\mathrm{Ind}_{W_E}^{W_F}\xi$ by $\mathrm{Ind}_{E/F}\xi$.

\section{Automorphic induction}\label{section autom ind}

We first recall some facts in the theory of \emph{automorphic induction}. These facts are well-known and can be found in \cite{Hen-Herb}, \cite{BH-ET3}, \cite{HL2011}. Let $K/F$ be a cyclic extension of degree $d$ and $\kappa $ be a character of $F^\times$ of order $d$ and with kernel $N_{K/F}(K^\times)$. In other words, $\kappa$ is a character of the extension $K/F$ by local class field theory. Write $G=\mathrm{GL}_n$ and $H=\mathrm{Res}_{K/F}\mathrm{GL}_m$ where $m=n/d$. By fixing an embedding, we can regard $H(F)=\mathrm{GL}_m(K)$ as a subgroup of $G(F)$. Suppose that $\pi$ is a supercuspidal representation of $G(F)$ and is isomorphic to $\kappa \pi =(\kappa\circ\det)\otimes \pi$, then there exists $\rho\in \mathcal{A}_m^0(K) $, a supercuspidal representation of $H(F)$, such that $\pi$ is automorphically induced by $\rho$. We write $A_{K/F}(\rho) = \pi$.

We can characterize such automorphic induction using character relation as follows. Suppose that $\Psi$ is a non-zero intertwining operator between $\kappa \pi$ and $\pi$, in the sense that
\begin{equation*}
\Psi \circ \kappa \pi (x) = \pi (x) \circ \Psi\text{, for all }x\in G(F).
\end{equation*}
This condition determines $\Psi$ up to a scalar. We then define the $(\kappa, \Psi)$-trace of $\pi$ to be the distribution $$\Theta^{\kappa, \Psi}_\pi(f) = \text{tr}(\Psi \circ \pi(f))\text{, for all }f \in \mathcal{C}^\infty_c(G(F)).$$
By 3.9 Corollary of \cite{Hen-Herb}, following the idea of Harish-Chandra \cite{HC-AID}, we can view $\Theta^{\kappa, \Psi}_\pi$ as a function on the subset $G(F)_{\mathrm{rss}}$ of regular semi-simple elements of $G(F)$. Similarly, and again by \cite{HC-AID}, we can view the distribution $\Theta_\rho$ of the character of $\rho$ as a function on $H(F)_{\mathrm{rss}}$. Let $G(F)_{\mathrm{ell}}$ be the subset of elliptic elements in $G(F)_{\mathrm{rss}}$. The character identity of automorphic induction is
\begin{equation}\label{automorphic-induction-character-relation}
  \Theta^{\kappa, \Psi}_\pi (h) = c(\rho, \kappa, \Psi) \Delta^2(h) \Delta^1(h)^{-1} \sum_{g \in \Gamma_{K/F}} \Theta^g_\rho(h) \text{, for all }h \in H(F) \cap G(F)_{\mathrm{ell}}.
\end{equation}
 Here $\Delta^1$ and $\Delta^2$ are the \emph{transfer factors} defined in 3.2, 3.3 of \cite{Hen-Herb}. We will recall their constructions in section \ref{section Relations of transfer factors}. The constant $c(\rho, \kappa, \Psi)$ is a normalization depending only on $\rho, \kappa$ and $\Psi$, and is called the automorphic induction constant. We will focus on this constant in chapter \ref{chapter comparing character}.

From (\ref{automorphic-induction-character-relation}) we know that all $\rho^g$, where $g$ runs through $\Gamma_{K/F}$, are automorphically induced to $\pi$. Given a supercuspidal $\pi$, the corresponding $\rho$ is $K/F$-regular, which means that all $\rho^g$ are mutually inequivalent. Conversely, if $\rho$ is $K/F$-regular, then we have a supercuspidal $\pi$ uniquely determined by (\ref{automorphic-induction-character-relation}). The relation between $\rho$ and $\pi$ should be independent of $\kappa$.

We explain the meaning of the compatibilities in property (\ref{point AI and BC}) in Theorem \ref{langlands corresp}. The first half asserts that the automorphic induction on $\mathcal{A}_n^{\mathrm{et}}(F)$ is compatible with induction on $\mathcal{G}_n^{\mathrm{et}}(F)$ via $\mathcal{L}_n$. Let $\mathcal{A}_m^0(K)_{K/F\text{-reg}}$ be the set of isomorphism classes of $K/F$-regular supercuspidal $\rho$ of $\mathrm{GL}_m(K)$, and $\mathcal{G}_m^{0}(K)_{K/F\text{-reg}}$ be the set of isomorphism classes of irreducible representations $\tau$ of $W_K$ of degree $m$ and being $K/F$-regular. The last condition means that if we denote by $\tau^g$, with $g\in W_F$, the representation $\tau^g(h)=\tau(ghg^{-1})$, for all $h\in W_K$, then all $\tau^g$ are mutually inequivalent. We have the following commutative diagram
\begin{equation*}
\xymatrixcolsep{5pc}\xymatrix{
\mathcal{G}^0_{m}(K)_{K/F\text{-reg}} \ar[d]^{\mathrm{Ind}_{K/F}} \ar[r]^{\mathcal{L}_m} &\mathcal{A}^0_{m}(K)_{K/F\text{-reg}} \ar[d]^{\mathrm{A}_{K/F}}\\
\mathcal{G}^0_n(F) \ar[r]^{\mathcal{L}_n} &\mathcal{A}^0_n(F).}
\end{equation*}
There is a dual diagram of the above, with induction on the Galois side replaced by restriction. The corresponding operation on the p-adic side is called the base change. For a comprehensive study on this topic, we refer the readers to \cite{AC}, \cite{HL2011} and skip further discussion to avoid introducing more notations. We would only indicate at the place where we make use of base change.

We briefly describe the process in \cite{BH-ET1}, \cite{BH-ET2}, \cite{BH-ET3} of how an admissible character determines an essentially tame supercuspidal representation using automorphic induction. Suppose that $\xi$ is an admissible character of $E^\times$ over $F$. We split the extension $E/F$ into intermediate sub-extensions
\begin{equation}\label{sequence of fields, ramified}
  F=K_{-1}\hookrightarrow K_0\hookrightarrow K_{1}\hookrightarrow \cdots\hookrightarrow K_{l-1}\hookrightarrow K_l\hookrightarrow E=K_{l+1}
\end{equation}
such that
\begin{enumerate}[(I)]
  \item $E/K_l$ is totally ramified of odd degree \cite{BH-ET1},\label{three cases totally-ram odd}
  \item each $K_i/K_{i-1}$, $i=1,\dots,l$, is quadratic totally ramified \cite{BH-ET2}, and \label{three cases totally-ram quad}
  \item $K_0/F$ is unramified \cite{BH-ET3}. \label{three cases unram}
\end{enumerate}
The extensions in (\ref{three cases totally-ram quad}) and (\ref{three cases unram}) are cyclic. Therefore each $A_{K_i/K_{i-1}}$, $i=0,\dots,l$, is defined by automorphic induction. The extension $E/K_l$ in case (\ref{three cases totally-ram odd}) is not necessarily cyclic, but it is so if we base change $K_l$ by an unramified extension. By a trick similar to the proof in 3.5 of \cite{BH-ET1}, we can define a map $$A_{E/K_l}:P(E/F)\hookrightarrow P(E/K_l)\hookrightarrow \mathcal{A}_1(E)_{E/K_l\text{-reg}}\rightarrow \mathcal{A}^{0}_{|E/K_l|}(K_l)$$ (with the first two maps being natural inclusions) such that the diagram
\begin{equation*}
\xymatrixcolsep{5pc}\xymatrix{
\mathcal{G}_1(E)_{E/K_l\text{-reg}} \ar[d]^{\mathrm{Ind}_{E/K_l}}  \ar[r]^{\mathcal{L}_1} &P(E/F) \ar[d]^{A_{E/K_l}}\\
\mathcal{G}^\mathrm{0}_{|E/K_l|}(K_l) \ar[r]^{\mathcal{L}_{|E/K_l|}} &\mathcal{A}^\mathrm{0}_{|E/K_l|}(K_l)}
\end{equation*}
commutes. Using condition (i) in the definition of the admissibility of $\xi$, we can compose the maps $A_{K_i/K_{i-1}}$, $i=0,\dots,l+1$, above and define
 $$A_{E/F}=A_{K_0/F}\circ A_{K_1/K_0}\circ \cdots \circ A_{K_l/K_{l-1}}\circ A_{E/K_l}:P(E/F)\rightarrow \mathcal{A}_n^{0}(F).$$
The `union' of the above maps
$$\bigsqcup_{\begin{smallmatrix}
 E/F\text{ tamely ramified}\\ |E/F|=n
\end{smallmatrix}}A_{E/F}:\bigsqcup_{\begin{smallmatrix}
 E/F\text{ tamely ramified}\\ |E/F|=n
\end{smallmatrix}}P(E/F)\rightarrow \mathcal{A}_n^{0}(F)$$
descends to a map $$A_n=A_n(F):P_n(F)\rightarrow \mathcal{A}_n^{0}(F).$$
\begin{prop}
  The map $A_n$ is injective with image in $\mathcal{A}^\mathrm{et}_n(F)$.
\end{prop}
\begin{proof}
  Indeed we constructed $A_n$ such that it is equal to the composition $$P_n(F) \xrightarrow{\Sigma_n} \mathcal{G}^\mathrm{et}_n(F) \xrightarrow{\mathcal{L}_n} \mathcal{A}^\mathrm{et}_n(F).$$ We refer the details to section 3 of \cite{BH-ET1}.
\end{proof}

\section{Rectifiers}\label{section rectifiers}

The bijection $A_n$ depends on the character relation of automorphic induction and is not quite explicit. Using section 2 of \cite{BH-ET1}, we can construct directly an essentially tame supercuspidal $\pi_\xi$ from an admissible character $\xi$ using solely representation theoretic method. Such supercuspidal representation is a compact induction
\begin{equation}\label{supercusp as compact induction}
  \pi_\xi=\mathrm{cInd}_{\mathbf{J}(\xi)}^{G(F)}\Lambda_\xi
\end{equation} of certain representation $\Lambda_\xi$ on a compact-mod-center subgroup ${\mathbf{J}(\xi)}$ of $G(F)$. This representation $(\mathbf{J}(\xi),\Lambda_\xi)$ depends on the admissible character $\xi$, and is an extension of certain representation called maximal type defined in \cite{BK}.

\begin{prop}\label{prop bijection of P and A} The map $$\Pi_n:P_n(F)\rightarrow\mathcal{A}^\mathrm{et}_n(F),\,(E/F,\xi)\mapsto\pi_\xi,$$
is a bijection, with $f(\pi_\xi)=f(E/F)$.
\end{prop}
\begin{proof}
The proof is summarized in chapter 2 of \cite{BH-ET1}.
\end{proof}

The whole theory comes from the fact that the maps $A_n$ and $\Pi_n$ are unequal in general. In other words, the composition of the bijections in Proposition \ref{prop bijection of P and G}, Theorem \ref{langlands corresp} and the inverse of the one in Proposition \ref{prop bijection of P and A},
\begin{equation*}
\mu: P_n(F) \xrightarrow{\Sigma_n} \mathcal{G}^\mathrm{et}_n(F) \xrightarrow{\mathcal{L}_n} \mathcal{A}^\mathrm{et}_n(F) \xrightarrow{\Pi_n^{-1}} P_n(F),
\end{equation*}
is not the identity map on $P_n(F)$. Bushnell and Henniart proved in \cite{BH-ET3} that, for each admissible character $\xi$ of $E^\times$, there is a unique tamely ramified character ${}_F\mu_\xi$ of $E^\times$, depending on the wild part of $\xi$ (i.e., the restriction $\xi|_{U^1_E}$), such that ${}_F\mu_\xi \cdot \xi$ is also admissible and
\begin{equation*}
\mu(E/F, \xi) = (E/F, {}_F\mu_\xi \cdot \xi).
\end{equation*}
We call the character ${}_F\mu_\xi$ the \emph{rectifier} of $\xi$. In the series \cite{BH-ET1}, \cite{BH-ET2}, \cite{BH-ET3}, we can express ${}_F\mu_\xi$ explicitly, and hence also the correspondence $\mathcal{L}_n$.

The rectifier ${}_F\mu_\xi$ is a product
\begin{equation}\label{product of rectifier}
{}_{F}\mu_\xi=({}_{K_l}\mu_\xi)({}_{K_l/K_{l-1}}\mu_\xi)\cdots({}_{K_1/K_0}\mu_\xi)({}_{K_0/F}\mu_\xi).
\end{equation}
Here the intermediate subfields are those in the sequence (\ref{sequence of fields, ramified}). The first factor ${}_{K_l}\mu_\xi$ is the rectifier by considering $\xi$ as being admissible over $K_l$. This rectifier is computed in \cite{BH-ET1}. The remaining factors are of the form ${}_{K_i/K_{i-1}}\mu_\xi$. These characters, called the $\nu$-rectifiers in this article, are constructed in \cite{BH-ET2} for $i=1,\dots,l$ and in \cite{BH-ET3} for $i=0$. We briefly describe the construction of the rectifiers and the $\nu$-rectifiers by the following inductive process. Suppose that $K/F$ is a sub-extension of $E/F$ of degree $d$ and assume that we have constructed ${}_{K}\mu_\xi$. Suppose that the essentially tame supercuspidal $\pi$
\begin{enumerate}[(i)]
 \item is automorphically induced from $\rho={}_{K}\pi_\xi$ and
  \item is compactly induced from an extended maximal type $\Lambda=\Lambda_{\nu\xi}$ depending on an admissible character $\nu\xi$, where $\nu$ is some tamely ramified character of $E^\times$. We know that such $\nu$ exists by Proposition 3.4 of \cite{BH-ET1}.
\end{enumerate}
By choosing suitably normalized $\Psi$, we can equate the characters of the automorphic induction (\ref{automorphic-induction-character-relation}) from $\rho$ and of the (twisted) Mackey induction from $\Lambda$. We can write down the latter explicitly as
 \begin{equation}\label{twisted mackey induction}
   \Theta^{\kappa}_\pi (h) =  \sum_{x\in G(F)/\mathbf{J}(\xi)} \kappa(\det x^{-1})\text{ trace}\Lambda(x^{-1}hx) \text{, for all }h \in H(F) \cap G(F)_{\mathrm{ell}}.
 \end{equation}
 Here we regard trace$\Lambda$ as a function of $G(F)$ vanishing outside $\mathbf{J}(\xi)=\mathbf{J}({\nu\xi})$. By decomposing both characters into finer sums and studying each terms in the sums, we solve for our desired $\nu$-rectifier $\nu={}_{K/F}\mu_\xi$. We then define ${}_{F}\mu_\xi=({}_{K/F}\mu_\xi)({}_{K}\mu_\xi).$

To express the values of the rectifiers and the $\nu$-rectifiers, we need to understand the construction of the compact induction (\ref{supercusp as compact induction}). We will come to this point in section \ref{Supercuspidal Representations}.

\section{The Archimedean case}\label{the Archi case}

We follow the description of relative discrete series representations of $G(F)=\mathrm{GL}_2(\mathbb{R})$ in chapter 1, \S 5 of \cite{Jac-Langlands-GL(2)} to express the local Langlands correspondence for this group. There is also a brief summary in Example 10.5 of \cite{Borel-autom-L} of classifying discrete series representations for a semi-simple split real group that possesses a compact maximal sub-torus.

We consider, when $F=\mathbb{R}$, the group $G(F)=\mathrm{GL}_2(\mathbb{R})$ and a maximal torus $T(F)\cong \mathbb{C}^\times$. Each character of $\mathbb{C}^\times$ is of the form
\begin{equation*}
\chi_{2s,n}: \mathbb{C}^\times \rightarrow \mathbb{C}^\times,\, re^{i \theta} \mapsto r^{2s}e^{in \theta},\,\text{for some } s\in \mathbb{C},\, n \in \mathbb{Z}.
\end{equation*}
We call $\chi_{2s,n}$ admissible over $\mathbb{R}$ (or regular) if $\chi_{2s,n}$ does not factor through
$N_{\mathbb{C}/\mathbb{R}}: z \mapsto z\bar{z} = |z|^2.$ This is equivalent to the condition that $n \neq 0$. Clearly, $\chi_{2s,n}$ and $\chi_{2s,-n}$ are conjugate under the non-trivial action in $ \Gamma_{\mathbb{C/R}}$. We then denote the equivalence class of $\chi_{2s,n}$ by $(\mathbb{C/R}, \chi_{2s,n})$ and the set of equivalence classes by
\begin{equation*}
P_2(\mathbb{R}) = \{ (\mathbb{C/R}, \chi_{2s,n}) | s \in \mathbb{C}, n>0\}.
\end{equation*}
By \cite{Tate-NTB}, any irreducible 2-dimensional representation of $W_\mathbb{R}$ is induced from a regular character of $W_\mathbb{C}\cong \mathbb{C}^\times$. We write
\begin{equation} \label{2-dim induced representation W-real case}
\varphi_{2s,n} = \text{Ind}_{\mathbb{C/R}} \chi_{2s,n}.
\end{equation}
It is easy to check that $\varphi_{2s,n} \cong \varphi_{2s,-n}$. Therefore the set of equivalence classes of 2-dimensional representations of $W_\mathbb{R}$ is
\begin{equation*}
\mathcal{G}^0_2(\mathbb{R}) = \{\varphi_{2s,n}|s \in \mathbb{C}, n > 0\},
\end{equation*}
and the map
\begin{equation}\label{induction of real weil group}
P_2(\mathbb{R}) \rightarrow \mathcal{G}^0_2(\mathbb{R}),\,\chi \mapsto \text{Ind}_{\mathbb{C/R}}\chi
\end{equation} is clearly bijective.

Each character of $\mathbb{R}^\times$ is of the form
\begin{equation*}
\mu_{s,m}: \mathbb{R}^\times \rightarrow \mathbb{C}^\times,\,\pm r \mapsto (\pm1)^m r^s,\, s \in \mathbb{C},\, m \in  \{0,1\}.
\end{equation*}
Using the notation in Theorem 5.11 of \cite{Jac-Langlands-GL(2)}, we consider the representation
\begin{equation*}
\sigma(\mu_{s+\frac{n}{2}, [n+1]}, \mu_{s-\frac{n}{2},0}),\,s \in \mathbb{C},\, n \geq 1,
\end{equation*}
where $[n]\in \mathbb{Z}/2$ is the parity of $n$. These representations exhaust the collection $\mathcal{A}^0_2(\mathbb{R})$ of all relative discrete series representations by Harish-Chandra's classification. The Langlands correspondence \cite{Langlands-real-alg-gp} is given by
\begin{equation}\label{langlands corresp CR case}
\mathcal{L}:\mathcal{G}^0_2(\mathbb{R}) \rightarrow \mathcal{A}^0_2(\mathbb{R}),\,\varphi_{2s,n} \mapsto \sigma(\mu_{s+\frac{n}{2}, [n+1]}, \mu_{s-\frac{n}{2},0}).
\end{equation}
The bijection is well-defined and injective because, by Theorem 5.11 of \cite{Jac-Langlands-GL(2)}, we have
\begin{equation*}
\sigma(\mu_{s-\frac{n}{2}, [-n+1]}, \mu_{s+\frac{n}{2},0}) \cong \sigma(\mu_{s+\frac{n}{2},[n+1]}, \mu_{s-\frac{n}{2},0})
  \end{equation*}
and
\begin{equation*}
\sigma(\mu_{s+\frac{m}{2}, [m+1]}, \mu_{s-\frac{m}{2},0}) \ncong \sigma(\mu_{s+\frac{n}{2},[n+1]}, \mu_{s-\frac{n}{2},0})\text{ for all }m\neq \pm n.
  \end{equation*}

Let us recall the construction of the relative discrete series representations, and refer to chapter 3 of \cite{Langlands-real-alg-gp} for details. For each character $\chi = \chi_{2s,n}$, we introduce a `$\rho$-shift' on $\chi$ as follows. Let \begin{equation*}
\rho: \mathbb{C}^\times \rightarrow \mathbb{C}^\times,\,z \mapsto ({z}{\bar{z}}^{-1})^{1/2}.
\end{equation*}
Hence indeed $\rho = \chi_{0,1}$. This is the based $\chi$-data used in \cite{Shelstad-temp-endo-for-real-group-1}. We will make use of a generalization of this to construct admissible embeddings of L-groups in section \ref{section langlands shelstad chi data}. We define
\begin{equation}\label{rectifier in archi case}
\chi_0 = \chi \rho^{-1} = \chi_{2s,n-1}.
\end{equation}
For every such $\chi_0$ we assign
\begin{equation}\label{local construction in archi case}
\pi(\chi_0) = \sigma(\mu_{s+ \frac{n}{2}, [n+1]}, \mu_{s-\frac{n}{2},0})=\mathcal{L}(\text{Ind}_{\mathbb{C/R}} \chi_0\rho),
\end{equation}
whose character formula is given by
\begin{equation}\label{character formula for pi(chi0)}
-\frac{\chi_0(z)\rho(z)-\chi_0(\bar{z})\rho(\bar{z})}{\Delta(z)\rho(z)} = -\frac{\chi(z)-\chi(\bar{z})}{\Delta(z)\rho(z)}\text{, for all } z\in\mathbb{C}^\times-\mathbb{R}.
\end{equation}
Here $\mathbb{C}^\times$ is embedded in $\text{GL}(2,\mathbb{R})$ as a compact-mod-center Cartan subgroup, and $\Delta(z)$ is the Weyl discriminant. The character formula is expressed as on the left side of (\ref{character formula for pi(chi0)}) in order to match the character formula, given in \cite{Langlands-real-alg-gp}, of a discrete series representation under the Harish-Chandra's classification. This result is comparable to the description $\rho+X^*(T)$ of the discrete spectrum of $\mathrm{SL}_2$ using infinitesimal characters, given in Example 10.5 of \cite{Borel-autom-L}. Therefore, in order to obtain the correct correspondence, we have to rectify (\ref{induction of real weil group}) by the character $\rho$. More precisely, we have the diagram
\begin{equation*}
\xymatrixcolsep{5pc}\xymatrix{
\mathcal{G}^0_2(\mathbb{R})   \ar[r]^{\mathcal{L}} &\mathcal{A}^0_{2}(\mathbb{R}) \\
P_2(\mathbb{R}) \ar[u]^{\mathrm{Ind}_{\mathbb{C}/\mathbb{R}}}\ar[r]^{\chi_0\rho\mapsto \chi_0}  &P_2(\mathbb{R}) \ar[u]^{\pi},
}
\end{equation*}
by combining (\ref{induction of real weil group}), (\ref{langlands corresp CR case}), (\ref{rectifier in archi case}), and (\ref{local construction in archi case}).

\chapter{Endoscopy}\label{endoscopy}

Let $K/F$ be a cyclic extension of degree $d$ and $m=n/d$. The aim of this chapter is to compute the Langlands-Shelstad transfer factor $\Delta_0$, defined in \cite{LS}, \cite{KS}, in the case when $G=\mathrm{GL}_n$ and $H=\mathrm{Res}_{K/F}\mathrm{GL}_m$ regarded as a twisted endoscopic group of $G$. The transfer factor $\Delta_0$ is a product of several other transfer factors $$\Delta_0=\Delta_\mathrm{I}\Delta_\mathrm{II}\Delta_\mathrm{III_1}\Delta_\mathrm{III_2}\Delta_\mathrm{IV}.$$ In our case these factors are functions evaluated at the regular semi-simple elements in an elliptic torus $T$ regarded as lying in both $G$ and $H$. We would recall the definitions and mention some properties of these transfer factors.

The topics of this chapter are summarized as follows.
\begin{enumerate}[(i)]
  \item We define the transfer factor $\Delta_0$ and explain its significance in terms of the transfer principle in section \ref{section Endoscopic group}.
      \item We construct $\Delta_\mathrm{I}$ in terms of several auxiliary data in section \ref{section The splitting invariant}. Then we show, in section \ref{section Trivializing the splitting invariant}, that we can trivialize $\Delta_\mathrm{I}$. More precisely, we show that $\Delta_\mathrm{I}\equiv 1$ if we choose specific data.
          \item In section \ref{section Induction as Admissible Embedding}, we interpret the induction of representations of Weil groups in terms of admissible embeddings of L-groups. We then specify those embeddings constructed by $\chi$-data in section \ref{section langlands shelstad chi data} and \ref{section explicit-Delta-III-2}, where we can express the transfer factor $\Delta_\mathrm{III_2}$ explicitly. In section \ref{section restriction property}, we state a restriction property of $\Delta_\mathrm{III_2}$, which is related to condition (\ref{ETLLC det is central char}) of the local Langlands correspondence in Theorem \ref{langlands corresp} (see Remark \ref{remark res rectifier is delta} for instance).
              \item In section \ref{section Relations of transfer factors}, we relate the above transfer factors to the one defined in \cite{Hen-Herb}.
\end{enumerate}

\section{The transfer principle}\label{section Endoscopic group}
Let $d$ be a divisor of $n$ and $K/F$ be a cyclic extension of degree $d$. Write $m=n/d$ and $H = \text{Res}_{K/F}\text{GL}_{m}$.
Then we can regard $H$ as a twisted endoscopic group of $G$ by the theory in \cite{KS} as follows. Take a torus $T$ over $F$ of rank $n$ (but not necessarily split over $F$). We can embed $T$ into $H$ and $G$ by certain isomorphisms $\iota_H:T\rightarrow T_H$ and $\iota_G:T\rightarrow T_G$ such that the isomorphism $\iota := \iota_G\iota_H^{-1}: T_H \rightarrow T_G$ is defined over $F$ and is admissible with respect to certain fixed splittings \footnote{The word `splitting' is also called `pinning' in some literature, coming from the French terminology `\'{e}pinglage.'} $\text{spl}_H$ of $H$ and $\text{spl}_G$ of $G$ (see section 3 of \cite{KS}). On the dual side, we take
\begin{equation*}
\hat{G} = \text{GL}_n( \mathbb{C}),\,  \hat{H} = \text{GL}_m(\mathbb{C})^d  \text{, and }  \hat{T} = (\mathbb{C}^\times)^n
\end{equation*}
as abstract groups. The $W_F$-action on $\hat{H}$ is defined by cyclicly permutes the $d$ components, so it factors through that of the cyclic quotient $\Gamma_{K/F}=W_F/W_K$. Hence it is more accurate to write $\hat{H} = \mathrm{Ind}_{K/F}\text{GL}_m(\mathbb{C}) $ and similarly $\hat{T} = \mathrm{Ind}_{E/F}\mathbb{C}^\times$. For convenience, we can choose splittings $\text{spl}_{\hat{H}}$ and $\text{spl}_{\hat{G}}$ such that both consist of diagonal subgroups as their respective tori and embed $\hat{T}$ into both $\hat{G}$ and $\hat{H}$ diagonally.

Together with $G$ there is a triple of datum $(G, \theta, \kappa)$ such that, in our case here, $\theta$ is the trivial automorphism of $G$ and $\kappa$ is a character of $F^\times$ associated to the cyclic extension ${K/F}$. By $H$ being a twisted endoscopic group of $(G, \theta, \kappa)$, we mean that there exists a quadruple of data $(H, \mathcal{H}, s, \hat{\xi})$ satisfying the conditions in section 2.1 of \cite{KS}. In our case, we can take
\begin{equation*}
\mathcal{H} = {}^LH = \text{GL}_m(\mathbb{C})^d \rtimes W_F,\, s = \left( \begin{smallmatrix} {} &{} &{} &1_m \\ 1_m &{} &{} &{} \\ {} &\ddots &{} &{} \\ {} &{} &1_m &{} \end{smallmatrix} \right),
\end{equation*}
and
\begin{equation*}
\hat{\xi}: {}^LH \rightarrow {}^LG,\,\left(\begin{smallmatrix} h_1 &{} &{}\\{} &\ddots &{} \\ {} &{} &h_d \end{smallmatrix}\right) \rtimes_{\hat{H}} w \mapsto \left(\begin{smallmatrix} h_1 &{} &{}\\{} &\ddots &{} \\ {} &{} &h_d \end{smallmatrix}\right) \bar{N}_{HG}(w) \rtimes_{\hat{G}} w,
\end{equation*}
where $1_{m}$ is the identity matrix of size $m \times m$, the entries $h_i \in \text{GL}_m(\mathbb{C})$, and $\bar{N}_{HG}(w)$ is the appropriate permutation matrix that measures the difference of the actions of $w\in W_F$ on $\hat{H}$ and on $\hat{G}$.

We suppose that $\gamma \in T(F)$ is strongly $G$-regular, which means that $\iota_G(\gamma)$ is strongly regular in $T_G$. With $\iota$ being admissible, it can be shown that $\iota_H(\gamma)$ is also strongly regular. We can compute the \emph{transfer factor} $$\Delta_0(\iota_H(\gamma), \iota_G(\gamma))$$ using the recipes in \cite{LS} or \cite{KS}. It is a product of five factors, $\Delta_\mathrm{I}, \Delta_\mathrm{II}, \Delta_\mathrm{III_1}, \Delta_\mathrm{III_2}$, and $\Delta_\mathrm{IV}$, all evaluated at the point $(\iota_H(\gamma), \iota_G(\gamma))$. For simplicity we write
\begin{enumerate}[(i)]
\item  $\Delta_*(\gamma) = \Delta_*(\iota_H(\gamma), \iota_G(\gamma))$ for every subscript $*=0,\mathrm{I},\dots,\mathrm{IV}$ above,
\item  $\Delta_{i,j,\dots}$ for the product $\Delta_i\Delta_j\cdots$, and
\item $\Delta_\mathrm{III}=\Delta_\mathrm{III_1,III_2}$.
\end{enumerate}
We can compute some of the factors with little effort.
\begin{enumerate}[(i)]
  \item By regarding $H$ as a subgroup of $G$ and fixing the embedding $T\hookrightarrow H\hookrightarrow G$, we have $$\Delta_\mathrm{III_1}(\gamma) = 1.$$
\item
By the definition in (3.6) of \cite{LS}, we have $$\Delta_\mathrm{IV}(\gamma)=\left(|D_G(\gamma)|_F|D_H(\gamma)|_E^{-1}\right)^{1/2},$$ where $D_G(\gamma)=\det(1-\mathrm{Ad}(\gamma))|_{\mathfrak{g}/\mathfrak{g}_\gamma}$ and $D_H(\gamma)=\det(1-\mathrm{Ad}(\gamma))|_{\mathfrak{h}/\mathfrak{h}_\gamma}$ are the Weyl determinants. It is equal to the factor $\Delta^1$ appearing in the automorphic induction identity (\ref{automorphic-induction-character-relation}), by the definition in (3.2) of \cite{Hen-Herb}. This factor turns out to be of very little interest to us.
\end{enumerate}
We have to choose some auxiliary data to compute $\Delta_\mathrm{I},\,\Delta_\mathrm{II}$ and $ \Delta_\mathrm{III_2}$. These will occupy the subsequent sections of this chapter.

We recall briefly the transfer principle. For details, see section 3 of \cite{Hen-Herb} and section 5.5 of \cite{KS}. Take $\gamma\in G(F)_{\mathrm{rss}}$ and denote by $G(F)_\gamma$ the centralizer of $\gamma$ in $G(F)$. For every $f\in \mathcal{C}^\infty_c(G(F))$ we define the $\kappa$-orbital integral $$O_\gamma^\kappa(f)=\int_{G(F)_\gamma\backslash G(F)}\kappa\circ\det(g)f(g^{-1}\gamma g)\frac{dg}{dt},$$ where $dg$ and $dt$ are Haar measures on $G(F)$ and $G(F)_\gamma$ respectively. Similarly, for $\gamma_H\in H(F)_\mathrm{rss}$ and $f^H\in \mathcal{C}^\infty_c(H(F))$, we define the orbital integral
$$O_{\gamma_H}(f^H)=\int_{H(F)_{\gamma_H}\backslash H(F)}f(h^{-1}\gamma_H h)\frac{dh}{dt_H},$$
where $dh$ and $dt_H$ are Haar measures on $H(F)$ and $H(F)_{\gamma_H}$ respectively. Suppose that $\iota(\gamma_H)=\gamma$. We call that $f^H$ and $f$ are matching functions if they satisfy \begin{equation}\label{twisted orbital integrals}
  O_{\gamma_H}(f^H)=\Delta_0(\gamma)O_\gamma^\kappa(f).
\end{equation} The Fundamental Lemma \cite{Walds-GLn} implies that for every $f$ such $f^H$ exists.

Suppose that $\rho\in \mathcal{A}_m^0(K)$. We call an admissible representation $\pi$ of $G(F)$ a spectral transfer of $\rho$ if they satisfy the character ideitity \begin{equation}\label{definition of spectral transfer}
  \Theta_\rho(f^H)=\text{constant}\cdot \Theta^\kappa_\pi(f).
\end{equation}
It is known (or see Proposition \ref{delta-2-equals-delta-II-III} for instance) that the Langlands-Shelstad transfer factor $\Delta_0$ is equal to the Henniart-Herb transfer factor $$\Delta_G^H(\gamma)=\Delta^1(\gamma)\Delta^2(\gamma)\text{, for all }\gamma\in H(F)\cap G(F)_\mathrm{ell},$$ up to a constant. More precisely, since $\Delta_\mathrm{IV}=\Delta^1$, the factors $\Delta_{\mathrm{I,II,III}}$ and $\Delta^2$ differ by a constant. Hence $\Delta_{\mathrm{I,II,III}}$ also satisfies the properties of $\Delta^2$. Using the technique involving the Weyl integration formula (see the proof of Theorem 3.11 of \cite{Hen-Herb}), we can show that the matching of orbital integrals (\ref{twisted orbital integrals}) and the spectral transfer (\ref{definition of spectral transfer}) give rise to an analogous identity of (\ref{automorphic-induction-character-relation}),
\begin{equation}\label{spetral transfer-character-relation}
\Theta^{\kappa }_\pi(h) = \text{constant}\cdot \Delta_{\mathrm{I,II,III}}(h) \Delta_{\mathrm{IV}}(h)^{-1} \sum_{g \in \Gamma_{K/F}}\Theta^g_\rho(h)\text{, for all }h\in H (F)\cap G(F)_{\mathrm{ell}}.
 \end{equation}
Therefore, by \cite{Hen-Herb}, the spectral transfer $\pi$ of $\rho$ exists and is exactly the automorphic induction of $\rho$. In chapter \ref{chapter comparing character}, we will show that the automorphic induction identity (\ref{automorphic-induction-character-relation}) and the spectral transfer identity (\ref{spetral transfer-character-relation}) are equal when both are suitably normalized.

\section{The splitting invariant and $\Delta_{\mathrm{I}}$}\label{section The splitting invariant}
We now sketch the construction of $\Delta_\mathrm{I}$ in \cite{LS}. By fixing an embedding $T\hookrightarrow G$, we assume that $T$ is a maximal torus contained in $G$. We have to recall several ingredients.
\begin{enumerate}[(i)]
  \item We choose an $F$-splitting $\mathbf{spl}_G = (\mathbf{B}, \mathbf{T}, \{\mathbf{X}_\alpha\})$ of $G$ and a Borel subgroup $B$ of $G$ (not necessarily defined over $F$) such that $h^{-1}(B,T)h = (\mathbf{B}, \mathbf{T})$ for some $h \in G$. We write $$\omega:W_F\rightarrow\Omega({G},\mathbf{T})\rtimes W_F,\,w\mapsto \omega(w)=\bar{\omega}(w)\rtimes w,$$ such that the action $\omega(w)\rtimes w $ on $t \in \mathbf{T}$ is the transferred action of ${w_T}$ by $\text{Int}(h)$.

      \item We recall the section $ \Omega(G, \mathbf{T}) \rightarrow N_G(\mathbf{T})$ defined in section 2.1 of \cite{LS}. This section depends on the choices of simple roots defining the splitting $\mathbf{spl}_G$. We denote this section by $n_S$ and define
$$   n:W_F \rightarrow N_G(\mathbf{T}) \rtimes W_F,\,w \mapsto n(w) = \bar{n}(w) \rtimes w := n_\mathrm{S} (\bar{\omega}(w)) \rtimes w. $$
This map may not be a morphism of groups.

\item We take a collection  $\{ a_\lambda \in \bar{F}^\times | \lambda \in  \Phi(G,T)\}$, called a set of $a$-data, satisfying the conditions
     \begin{equation}\label{definition of a data}
      a_{g\lambda} = {}^g(a_\lambda)\text{ for all }g \in \Gamma_F\text{ and }a_{-\lambda} = - a_\lambda.
    \end{equation}
    We may regard the indexes $\lambda$ as in $\Phi(G, \mathbf{T})$ by using the transport $\text{Int}(h)$. We then write $$x(g) = \prod_{\lambda \in \Phi(\mathbf{B}, \mathbf{T}), {}^{g^{-1}} \lambda \notin \Phi(\mathbf{B}, \mathbf{T})} a_\lambda^{\lambda^\vee} \in \mathbf{T}.$$
\end{enumerate}
Here is a remark about those $a$-data in (\ref{definition of a data}). By Proposition \ref{explicit expression of double coset}, every $\lambda$ can be written as $\left[\begin{smallmatrix}  g\\h\end{smallmatrix}\right]$ with $g,h\in \Gamma_F/\Gamma_E$. We sometimes denote $a_\lambda$ by $a_{g,h}$. Hence by (\ref{definition of a data}) the whole collection $\{a_\lambda|\lambda\in \Phi\}$ of $a$-data depends on the subcollection $$\{a_{1,g}|g\in \mathcal{D}_\mathrm{sym}\sqcup\mathcal{D}_{\mathrm{asym}/\pm}\},$$ with notations introduced in section \ref{section Galois groups}.
 \begin{prop}\label{1-cocycle of splitting invar}
 The map $$g \rightarrow h x(g) \bar{n}(g) {}^gh^{-1},\,g\in \Gamma_F,$$ is a 1-cocycle of $\Gamma_F$ with image lying in $T$.
 \end{prop}
 \begin{proof}
   See section (2.3) of \cite{LS}.
 \end{proof}
We denote by $\lambda(\{a_\lambda\}, T)$ the class defined by this cocycle in $H^1(\Gamma_F, T)$. Following the recent preprint \cite{KS2012}, we call this class the \emph{splitting invariant}. Finally, $\Delta_\mathrm{I}(\gamma)$ is defined to be the Tate-Nakayama product $$\left<\lambda(\{a_\lambda\},T), s_T\right>$$ for certain $s_T\in \hat{T}$ depending on the datum $s$ of the endoscopic group $H$.

\begin{rmk}\label{splitting invariant indep of basis}
By (2.3.3) of \cite{LS}, we know that the splitting invariant is independent of the choice of Borel subgroup $B$.
\end{rmk}

\section{Trivializing the splitting invariant}\label{section Trivializing the splitting invariant}

Let $E/F$ be a tamely ramified extension. The aim of this section is to choose an embedding of the torus $T(F) = E^\times$ into $\text{GL}_n(F)$ and certain $a$-data $\{a_\lambda\}$ to trivialize the splitting invariant class $\lambda(\{a_\lambda\} ,T)$. We can even make things better: the 1-cocycle in Proposition \ref{1-cocycle of splitting invar} representing this class is trivial by selecting suitable data. The idea originates from the construction in \cite{HI}.

Take a primitive root of unity $\zeta\in \mu_E$. We choose an $F$-basis of $E$ by
\begin{equation}\label{chosen basis}
  \mathfrak{b} = \{ 1, \zeta, \dots, \zeta^{f-1}, \varpi_E, \varpi_E\zeta,
\dots, \varpi_E \zeta^{f-1}, \dots, \varpi_E^{e-1} \zeta^{f-1} \}
\end{equation}
We identify $G$ with $\mathrm{Aut}_F(E)$, the group of $F$-automorphisms of $E$, using the above basis. We may also assume that the torus $\mathbf{T}$ in the chosen $F$-splitting $\mathbf{spl}_G$ is the diagonal subgroup. We write a row vector $$\vec{v} = (  1, \zeta, \dots, \zeta^{f-1}, \varpi_E, \varpi_E\zeta, \dots, \varpi_E \zeta^{f-1}, \dots, \varpi_E^{e-1} \zeta^{f-1}) \in E^n$$ and embed $\iota: T \rightarrow G$ by
\begin{equation}\label{canonical embedding for trivial delta 1}
  x \cdot \vec{v} = \vec{v} \cdot \iota(x),
\end{equation}
 where the left side is coordinate-wise multiplication and the right side is a multiplication of a row vector with a matrix. We can assume that $T$ lies in $H$ and that $\iota$ is admissible.

We now define a total order $<_\phi$ on the `Galois set' $\Gamma_{E/F}= \{ \phi^i \sigma^k| 0\leq i< f,\, 0\leq k <e\}$ using the multiplicative $\phi$-orbits of $\left<\sigma\right>$. Take a set $q^f\backslash (\mathbb{Z}/e)$ of representatives in $\mathbb{Z}/e$ of the $q^f$-orbits (see the discussion before Lemma \ref{number of double cosets}) and choose an arbitrary order $<_\phi$ on $\{\sigma^k|k\in q^f\backslash (\mathbb{Z}/e)\}$. For simplicity we may choose the orbit of 1 to be $<_\phi$-minimum. If $g,h\in \Gamma_F/\Gamma_E$ lie in different $\phi$-orbits, say $g\in \left<\phi\right> \sigma^k,\,h\in \left<\phi\right> \sigma^l$, and $k\neq l$ in $ q^f\backslash (\mathbb{Z}/e),$ then we order $g$ and $h$ by $$g<_\phi h\text{ if and only if }k<_\phi l.$$ We finally consider the multiplicative $\phi$-orbit of $\sigma^k$ for a fixed $k\in q^f\backslash (\mathbb{Z}/e)$. This is the subset $\{\phi^i\sigma^k|0\leq i<q^{fs_k}\}$ where $s_k$ is the cardinality of the $q^f$-orbit $k\in\mathbb{Z}/e$. We order the $\phi$-orbit of $\sigma^k$ by $$\phi^i\sigma^k<_\phi\phi^j\sigma^k\text{ if and only if }0\leq i<j<q^{fs_k}.$$ The above procedures then define a total order $<_\phi$ on $\Gamma_F/\Gamma_E$. For example, if we choose the order $0<_\phi1<_\phi\cdots<_\phi k<_\phi\cdots $ on $q^f\backslash (\mathbb{Z}/e)$, then the order on $ \Gamma_F/\Gamma_E$ looks like
\begin{equation*}
\begin{split}
  &1 <_\phi \phi <_\phi \cdots <_\phi \phi^{f-1}
  \\
  <_\phi& \sigma <_\phi \phi\sigma <_\phi \cdots<_\phi \phi^{f-1}\sigma <_\phi \sigma^{1_{+_\phi 1}}<_\phi \phi\sigma^{1_{+_\phi 1}} <_\phi \cdots <_\phi \sigma ^{1_l} <_\phi \cdots <_\phi\phi^{f-1}\sigma^{1_l}
  \\
  <_\phi &\cdots
   \\
  <_\phi& \sigma^k <_\phi \phi\sigma^k <_\phi \cdots<_\phi \phi^{f-1}\sigma <_\phi \sigma^{k_{+_\phi 1}}<_\phi \phi\sigma^{k_{+_\phi 1}} <_\phi \cdots <_\phi \sigma ^{k_l} <_\phi \cdots <_\phi\phi^{f-1}\sigma^{k_l}
  \\
   <_\phi &\cdots.
\end{split}
\end{equation*}
We justify some notations in the above sequence.
\begin{enumerate}[(i)]
  \item We denote by $k_{+_\phi 1}$ the `successor' of $k$ under the order induced by $<_\phi$. It satisfies $k_{+_\phi 1}\equiv kq^f\mod e$.
      \item We denote by $k_l$ the maximum within the $q^f$-orbit of $k$. It satisfies $k_l\equiv q^{f(s_k-1)}k \mod e$.
\end{enumerate}
We remark here that such order is defined only for computation. By Remark \ref{splitting invariant indep of basis}, we know that the splitting invariant is independent of the order of the chosen basis.

Suppose that $\Gamma_{E/F} = \{ g_1 = 1, \dots, g_n\}$ is ordered by the above order $<_\phi$. By putting $$\mathbf{g} = (\vec{v}^t,
{}^{g_2}\vec{v}^t, \dots, {}^{g_n}\vec{v}^t)^t \in \mathrm{GL}_n (E) ,$$ we have $\mathbf{g}{\iota(x)}\mathbf{g}^{-1} = \text{diag}(x, {}^{g_2}x, \dots,, {}^{g_n}x)$ for all $x \in E^\times$. The $(\alpha, j)$-block of $\mathbf{g}$ is $$\mathbf{g}_{\alpha j}= \left(
\begin{matrix}
{}^{\sigma^\alpha} (\varpi^j_E, \varpi^j_E \zeta, \dots, \varpi^j_E \zeta^{f-1}) \\
{}^{\phi\sigma^\alpha} (\varpi^j_E, \varpi^j_E \zeta, \dots, \varpi^j_E \zeta^{f-1}) \\
\vdots \\
{}^{\phi^{f-1}\sigma^\alpha} (\varpi^j_E, \varpi^j_E \zeta, \dots, \varpi^j_E
\zeta^{f-1})
\end{matrix}\right).$$ This is equal to $g_{\alpha j}u$, where $g_{\alpha j} =
\text{diag}({}^{\sigma^\alpha} \varpi_E^j, \dots,
{}^{\phi^{f-1}\sigma^\alpha} \varpi_E^j)$ and
$$u = \left(
\begin{matrix}
&1 & {}^{\sigma^\alpha}\zeta &\dots &{}^{\sigma^\alpha} \zeta^{f-1}\\
&1 &{}^{\phi\sigma^\alpha}\zeta &\dots &{}^{\phi\sigma^\alpha} \zeta^{f-1}\\
&&\vdots&&\\
&1 &{}^{\phi^{f-1}\sigma^\alpha}\zeta &\dots &{}^{\phi^{f-1}\sigma^\alpha}
\zeta^{f-1}
\end{matrix}
\right).$$
Therefore we have $\det  \mathbf{g} = \det {g}(\det u)^e$, where $$\det g = \prod^{f-1}_{k=0} \prod_{0 \leq_\phi \beta <_\phi \alpha < e}({}^{\phi^k \sigma^\alpha} \varpi_E - {}^{\phi^k\sigma^\beta} \varpi_E)$$ and $$\det u = \prod_{0 \leq i <k \leq f-1}({}^{\phi^k}\zeta - {}^{\phi^i} \zeta).$$
Therefore if we take
\begin{equation}\label{definition of b-data}
 b_{\phi^k \sigma^\alpha} = \prod_{0 \leq_\phi \beta <_\phi \alpha}({}^{\phi^k \sigma^\alpha} \varpi_E - {}^{\phi^k \sigma^\beta}
\varpi_E)^{-1} \prod_{0 \leq i < k}({}^{\phi^k} \zeta - {}^{\phi^i} \zeta)^{-1}
\end{equation} and write  $$h^{-1}
= (b_1 \vec{v}^t, b_{g_2}{}^{g_2} \vec{v}^t, \dots, b_{g_n}{}^{g_n} \vec{v}^t)^t,$$ then $\det  h=1$ and $h^{-1} \iota(x)h =
\mathbf{g} \iota(x) \mathbf{g}^{-1}$ for all $x \in E^\times$.

\begin{prop}\label{prop trivial cocycle}
  Fix a splitting $\mathbf{spl}_G$ of $G$ defined by the basis (\ref{chosen basis}) above. The cocycle defined in Proposition \ref{1-cocycle of splitting invar} is trivial by choosing suitable $a$-data.
 \end{prop}

The proof occupies the remaining of this section. First notice that if $L/F$ is the Galois closure of $E/F$, then $\Gamma_L$ acts trivially on $ G(E)$. Moreover, $x(g)$ and $\bar{n}(g)$ are identity matrices. Hence it is enough to show that the cocycle is trivial at the generators $\sigma$ and $\phi$ of $\Gamma_{L/F}$. In other words, we will show that there exist $a$-data satisfying the matrix equations
\begin{equation}\label{delta-1-trivial-cocycle}
x(\sigma) \bar{n}(\sigma)^\sigma h^{-1} = h^{-1}  \text{ and }
x(\phi)\bar{n}(\phi) {}^\phi h^{-1} = h^{-1}.
\end{equation}
For each equation in (\ref{delta-1-trivial-cocycle}), we can concentrate on the rows indexed by the cosets in $\Gamma_{E/F}$ lying within a $\sigma$-orbit or a $\phi$-orbit respectively. We would apply the following.
\begin{lemma}
Let $\tau$ be an arbitrary element in $\Gamma_{L/F}$. If $\{g, \tau g, \dots,\tau^{k-1}g\}$ is an ordered $\tau$-orbit, then with respect to such basis,
\begin{equation*}
\bar{n}(\tau) = \left(
\begin{smallmatrix}
&&& &1\\
&-1 &&&\\
& &\ddots &&\\
&&&-1 &
\end{smallmatrix}
\right)
\text{ and }
x(\tau) = \mathrm{diag}\left(
\prod^{k-1}_{i=1} a_{g,{\tau^i}g},a^{-1}_{g, {\tau} g},a^{-1}_{g, {\tau^2} g},\dots,a^{-1}_{g, {\tau^{k-1}} g}
\right).
\end{equation*}
\end{lemma}

\begin{proof}
Recall that $$x(\tau) = \prod_{\lambda >0 ,\,{}^{\tau^{-1}} \lambda <0}
a_\lambda^{\lambda^\vee}.$$ When $\lambda = \left[\begin{smallmatrix} g \\ h
\end{smallmatrix}\right] = \left[ \begin{smallmatrix}
   gW_E \\ h W_E \end{smallmatrix}\right]$, we have denoted $a_\lambda$ by $a_{g,h}$. If we choose the basis as above, then those positive $\lambda = \left[
\begin{smallmatrix} \tau^i g \\ \tau^j g\end{smallmatrix}\right]$ with $\tau^{-1}\lambda <0$ are those with $i=0$ and $j>0$.
Therefore $$x(\tau) = \prod^{k-1}_{j=1}a^{\left[ \begin{smallmatrix}  g \\
\tau^j g\end{smallmatrix}\right]^\vee}_{g, \tau^j g},$$ and is equal to the one in the assertion. The action of $\bar{n}(\tau)$ on the root system is given by the corresponding cyclic permutation
$$(g, \tau g, \dots, {\tau^{k-1}}g) = (g,
{\tau^{k-1}}g)(g, {\tau^{k-2}}g) \cdots (g, {\tau}g).$$
Therefore $\bar{n}(\tau)$ is equal to
\begin{equation*}
\left( \begin{smallmatrix}
& & && & &1\\
& &1 && & &\\
& & &\ddots && &\\
&&&&&&\\
  &   & &  & &1 &\\
&-1 & & && &
\end{smallmatrix}\right)\left(\begin{smallmatrix}
& & & & &1&\\
& &1 & && &\\
& & &\ddots && &\\
 &  & & &1 &&\\
&-1 & & & &&\\
&&&&&&1
\end{smallmatrix}\right) \cdots\left(\begin{smallmatrix}
& &1 && & &\\
&-1 && & & &\\
&& &1 & & &\\
&& & & &\ddots &\\
&&&&&&\\
& & & && &1
\end{smallmatrix}\right) =\left(\begin{smallmatrix}
& & & &1\\
-1 & & &\\
& &\ddots & &\\
& & &-1 &
\end{smallmatrix}\right).
\end{equation*}
\end{proof}

We start to solve for $\{ a_\lambda\}$ in the first equation of (\ref{delta-1-trivial-cocycle}). A $\sigma$-orbit is of the form $\{ \sigma^k \phi^j | 0 \leq k<e\}$ with a fixed $j=0, \dots, f-1$. Using Remark \ref{splitting invariant indep of basis}, we order this basis by the usual order $<$ on $\{0,\dots,e-1\}$, so that the minimal element in this orbit is $\phi^j$. Notice that we have to modify the elements in (\ref{definition of b-data}): just replace the order $<_\phi$ by the usual $<$ order. Focusing on the rows of (\ref{delta-1-trivial-cocycle}) corresponding to this orbit, the equation reads as
\begin{equation}\label{delta-1-equation-for-sigma}
\begin{split}
   & \text{diag}\left( \prod^{e-1}_{k=1} a_{\phi^j, \sigma^k\phi^j},
a_{\phi^j, \sigma \phi^j}^{-1}, \dots, a^{-1}_{\phi^j, \sigma^{e-1}
\phi^j}\right)\cdot \left( \begin{smallmatrix}
     & & &1\\
     -1 & & &\\
     &\ddots & &\\
     & &-1 &
  \end{smallmatrix} \right)\cdot
  \\
    &{}^\sigma (b_{\phi^j}{}^{\phi^j}\vec{v}^t, b_{\sigma
\phi^j}{}^{\sigma \phi^j}\vec{v}^t, \dots, b_{\sigma^{e-1}
\phi^j}{}^{\sigma^{e-1}\phi^j}\vec{v}^t)^t
  \\
  = &(b_{\phi^j}{}^{\phi^j}\vec{v}^t, b_{\sigma \phi^j}{}^{\sigma
\phi^j}\vec{v}^t, \dots, b_{\sigma^{e-1}
\phi^j}{}^{\sigma^{e-1}\phi^j}\vec{v}^t)^t   .
\end{split}
\end{equation}
The $(k+1)$th row for $1 \leq k < e$ reads
$$-a^{-1}_{\phi^j, \sigma^k
\phi^j}{}^\sigma b_{\sigma^{k-1}\phi^j} = b_{\sigma^k \phi^j}.$$
From the modified (\ref{definition of b-data}), we
solve that
\begin{equation}\label{delta-1-solution-for-a-sigma}
  a_{\phi^j, \sigma^k \phi^j} = - {}^\sigma b_{\sigma^{k-1}
\phi^j} b^{-1}_{\sigma^k \phi^j} = {}^{\phi^j} \varpi_E - {}^{\sigma^k
\phi^j} \varpi_E .
\end{equation}
These elements clearly satisfy the second condition of $a$-data in (\ref{definition of a data}). The first row reads $$\left( \prod^{e-1}_{k=1} a_{\phi^j,
\sigma^k\phi^j}\right) {}^\sigma b_{\sigma^{e-1} \phi^j} =
b_{\phi^j}.$$
By putting (\ref{delta-1-solution-for-a-sigma}), it changes into $${}^\sigma \left(
\prod^{e-1}_{k=0} b_{\sigma^k\phi^j}\right) = (-1)^{e-1}
\prod^{e-1}_{k=0}b_{\sigma^k\phi^j},$$ which can be easily checked to
be true. We have solved the first equation of (\ref{delta-1-trivial-cocycle}).

We then solve for $\{a_\lambda\}$ in the second equation of (\ref{delta-1-trivial-cocycle}). Take the $\phi$-orbit $$\{ \phi^j \sigma^\alpha | j=0,
\dots, f-1\text{ and }\alpha \in \left<q^f\right> k\}$$ of $\sigma^k$, ordered by $<_\phi$ defined above. We have the matrix equation similar to (\ref{delta-1-equation-for-sigma}) corresponding to this orbit. We then check the equations given by the rows, which are distinguished between the following cases.
\begin{enumerate}[(i)]
\item  For fixed $\alpha$, if $j=1, \dots, f-1$, then $$a_{\sigma^k, \phi^j \sigma^\alpha} = - {}^\phi b_{\phi^{j-1}\sigma^{\alpha}}b^{-1}_{\phi^j\sigma^\alpha} = \zeta - {}^{\phi^j}\zeta.$$
    These elements satisfy the second condition of $a$-data in (\ref{definition of a data}).

\item If $\alpha \neq k$ and $j=0$, then $$a_{\sigma^k, \sigma^\alpha}
= -{}^\phi b_{\phi^{f-1}\sigma^{\alpha_{-_\phi 1}}}b^{-1}_{\sigma^\alpha}. $$
Here $\alpha_{-_\phi 1} $ satisfies $ \alpha\equiv \alpha_{-_\phi 1} q^f\mod e $. This is the `predecessor' of $\alpha$ with respect to $<_\phi$ in the $\phi$-orbit of $\{\sigma^k\}$. Notice that $\phi$ permutes the elements in the orbits $<_\phi$-strictly smaller than the orbit of $\sigma^k$ and moves the elements in the orbit of $\sigma^k$ `one step forward.' We then have \begin{equation*}
  \begin{split}
    {}^\phi b_{\phi^{f-1}\sigma^{\alpha_{-_\phi1}}} = &{}^\phi \left(
\prod_{0 \leq_\phi \beta <_\phi \alpha_{-_\phi 1}}
({}^{\phi^{f-1}\sigma^{\alpha_{-_\phi 1}}}\varpi_E -
{}^{\phi^{f-1}\sigma^{\beta}}\varpi_E)^{-1} \prod_{0 \leq i < f-1 }
({}^{\phi^{f-1}} \zeta - {}^{\phi^i}\zeta)^{-1} \right)
\\
=& \left(\prod_{0 \leq_\phi \beta <_\phi \alpha} ({}^{\sigma^{\alpha}}\varpi_E -
{}^{\sigma^{\beta}}\varpi_E)^{-1}\right) ({}^{\sigma^\alpha}\varpi_E -
{}^{\sigma^k}\varpi_E) \prod_{0 \leq i < f-1 } ( \zeta -
{}^{\phi^{i+1}}\zeta)^{-1}.
  \end{split}
\end{equation*}
The first bracket on the right side is just $b_{\sigma^\alpha}$. Therefore
\begin{equation*}
  a_{\sigma^k, \sigma^\alpha} = ({}^{\sigma^k}\varpi_E -
{}^{\sigma^\alpha} \varpi_E) \prod_{0 \leq i < f-1}(\zeta -
{}^{\phi^{i+1}}\zeta)^{-1}.
\end{equation*}
    Again these elements satisfies the second condition of $a$-data in (\ref{definition of a data}). Notice that \begin{equation}\label{delta-1-equation-for-phi-middle-row}
 a_{\sigma^k, \sigma^\alpha} =({}^{\sigma^k}\varpi_E -
{}^{\sigma^\alpha}\varpi_E)\left( \prod_{0 \leq i <f-1}
(-{}^{\phi}b_{\phi^i \sigma^\alpha}
b^{-1}_{\phi^{i+1}\sigma^\alpha})\right)^{-1}.
\end{equation}

\item  If $\alpha = k$ and $j=0$, the equation reads
\begin{equation}\label{delta-1-equation-for-phi-first-row}
  \prod^{s_kf-1}_{i=0} a_{\sigma^k, \phi^i\sigma^k} = {}^\phi
b^{-1}_{\phi^{f-1} \sigma^{ k_l}}b_{\sigma^k}.
\end{equation}
Here $ k_l$ is the $<_\phi$-largest in the orbit of $\sigma^k$. We have
\begin{equation*}
\begin{split}
   {}^\phi b_{\phi^{f-1}\sigma^{ k_l}} &= {}^\phi\left(\prod_{0\leq_\phi\beta<_\phi k_l}\left({}^{\phi^{f-1}\sigma^{ k_l}}\varpi_E-{}^{\phi^{f-1}\sigma^{\beta}}\varpi_E\right)^{-1}
 \prod_{0\leq i<f-1}\left({}^{\phi^{f-1}}\zeta-{}^{\phi^{i}}\zeta\right)^{-1}\right)
 \\
 =&\prod_{0\leq_\phi j<_\phi k}\left({}^{\sigma^{k}}\varpi_E-{}^{\sigma^{j}}\varpi_E\right)^{-1}
 \prod_{k\leq_\phi\beta<_\phi \alpha_l}\left({}^{\sigma^{k}}\varpi_E-{}^{\sigma^{\beta_{+_\phi 1}}}\varpi_E\right)^{-1}
 \prod_{1\leq i<f}\left( \zeta-{}^{\phi^{i}}\zeta \right)^{-1}
\end{split}
\end{equation*} and $$ b_{\sigma^k} = \prod_{0 \leq_\phi j <_\phi k}({}^{\sigma^k} \varpi_E -
{}^{\sigma^j} \varpi_E)^{-1}.$$
Hence the right side of (\ref{delta-1-equation-for-phi-first-row}) is $$\prod_{k \leq_\phi \beta <_\phi  k_l}({}^{\sigma^k}\varpi_E - {}^{\sigma^{\beta_{+_\phi 1}}}\varpi_E)  \prod_{1 \leq i <f}(\zeta - {}^{\phi^i}\zeta),$$ and is equal to the left side. Indeed in the above product each ${}^{\sigma^k}\varpi_E - {}^{\sigma^{\beta_{+_\phi 1}}}\varpi_E$ is coming from (\ref{delta-1-equation-for-phi-middle-row}) and each $\zeta - {}^{\phi^i}\zeta$ is
coming from $a_{\sigma^k, \phi^i \sigma^{ k}}$ for $i =1,
\dots, f-1$.
\end{enumerate}
We have solved both equations in (\ref{delta-1-trivial-cocycle}) for particular $a$-data of the form $a_{g,\sigma^kg}$ and $a_{g,\phi^ig}$, with $g\in \Gamma_F/\Gamma_E$. The other $a$-data are expected to be more complicated to express, but we do not need to deal with them. We conclude that we have proved Proposition \ref{prop trivial cocycle}, and so the splitting invariant $\lambda(\{a_\lambda\},T)$ is trivial. Recall that $\Delta_\mathrm{I}(\gamma)=\left<\lambda(\{a_\lambda\},T), s_T\right>$ for certain $s_T\in \hat{T}$. Therefore we assume from now on that $\Delta_\mathrm{I}$ is always trivial, under suitable choices of basis and $a$-data.

When $E/F$ is totally ramified, we have $\Gamma_{E/F} = \{\sigma^{k}|0\leq k<e\}$ and $\mathfrak{b}=\{\varpi_E^k|0\leq k<e\}$. In this case we choose $a$-data
\begin{equation}\label{delta-1-solution-for-a-totally-ramified}
  a_{\sigma^i, \sigma^j} = {}^{\sigma^i} \varpi_E -
{}^{\sigma^j}\varpi_E.
\end{equation}

\section{Induction as admissible embedding}\label{section Induction as Admissible Embedding}

We would like to change our notations until Remark \ref{changing back from right to left}, that an element $g\in W_F$ represents a right coset $W_Eg\in W_E\backslash W_F$ instead of a left coset adopted from section \ref{section root system}. The reason is that the formulae would match those in \cite{LS} by doing so. A root $\left[\begin{smallmatrix}g\\h\end{smallmatrix}\right]$ then stands for $\left[\begin{smallmatrix}W_Eg\\W_Eh\end{smallmatrix}\right]$. It evaluates at $t\in E^\times$ as $\left[\begin{smallmatrix}g\\h\end{smallmatrix}\right](t)={}^{g^{-1}}t({}^{h^{-1}}t)^{-1}$. The $W_F$-action on the roots is given by $
{}^w \left[ \begin{smallmatrix}
 g\\h
\end{smallmatrix}
\right]  =\left[ \begin{smallmatrix}
gw^{-1}\\
hw^{-1}
\end{smallmatrix}\right]$, for $w\in W_F$. This change does not affect the symmetry of the roots. In Remark \ref{changing back from right to left}, we will justify that we can change back our notations using left cosets.

In this section we recall some basic facts about admissible embeddings defined in section (2.6) of \cite{LS}. Let $T$ be a maximal torus of $G$ defined over $F$. By taking a conjugate of $T$ in $G$, which is still denoted by $T$ for brevity, we assume that $T$ is contained in a Borel subgroup $B$ defined over $F$. Choose a $W_F$-invariant splitting $(\mathcal{T},\mathcal{B},\{\hat{X}_\alpha\})$ of $\hat{G}$ and an isomorphism $\iota:\hat{T}\rightarrow \mathcal{T}$ that maps the basis determined by $\hat{B}$ to the basis determined by $\mathcal{B}$. For notation convenience we usually omit $\iota$ and write $t=\iota(t)\in \mathcal{T}$ for $t\in \hat{T}$, but we should bear in mind that $\hat{T}$ and $\mathcal{T}$ may have different $W_F$-actions. We denote the action of $w\in W_F$ on $\hat{G}\supseteq \mathcal{T}$ by $w_{\hat{G}}$ and on $\hat{T}$ by $w_{\hat{T}}$.


An \emph{admissible embedding} from ${}^L \!T$ to ${}^L \!G$ is a morphism of groups $\chi: {}^L\!T \rightarrow {}^L\!G$ of the form
\begin{equation*}
\chi(t \rtimes w) = t \bar{\chi}(w) \rtimes w\text{ for all }t\rtimes w\in {}^LT,
\end{equation*}
for some map $\bar{\chi}: W_F \rightarrow \hat{G}$. By expanding the group law $\chi(s\rtimes v)\chi(t\rtimes w)=\chi((s\rtimes v)(t\rtimes w))$, we can show that
\begin{equation}\label{condition on bar-chi}
(\text{Int} \bar{\chi}(v))({}^{v_{\hat{G}}}t) ={}^{v_{\hat{T}}} t\text{ and }\bar{\chi}(vw) = \bar{\chi}(v)^{v_{{\hat{G}}}}  \bar{\chi}(w)\text{ for all }t \in \hat{T},\,v, w \in W_F.
 \end{equation}
 Conversely if $\bar{\chi}$ satisfies (\ref{condition on bar-chi}), then $\chi$ is an admissible embedding. Clearly $\bar{\chi}$ has image in $N_{\hat{G}}(\mathcal{T})$. 
Given a subgroup $\mathcal{H}$ of $\hat{G}$, we call two admissible embeddings $\chi_1$, $\chi_2$ being $\mathrm{Int}(\mathcal{H})$-equivalent if there is $x\in \mathcal{H}$ such that $$\chi_1(t\rtimes w)=(x\rtimes 1)\chi_2(t\rtimes w)(x\rtimes 1)^{-1}\text{ for all }t\rtimes w\in {}^LT,$$ or equivalently, using (\ref{condition on bar-chi}), that $$\bar{\chi_1}(w)=x\bar{\chi}_2(w){}^{w_{\hat{G}}}x^{-1}\text{ for all }w\in W_F.$$

\begin{rmk}
We can choose splittings of $G$ and $\hat{G}$ such that we have a duality on the bases of $T$ and $\mathcal{T}$ for explicit computations. For example, we can choose a basis of $\mathcal{T}$ for constructing the section $n_S$ defined in section \ref{section The splitting invariant}. Our main results would be independent of these choices. For instance, the $\hat{G}$-conjugacy class of an admissible embedding is independent of the choices of the Borel subgroup $B$ containing $T$ and the splitting $(\mathcal{T}, \mathcal{B}, \{ \hat{X}_\alpha \})$ of $\hat{G}$ (see (2.6.1) and (2.6.2) of \cite{LS}).
\qed\end{rmk}

Langlands and Shelstad constructed certain admissible embeddings using $\chi$-data. Such data are defined in section (2.5) of \cite{LS}. We will give the construction of such embeddings in section \ref{section langlands shelstad chi data}. Let us assume such construction for a moment, and denote by $\mathrm{AE}({}^L\!T, {}^L\!G, (\hat{T}, \hat{B}) \rightarrow (\mathcal{T}, \mathcal{B}))$, or just $ \mathrm{AE}({}^L\!T, {}^L\!G)$ for simplicity, the set of admissible embeddings ${}^L\!T \rightarrow  {}^L\!G$ associated to the choice of the isomorphism $(\hat{T}, \hat{B}) \xrightarrow{} (\mathcal{T}, \mathcal{B})$. It is not difficult to describe this set.

\begin{prop}\label{set of adm embed as H1}
The set $\mathrm{AE}({}^L \! T,{}^L \! G)$ is a $Z^1(W_F,\hat{T})$-torsor, and the set of its $\mathrm{Int}(\mathcal{T})$-equivalence classes  $\mathrm{Int}(\mathcal{T})\backslash \mathrm{AE}({}^L \! T, {}^L \! G)$ is an $H^1(W_F, \hat{T})$-torsor.
\end{prop}

\begin{proof}
We fix an embedding $\chi_0\in \mathrm{AE}({}^L \! T,{}^L \! G)$ and take $\bar{\chi}_0: W_F  \rightarrow  \hat{G}$ as in the definition of admissible embedding. Then for each $\chi \in \mathrm{AE}({}^L \! T,{}^L \! G)$, the difference $\bar{\chi} \bar{\chi_0}^{-1}$ is a 1-cocycle of $W_F$ valued in $\hat{T}$. Indeed for a fixed $w\in W_F$ both $\bar{\chi}(w)$ and $ \bar{\chi_0}(w)$ project to the same element in $\Omega({\hat{G}},\mathcal{T})=N_{\hat{G}}(\mathcal{T})/\mathcal{T}$. Using (\ref{condition on bar-chi}) we have that
\begin{equation*}\begin{split}
\bar{\chi} \bar{\chi_0}^{-1}(vw)
&= \bar{\chi}(v) {}^{v_{\hat{G}}} \bar{\chi}(w) {}^{v_{\hat{G}}} \bar{\chi_0}(w)^{-1} \bar{\chi_0}(v)^{-1}
\\
&= \bar{\chi}(v) \bar{\chi_0}(v)^{-1} {}^{v_{\hat{T}}} \bar{\chi}(w) {}^{v_{\hat{T}}} \bar{\chi_0}(w)^{-1}
\end{split}
\end{equation*}
for all $v,w\in W_F$. We can readily verify that the map $$\mathrm{AE}({}^L \! T,{}^L \! G)\rightarrow Z^1(W_F, \hat{T}),\, \chi\mapsto \bar{\chi} \bar{\chi_0}^{-1},$$ is bijective. From the equality
$t\bar{\chi}(v) {}^{v_{\hat{G}}} t^{-1} \bar{\chi_0}^{-1}(v) = t \bar{\chi}(v) \bar{\chi_0}^{-1}(v) {}^{v_{\hat{T}}} t^{-1}$ for all $t\in \mathcal{T}$, we know that two embeddings are $\mathrm{Int}(\mathcal{T})$-equivalent if and only if the corresponding 1-cocycles differ by a coboundary in $Z^1(W_F, \hat{T})$.
\end{proof}

\begin{rmk}\label{an explicit admissible embedding}
For $G=\mathrm{GL}_n$ we can construct an explicit embedding ${}^LT \rightarrow {}^LG$. Choose $\mathcal{T}$ to be the diagonal subgroup of $\hat{G}$. We embed $\hat{T}$ into $\hat{G}$ with image $\mathcal{T}$ and define $$W_F \rightarrow N_{\hat{G}}(\hat{T}),\,w \mapsto N(w),$$ where $N(w)$ is a permutation matrix, the matrix with entries being either 0 or 1. We assign $N(w)$ according to the $W_F$-action on $\hat{T}$, or more precisely that $\mathrm{Int}(N(w))t={}^{w_{\hat{T}}}t$ for all $t\in \hat{T}$. Clearly the map ${}^LT \rightarrow {}^LG,\,t \rtimes w \mapsto t N(w) \times w$,
defines an admissible embedding.
\qed\end{rmk}

We now take $T$ to be the elliptic torus $\mathrm{Res}_{E/F}\mathbb{G}_m$. By local class field theory \cite{Tate-NTB} and Shapiro's Lemma (see the Exercise in VII \S5 of \cite{Serre-local-fields}), we have a special case of Langlands correspondence for torus
\begin{equation}\label{hom as H1}
\text{Hom}(E^\times, \mathbb{C}^\times)\cong\text{Hom}(W_E, \mathbb{C}^\times)  \cong H^1(W_F, \hat{T}) .
\end{equation}
The precise correspondence is given as follows. Suppose that $\xi$ is a character of $E^\times$, regarded as a character of $W_E$ by class field theory. Take a collection of coset representatives $\{g_1,\dots,g_n\}$ of $W_E\backslash W_F$. Define for each $g_i$ a map $u_{g_i}:W_F\rightarrow W_E$ given by
\begin{equation}\label{pre-transfer}
g_iw=u_{g_i}(w)g(g_i,w)\text{ for }g(g_i,w)\in \{g_1,\dots,g_n\}.
\end{equation}
Then define
\begin{equation*}
\tilde{\xi}:W_F\rightarrow \hat{T} ,\,w\mapsto \left(\xi(u_{g_1}(w)),\dots,\xi(u_{g_n}(w))
\right).
\end{equation*}
We can check that $\tilde{\xi}$ is a 1-cocycle in $Z^1(W_F,\hat{T})$, and different choices of coset representatives give cocycles different from $\tilde{\xi}$ by a 1-coboundary. Hence the 1-cohomology class of $\tilde{\xi}$ is defined. By abusing of language, we call $\tilde{\xi}$ a \emph{Langlands parameter} of $\xi$. Moreover, combining Proposition \ref{set of adm embed as H1} and (\ref{hom as H1}), we have
\begin{equation}\label{set of adm embed as Hom}
\text{Hom}(E^\times, \mathbb{C}^\times) \cong \mathrm{Int}(\mathcal{T})\backslash\mathrm{AE}({}^L\!T, {}^L\!G).
\end{equation}
Explicitly, if we have a character $\mu$ of $E^\times$, then define
\begin{equation*}
\chi:{}^LT\rightarrow {}^LG,\,t\rtimes w\mapsto t\left(
\begin{smallmatrix}\mu(u_{g_1}(w))&&\\&\ddots&\\&&\mu(u_{g_n}(w))\end{smallmatrix}
\right)N(w)\times w.
\end{equation*}
Here $N(w)$ is the permutation matrix introduced in Remark \ref{an explicit admissible embedding}. Notice that this bijection is non-canonical. We consider the composition
\begin{equation*}
 H^1(W_F, \hat{T})\times \mathrm{Int}(\mathcal{T})\backslash \mathrm{AE}({}^L\!T,{}^L\!G)
\rightarrow \mathrm{Int}(\hat{G})\backslash\mathrm{Hom}_{W_F}(W_F, {}^L\!G)
\end{equation*}
such that $(\tilde{\xi}, \chi) \mapsto \chi \circ \tilde{\xi}$. Combining the bijections (\ref{hom as H1}) and (\ref{set of adm embed as Hom}), we have the following result.
\begin{prop}\label{composition as induction}
Suppose that $\xi$ and $\mu$ come from $\tilde{\xi}$ and $\chi$ by the bijections (\ref{hom as H1}) and (\ref{set of adm embed as Hom}) respectively. The composition $ \chi \circ \tilde{\xi}$, when projected to $\hat{G}\cong \mathrm{GL}_n(\mathbb{C})$, is isomorphic to $\mathrm{Ind}_{E/F}(\xi\mu)$ as a representation of $W_F$.
\end{prop}
\begin{proof}
Choose a suitable basis on the representation space of $\mathrm{Ind}_{E/F}(\xi\mu)$. For example, if we realize our induced representation by the subspace of functions $$\{f:W_F\rightarrow\mathbb{C}|f(xg)=\xi\mu(x)f(g)\text{ for all }x\in W_E,g\in W_F\},$$ then we choose those $f_i$ determined by $f_i(g_j)=\delta_{ij}$ (Kronecker delta) as basis vectors. The matrix coefficient of $\mathrm{Ind}_{E/F}(\xi\mu)$ is therefore $$
w\mapsto\left(
\begin{smallmatrix}\mu\xi(u_{g_1}(w))&&\\&\ddots&\\&&\mu\xi(u_{g_n}(w))\end{smallmatrix}
\right)N(w)\text{, for all }w\in W_F.
$$
This is the same as the image of $\chi \circ \tilde{\xi}$.
\end{proof}
\begin{rmk}\label{recover}
We can recover ${\xi}$ from $\text{Ind}_{E/F}{\xi}$ as follows. We choose the first $k$ coset representatives $g_1=1,g_2\dots,g_k$ to be those in the normalizer $N_{W_F}(W_E)=\mathrm{Aut}_F(E)$. By choosing suitable basis, we then consider the matrix coefficients of $\mathrm{Res}_{E/F}\mathrm{Ind}_{E/F}{\xi}$. The first $k$ diagonal entries are always non-zero and give the characters ${\xi}^{g_i}$, $i=1,\dots,k$.
\qed\end{rmk}

\section{Langlands-Shelstad $\chi$-data}\label{section langlands shelstad chi data}

In this section we recall the construction of admissible embeddings ${}^LT \rightarrow {}^LG$ given in chapter 2 of \cite{LS}. The actual construction applies to general connected reductive algebraic groups defined and quasi-split over $F$. We take a maximal torus $T$ in $G$ also defined over $F$ and a maximal torus $\mathcal{T}$ of $\hat{G}$ as part of a chosen splitting $\mathrm{spl}_{\hat{G}}=(\mathcal{T},\, \mathcal{B},\,\{\hat{X}_\alpha\})$ for $\hat{G}$. We emphasize again that different choices of auxiliary data yield $\mathrm{Int}(\hat{G})$-equivalent admissible embeddings. For computational convenience we choose $\mathcal{T}$ to be the diagonal subgroup and $\mathcal{B}$ to be the group of upper triangular matrices. The tori $\hat{T}$ and $\mathcal{T}$ are isomorphic as groups but equipped with different $W_F$-actions.

Recall, from Proposition \ref{set of adm embed as H1}, that the existence of an embedding $\chi: {}^LT \rightarrow {}^LG$ with a fixed restriction $\hat{T}\rightarrow \mathcal {T}$ is equivalent to the existence of a 1-cocycle $\bar{\chi}\in Z^1(W_F,N_{\hat{G}}(\mathcal{T}))$ as in (\ref{condition on bar-chi}). Following section 2.6 of \cite{LS}, we construct an admissible one directly as follows. Recall the section $ \Omega(G, \mathbf{T}) \rightarrow N_G(\mathbf{T})$ defined in section \ref{section The splitting invariant}. We apply the same construction and get a section $ \Omega(\hat{G}, \mathcal{T}) \rightarrow N_{\hat{G}}(\mathcal{T})$ with respect to the choices of simple roots in $\mathrm{spl}_{\hat{G}}$. In turn this section yields a map
\begin{equation*}
n:W_F\rightarrow N_{\hat{G}}(\mathcal{T})\rtimes W_F,\,w\mapsto n(w)=\bar{n}(w)\rtimes w.
\end{equation*}
This map is not necessarily a morphism of groups, yet $\bar{n}$ satisfies the first equation in (\ref{condition on bar-chi}) in place of $\bar{\chi}$. We write \begin{equation}\label{Borel-pos 2 cocycle}
t_\mathrm{b}(v,w)= n(v)n(w)n(vw)^{-1} = \bar{n}(v){}^{v_{\hat{G}}}\bar{n}(w)\bar{n}(vw)^{-1}.
\end{equation}
This is a 2-cocycle of $W_F$ with values in $\{\pm1\}^n\subseteq \mathcal{T}$, by Lemma 2.1.A of \cite{LS}. Hence the problem of seeking such $\bar{\chi}$ is equivalent to looking for a map $r_\mathrm{b}:W_F \rightarrow \hat{T}$ that splits $t^{-1}_\mathrm{b}$. In other words, the map $r_\mathrm{b}$ should satisfy
\begin{equation}\label{def of r-b that splits t-b}
r_\mathrm{b}(v) {}^{{v}_{\hat{T}}}r_\mathrm{b}(w)r_\mathrm{b}(vw)^{-1} = t_\mathrm{b}(v,w)^{-1}.
\end{equation}

The idea in \cite{LS} of constructing such splitting $r_\mathrm{b}$ is to choose a set of characters (recall the definition of the fields $E_\lambda$ and $E_{\pm\lambda}$ in section \ref{section root system})
$$\{ \chi_{\lambda} \}_{\lambda \in \Phi},\,\chi_{\lambda}:E^\times_{\lambda} \rightarrow \mathbb{C}^\times,$$
called \emph{$\chi$-data}, such that the following conditions hold.
\begin{enumerate}[(i)]
\item For each $\lambda\in\Phi$, we have $ \chi_{-\lambda} = \chi_\lambda^{-1}$ and $\chi_{{}^w\lambda} = \chi_\lambda^{w^{-1}}$ for all $w \in W_F$. \label{chi-data condition transition}
\item If $\lambda$ is symmetric, then $\chi|_{E^\times_{\pm \lambda}}$ equals the quadratic character $\delta_{E_{\lambda}/E_{\pm \lambda}}$ attached to the extension ${E_{\lambda}/E_{\pm \lambda}}$. \label{chi-data condition quadratic}
\end{enumerate}
If we choose $\mathcal{R}_\pm=\mathcal{R}_\mathrm{sym}\sqcup \mathcal{R}_{\mathrm{asym}/\pm}$ to be a subset of $\Phi$ consisting of representatives of
$ W_F\backslash\Phi_\text{sym}$ and $ W_F\backslash\Phi_{\mathrm{asym}/\pm}$, then by condition (\ref{chi-data condition transition}) the set of $\chi$-data depends only on the subset $\{ \chi_\lambda\}_{\lambda \in \mathcal{R}_\pm}$. As before, we may regard each $\chi_\lambda$ as a character of the Weil group $W_{E_\lambda}$. Following section (2.5) of \cite{LS}, we define for each $\lambda \in \mathcal{R}_\pm$ a map
\begin{equation}\label{definition of r-lambda}
r_\lambda:W_F \rightarrow \mathcal{T},\, w \mapsto \prod_{g_i\in W_{\pm\lambda}\backslash W_F}\chi_\lambda(v_1(u_{g_i}(w))) ^{g_i^{-1} \lambda},
\end{equation}
where $u_{g_i}$ is the map (\ref{pre-transfer}) for $W_{\pm\lambda}\backslash W_F$ and $v_1$ is defined similarly for $W_{+\lambda}\backslash W_{\pm\lambda}$. We then define
\begin{equation}\label{definition of the splitting rg}
r_\mathrm{g}=\prod_{\lambda \in \mathcal{R}_\pm} r_\lambda.
\end{equation}
By Lemma 2.5.A of \cite{LS}, such construction yields a 2-cocycle
\begin{equation}\label{Galois-pos 2 cocycle}
t_\mathrm{g}(v,w)=r_\mathrm{g}(v){}^{{v}_{\hat{T}}}r_\mathrm{g}(w)r_\mathrm{g}(vw)^{-1} \in Z^2(W_F, \{\pm1\}^n).
\end{equation}
In constructing the 2-cocycles (\ref{Borel-pos 2 cocycle}) and (\ref{Galois-pos 2 cocycle}), we implicitly used two different notions of gauges (defined right before Lemma 2.1.B of \cite{LS}) on the root system $\Phi$. To relate them we recall a map $s=s_{\mathrm{b}/\mathrm{g}}:W_F \rightarrow \{\pm1\}^n$ in section (2.4) of \cite{LS} such that
\begin{equation}\label{def of s-b-g}
s(v){}^{{v}_{\hat{T}}}s(w)s(vw)^{-1} = t_\mathrm{b}(v,w)t_\mathrm{g}(v,w)^{-1}.
\end{equation}
Write $r_\mathrm{b} = s_{\mathrm{b}/\mathrm{g}}r_\mathrm{g}$ and $\bar{\chi} = r_\mathrm{b}\bar{n}$.

 \begin{prop}\label{define an admissible embedding}
 The map $\chi$ defines an admissible embedding ${}^LT\rightarrow {}^LG$.
\end{prop}
\begin{proof}It suffices to show that $\bar{\chi}$ satisfies the two conditions in (\ref{condition on bar-chi}). The first condition is just from the definition of $n(w)$, while the second condition is a straightforward calculation using (\ref{Borel-pos 2 cocycle}), (\ref{def of r-b that splits t-b}), (\ref{Galois-pos 2 cocycle}), and (\ref{def of s-b-g}).
\end{proof}

Using the bijection in Proposition \ref{orbit of roots as double coset} we have the natural bijection
\begin{equation}\label{bijection between representatives}
  \mathcal{R}_\pm=\mathcal{R}_\text{sym}\bigsqcup\mathcal{R}_{\mathrm{asym}/\pm}\rightarrow \mathcal{D}_\pm=\mathcal{D}_\text{sym}\bigsqcup\mathcal{D}_{\mathrm{asym}/\pm},\,\lambda=\left[\begin{smallmatrix}1\\g\end{smallmatrix}\right]\mapsto g.
\end{equation}
We usually write $\chi_g=\chi_\lambda$ using the bijection above and denote a collection of $\chi$-data $\{\chi_\lambda\}_{\lambda\in\mathcal{R}_\pm}$ by $\{\chi_g\}_{g\in\mathcal{D}_\pm}$. We also denote by $E_g$ and $E_{\pm g}$ the fields $E_ \lambda$ and $E_{\pm \lambda}$ respectively.

\section{Explicit $\Delta_{\mathrm{III}_2}$}\label{section explicit-Delta-III-2}

Let $\chi:{}^LT\rightarrow {}^LG$ be the admissible embedding defined by some chosen $\chi$-data $\{\chi_g\}_{g\in\mathcal{D}_\pm}$, as in section \ref{section langlands shelstad chi data}. Let $\xi$ be a character of $E^\times$ and $\tilde{\xi}\in Z^1(W_F, \hat{T})$ be a Langlands parameter of $\xi$. The choice of $\tilde{\xi}$ would be irrelevant. In Proposition \ref{composition as induction} we described the induced representation $\mathrm{Ind}_{E/F}\xi$ as certain embedding of the image of $\tilde{\xi}$ into $\mathrm{GL}_n(\mathbb{C})$. Here we have the reverse: given a representation of $W_F$ defined by an admissible embedding, we would recover the character that induces such representation.
\begin{prop}\label{recover the character from induction}
Given $\chi$-data $\{\chi_g\}=\{\chi_g\}_{g\in\mathcal{D}_\pm}$, we define
\begin{equation*}
\mu =\mu_{\{\chi_g\}}=\prod_{[g]\in( W_E \backslash W_F/W_E)'}\mathrm{Res}^{E^\times_g}_{E^\times}\chi_g.
\end{equation*}
Let $\chi$ be the admissible embedding defined by $\{\chi_g\}_{g\in\mathcal{D}_\pm}$. Then for all character $\xi$ of $E^\times$, the composition $$W_F \xrightarrow{\tilde{\xi}} {}^LT \xrightarrow{\chi} {}^LG \xrightarrow{\mathrm{proj}}\mathrm{GL}_n(\mathbb{C})$$ is isomorphic to $\mathrm{Ind}_{E/F}(\xi \mu)$ as a representation of $W_F$.
\end{prop}

\begin{rmk} \label{two remarks about product}
\begin{enumerate}[(i)]
 \item Notice that the product in Proposition \ref{recover the character from induction} is independent of the representatives $g\in W_E\backslash W_F$ of $[g]\in( W_E \backslash W_F/W_E)'$. Indeed if ${}^x\left[\begin{smallmatrix}1\\g\end{smallmatrix}\right]=\left[\begin{smallmatrix}1\\h\end{smallmatrix}\right]$ for some $x\in W_E$, then ${}^xE_g=E_h$ and so $\mathrm{Res}^{E^\times_h}_{E^\times}\chi_h=\mathrm{Res}^{{}^xE^\times_g}_{E^\times}\chi_g^{x^{-1}}=\mathrm{Res}^{E^\times_g}_{E^\times}\chi_g$,
by (\ref{chi-data condition transition}) in the definition of $\chi$-data. \label{remark indep of choices of double coset}
\item
Suppose that we have fixed a character $\xi$ of $E^\times$. Take a subset $\{g_1=1,g_2\dots,g_k\}$ of coset representatives of $W_E\backslash W_F$ in the normalizer $N_{W_F}(W_E)=\mathrm{Aut}_F(E)$ and write $\mu_1=\mu_{\{\chi_g\}}$ as in Proposition \ref{recover the character from induction}. Then all other characters $\mu_j$ such that $\mathrm{Ind}_{E/F}(\xi \mu_j)\cong  \mathrm{Ind}_{E/F}(\xi \mu_1)$ are of the form $\mu_j=\xi^{g_j-1}\mu_1$, $j=1,\dots,k$. This character $\mu_j$ also has a factorization as in Proposition \ref{recover the character from induction}. We can arrange the factors such that they are the same as those of $\mu$ except the character $\chi_{g_j}$ is changed according to the following.
\begin{enumerate}[(i)]
\item If ${g_j}$ is symmetric, then $\chi_{g_j}$ is replaced by $\xi^{{g_j}-1}\chi_{g_j}$.
\item If ${g_j}$ is asymmetric, then $\chi_{g_j}$ is replaced by $\xi^{{g_j}}\chi_{g_j}$ and so $\chi_{{g_j}^{-1}}$ by $\xi^{-1}\chi_{{g_j}^{-1}}$.
\end{enumerate}
\end{enumerate}
\qed\end{rmk}

\begin{proof} (of Proposition \ref{recover the character from induction})
In this proof we abbreviate $H = W_F$ and $K = W_E$. For each $\lambda=\left[\begin{smallmatrix}1\\g\end{smallmatrix}\right]$, we write $K_g = K \cap g^{-1}Kg$, which equals $W_{\lambda}$. If $[g] \in (K \backslash H/K)_{\mathrm{sym}}$, then, because $KgK=Kg^{-1}K$, we can replace $g$ by an element in $Kg$ such that $g^2\in K$. Subsequently we have $g\in W_{\pm\lambda}$ and $g^2\in K_g= W_{\lambda}$. We denote by $K_{\pm g}$ the group generated by $K_g$ and $g$, which equals $W_{\pm \lambda}$.

By (\ref{chi-data condition transition}) in the definition of $\chi$-data, we rewrite the product in Proposition \ref{recover the character from induction} as
\begin{equation}\label{expanded recover the character from induction}
\prod_{[g]\in(K \backslash H /K)_\mathrm{asym/\pm}} ( \mathrm{Res}^{E^\times_{g}}_{E^\times} \chi_g) ( \mathrm{Res}^{{}^{g}E_{g}^\times}_{E^\times} \chi^{g^{-1}}_g) ^{-1} \prod_{[g]\in(K \backslash H /K)_\text{sym}} ( \mathrm{Res}^{E^\times_{g}}_{E^\times} \chi_g ).
\end{equation}
Recall that our dual group $\mathcal{T}$ is the diagonal subgroup. In order to check that $\chi$ gives rise to a character $\mu$ as (\ref{expanded recover the character from induction}), it is enough to consider the first entry of $r_{\mathrm{g}}$ (see (\ref{definition of the splitting rg}) and Remark \ref{recover}). From (\ref{definition of r-lambda}) we have
\begin{equation*}
\begin{split}
r_{\mathrm{g}}(w)=&\left( \prod_{[g]\in(K \backslash H /K)_\mathrm{asym/\pm}} \prod_{g_i \in K_g\backslash H}  \chi_g(u_{g_i}(w))^{\left[\begin{smallmatrix}g_i\\gg_i\end{smallmatrix}\right]}\right)
\\
&\left(\prod_{[g]\in(K \backslash H /K)_\text{sym}}  \prod_{g_i \in K_{\pm g}\backslash H}  \chi_g(v_1u_{g_i}(w))^{\left[\begin{smallmatrix}g_i\\gg_i\end{smallmatrix}\right]}\right).
\end{split}
\end{equation*}
By restricting to $ W_E$, we get the first entry of $r_{\mathrm{g}}(w)$ as
\begin{equation}\label{first entry of r as 4 products}
\begin{split}
r_{\mathrm{g}}(w)_1=&\left(\prod_{[g] \in (K \backslash H /K)_\mathrm{asym/\pm}} \left(\prod_{g_i \in K_g \backslash K }\chi_g(u_{g_i}(w))\right)\left(\prod_{\begin{smallmatrix}{g_i \in K_g \backslash H}\\{gg_i\in K}\end{smallmatrix}} \chi_g(u_{g_i}(w))^{-1}\right)\right)
\\
&\left(\prod_{[g] \in (K \backslash H /K)_\text{sym}} \left(\prod_{\begin{smallmatrix}{g_i \in K_{\pm g} \backslash K}\\{}\end{smallmatrix}}\chi_g(v_1u_{g_i}(w))\right)\left(\prod_{\begin{smallmatrix}{g_i \in K_{\pm g} \backslash H}\\{gg_i\in K}\end{smallmatrix}} \chi_g(v_1u_{g_i}(w))^{-1}\right)\right).
\end{split}
\end{equation}

We now analyze the products in (\ref{first entry of r as 4 products}) and match them to those in (\ref{expanded recover the character from induction}). First, for $[g]\in (K \backslash H /K)_\mathrm{asym/\pm}$, the first product of (\ref{first entry of r as 4 products}) $$\prod_{g_i \in K_g \backslash K } u_{g_i}(w),\,w \in K,$$  is the transfer map $T^K_{K_g}:K^{\mathrm{ab}}\rightarrow (K_g)^{\mathrm{ab}}$. By class field theory \cite{Tate-NTB}, it corresponds to the inclusion $E^\times \hookrightarrow E^\times_{g}$. Therefore $$\prod_{g_i \in K_g \backslash K }\chi_g(u_{g_i}(w)) = \text{Res}^{E^\times_{g}}_{E^\times} \chi_g(w),$$ which is the first factor in (\ref{expanded recover the character from induction}).

Next we consider the inverse of the second product of (\ref{first entry of r as 4 products}) $$\prod_{\begin{smallmatrix}{g_i \in K_g \backslash H}\\{gg_i\in K}\end{smallmatrix}}u_{g_i}(w),\,w\in K.$$ For $g_i \in K_g \backslash H$ such that $gg_i \in K$, we can write $g_i=g^{-1}x_i$ for some $x_i$ running through a set in $K$ of representatives of $  K _{g^{-1}} \backslash K$. On one hand, if $u_{x_i}$ is the map  (\ref{pre-transfer}) for $  K_{g^{-1}} \backslash K$, then we have
\begin{equation}\label{first relation in proof of recover the character from induction}
g^{-1}(x_iw) = g^{-1}(u_{x_i}(w)x_{j(x_i,w)}),
\end{equation}
where $u_{x_i(w)} \in K_{g^{-1}}$. On the other hand, by regarding $g^{-1}x_i \in K_g  \backslash H$ we have
\begin{equation} \label{second relation in proof of recover the character from induction}
g^{-1}x_iw = u_{g^{-1}x_i}(w)g_{j(g^{-1}x_i,w)},
\end{equation}
where $u_{g^{-1}x_i}(w) \in K_g$ and $g_{j(g^{-1}x_i,w)}$ is of the form $g^{-1}x_j$ for some $j$. By comparing (\ref{first relation in proof of recover the character from induction}) and (\ref{second relation in proof of recover the character from induction}) we have
$g^{-1}u_{x_i}(w)g = u_{g^{-1}x_i}(w)$. Therefore $$\prod_{\begin{smallmatrix}{g_i \in K_g \backslash H}\\{gg_i\in K}\end{smallmatrix}} u_{g_i}(w) = g^{-1}\left(\prod_{\begin{smallmatrix}{g_i \in K_g \backslash H}\\{gg_i\in K}\end{smallmatrix}} u_{x_i}(w)\right)g = g^{-1} T^{K}_{K_{g^{-1}}}(w)g.$$ This implies that $$\prod_{\begin{smallmatrix}{g_i \in K_g \backslash H}\\{gg_i\in K}\end{smallmatrix}}\chi_g( u_{g_i}(w)) = \chi_g^{g^{-1}}(T^{K}_{K_{g^{-1}}}(w)) = (\text{Res}^{{}^gE^\times_{g}}_{E^\times} \chi_g^{g^{-1}})(w),$$ which is the inverse of the second factor in (\ref{expanded recover the character from induction}).

Finally, for $[g]\in (K\backslash H/K)_\mathrm{sym}$, we choose coset representatives $g_1,\dots,g_k$, $g_1g,\dots,g_kg$ for  $K_g\backslash H$ such that $g_1,\dots,g_k$ are those of $K_{\pm g}\backslash H$. Moreover we can assume that $$g_1,\dots,g_h,gg_{h+1},\dots,gg_{2h}\in K.$$ Hence the third product in (\ref{first entry of r as 4 products}) is
\begin{equation}\label{third product in chi-data}
\prod_{\begin{smallmatrix}{g_i \in K_{\pm g} \backslash K}\end{smallmatrix}}\chi_g(v_1(u_{g_i}(w)))
=
\prod_{i=1}^h\chi_g(v_1(u_{g_i}(w))).
\end{equation}
Here $u_{g_i}$ is the map (\ref{pre-transfer}) for $K_{\pm g}\backslash H$ and $v_1 u_{g_i}$ is the one for $ K_g\backslash H$. For the fourth product in (\ref{first entry of r as 4 products}), because $\chi_g^g=\chi_g^{-1}$ (by (\ref{chi-data condition quadratic}) in the definition of $\chi$-data) and $g(v_1(u_{g_i}(w)))g^{-1}=v_1(u_{gg_i}(w)),$ we have indeed
\begin{equation}\label{fourth product in chi-data}
\prod_{\begin{smallmatrix}{g_i \in K_{\pm g} \backslash H}\\{gg_i\in K}\end{smallmatrix}} \chi_g(v_1u_{g_i}(w))^{-1}
=
\prod_{i=h+1}^{2h}\chi_g(v_1(u_{gg_i}(w))).
\end{equation}
Therefore the product of (\ref{third product in chi-data}) and (\ref{fourth product in chi-data}) is
$\chi_g(T^{K}_{K_g}(w))=(\mathrm{Res}^{E^\times_{g}}_{E^\times} \chi_g)(w)$,
which is the last factor of (\ref{expanded recover the character from induction}).
\end{proof}

We have similar result for $H=\mathrm{Res}_{K/F}\mathrm{GL}_m$ as follows. We first regard $H=\mathrm{GL}_m$ as an reductive group over $K$. Let $\tilde{\xi}_K \in Z^1(W_K, \hat{T})$ be a Langlands parameter of the character $\xi$ of $E^\times$. Let $\mathcal{D}(K)_\pm$ be the sub-collection in $ W_K / W_E$ defined analogously as $\mathcal{D}_\pm= \mathcal{D}(F)_\pm$. Take the sub-collection $\{ \chi_g  \}_{g\in \mathcal{D}(K)_\pm}$ of $\chi$-data and define the admissible embedding $$(\chi_H)_K : {}^LT_K:=\hat{T} \rtimes W_K \rightarrow {}^LH_K:=\hat{H} \times W_K.$$ Then similar to Proposition \ref{recover the character from induction}, the composition
\begin{equation}\label{admissible-embedding-induction-for-H}
  W_K \xrightarrow{\tilde{\xi}_K} {}^LT_K \xrightarrow{(\chi_H)_K} {}^LH_K \rightarrow \text{GL}_m(\mathbb{C})
\end{equation}
is isomorphic to $\text{Ind}_{E/K}(\xi \mu_H)$ as a representation of $W_K$, where
\begin{equation*}
  \mu_H = \prod_{[g] \in (W_E \backslash W_K / W_E)'} \text{Res}{}^{E_g^\times}_{E^\times} \chi_g.
\end{equation*}
We also write $\mu_G$ to be the character $\mu$ in Proposition \ref{recover the character from induction}.

\begin{cor}\label{definition of delta III2}
  For all $\gamma \in E^\times$ regular, we have  $$ \Delta_\mathrm{III_2}(\gamma) = \mu_G(\gamma)^{-1} \mu_H(\gamma)= \prod_{[g]\in W_E \backslash W_F / W_E - W_E \backslash W_K / W_E} \left( \mathrm{Res}^{E_g^\times}_{E^\times} \chi_g \right)(\gamma)^{-1}.$$
\end{cor}
\begin{proof}
We regard both $G$ and $H$ as reductive groups over $F$. Let $\chi_G:{}^LT\rightarrow {}^LG$ be the admissible embedding defined by $\{\chi_\lambda\}_{\lambda\in \Phi(G,T)}$ and $\chi_H:{}^LT\rightarrow {}^LH$ be the one defined by $\{\chi_\lambda\}_{\lambda\in \Phi(H,T)}$. Recall by definition in (3.5) of \cite{LS} that $\Delta_\mathrm{III_2}(\gamma) = \langle \mathbf{a}, \gamma \rangle$. Here $\mathbf{a}$ is the class in $H^1(W_F, \hat{T})$ defined by the cocycle $a$ satisfying $$\hat{\xi} \circ \chi_H = a\chi_G.$$ We make use of the bijection (\ref{set of adm embed as Hom}) defined by torsor. Explicitly we consider the commutative diagram
\begin{equation*}
  \xymatrixcolsep{4pc}\xymatrix{
&W_K    \ar[r]^{\tilde{\xi}_K}     \ar[d]           & {}^LT_K     \ar[r]^{(\chi_H)_K} \ar[d]
&{}^LH_K  \ar[r] \ar[d]       &\mathrm{GL}_m(\mathbb{C})  &\\
&W_F     \ar[r]^{\tilde{\xi}}     & {}^LT \ar[r]^{\chi_H}  & {}^LH \ar[r]^{\hat{\xi}}&\hat{G}\times W_F\rightarrow \mathrm{GL}_n(\mathbb{C})
}
\end{equation*}
such that
\begin{enumerate}[(i)]
\item  the upper row is (\ref{admissible-embedding-induction-for-H}),
\item  the vertical maps are natural inclusions, and
\item  the morphism $\hat{\xi}$ is defined as part of the endoscopic data $(H, \mathcal{H}={}^LH_F, s, \hat{\xi})$ for $H$.
\end{enumerate}
Following the diagram, we see that the lower row is isomorphic to $\mathrm{Ind}_{E/F}(\xi \mu_H)$ as a representation of $W_F$. Comparing this to the representation $\mathrm{Ind}_{E/F}(\xi \mu_G)$ given in Proposition \ref{recover the character from induction}, we then proved the assertion immediately.
\end{proof}

\begin{rmk}\label{changing back from right to left}
From now on we change back our notations from right cosets to left cosets. If we denote the $\chi$-data we considered from section \ref{section langlands shelstad chi data} to here by $\chi_{(g)}=\chi_{\left[\begin{smallmatrix}W_E\\W_Eg\end{smallmatrix}\right]}$, then we define our new $\chi$-data by
$$\chi_{g}=\chi_{\left[\begin{smallmatrix}W_E\\gW_E\end{smallmatrix}\right]}:=\chi_{(g^{-1})} =\chi_{\left[\begin{smallmatrix}W_E \\W_Eg^{-1}\end{smallmatrix}\right]},$$
which are used from now on. Since the map $\left[\begin{smallmatrix}gW_E\\hW_E\end{smallmatrix}\right]\mapsto\left[\begin{smallmatrix}W_Eg^{-1}\\W_Eh^{-1}\end{smallmatrix}\right]$ is $W_F$-equivalent and symmetry preserving, we can easily check that $\{\chi_g\}_{g\in\mathcal{D}_\pm}$ are also $\chi$-data.
\qed\end{rmk}

We can change the notation of $\Delta_\mathrm{III_2}$ in Corollary \ref{definition of delta III2} in terms of roots
$$\Delta_{\mathrm{III}_2}(\gamma) = \prod_{\lambda \in \mathcal{R}_{\mathrm{asym}/\pm}-\Phi(H,T)} \chi_\lambda(\lambda(\gamma))^{-1} \prod_{\lambda \in \mathcal{R}_{\text{sym}}-\Phi(H,T)} \chi_\lambda|_{E^\times}(\gamma)^{-1}\text{, for }\gamma\in E^\times\text{ regular.}$$
In \cite{LS}, the transfer factor $\Delta_\mathrm{II}$ is defined by
 $$\Delta_{\mathrm{II}}(\gamma) = \prod_{\lambda \in \mathcal{R}_{\mathrm{asym}/\pm} - \Phi(H,T)} \chi_\lambda({}\lambda( \gamma)) \prod_{\lambda \in \mathcal{R}_{\text{sym}}-\Phi(H,T)} \chi_\lambda \left( \frac{\lambda(\gamma)-1}{a_\lambda}\right)\text{, for }\gamma\in E^\times\text{ regular.}$$
Therefore \begin{equation}\label{Delta-II,III2 as product of sym}
    \Delta_{\mathrm{III}_2}(\gamma) \Delta_{\mathrm{II}}(\gamma) = \prod_{\lambda \in \mathcal{R}_{\text{sym}}-\Phi(H,T)}\chi_\lambda \left( \frac{\lambda(\gamma) - 1}{\gamma a_\lambda} \right)
   =\prod_{g\in\mathcal{D}(F)_\mathrm{sym}-\mathcal{D}(K)}\chi_g\left( \frac{\gamma - {}^g \gamma}{a_{1,g}} \right).
 \end{equation}
If $g\in W_F$ is symmetric, then we can assume that $g^2 \in W_E$ and ${}^g a_{1,g} = -a_{1,g}$. Since $(\gamma - {}^g \gamma)/a_{1,g}\in E^\times_{\pm g}$, we have that $\chi_{g}((\gamma - {}^g \gamma)/a_g)$ is a sign. Also recall that $\Delta_{\mathrm{III_1}}(\gamma)$ and $\Delta_\mathrm{I}(\gamma)$ are all 1. Therefore we have $\Delta_\mathrm{I,II,III}(\gamma) = \Delta_{\mathrm{II},\mathrm{III_2}}(\gamma)$, which is also a sign.

\begin{prop}
  $\Delta_\mathrm{I,II,III}(\gamma)$ is independent of the choices of admissible embedding $T_H \rightarrow T_G$, $a$-data and $\chi$-data.
\end{prop}
\begin{proof}
  We only sketch the reasons and refer to chapter 3 of \cite{LS} for details. We recall the definitions of various transfer factors and check that
\begin{enumerate}[(i)]
\item  only $\Delta_{\mathrm{III_1}}$ and $\Delta_{\mathrm{I}}$ depend on the admissible embedding $T_H \rightarrow T_G$,
\item  only $\Delta_{\mathrm{I}}$ and $\Delta_{\mathrm{II}}$ depend on $a$-data, and
\item only $\Delta_{\mathrm{II}}$ and $\Delta_{\mathrm{III_2}}$ depend on $\chi$-data.
\end{enumerate}
  The effects of the choices cancel when we multiply the various factors together.
\end{proof}
The product $\Delta_\mathrm{I,II,III}$ still depends on the chosen $F$-splitting $\mathbf{spl}_G$ as shown in the factor $\Delta_I$. We refer to Lemma 3.2.A of \cite{LS} for the detail of such dependence.

\begin{prop}\label{transfer factor trivial when totally-ram}
  Suppose $E/F$ is totally ramified. Then $\Delta_\mathrm{I,II,III}(\varpi_E) = 1$.
\end{prop}
\begin{proof}
  We choose the embedding $\iota: E^\times \hookrightarrow G(F)$ defined by (\ref{canonical embedding for trivial delta 1}), $a$-data defined by (\ref{delta-1-solution-for-a-totally-ramified}) and arbitrary $\chi$-data. The statement then follows directly from (\ref{Delta-II,III2 as product of sym}) when $\gamma=\varpi_E$.
\end{proof}

\section{A restriction property of $\Delta_{\mathrm{III}_2}$}\label{section restriction property}
In this section, we show that the restriction on $F^\times$ of the product $\mu=\mu_{\{\chi_g\}}$ in Proposition \ref{recover the character from induction} has a specific form. First we recall the following elementary results. Given a group $H$, we denote by $1_H$ the trivial representation of $H$. Suppose that $K$ is a subgroup of $H$ of finite index. We denote by $T^H_K : H^\text{ab} \rightarrow K^\text{ab}$ the transfer morphism. For any $g\in H$, we write ${}^gK=gKg^{-1}$.
\begin{prop}\label{well-known results Mackey}Let $\sigma$ and $\pi$ be finite dimensional representations of $K$ and $H$ respectively. We have the following formulae.
\begin{enumerate}[(i)]
\item (Mackey's Formula) $$\mathrm{Res}_K^H\mathrm{Ind}^H_K \sigma\cong\bigoplus_{[g]\in K \backslash H / K}\mathrm{Ind}^K_{K\cap {}^gK}\mathrm{Res}^{{}^gK}_{K\cap {}^gK}({}^g\sigma).$$ \label{Mackey Formula}
\item $\det\mathrm{Ind}^H_K \sigma \cong (\det\mathrm{Ind}^H_K 1_K)^{\mathrm{dim} \sigma} \otimes(\det\sigma \circ T^H_K)$. \label{det ind}
\item $(\det\mathrm{Res}_K^H\pi)\circ T^H_K=(\det\pi)^{|H/K|}$. \label{det res}
\end{enumerate}
\end{prop}
\begin{proof} Formulae (\ref{Mackey Formula}) and (\ref{det ind}) are well-known, for example (\ref{Mackey Formula}) is proved in 7.3 of \cite{Serre-repres}, and (\ref{det ind}) can be found in the Exercise in VII \S8 of \cite{Serre-local-fields}. Formula (\ref{det res}) is direct from (\ref{det ind}) if we take $\sigma=\mathrm{Res}_K^H\pi$.
\end{proof}
In particular, if $\chi$ is a character of $K$, then by (\ref{det ind}) we have
\begin{equation}\label{star3}
\chi \circ T^H_K \cong \left( \det \mathrm{Ind}^H_K \chi \right) \left( \det \mathrm{Ind}^H_K 1_K\right).
\end{equation}

\begin{lemma}\label{det ind decomp into double cosets}
We have the formula $$\det \mathrm{Ind}^H_K 1_K=\prod_{[g]\in( K \backslash H / K)'  }\det \mathrm{Ind}^H_{K_g} 1_{K_g}.$$
\end{lemma}
\begin{proof} By applying Mackey's formula on $\sigma=1_K$, we obtain
$$\mathrm{Res}_K^H\mathrm{Ind}^H_K 1_K\cong\bigoplus_{[g]\in K \backslash H / K}\mathrm{Ind}^K_{K_g}1_{K_g}.$$
We take determinant and then transfer morphism $T^H_K$ on both sides. By (\ref{det ind}) and (\ref{det res}) of Proposition \ref{well-known results Mackey}, we obtain $$\left( \det \mathrm{Ind}^H_K 1_K\right)^{|H/K|}=\prod_{[g]\in K \backslash H / K}\left( \det \mathrm{Ind}^H_{K_g} 1_{K_g}\right)\left( \det \mathrm{Ind}^H_K 1_K\right)^{|K/K_g|}.$$
The sum of $|K/K_g|$, for $[g]$ runs through $K \backslash H / K$, is $|H/K|$. Hence
$$\prod_{[g]\in K \backslash H / K} \det \mathrm{Ind}^H_{K_g} 1_{K_g}=1,$$
which is just the desired formula.
\end{proof}
Notice that $\det \mathrm{Ind}^H_{K_g} 1_{K_g}$ is independent of the choice of representative $g$ of the double coset $[g]$ if we interpret the character as the sign of the $H$-action on $H/K_g$ by left multiplication. For all representatives of $[g]$, the corresponding actions are equivalent each other. We write $\delta_{H/K}$ for the character $\det \mathrm{Ind}^H_K 1_K$. If $H=W_F$ and $K=W_E$ for some field extension $E/F$, then we write $\delta_{E/F}$ as $\delta_{W_F/W_E}$. We can easily check the formula
\begin{equation}\label{transition of delta}
  \delta_{E/F}=\delta_{E/K}|_{F^\times}\cdot \delta_{K/F}^{|E/K|}
\end{equation}
as a character of $F^\times$.
\begin{prop}\label{restriction of mu to F}
For all $\chi$-data $\{\chi_g\}_{g\in\mathcal{D}_\pm}$, if $\mu$ is the character of $E^\times$ defined by $\{\chi_g\}_{g\in\mathcal{D}_\pm}$  in Proposition \ref{recover the character from induction}, then $\mu|_{F^\times}= \delta_{E/F}$.
\end{prop}

\begin{proof}We first abbreviate $H = W_F$, $K = W_E$. For each $\lambda=\left[\begin{smallmatrix}1\\g\end{smallmatrix}\right]$, we write $K_g =K\cap {}^gK= W_{\lambda}$ and $K_{\pm g} =W_{\pm \lambda}$.
The equality $\mu|_{F^\times}= \delta_{E/F}$ can be rephrased as
\begin{equation*}
\prod _{[g]\in \left( K \backslash H / K \right)'}\chi_g \circ T^H_{K_g}= \delta_{H/K}.
\end{equation*}
By Lemma \ref{det ind decomp into double cosets}, we have to show that
\begin{equation}\label{chi data product and det ind}
\prod _{[g]\in \left( K \backslash H / K \right)'}\chi_g \circ T^H_{K_g}= \prod_{[g]\in( K \backslash H / K)'  }\delta_{H/{K_g}}.
\end{equation}
By comparing (\ref{chi data product and det ind}) termwise, we have the following claims.
\begin{enumerate}[(i)]
\item If $[g]\in  \left( K \backslash H / K \right)_\mathrm{asym/\pm}$, then $$\left( \chi_g \circ T^H_{K_g} \right)  \left( \chi_{g^{-1}} \circ T^H_{K_{g^{-1}}}\right) =\delta_{H/{K_g}}\delta_{H/{K_{g^{-1}}}}=1.$$
\item If $[g]\in  \left( K \backslash H / K \right)_\text{sym}$, then $\chi_g \circ T^H_{K_g}=\delta_{H/{K_g}}$.
\end{enumerate}

If $[g]\in \left( K \backslash H / K \right)_\mathrm{asym/\pm}$, then we have $K_{g^{-1}}={{}^{g}K_g}$, which is the stabilizer of the root $\left[\begin{smallmatrix}1\\g^{-1}\end{smallmatrix}\right]$. On one hand, because $\chi_{g^{-1}}=({}^g\chi_g)^{-1}$ by (\ref{chi-data condition transition}) in the definition of $\chi$-data, we have
\begin{equation*}
\left( \chi_g \circ T^H_{K_g} \right)  \left( \chi_{g^{-1}} \circ T^H_{K_{g^{-1}}}\right) =\left( \chi_g \circ T^H_{K_g} \right) \left( {}^{g} \chi^{-1}_{g} \circ T^H_{{}^{g}K_g}\right)\equiv 1.
\end{equation*}
On the other hand, since the $H$-action on the orbit ${}^H\left[\begin{smallmatrix}g\\1\end{smallmatrix}\right]={}^H\left[\begin{smallmatrix}1\\g^{-1}\end{smallmatrix}\right]$ is equivalent to that on ${}^H\left[\begin{smallmatrix}1\\g\end{smallmatrix}\right]$, we have
$\mathrm{Ind}^H_{K_g} 1_{K_g} \cong \mathrm{Ind}^H_{K_{g^{-1}}} 1_{K_{g^{-1}}}$. Therefore $\delta_{H/{K_g}}\delta_{H/{K_{g^{-1}}}}=1$.
We have proved the first claim.

If $[g]\in  \left( K \backslash H / K \right)_\text{sym}$, then we have an isomorphism $\mathrm{Ind}^{K_{\pm g}}_{K_{ g}}1_{K_{g}}\cong 1_{K_{\pm g}} \oplus \delta_{K_{\pm g}/K_{g}}$ as representations of $K_{\pm g}$. Here $\delta_{K_{\pm g}/K_{g}}$ is the quadratic character of $K_{\pm g}/K_{g}$. We denote this character by $\delta$. Hence $\mathrm{Ind}^H_{K_g} 1_{K_g} \cong \mathrm{Ind}^H_{K_{\pm g}} 1_{K_{\pm g}} \oplus \mathrm{Ind}^H_{K_{\pm g}}\delta$ and
\begin{equation}\label{star2}
\delta_{H/{K_g}} = \delta_{H/{K_{\pm g}}}\cdot \det\mathrm{Ind}^H_{K_{\pm g}}\delta\end{equation}
by taking determinant. Now the condition (\ref{chi-data condition quadratic}) in the definition of $\chi$-data, that $\chi_g \circ T^{K_{\pm g}}_{K_{g}}=\delta$, gives $\chi_g \circ T^H_{K_g} = \delta \circ T^H_{K_{\pm g}}$. By (\ref{star3}), this is just the right side of (\ref{star2}). We have proved the second claim and therefore Proposition \ref{restriction of mu to F}.
\end{proof}

Strictly speaking, $\Delta_{\mathrm{III}_2}$ is not defined on $F^\times$. However, we can use its definition $\left<\mathbf{a},\cdot\right>$ in the proof of Corollary \ref{definition of delta III2} and identify this character with $\mu_G^{-1}\mu_H$. Then we can write \begin{equation*}
 \Delta_{\mathrm{III}_2}|_{F^\times}=\delta_{K/F}^{|E/K|}
\end{equation*} as a consequence of (\ref{transition of delta}) and Proposition \ref{restriction of mu to F}.

\section{Relation between transfer factors}\label{section Relations of transfer factors}
We explain, in a bit more detail, the relation between the transfer factor in the theory of endoscopy and the one defined in 3.3 of \cite{Hen-Herb}. This should be already well-known, as mentioned in the second Remark in 3.7 of \cite{Hen-Herb} and Remarque (3) of 2.3 of \cite{HL2010}. Recall that, to establish the automorphic induction for $G=\mathrm{GL}_n$ and $H=\mathrm{Res}_{K/F}\mathrm{GL}_m$, we have defined the transfer factor of the form $$\Delta_G^H(\gamma)=\Delta^1(\gamma)\Delta^2(\gamma)\text{, for }\gamma\text{ regular.}$$
We have also mentioned that $\Delta^1$ is equal to $\Delta_\mathrm{IV}$ in section \ref{section Endoscopic group}. To define $\Delta^2$, we have to fix a transfer system $(\sigma,\kappa,e)$ such that
  \begin{enumerate}[(i)]
    \item $\sigma$ is a generator of $\Gamma_{K/F}$,
    \item $\kappa$ is a character associated to $K/F$, and
    \item $e\in K^\times$ such that ${}^\sigma e=(-1)^{m(d-1)}e$.
  \end{enumerate}
  There are certain choices of $e$ from 3.2 Lemma of \cite{Hen-Herb}.
  \begin{enumerate}[(i)]
\item If $m(d-1)$ is even or if char$(F) = 2$, then we can choose $e=1$.
\item  If $E/F$ is unramified, then we can choose $e \in U_E$.
\end{enumerate}
We define $\Delta^2$ to be
$$\Delta^2(\gamma) = \kappa\left( e\left(\prod_{0 \leq i<j \leq d-1} \prod^{m}_{k,\ell = 1}\left( ({}^{\sigma^i} \gamma)_k-({}^{\sigma^j}\gamma)_\ell\right)\right)\right)\text{ for }\gamma\text{ regular},$$
where $({}^{\sigma^i} \gamma)_k$, $k=1,\dots,m$, are the $m$ distinct eigenvalues of ${}^{\sigma^i} \gamma$.
\begin{prop}\label{delta-2-equals-delta-II-III}
  The product $\Delta_\mathrm{II,III}$ is equal to $\Delta^2$ up to a constant depending on the choices of the transfer system $(\sigma, \kappa, e)$ and $a$-data.
\end{prop}
\begin{proof}
  We fix the choices of $\sigma$ and $\kappa$ in the transfer system. We start from the equality \begin{equation}\label{delta-product-pver-roots}
    \left(\prod_{0 \leq i<j \leq d-1} \prod^{m}_{k,\ell = 1}\left( ({}^{\sigma^i} \gamma)_k-({}^{\sigma^j}\gamma)_\ell\right)\right)=\prod_{\lambda \in \Phi(G,T)_+ - \Phi(H,T)}(\lambda (\gamma) - 1).
  \end{equation}
  Here $\Phi(G,T)_+$ is the set of positive roots in $\Phi(G,T)$. We then check the contribution of each $\Gamma_F$-orbits of roots in the product (\ref{delta-product-pver-roots}). If $[\lambda]$ is an asymmetric orbit, then up to a sign its contribution is $$\prod_{\mu \in [\lambda]}(\mu(\gamma) - 1).$$ Since the product runs through a $\Gamma_F$-orbit, we have $$\kappa\left( \prod_{\mu \in [\lambda]} (\mu(\gamma) -1)\right) = 1.$$ If $[\lambda]$ is symmetric, then only those positive roots in this orbit contribute to the product. Write this subset by $[\lambda]_+$. Therefore by choosing a constant $e_{[\lambda]}$ such that $$e_{[\lambda]} \prod_{\mu \in [\lambda]_+}(\mu(\gamma) - 1)\in F^\times,$$ the contribution is equal to $$\kappa\left( e_{[\lambda]} \prod_{\mu \in [\lambda]_+}(\mu(\gamma) - 1)\right) = \pm 1.$$ We can choose $e_{[\lambda]}$ such that this sign is equal to $$\chi_\lambda \left( \frac{\lambda(\gamma) - 1}{a_\lambda}\right).$$  Hence by choosing $e$ in the transfer system as the product of $ e_{[\lambda]}$ with $[\lambda]$ runs through all symmetric orbits, we have both $\Delta^2(\gamma)$ and $ \Delta_\mathrm{II,III}(\gamma)$ equal to $$\prod_{\lambda\in \mathcal{R}_{\text{sym}}} \chi_\lambda \left( \frac{\lambda(\gamma) - 1}{a_\lambda}\right).$$ We remark again that each factor above is independent of the choice of the root representing its orbit.
\end{proof}

\chapter{An interlude}\label{chapter interlude}

In this small chapter, we briefly look ahead to the end of the article and state the main results in more detail, based on the background we have provided in chapter \ref{chapter basic}-\ref{endoscopy}.

Recall that, in order to describe the essentially tame local Langlands correspondence, we have to make use of the automorphic induction (with the help of base change), which is known to be part of the transfer principle. On one hand, we rewrite the character identity of automorphic induction given in (\ref{automorphic-induction-character-relation}) as
\begin{equation}\label{interlude autom ind formula}
  \Theta^{\kappa,\Psi}_\pi (\gamma) = c_\theta\Delta^2(\gamma) \Delta^1(\gamma)^{-1} \sum_{g \in \Gamma_{K/F}} \Theta^g_\rho(\gamma) \text{, for all }\gamma \in H(F) \cap G(F)_{\mathrm{ell}}.
\end{equation}
Here we replace the notation of the constant $ c(\rho, \kappa, \Psi)$ in (\ref{automorphic-induction-character-relation}) by a more precise notation $c_\theta$ used in \cite{BH-ET3}. The notation $\theta$ denotes a simple character of $\xi$, which is an inflation of the wild part $\xi|_{U_E^1}$ of $\xi$ to a character of a compact subgroup of $G(F)$. We refer the notion of simple character to chapter 3 of \cite{BK}. A particular choice comes from the first step of constructing the supercuspidal $\pi$ from $\xi$, which is the bijection in Proposition \ref{prop bijection of P and A} and will be briefly described in section \ref{Supercuspidal Representations}.

On the other hand, we can specify a constant in the spectral transfer identity (\ref{spetral transfer-character-relation}). We write
\begin{equation}\label{interlude spec transfer formula}
\Theta^{\kappa }_\pi(\gamma) = \epsilon_\mathrm{L}(V_{G/H}) \Delta_{\mathrm{I,II,III}}(\gamma) \Delta_{\mathrm{IV}}(\gamma)^{-1} \sum_{g \in \Gamma_{K/F}}\Theta^g_\rho(\gamma)\text{, for all }\gamma\in H (F)\cap G(F)_{\mathrm{ell}}.
 \end{equation}
Here $\epsilon_\mathrm{L}(V_{G/H})$ is the Langlands constant defined in section 5.3 of \cite{KS}. The normalized transfer factor then depends on a pair consisting of the unipotent radical $\mathscr{U}$ of a chosen $F$-Borel subgroup and certain character of $\mathscr{U}$. Following 1.5 of \cite{BH-ET3}, we call this pair a Whittaker datum of $G(F)$. We may choose this Borel subgroup from our $F$-splitting $\mathbf{spl}_G$ as in section \ref{section Trivializing the splitting invariant} and call the Whittaker datum the standard one. From section 5.3 of \cite{KS}, we know that this normalized transfer factor is independent of the chosen splitting $\mathbf{spl}_G$ giving rise to the Whittaker datum.

The first result is to compare these two identities (\ref{interlude autom ind formula}) and (\ref{interlude spec transfer formula}). As mentioned in section \ref{section main results of author}, we would find out the difference of the two normalizations in terms of a constant. This constant, denoted by $\kappa(x)$, is deduced by studying different Whittaker data and depends on the element $x\in G(F)$ that conjugates these data. In section \ref{section 3 lemmas}, we will first figure out the correct Whittaker datum that gives the normalization in (\ref{interlude autom ind formula}). We will then compute the constant $\kappa(x)$ in the subsequent sections of chapter \ref{chapter comparing character}.

A remark here is that, the product $\kappa(x) c_\theta \Delta^2$ depends only on the standard Whittaker datum, and is independent of the representation $\pi$. However the factors $\kappa(x)$ and $c_\theta$ certainly depend on $\pi$. Similar happens on the other side. We have chosen the splitting $\mathbf{spl}_G$, the admissible embedding of the maximal torus, and the $a$-data as in section \ref{section The splitting invariant} and \ref{section Trivializing the splitting invariant} in favor of our calculations. For example, we can choose specific $a$-data to trivialize $\Delta_\mathrm{I}$. Recall that our second result is to choose suitable $\chi$-data whose corresponding admissible embedding of L-groups yields the Langlands parameter of the specified representation $\pi$. Since the transfer factors $\Delta_\mathrm{II}$ and $\Delta_\mathrm{III_2}$ are built by $\chi$-data, they also depend on the representation $\pi$.

Most of the idea of our second result is in the Introduction. Suppose that $\xi$ is the admissible character giving rise to the supercuspidal representation $\pi$. With those chosen $\chi$-data $\{\chi_{\lambda,\xi}\}_{\lambda\in\Phi}$ in Theorem \ref{chi-data factor of BH-rectifier}, we have the following consequences on the rectifier. Let $K$ be an intermediate subfield between $E/F$, not necessarily cyclic. If we regard $\xi$ as being admissible over $K$, then we have, using Corollary \ref{factorization of rectifier over L}, that  $${}_{K}\mu_\xi=\prod_{\begin{smallmatrix}{\lambda}\in W_F\backslash \Phi\\ \lambda|_K\equiv 1\end{smallmatrix}}\chi_{\lambda,\xi}|_{E^\times}.$$
Notice that the index set $\{{\lambda}\in W_F\backslash \Phi,\,\lambda|_K\equiv 1\}$ above can be identified with the $W_K$-orbit of the root system $\Phi(\mathrm{GL}_{|E/K|},\mathrm{Res}_{E/K}\mathbb{G}_m)$ over $K$. In the case when $K/F$ is a cyclic sub-extension, we have seen that $H=\mathrm{Res}_{K/F}\mathrm{GL}_{|E/K|}$ is a twisted endoscopic group of $G$. We can define the transfer factor $\Delta_\mathrm{III_2}$ for $(G,H)$ and obtain, in Corollary \ref{rectifier as transfer factor}, the relation $$ {}_F\mu_\xi(\gamma) {}_K\mu_\xi(\gamma)^{-1}=\prod_{\begin{smallmatrix}{\lambda}\in W_F\backslash \Phi\\ \lambda|_K\neq 1\end{smallmatrix}}\chi_{\lambda,\xi}(\gamma)=\Delta_\mathrm{III_2}(\gamma)$$ for every $\gamma\in T(F)=E^\times$ regular in $G(F)$.

We can express the explicit values of the above normalization constant $c_\theta$ and the $\chi$-data in terms of t-factors of finite modules, whose constructions are summarized as follows and whose details are referred to chapter \ref{chapter finite cymplectic module} and \ref{chapter Tame supercuspidal representations}. Denote by $\Psi_{E/F}$ the finite group $E^\times/F^\times U_E^1$. For each admissible character $\xi$, there is a finite symplectic $\mathbf{k}_F\Psi_{E/F}$-module $V=V_\xi$ emerging from the bijection $\Pi_n$ in Proposition \ref{prop bijection of P and A}. Equipped with $V$ is an alternating bilinear form $h_\theta$ naturally defined by the chosen simple character $\theta$ of $\xi$. We can embed each $V$ into a fixed finite $\mathbf{k}_F\Psi_{E/F}$-module $U$, which is independent of the admissible character $\xi$. Practically we can think of $U$ as the largest possible $V$ as $\xi$ runs through all admissible characters in $P(E/F)$. We will state in Proposition \ref{decomp of standard module} that this fixed module $U$ admits a complete decomposition
$$U=\bigoplus_{\lambda\in W_F\backslash\Phi}U_{\lambda},$$
called the {residual root space decomposition}. This is analogous to the root space decomposition of a Lie algebra in the absolute case. By restriction, the submodule $V$ admits a decomposition
\begin{equation}\label{interlude fine decomp}
  V=\bigoplus_{\lambda\in W_F\backslash\Phi}V_{\lambda}.
\end{equation}We can show that $V_\lambda$ is either trivial or isomorphic to $U_\lambda$. We will prove in Proposition \ref{complete decomp orthogonal} that such decomposition on $V$ is orthogonal with respect to $h_\theta$, with symplectic isotypic components of the form
\begin{equation*}
\mathbf{V}_\lambda=\begin{cases}
V_\lambda\oplus V_{-\lambda} & \text{if }{}^{W_F}\lambda\neq {}^{W_F}(-\lambda)\text{ (i.e., }\lambda\text{ is asymmetric)},\\
V_\lambda & \text{if }{}^{W_F}\lambda= {}^{W_F}(-\lambda)\text{ (i.e., }\lambda\text{ is symmetric)}.
\end{cases}
 \end{equation*}
Let $W_{\lambda}$ be the complementary module of $V_\lambda$ in the sense that $$V_\lambda\oplus W_{\lambda}=U_\lambda.$$ We remark that we already had a decomposition of $V$ in \cite{BH-ET3} in terms of the jump data of $\xi$, see (\ref{coarse decomp of V}) for instance. The decomposition (\ref{interlude fine decomp}) is a finer one, and is complete in the sense that the isotypic components are all known.

Let $\mu$ and $\varpi$ be the images in $\Psi_{E/F}$ of the subgroup $\mu_E$ of the roots of unity and the subgroup generated by a chosen prime element $\varpi_E$ respectively. We are interested in the $\mathbf{k}_F\mu$-module and the $\mathbf{k}_F\varpi$-module structures of these $V_\lambda$ and $W_\lambda$. These lead to define certain invariants as follows.
 \begin{enumerate}[(i)]
\item By regarding each symplectic component $\mathbf{V}_\lambda$ as a $\mathbb{F}_p\mu$-module and a $\mathbb{F}_p\varpi$-module respectively, we have the symplectic signs $$t_\mu^0(\mathbf{V}_\lambda),\, t_\mu^1(\mathbf{V}_\lambda),\,t_\varpi^0(\mathbf{V}_\lambda)\text{, and }t_\varpi^1(\mathbf{V}_\lambda).$$ Here $t_\mu^0(\mathbf{V}_\lambda)$ and $t_\varpi^0(\mathbf{V}_\lambda)$ are signs $\pm1$, while  $t_\mu^1(\mathbf{V}_\lambda):\mu\rightarrow\{\pm1\}$ and $t_\varpi^1(\mathbf{V}_\lambda):\varpi\rightarrow\{\pm1\}$ are quadratic characters. These are all defined by the algorithm in section 3 of \cite{BH-ET3} and are computed in Proposition \ref{summary of t-factors} for each $\lambda \in W_F\backslash \Phi$.
\item For each $W_{\lambda}$ we attach a Gauss sum $t(W_{\lambda})$ with respect to certain quadratic form on $W_{\lambda}$. Usually $t(W_{\lambda})$ is a sign $\pm1$, except in one case it is at most a 4th root of unity. These factors are defined in section 4 of \cite{BH-ET2} and are computed in section \ref{section comp mod}.
\end{enumerate}
We then assign a collection of tamely ramified characters $\{\chi_{\lambda,\xi}\}_{\lambda\in \Phi}$ in terms of the t-factors above. These characters are those $\chi$-data associated to the admissible character $\xi$ as mentioned above. Their precise values are given in Theorem \ref{chi-data factor of BH-rectifier}. We also need certain t-factors to compute the constant $c_\theta$. During the way of computations in chapter \ref{chapter comparing character} and \ref{chapter rectifier and transfer}, we will make strong use of the jump data of $\xi$ and run into some technical but interesting sign checking.

We end this Interlude with a couple of remarks.
\begin{rmk}
  \begin{enumerate}[(i)]
    \item We emphasize again that, for explicit computations, we would make use of the bijection in Proposition \ref{orbit of roots as double coset} between non-trivial double cosets in $\Gamma_E\backslash \Gamma_F/\Gamma_E$ and $\Gamma_F$-orbits of the root system $\Phi$. Alternatively, we can choose a bijection between their respective representatives, as in (\ref{bijection between representatives}). The statements above involving roots or $\Gamma_F$-orbits of roots would be replaced by analogous statements involving representatives in $\Gamma_F/\Gamma_E$ of non-trivial double cosets.
    \item The modules $U$ and $V$ are quotients of certain compact subgroups of $G(F)$, which are examples of those in the Moy-Prasad filtrations \cite{MP1} of $G(F)$. The fact that each such quotient is isomorphic to certain Lie algebra is known in \cite{MP1}, see also Corollary 2.3 of \cite{Yu}. The treatment in this thesis is close to those in section 4 of \cite{BH-ET2} and section 3, 7 of \cite{BH-ET3}, since we will make use of their results for explicit computations.
  \end{enumerate}
    \end{rmk}

    \chapter{Finite symplectic modules}\label{chapter finite cymplectic module}

Again $E/F$ is tamely ramified. Let $\mathfrak{A}$ be the hereditary order of $\mathrm{End}_F(E)$ corresponding to the $\mathfrak{o}_F$-lattice chain $\{\mathfrak{p}_E^k|k\in\mathbb{Z}\}$ in $E$ and $\mathfrak{P_A}$ be the Jacobson radical of $\mathfrak{A}$. They are introduced in chapter 1 of \cite{BK}. Denote $\Psi_{E/F}=E^\times/F^\times U^1_E$. The main result, Proposition \ref{decomp of standard module} of section \ref{section standard module}, is to deduce a $\mathbf{k}_F \Psi_{E/F}$-module structure of the quotient space $U:=\mathfrak{A/P_A}$ naturally derived from the conjugate action of $E^\times$ on $\mathrm{End}_F (E)$. We would describe the complete decomposition of $U$ into $\mathbf{k}_F \Psi_{E/F}$-submodules, called the residual root space decomposition.

In section \ref{section symplectic modules} and \ref{section comp mod} we attach certain invariants, called t-factors, on the submodules of $U$. These invariants arise from certain symplectic structure imposed on a submodule $V$ of $U$ and the structure of symmetry on the complementary module $W$ of $V$ in $U$. By studying these structures, we can compute the values of those t-factors explicitly, as in Proposition \ref{summary of t-factors} and Proposition \ref{n-factor of comp-mod}. We will bring the t-factors into play frequently in subsequent chapters.

\section{The standard module}\label{section standard module}

The content in this section is classical. We re-interpret the results in the language here to build up the results of later sections. Many standard algebra textbooks can deduce the results here. For example, one may consult \cite{reiner-max-order}. Some of the details can be found in section 7 of \cite{BH-ET3}.

We consider the following $F$-vector spaces with $E^\times$-actions. For all $t\in E^\times$,
\begin{equation}\begin{split}\label{F-spaces E-modules}
&E \otimes E\text{, with }{}^t(x \otimes y) \mapsto tx \otimes yt^{-1}\text{ for all } x,y\in E,
\\
&\mathrm{End}_F(E)\text{, with }({}^tA)(v) = tA(t^{-1}v)\text{ for all }  A \in \mathrm{End}_F(E),\,v\in E,\text{ and }
\\
&\bigoplus_{[g] \in \Gamma_E \backslash \Gamma_F /\Gamma_E} {}^gEE\text{, with }{}^t({}^gx_{[g]}y_{[g]}) _{[g]}= ({}^gtt^{-1}{}^gx_{[g]}{}y_{[g]})_{[g]}\text{ for all } x_{[g]}, y_{[g]} \in E.
\end{split}\end{equation}

\begin{prop}\label{isom F spaces E modules} The $F$-linear maps
\begin{enumerate}[(i)]
\item $E \otimes E \rightarrow \mathrm{End}_F(E),\, x \otimes y \mapsto (v \mapsto \mathrm{tr}_{E/F} (yv)x)$, and \label{isomorphism tensor to endomorphism}
\item $E \otimes E \rightarrow \bigoplus_{[g] \in \Gamma_E \backslash \Gamma_F/\Gamma_E} {}^gEE,\,x\otimes y \mapsto ({}^gxy)_{[g]}$, \label{isoomrphism tensor to root spaces}
\end{enumerate}
are isomorphisms of $E^\times$-modules.
\end{prop}
\begin{proof}
 Indeed (\ref{isomorphism tensor to endomorphism}) is isomorphic by the non-degeneracy of the trace form $\mathrm{tr}_{E/F}$, while (\ref{isoomrphism tensor to root spaces}) is isomorphic by considering all possible $F$-algebra embeddings $E \otimes E \rightarrow \bar{F}$. Notice that the isomorphism (\ref{isoomrphism tensor to root spaces}) is moreover an $F$-algebra one. The $E^\times$-invariance of each of the morphisms is clear.
 \end{proof}

By choosing a suitable $F$-basis of $E$, we can identify $\mathrm{End}_F(E)$ with $\mathfrak{g}(F) = \mathfrak{gl}_n(F)$ and its subalgebra $E$ with a Cartan subalgebra $\mathfrak{g}(F)_0$ of $\mathfrak{g}(F)$. Recall that the roots of the elliptic maximal $F$-torus $T = \mathrm{Res}_{E/F} \mathbb{G}_m$ in the $F$-reductive group $G = \mathrm{GL}_n$ are of the form $[\begin{smallmatrix}
g \\ h
\end{smallmatrix}]$ with $g, h \in \Gamma_F/\Gamma_E$, $g\neq h$, such that
\begin{equation*}
[\begin{smallmatrix}
g \\ h
\end{smallmatrix}](t) = {}^gt({}^ht)^{-1}\text{ for all } t \in E^\times =T(F).
 \end{equation*}
 We denote the $\Gamma_F$-orbit of a root $\lambda$ by $[\lambda]$.
\begin{prop} \label{rational root space decomp}
\begin{enumerate}[(i)]
 \item  The $F$-Lie algebra $\mathfrak{g}(F)$ decomposes into $\mathrm{Ad}(E^\times)$-invariant subspaces $$ \mathfrak{g}(F) = \mathfrak{g}(F)_0 \oplus \bigoplus_{[\lambda ]\in \Gamma_F \backslash \Phi} \mathfrak{g}(F)_{[\lambda]}.$$
\item This decomposition is compatible with the isomorphism
$$\mathrm{End}_F(E)\cong \bigoplus_{[g] \in \Gamma_E \backslash \Gamma_F/\Gamma_E} {}^gEE$$
induced by Proposition \ref{isom F spaces E modules}, such that $\mathfrak{g}(F)_0\cong E$ and $\mathfrak{g}(F)_{\left[\left[\begin{smallmatrix}
1 \\g
\end{smallmatrix}\right]\right]}\cong {}^gEE.$
\end{enumerate}
\end{prop}
\begin{proof}
The first assertion can be derived from the absolute case by a simple Galois descent argument. More precisely, for every orbit $[\lambda]\in \Gamma_F\backslash \Phi$ there is a subspace $\mathfrak{g}(F)_{[\lambda]}$ in $\mathfrak{g}(F)$ such that $$\mathfrak{g}(F)_{[\lambda]}\otimes_F\bar{F}=\bigoplus_{\mu\in[\lambda]}\mathfrak{g}(\bar{F})_{\mu},$$ the direct sum of the root space $\mathfrak{g}(\bar{F})_{\mu}$ for $\mu\in\Phi$. The second assertion is clear by Proposition \ref{explicit expression of double coset}.(\ref{explicit expression of double coset double coset}).
\end{proof}
We call the decomposition of $\mathrm{End}_F(E) \cong \mathfrak{g}(F)$ in Proposition \ref{rational root space decomp} the \emph{rational root space decomposition}. We are going to show that such decomposition descends to the one of certain $\mathfrak{o}_F$-sublattices of $\mathfrak{g}(F)$. Let $\mathfrak{A}$ be the hereditary order of $\mathrm{End}_F(E)$ corresponding to the $\mathfrak{o}_F$-lattice chain $\{\mathfrak{p}_E^k|k\in\mathbb{Z}\}$ in $E$, as introduced in (1.1) of \cite{BK}. Explicitly, we can define $\mathfrak{A}$ as
\begin{equation}\label{hereditary order A}
  \mathfrak{A} = \{ X \in \mathrm{End}_F(E) | X\mathfrak{p}_E^k \subseteq \mathfrak{p}_E^k \text{ for all } k\in \mathbb{Z} \}.
\end{equation}
If we identify $\mathrm{End}_F(E)\cong \mathfrak{gl}_n(F)$ by choosing suitable basis, then the lattice $\mathfrak{A}$ can be expressed in matrices partitioned into $e \times e$ blocks of size $f \times f$, with entries in $\mathfrak{o}_F$, and blocked upper triangular mod $\mathfrak{p}_F$.

Let $K$ be the maximal unramified extension of $F$ in $E$. Consider the following $\mathfrak{o}_F$-lattices contained respectively in the $F$-vector spaces in (\ref{F-spaces E-modules}),
\begin{equation}\label{isom o-F-lattices}
\begin{split}
&\mathfrak{o}_{E \otimes E} := (\mathfrak{o}_E \otimes_{\mathfrak{o}_F} \mathfrak{o}_E)+ \sum_{\begin{smallmatrix}1 \leq k \leq \ell \leq e-1 \\ k+\ell \geq e\end{smallmatrix}} \mathfrak{o}_K \varpi^{-1}_F (\varpi^k_E \otimes \varpi^\ell_E) \subseteq E \otimes E,
\\
&\mathfrak{A} \subseteq \mathrm{End}_F(E),\text{ and}
\\
&\bigoplus_{[g]\in \Gamma_E\backslash \Gamma_F / \Gamma_E} \mathfrak{o}_{{}^gEE} \subseteq \bigoplus_{[g]\in \Gamma_E\backslash \Gamma_F / \Gamma_E}{}^gEE.
\end{split}
\end{equation}
They are all $E^\times$-conjugate invariant.

\begin{prop}\label{prop isom of lattices} The isomorphisms in Proposition
\ref{isom F spaces E modules} induce isomorphisms of $\mathfrak{o}_F$-lattices as well as of $E^\times$-modules in (\ref{isom o-F-lattices}).
\end{prop}

\begin{proof}
The $\mathfrak{o}_F$-morphism $\mathfrak{o}_{E \otimes E} \rightarrow \mathfrak{A}$ restricted from (\ref{isomorphism tensor to endomorphism}) of Proposition
\ref{isom F spaces E modules} is clearly injective and $E^\times$-invariant. To show the surjectivity,  we choose an $\mathfrak{o}_F$-basis $\{ w_1, \dots, w_n\}$ of $\mathfrak{o}_E$ such that
\begin{equation*}
v_E(w_{jf+1}) = \dots = v_E(w_{(j+1)f}) = j\text{ for each }j=0,\dots,e-1.
\end{equation*}
We choose another $\mathfrak{o}_F$-basis $\{ w^*_i\}$ of $\mathfrak{o}_E$ dual to $\{w_i\}$ in the sense that
\begin{equation*}
\mathrm{tr}_{E/F}(w^*_i w_j) = 0 \text{ for }i \neq j\text{ and }\mathrm{tr}_{E/F}(w^*_iw_i) =
\begin{cases}
1 &\text{ for }  i = 1, \dots, f,
\\
\varpi_F &\text{ for }  i = f+1, \dots, n.
\end{cases}
\end{equation*}
Then we can readily show that, under the isomorphism  $E \otimes E \rightarrow \mathrm{End}_F(E) \rightarrow \mathfrak{gl}_n (F)$, the element $\sum_{i,j}a_{ij}(w^*_i \otimes w_j)$ in $E\otimes E$ is mapped to the matrix $(A_{ij})$ where $A_{ij} = a_{ij} \mathrm{tr}(w^*_i w_i)$. We check the $F$-valuations of these entries. Suppose that $\lfloor i/f\rfloor=k$ and $\lfloor j/f\rfloor=\ell$.
\begin{enumerate}[(i)]
\item
If $k=1$, then we have $v_E(w^*_i) = v_E(w_i) =0 $ and so $a_{ij}\mathrm{tr}(w^*_iw_i) = a_{ij} \in \mathfrak{o}_F$.
\item If $\ell \geq k \gneq 1$, then $a_{ij} \in \mathfrak{o}_F \varpi^{-1}_F$, $v_E(w^*_i) = e + 1 - k$ and $v_E(w_j) = \ell -1$. Hence $a_{ij}\mathrm{tr}(w^*_iw_i) \in \mathfrak{o}_F$.
\item
When $k >\ell$, we have $a_{ij} \in \mathfrak{o}_F$ and $a_{ij} \mathrm{tr}(w^*_iw_i) = a_{ij} \varpi_F \in \mathfrak{p}_F$.
\end{enumerate}
We have just shown that $\mathfrak{o}_{E \otimes E} \rightarrow \mathfrak{A}$ is surjective and therefore isomorphic.

To deal with another isomorphism, we use the standard technique in chapter 4 of \cite{reiner-max-order}. Let $M\subseteq N$ be two free $\mathfrak{o}_F$-modules of the same rank. Suppose that the quotient $N/M$ is isomorphic to $\oplus_j\mathfrak{o}_F/\mathfrak{p}_F^{n_j}$ as $\mathfrak{o}_F$-modules. Define the order ideal $$\mathrm{ord}_{\mathfrak{o}_F}(N/M):=\prod_j\mathfrak{p}_F^{n_j}.$$ We can compute the order alternatively as follows. If $M=\oplus_j\mathfrak{o}_Fx_j$ and $N=\oplus_j\mathfrak{o}_Fy_j$ such that $x_i=\sum_ja_{ij}y_j$ for some $a_{ij}\in \mathfrak{o}_F$, then $\mathrm{ord}_{\mathfrak{o}_F}(N/M)={\mathfrak{o}_F}\det(a_{ij})$. We then consider the vector space $W=M\otimes_{\mathfrak{o}_F} F$. Suppose that we have a non-degenerate symmetric $F$-bilinear form $\mathrm{tr}:W\times W\rightarrow F$. We define the discriminant ideal of $M$ with respect to the form $\mathrm{tr}$ to be $d(M)=\det\mathrm{tr}(x_ix_j)\mathfrak{o}_F$.
\begin{lemma} \label{lemma about disc}
\begin{enumerate}[(i)]
\item We have the relation $d(M)=\mathfrak{o}_F(N/M)^2d(N)$. \label{item disc and order}
\item If $P$ and $Q$ are free $\mathfrak{o}_F$-modules, then
\begin{enumerate}
\item $d(P\oplus Q)=d(P)d(Q)$, and \label{item direct sum of disc}
\item $d(P\otimes_{\mathfrak{o}_F} Q)=d(P)^{\mathrm{rank}Q}d(Q)^{\mathrm{rank}P}$.  \label{item tensor prod of disc}
\end{enumerate}\label{item direct sum and tensor prod of disc}
\end{enumerate}
\end{lemma}
\begin{proof}
(\ref{item disc and order}) comes from the identity $\det\mathrm{tr}(x_ix_j)=\det(a_{ij})^2\det\mathrm{tr}(y_iy_j)$, (\ref{item direct sum of disc}) is easy, and (\ref{item tensor prod of disc}) is an elementary calculation of the determinant of tensor product of matrices.
\end{proof}
In the case when $E$ is a tame extension of $F$ and $M=\mathfrak{o}_E$, we denote the discriminant ideal $d(\mathfrak{o}_E)$ of $\mathfrak{o}_F$ by $d(E/F)$. This is known to be $\mathfrak{p}_F^{f(e-1)}$.

Now we show that the second $\mathfrak{o}_F$-morphism $\mathfrak{o}_{E \otimes E} \rightarrow \bigoplus_{[g]} \mathfrak{o}_{{}^gEE}$ is isomorphic. It is clearly injective and $E^\times$-invariant. To show that the map is surjective, we compare the sizes of $\bigoplus_{[g]} \mathfrak{o}_{{}^gEE}$ and the image of the sub-lattice $\mathfrak{o}_E \otimes_{\mathfrak{o}_F} \mathfrak{o}_E$ of $\mathfrak{o}_{E \otimes E}$. We apply Lemma \ref{lemma about disc} on
$$M=\mathfrak{o}_E \otimes_{\mathfrak{o}_F} \mathfrak{o}_E,\,N=\bigoplus_{[g]\in \Gamma_E \backslash \Gamma_F / \Gamma_E} \mathfrak{o}_{{}^gEE}.$$
First notice that $d(M)=d(E/F)^n$ by (\ref{item tensor prod of disc}) of Lemma \ref{lemma about disc}. Let $m$ be the $F$-valuation of the ideal $$d(E/F)^{2n} \prod_{[g]\in \Gamma_E\backslash \Gamma_F/\Gamma_E}d({}^gEE/F)^{-1}.$$ In the tame case, we can compute
$$m=2nf(e-1) - (e-1) \sum_{[g]\in \Gamma_E \backslash \Gamma_F / \Gamma_E}f({{}^gEE/F}).$$
The sum appearing on the right side is equal to $fn$, since we know that $$\sum_{[g]\in \Gamma_E \backslash \Gamma_F / \Gamma_E}|{}^gEE/F |=\dim_F(E\otimes _FE)= n^2$$ and each ${}^gE E/F$ has the same ramification degree $e$. Therefore, by (\ref{item disc and order}) of Lemma \ref{lemma about disc}, the order of $$\left(\bigoplus_{[g]\in \Gamma_E \backslash \Gamma_F / \Gamma_E}\mathfrak{o}_{{}^gEE}\right)/(\mathfrak{o}_E\otimes_{\mathfrak{o}_F}\mathfrak{o}_E)$$ is $q^{m/2}=q^{f^2e(e-1)/2}$. Using the expression of (\ref{isom o-F-lattices}), we can check that the order of $$\mathfrak{o}_{E\otimes E} / (\mathfrak{o}_E \otimes_{\mathfrak{o}_F} \mathfrak{o}_E)$$ is also $q^{m/2}$. Therefore the morphism $\mathfrak{o}_{E \otimes E} \rightarrow \bigoplus_{[g]} \mathfrak{o}_{{}^gEE}$ is surjective and hence isomorphic.
\end{proof}

Let $\mathfrak{P_A}$ be the Jacobson radical of $\mathfrak{A}$. Explicitly,
\begin{equation}\label{hereditary order PA}
  \mathfrak{P_A} = \{ X \in \mathrm{End}_F(E) | X\mathfrak{p}_E^k \subseteq \mathfrak{p}_E^{k+1} \text{ for all } k\in \mathbb{Z} \}.
\end{equation}
By arguments similar to those in the proof of Proposition \ref{prop isom of lattices}, we can show that the following $\mathfrak{o}_F$-sublattices of those lattices in (\ref{isom o-F-lattices}),
\begin{equation*}
\begin{split}
&\mathfrak{P}_{E \otimes E} := \mathfrak{p}_E \otimes_{\mathfrak{o}_F} \mathfrak{p}_E \subseteq \mathfrak{o}_{E \otimes E},
\\
&\mathfrak{P}_{\mathfrak{A}}  \subseteq \mathfrak{A}\text{, and}
\\
&\bigoplus_{[g]\in \Gamma_E \backslash \Gamma_F / \Gamma_E}\mathfrak{p}_{{}^gEE }\subseteq \bigoplus_{[g]\in \Gamma_E \backslash \Gamma_F / \Gamma_E}\mathfrak{o}_{{}^gEE},
\end{split}
\end{equation*}
are all isomorphic. We therefore have a $\mathbf{k}_F$-isomorphism
\begin{equation}\label{residual root space decomp}
\mathfrak{A} / \mathfrak{P}_{\mathfrak{A}} \cong \bigoplus_{[g]\in \Gamma_E \backslash \Gamma_F / \Gamma_E} \mathfrak{o}_{{}^gEE }/ \mathfrak{p}_{{}^gEE} = \bigoplus_{[g]\in\Gamma_E \backslash \Gamma_F / \Gamma_E}\mathbf{k}_{{}^gEE},
\end{equation}
which is moreover $E^\times$-equivalent. The $E^\times$-conjugate actions on both sides factor through the finite group $$\Psi_{E/F}:=E^\times / F^\times U^1_E.$$
We call the decomposition of the $\mathbf{k}_F\Psi_{E/F}$-module $\mathfrak{A} / \mathfrak{P}_{\mathfrak{A}}$ in (\ref{residual root space decomp}) the \emph{residual root space decomposition}.

Let us describe the action of $\Psi_{E/F}$ on $\mathfrak{A}/\mathfrak{P}_{\mathfrak{A}}$ more precisely. Write $\mathfrak{m} = \mathfrak{gl}_f(\mathbf{k}_F)$. Since $\mathfrak{A}/\mathfrak{P}_{\mathfrak{A}}$ is isomorphic to $\mathfrak{m}^e$, we regard $\mathfrak{m}^e$ as being embedded into diagonal blocks in $\mathfrak{gl}_n(\mathbf{k}_F)$. We first consider $\mathfrak{m}$ as a $\mathbf{k}_F \mu$-module. By embedding $\mathbf{k}_E \hookrightarrow \mathfrak{gl}_f(\mathbf{k}_F)$ (the choice of such embedding is irrelevant), we have from section 7.3 of \cite{BH-ET3} that
\begin{equation*}
\mathfrak{m} \cong \bigoplus_{i=0}^{f-1} \mathfrak{m}_i
\end{equation*}
where $\mathfrak{m}_i \cong \mathbf{k}_E$ as a $\mathbf{k}_F$-vector space and $\zeta \in \mu_E$ acts by $v \mapsto (\zeta^{q^i-1})^{-1}v$ for all $v \in \mathfrak{m}_i$. Each of the characters $\zeta\mapsto (\zeta^{q^i-1})^{-1}$, $i=0,\dots,f-1$,
is trivial on $\mu_F$, and hence is a character of $$\mu:=\mu_E/\mu_F.$$ The $\Psi_{E/F}$-module we are interested in is $$U=\text{Ind}_\mu^{\Psi_{E/F}}\mathfrak{m}.$$ Clearly $U\cong\mathfrak{A}/\mathfrak{P}_{\mathfrak{A}}$ as a $\mathbf{k}_F$-vector space.

\begin{prop} \label{decomp of standard module}
\begin{enumerate}[(i)]
\item  The $\Psi_{E/F}$-module $U$ decomposes into submodules
\begin{equation*}
U \cong \bigoplus_{\begin{smallmatrix}i\in \mathbb{Z}/f \\k \in q^f \backslash (\mathbb{Z} /e)\end{smallmatrix} }  U_{ki},
\end{equation*}
such that, the $\Psi_{E/F}$-action on each component $U_{ki}$ is given by the multiplication of the following images: $\zeta\mapsto (\zeta^{q^i-1})^{-1}$ for all $\zeta\in \mu_E$ and $\varpi_E\mapsto (\zeta^k_e \zeta_{\phi^i})^{-1}.$ \label{decomp of standard module first}
\item  The decomposition of $U$ is equivalent to the residual root space decomposition for $\mathfrak{A} / \mathfrak{P}_{\mathfrak{A}}$. \label{decomp of standard module second}
    \end{enumerate}
\end{prop}

\begin{proof}
\begin{enumerate}[(i)]
\item
 We have the decomposition $U = \bigoplus_{i \in \mathbb{Z}/f}U_i$ for $U_i = \text{Ind}^{\Psi_{E/F}}_{\mu}\mathfrak{m}_i$. Each $\mathfrak{m}_i$ is isomorphic to $\mathbf{k}_E$, hence $U_i$ is an $e$-dimensional $\mathbf{k}_E$-vector space. By (\ref{e-th power of prime is also prime}), we can choose a $\mathbf{k}_E$-basis for $U_i$ such that $\varpi_E$ acts as conjugation of the matrix  $$\left( \begin{smallmatrix}
&1&&\\&&\ddots&\\&&&1\\\zeta_{E/F}&&&
\end{smallmatrix} \right).$$
 The eigenvalues of such conjugation is $(\zeta_e^k \zeta_{\phi^i})^{-1}$ for some fixed $e$th root $\zeta_{\phi^i}$ of $\zeta_{E/F}^{q^i-1}$ and some $k=0,\dots,e-1$. Hence the subspaces with eigenvalues $(\zeta_e^k\zeta_{\phi^i})^{-1}$ in the same $\Gamma_{\mathbf{k}_E}$-orbit, which are those $(\zeta_e^k\zeta_{\phi^i})^{-1}$ with $k\in \mathbb{Z}/e$ in the same $q^f$-orbit, form a simple $\Psi_{E/F}$-module.

\item If we write $g=\sigma^k\phi^i$ as in Proposition \ref{explicit expression of double coset}.(\ref{explicit expression of double coset galois set}), then the action of $E^\times$ on $\mathbf{k}_{{}^gEE}$, a component of $\mathfrak{A/P_A}$, is induced from the conjugate action on ${}^gEE$. The actions of $\zeta\in\mu_E$ and $\varpi_E$ are given respectively by the multiplications of
$\zeta({}^{ \sigma^k \phi^i}\zeta)^{-1} =(\zeta^{q^i-1})^{-1}$ and $\varpi_E({}^{ \sigma^k \phi^i} \varpi_E)^{-1}  = (\zeta^k_e \zeta_{\phi^i})^{-1}$. These are just the actions on $U_{ki}$ defined in (\ref{decomp of standard module first}).
\end{enumerate}
\end{proof}

The $\mathbf{k}_F \Psi_{E/F}$-module $U$ is called the \emph{standard module}. We usually denote the submodule $U_{ki}$ by $U_{[\sigma^k\phi^i]}$. This notation is well-defined, which means that the finite module $U_{[g]}$ is independent of the coset representative of $[g]$, because ${}^gEE$ and ${}^hEE$ have the same residue field if $[g]=[h]$. For each subset $\mathcal{D}\subseteq (W_E\backslash W_F/W_E)'$, we write $$U_\mathcal{D}=\bigoplus_{[g]\in \mathcal{D}}U_{[g]}. $$ In particular, for every intermediate field extensions $F\subseteq K\subseteq L\subseteq E$, if $$\mathcal{D}=(W_E\backslash W_K/W_E)-(W_E\backslash W_L/W_E),$$ then we denote $U_\mathcal{D}$ by $U_{L/K}$. If $V$ is a submodule of $U$, then we write $V_{\mathcal{D}}=U_\mathcal{D}\cap V$.

\section{Symplectic modules}\label{section symplectic modules}

In this section we introduce a symplectic structure on certain submodules of the standard module $U=\mathfrak{A}/\mathfrak{P}_{\mathfrak{A}}$. We first give a brief summary on finite symplectic modules, whose details are referred to chapter 8 of \cite{BF} and chapter 3 of \cite{BH-ET3}. We consider a finite cyclic group $\Gamma$ whose order is not divisible by a prime $p$. We call a finite $\mathbb{F}_p\Gamma$-module \emph{symplectic} if there is a non-degenerate alternating $\Gamma$-invariant bilinear form $h:V\times V\rightarrow\mathbb{F}_p$. The $\Gamma$-invariant condition means that
\begin{equation*}
h(\gamma v_1,\gamma v_2)=h(v_1,v_2)\text{, for all }\gamma\in \Gamma,\,v_1,v_2\in V.
\end{equation*}
Let $V_{\lambda}=\mathbb{F}_p[\lambda(\Gamma)]$ be the simple $\mathbb{F}_p\Gamma$-module defined by the character $\lambda\in \mathrm{Hom}(\Gamma,\bar{\mathbb{F}}_p^\times)$. Its $\mathbb{F}_p$-linear dual is isomorphic to $V_{\lambda^{-1}}$. We usually regard $V_{\lambda}$ as a field extension of $\mathbb{F}_p$.
\begin{prop} \label{hyper and aniso and properties}
\begin{enumerate}[(i)]
\item Any indecomposable symplectic $\mathbb{F}_p\Gamma$-module is isomorphic to one of the following two kinds.
\begin{enumerate}
\item A \emph{hyperbolic} module is of the form $V_{\pm\lambda}\cong V_{\lambda}\oplus V_{\lambda^{-1}}$ such that either $\lambda^2=1$ or $V_{\lambda}\ncong V_{\lambda^{-1}}$.
\item An \emph{anisotropic} module is of the form $V_\lambda$ with $\lambda^2\neq1$ and $V_{\lambda}\cong V_{\lambda^{-1}}$.
\end{enumerate} \label{def hyper and aniso}
\item If $V_\lambda$ is anisotropic, then $ |\mathbb{F}_p[\lambda(\Gamma)] / \mathbb{F}_p|$ is even and $\lambda(\Gamma) \subseteq \ker( N_{\mathbb{F}_p[\lambda(\Gamma)] / \mathbb{F}_p[\lambda(\Gamma)]_\pm})$ the kernel of the norm map of the quadratic extension $\mathbb{F}_p[\lambda(\Gamma)] / \mathbb{F}_p[\lambda(\Gamma)]_\pm$. \label{even deg from def of aniso}
\item The $\Gamma$-isometry class of a symplectic $\mathbb{F}_p\Gamma$-module $(V,h)$ is determined by the underlying $\mathbb{F}_p\Gamma$-module $V$. \label{symplectic determined by underlying module}
\end{enumerate}
\end{prop}
\begin{proof}
All the proofs can be found in chapter 8 of \cite{BF} or in chapter 3 of \cite{BH-ET3}.
\end{proof}

We also call a symplectic $\mathbb{F}_p \Gamma$-module hyperbolic (resp. anisotropic) if it is a direct sum of hyperbolic (resp. anisotropic) indecomposable submodules. A special case is that if $V$ is anisotropic, then $V \oplus V$ is hyperbolic. By slightly abusing our terminology, we treat this as a special case of hyperbolic module and call it an \emph{even anisotropic} module. 

For each symplectic $\mathbb{F}_p\Gamma$-module $V$, we attach a sign $t^0_\Gamma(V)\in \{\pm1\}$ and a character $t^1_\Gamma(V):\Gamma\rightarrow \{\pm1\}$. We call these \emph{t-factors} of $V$. We choose a generator $\gamma$ of $\Gamma$ and set $t_\Gamma(V)=t^0_\Gamma(V)t^1_\Gamma(V)(\gamma)$. We give an algorithm in \cite{BH-ET3} on computing the t-factors.

\begin{enumerate}[(i)]
\item If $\Gamma$ acts on $V$ trivially, then $$t^0_\Gamma(V)=1\text{ and }t^1_\Gamma(V)\equiv1.$$
\item Let $V$ be an indecomposable symplectic $\mathbb{F}_p\Gamma$-module.
\begin{enumerate}
\item If $V=V_\lambda\oplus V_{\lambda^{-1}}$ is hyperbolic, then $$t^0_\Gamma(V)=1\text{ and }t^1_\Gamma(V)=\mathrm{sgn}_{\lambda(\Gamma)}(V_\lambda).$$
    Here $\mathrm{sgn}_{\lambda(\Gamma)}(V_\lambda):\Gamma\rightarrow\{\pm1\}$ is the character such that the image $\gamma\mapsto \mathrm{sgn}_{\lambda(\gamma)}(V_\lambda)$ is the sign of the multiplicative action of $\lambda(\gamma)$ on $V_\lambda$.
\item If $V=V_\lambda$ is anisotropic, then $$t^0_\Gamma(V)=-1\text{ and }t^1_\Gamma(V)(\gamma)=\left(\frac{\gamma}{\mathcal{K}}\right)\text{ for any }\gamma\in \Gamma.$$
    Here $\mathcal{K}=\ker(N_{\mathbb{F}_p[\lambda(\Gamma)]/\mathbb{F}_p[\lambda(\Gamma)]_\pm})$  and $\left(-\right)$ is the Jacobi symbol: for every finite cyclic group $H$, $$\left(\frac{x}{H}\right)=\begin{cases}
1 & \text{if }x\in H^2, \\
 -1 & \text{otherwise}.
\end{cases}$$
\end{enumerate}
\item If $V$ decomposes into an orthogonal sum $V_1\oplus\cdots\oplus V_m$ of indecomposable symplectic $\mathbb{F}_p\Gamma$-modules, then $$t^i_\Gamma(V)=t^i_\Gamma(V_1)\cdots t^i_\Gamma(V_m)\text{ for }i=0,1.$$
\end{enumerate}
If $p=2$, then the order of $\Gamma$ is odd. In this case $t^1_\Gamma(V)$ is always trivial because all sign characters and Jacobi symbols are trivial.

\begin{rmk}
If $V=V_{\lambda}$ is anisotropic indecomposable, then $V\oplus V$ is hyperbolic, or precisely even anisotropic. The t-factors are the same whether we consider $V\oplus V$ as hyperbolic or anisotropic. It is clear that $t^0_\Gamma (V\oplus V) = 1$ in both cases, while $t^1_\Gamma (V \oplus V)=\mathrm{sgn}_{\lambda(\Gamma)}V $ in the hyperbolic case and $t^1_\Gamma (V \oplus V) =t^1_\Gamma (V )^2= 1$ in the anisotropic case. Indeed $\mathrm{sgn}_{\lambda(\Gamma)}V \equiv 1$. It is clear for $p=2$. If $p$ is odd, then by Proposition \ref{hyper and aniso and properties}.(\ref{even deg from def of aniso}) we have that $s=|\mathbb{F}_p[\lambda(\Gamma)]/\mathbb{F}_p| $ is even and $\lambda(\Gamma) \subseteq \mu_{p^{s/2}+1}$. Therefore $|\mathbb{F}_p[\lambda(\Gamma)]^\times / \mu_{p^{s/2}+1}| = p^{s/2} - 1$ is even and $\mathrm{sgn}_{\lambda(\Gamma)} V=(\mathrm{sgn}_{\lambda(\Gamma)} \mu_{p^{s/2} + 1})^{p^{s/2} - 1}=1$.
\qed\end{rmk}

Suppose that $\Gamma$ is one of the cyclic subgroups $$\mu:=\mu_E/\mu_F\text{ and }\varpi:=\left<\varpi_E\right>/\left<\varpi_F\right>$$ of $\Psi_{E/F}$. We study the symplectic $\mathbb{F}_p\Gamma$-submodules of the standard $\mathbf{k}_F\Psi_{E/F}$-module $U$ and compute the t-factors of them. Recall that $[g] \in (\Gamma_E \backslash \Gamma_F / \Gamma_E)'$ is {symmetric} if $[g] = [g^{-1}]$ and is {asymmetric} otherwise.

If $[g]$ is asymmetric, then $U_{\pm [g]} = U_{[g]} \oplus U_{[g^{-1}]}$ is a hyperbolic $\mathbb{F}_p \Gamma$-module. Write $[g]=[\sigma^k\phi^i]$ using the description in Proposition \ref{explicit expression of double coset}. If $\Gamma=\mu$, then we have
    $$t^0_{\mu}(U_{\pm[\sigma^k\phi^i]})=1\text{ and }t^1_{\mu}(U_{\pm[\sigma^k\phi^i]}) :\zeta\mapsto \mathrm{sgn}_{\zeta^{q^{i}-1}}(U_{[\sigma^k\phi^i]})\text{ for all }\zeta\in \mu_E.$$
In particular, if $i=f/2$, then both $ U_{[\sigma^k\phi^{f/2}]}$ and $ U_{[(\sigma^k\phi^{f/2})^{-1}]}$ contain the $\mathbb{F}_p\mu$-module $\mathfrak{m}_{f/2}$. This module is anisotropic indecomposable by Proposition 19 of \cite{BH-ET3}. Hence $U_{\pm[\sigma^k\phi^{f/2}]}$ is even anisotropic, and so
     $$t^0_{\mu}(U_{\pm[\sigma^k\phi^{f/2}]}) = 1\text{ and }t^1_{\mu}(U_{\pm[\sigma^k\phi^{f/2}]}) \equiv 1.$$
If $\Gamma=\varpi$, then we have similarly
    $$t^0_{\varpi}(U_{\pm[\sigma^k\phi^i]})=1\text{ and }t^1_{\varpi}(U_{\pm[\sigma^k\phi^i]})(\varpi_E) = \mathrm{sgn}_{\zeta^k_e \zeta_{\phi^i}}(U_{[\sigma^k\phi^i]}).$$

Now we assume that $[g]$ is symmetric. We first give some properties of the submodule $U_{[g]}$. Let $t=t_k$ be the minimal solution of Proposition \ref{properties of symmetric [g]}.
\begin{lemma}\label{degree condition for symmetric double coset}
 Assume $[g]\neq[\sigma^{e/2}]$, then
$|U_{[g]}/\mathbf{k}_E|$ equals $2t_k$ in case $i=0$, and equals $2t_k+1$ in case $i=f/2$.
\end{lemma}
\begin{proof}
Since $U_{[g]}$ is the field extension of $\mathbf{k}_E$ generated by $\zeta^k_e$, the degree of $U_{[g]}/\mathbf{k}_E$ is equal to the minimal solution $s$ of $e|(q^{fs}-1)k$. Therefore $s$ must be the number indicated above.
\end{proof}

We denote by $U_{\pm [g]}$ the subfield of $U_{[g]}$ such that $|U_{[g]}/U_{\pm [g]}|=2$. Notice that $U_{\pm [g]}$ contains $\mathbf{k}_E$ as a sub-extension if and only if $[g]=[\sigma^{k}]$, with $k\neq 0$ and $e/2$. In the exception case when $[g] = [\sigma^{e/2}]$, the corresponding root $\lambda$ satisfies $\lambda^2=1$, the module $U_{[\sigma^{e/2}]} \cong \mathbf{k}_E$ as a $\mathbf{k}_F$-vector space, and the minimal solution in Proposition \ref{properties of symmetric [g]} is $t_{e/2} = 0$. As we will see in Proposition \ref{trivialness of some V[g]}, not all submodules of $U$ admit symplectic structures. This applies in particular to $U_{[\sigma^{e/2}]}$.

 Now we compute the t-factors of $U_{[g]}$ for symmetric $[g]$. Suppose that $[g]=[\sigma^{k}]$, with $k \neq 0$ or $e/2$. Since $\mu_E$ acts on $U_{[\sigma^{k}]}$ trivially, we have
\begin{equation*}
t^0_{\mu}(U_{[\sigma^k]})=1\text{ and }t^1_{\mu}(U_{[\sigma^k]}) \equiv1.
\end{equation*}
To compute $t^i_\varpi(U_{[\sigma^k]})$, $i=0,1$, we consider $U_{[\sigma^k]}$ as an $\mathbb{F}_p\varpi$-module.
Since $\left[\begin{smallmatrix}
1\\  \sigma^k
\end{smallmatrix}\right](\varpi_E)=\zeta_e^k$, we have that $U_{[\sigma^k]}$ isomorphic to $\mathbf{k}_E[\zeta^k_e]$ as a $\mathbf{k}_F$-vector space. Moreover, the simple submodule is anisotropic and is isomorphic to $\mathbb{F}_p[\zeta^k_e]$ as an $\mathbb{F}_p$-vector space. We have the following property about its multiplicity.
\begin{lemma}\label{order 2 trick}
The degree $r=r_{[\sigma^k]} = |\mathbf{k}_E[\zeta^k_e]/\mathbb{F}_p[\zeta^k_e]|$ is odd.
\end{lemma}
\begin{proof}
We see that $\left[\begin{smallmatrix}
1\\  \sigma^k
\end{smallmatrix}\right]( \varpi) $ is contained in $ \ker(N_{\mathbf{k}_E[\zeta^k_e]/\mathbf{k}_E[\zeta^k_e]_\pm})$, which is the group of $(q^{ft_k}+1)$-roots of unity in $\mathbf{k}_E[\zeta^k_e]$. Suppose that $r$ is even. If $\mathbb{F}_p[\zeta_e^k]$ has $Q$ elements, then $\mathbf{k}_E[\zeta^k_e]$ has $Q^r=q^{2ft_k}$ elements. We have
\begin{equation*}
\left[\begin{smallmatrix}
1\\  \sigma^k
\end{smallmatrix}\right](\varpi)\subseteq \mathbb{F}_p[\zeta_e^k]^\times\cap \ker(N_{\mathbf{k}_E[\zeta^k_e]/\mathbf{k}_E[\zeta^k_e]_\pm})=\mu_{Q-1}\cap \mu_{Q^{r/2}+1}\subseteq\{\pm1\},
\end{equation*}
forcing $k=0$ or $e/2$. This is a contradiction.
\end{proof}
By Lemma \ref{order 2 trick}, we have
\begin{equation*}
t^0_\varpi(U_{[\sigma^k]}) = (-1)^r= -1
\text{ and }
t^1_\varpi(U_{[\sigma^k]}):\varpi_E\mapsto\left(\frac{\zeta_e^k}{\mu_{p^s+1}}\right)^{r}=\left(\frac{\zeta_e^k}{\mu_{p^s+1}}\right),
\end{equation*}
where $|\mathbb{F}_p[\zeta^k_e]/\mathbb{F}_p|=2s$ and $$\mu_{p^s+1}=\ker(N_{\mathbb{F}_p[\zeta^k_e]/\mathbb{F}_p[\zeta^k_e]_\pm}) $$ for the quadratic extension $\mathbb{F}_p[\zeta^k_e]/\mathbb{F}_p[\zeta^k_e]_\pm$.

Now suppose $[g]=[\sigma^{k}\phi^{f/2}]$. We first consider $U_{[\sigma^k \phi^{f/2}]}$ as an $\mathbb{F}_p\mu$-module. Notice that $$\left[\begin{smallmatrix}
1\\  \sigma^k\phi^{f/2}
\end{smallmatrix}\right](\mu_E) = \ker (N_{\mathbf{k}_E/\mathbf{k}_{E\pm}})=\mu_{q^{f/2}+1}\text{ and }\mathbb{F}_p\left[\left[\begin{smallmatrix}
1\\  \sigma^k\phi^{f/2}
\end{smallmatrix}\right] (\mu_E)\right]=\mathbf{k}_E.$$ Since each simple $\mathbb{F}_p\mu$-submodule in $U_{[\sigma^k \phi^{f/2}]}$ is anisotropic, we have
\begin{equation*}
t^0_{\mu}(U_{[\sigma^k \phi^{f/2}]})=(-1)^{2t_k+1}=-1
\end{equation*}
and
\begin{equation*}
t^1_\mu(U_{[\sigma^k \phi^{f/2}]}):\zeta\mapsto\left(\frac{\zeta^{{q^{f/2}-1}}}{\mu_{q^{f/2}+1}}\right)^{2t_k+1}= \left(\frac{\zeta^{{q^{f/2}-1}}}{\mu_{q^{f/2}+1}}\right)\text{ for all }\zeta\in \mu_E.
\end{equation*}
Notice that the latter is the unique quadratic character of $\mu_E$. We then regard $U_{[\sigma^k \phi^{f/2}]}$ as an $\mathbb{F}_p\varpi$-module. The action of $\varpi_E$ is the multiplication of $\left[\begin{smallmatrix}
1\\  \sigma^k\phi^{f/2}
\end{smallmatrix}\right](\varpi_E)=(\zeta_e^k\zeta_{\phi^{f/2}})^{-1}$. We distinguish this value between the following cases.
\begin{enumerate}[(i)]
\item If $\zeta_e^k\zeta_{\phi^{f/2}}=1$, then $\varpi_E$ acts trivially and hence $$t^0_{\varpi}(U_{[\sigma^k \phi^{f/2}]}) = 1\text{ and }t^1_{\varpi}(U_{[\sigma^k \phi^{f/2}]}) \equiv 1.$$
\item If $\zeta_e^k\zeta_{\phi^{f/2}}=-1$, then $U_{[\sigma^k \phi^{f/2}]}\cong\mathbf{k}_E$ and $\mathbb{F}_p[\zeta_e^k\zeta_{\phi^{f/2}}]=\mathbb{F}_p$. Since $|\mathbf{k}_E/\mathbb{F}_p|$ is even, the module $U_{[\sigma^k \phi^{f/2}]}$ is even anisotropic. We have
$$t^0_{\varpi}(U_{[\sigma^k \phi^{f/2}]})=1\text{ and }t^1_{\varpi}(U_{[\sigma^k \phi^{f/2}]})(\varpi_E)=\mathrm{sgn}_{-1}(U_{\pm[\sigma^k \phi^{f/2}]})=(-1)^{\frac{1}{2}(q^{f/2}-1)}.$$
\item If $\zeta_e^k\zeta_{\phi^{f/2}}\neq\pm1$, then similar to Lemma \ref{order 2 trick} we have that $|U_{[\sigma^k\phi^{f/2}]}/\mathbb{F}_p[\zeta_e^k\zeta_{\phi^{f/2}}]|$ is odd. We have $$t^0_{\varpi}(U_{[\sigma^k\phi^{f/2}]}) = -1\text{ and }t^1_{\varpi}(U_{[\sigma^k\phi^{f/2}]}):\varpi_E\mapsto \left(\frac{\zeta_e^k\zeta_{\phi^{f/2}}}{\mu_{p^s+1}}\right),$$ where $|\mathbb{F}_p[\zeta_e^k\zeta_{\phi^{f/2}}]/\mathbb{F}_p|=2s$ and $$\mu_{p^s+1}=\ker(N_{\mathbb{F}_p[\zeta_e^k\zeta_{\phi^{f/2}}]/\mathbb{F}_p[\zeta_e^k\zeta_{\phi^{f/2}}]_\pm}) $$ for the quadratic extension $\mathbb{F}_p[\zeta_e^k\zeta_{\phi^{f/2}}]/\mathbb{F}_p[\zeta_e^k\zeta_{\phi^{f/2}}]_\pm$.
\end{enumerate}
We summarize the above values of t-factors in the following
\begin{prop} \label{summary of t-factors}
\begin{enumerate}[(i)]
  \item If $\Gamma=\mu$, then the t-factors are as follows.
  \begin{enumerate}
\item If $[g]=[\sigma^k\phi^i]$ is asymmetric, then $t^0_{\mu}(U_{\pm[g]})=1$ and\\ $t^1_{\mu}  (U_{\pm[g]}) :\zeta\mapsto    \mathrm{sgn}_{\zeta^{q^{i}-1}}(U_{[g]})$.
    \item If $[g]=[\sigma^k]$ is symmetric, then $t^0_{\mu}(U_{[g]})=1$ and $t^1_{\mu}(U_{[g]}) \equiv1.$
    \item If $[g]=[\sigma^k\phi^{f/2}]$ is symmetric, then $t^0_{\mu}(U_{[g]})=-1$ and $t^1_\mu(U_{[g]})$ is quadratic.
  \end{enumerate}
  \item If $\Gamma=\varpi$, then the t-factors are as follows.
  \begin{enumerate}
    \item If $[g]=[\sigma^k\phi^i]$ is asymmetric, then $t^0_{\varpi}(U_{\pm[g]})=1$ and $$t^1_{\varpi}(U_{\pm[g]})(\varpi_E) = \mathrm{sgn}_{\zeta^k_e \zeta_{\phi^i}}(U_{[g]}).$$
    \item If $[g]=[\sigma^k]$ is symmetric, then $
t^0_\varpi(U_{[g]})= -1$ and $$
t^1_\varpi(U_{[g]}):\varpi_E\mapsto\left(\frac{\zeta_e^k}{\ker(N_{\mathbb{F}_p[\zeta^k_e]/\mathbb{F}_p[\zeta^k_e]_\pm})}\right).$$
    \item If $[g]=[\sigma^k\phi^{f/2}]$ is symmetric,
        \begin{enumerate}[(I)]
\item if $\zeta_e^k\zeta_{\phi^{f/2}}=1$, then $t^0_{\varpi}(U_{[g]}) = 1$ and $t^1_{\varpi}(U_{[g]}) \equiv 1$;
\item if $\zeta_e^k\zeta_{\phi^{f/2}}=-1$, then $t^0_{\varpi}(U_{[g]})=1$ and $t^1_{\varpi}(U_{[g]})(\varpi_E)=(-1)^{\frac{1}{2}(q^{f/2}-1)}$;
\item if $\zeta_e^k\zeta_{\phi^{f/2}}\neq\pm1$, then $t^0_{\varpi}(U_{[g]}) = -1$ and $$t^1_{\varpi}(U_{[g]}):\varpi_E\mapsto \left(\frac{\zeta_e^k\zeta_{\phi^{f/2}}}{\ker(N_{\mathbb{F}_p[\zeta_e^k\zeta_{\phi^{f/2}}]/\mathbb{F}_p[\zeta_e^k\zeta_{\phi^{f/2}}]_\pm})}\right).$$
\end{enumerate}
  \end{enumerate}
\end{enumerate}\qed
\end{prop}

\section{Complementary modules}\label{section comp mod}

This section involves quadratic Gauss sums. Many standard textbooks on number theory, e.g. \cite{Lang-ANT}, can deduce the results here. Some of the details can be found in section 4 of \cite{BH-ET2}.

Let $W$ be a finite dimensional $\mathbf{k}_F$-vector space, $Q:W\rightarrow\mathbf{k}_F$ be a quadratic form, and $\psi_F$ be a non-trivial character of $\mathbf{k}_F$. We define the Gauss sum $$\mathfrak{g}(Q,\psi_F)=\sum_{x\in W}\psi_F(Q(x,x)).$$ The simplest example is when $W=\mathbf{k}_F$ and $Q(x,x)=x^2$. In this case we denote the Gauss sum by $\mathfrak{g}(\psi_F)$. Suppose that $Q$ is non-degenerate, in the sense that $\det Q\neq 0$, then clearly
\begin{equation}\label{gauss sum reduction}
\mathfrak{g}(Q,\psi_F)=\left(\frac{\det Q}{\mathbf{k}_F^\times}\right)\mathfrak{g}(\psi_F)^{\dim_{\mathbf{k}_F} W}.
\end{equation}
We also define the normalized Gauss sums $\mathfrak{n}(Q,\psi_F)=(\#W)^{-1/2}\mathfrak{g}(Q,\psi_F)$ and $\mathfrak{n}(\psi_F)=q^{-1/2}\mathfrak{g}(\psi_F)$. The equation (\ref{gauss sum reduction}) is still true if we replace $\mathfrak{g}$ by $\mathfrak{n}$. We can easily show that the Gauss sum is a convoluted sum $$\mathfrak{g}(\psi_F)=\sum_{x\in\mathbf{k}_F}\left(\frac{x}{\mathbf{k}_F^\times}\right)\psi_F(x).$$
Using this and the technique similar to \cite{Lang-ANT} IV, \S3, $\textbf{GS2}$, we can deduce that $$\mathfrak{n}(Q,\psi_F)^2=\left(\frac{-1}{q}\right),$$ which implies that the normalized sum $\mathfrak{n}(Q,\psi_F)$ is a 4th root of unity.

We now assume that $W$ is a $\mathbf{k}_F\Psi_{E/F}$-submodule of the standard module $U$. Let $Q$ be a non-degenerate $\Psi_{E/F}$-invariant bilinear form of $W$. Let $Q_{[g]}$ be the restriction of $Q$ on $\mathbf{W}_{[g]} = W_{[g]}\oplus W_{[g^{-1}]}$ if $[g]$ is asymmetric, and on $\mathbf{W}_{[g]} = {W}_{[g]} $ if $[g]$ is symmetric. We assume that each $Q_{[g]}$ is non-degenerate if $\mathbf{W}_{[g]}$ is non-trivial.
\begin{prop}\label{n-factor of comp-mod}
If ${[g]}\neq[\sigma^{e/2}]$ and $\mathbf{W}_{[g]}$ is non-trivial, then the following holds.
  \begin{enumerate}[(i)]
     \item The sum $\mathfrak{n}(Q_{[g]},\psi_F)$ is equal to $1$ if $[g]$ is asymmetric, and is equal to $=-1$ if $[g]$ is symmetric. \label{n-factor of comp-mod first}
    \item The sum $\mathfrak{n}(Q_{[g]},\psi_F)$ depends only on the symmetry of $[g]$ and is independent of the quadratic form $Q$. \label{n-factor of comp-mod second}
  \end{enumerate}
\end{prop}
In the exceptional case $[g]=[\sigma^{e/2}]$, the sum  $\mathfrak{n}(Q_{[g]},\psi_F)$ is an arbitrary 4th root of unity.
\begin{proof}(of Proposition \ref{n-factor of comp-mod})
 We recall, from Proposition 4.4 of \cite{BH-ET2}, that if $\Psi_{E/F}$ is cyclic and $W$ is a non-trivial indecomposable $\mathbf{k}_F\Psi_{E/F}$-submodule of $U_{[g]}$, then we have
\begin{equation}\label{jacobi symbol depend on type only}
\left(\frac{\det Q|_W}{\mathbf{k}_F^\times}\right)=
\begin{cases}
\left(\frac{-1}{q}\right)^{\dim_{\mathbf{k}_F} W } & \text{if }{[g]} \text{ is asymmetric},\\
-\left(\frac{-1}{q}\right)^{\dim_{\mathbf{k}_F} W/2 } & \text{if }{[g]} \text{ is symmetric}.
\end{cases}
\end{equation}
The formula (\ref{jacobi symbol depend on type only}) generalizes to the case when $\Psi_{E/F}$ is not necessarily cyclic and when $W$ is replaced by $\mathbf{W}_{[g]}$. It is clear if $[g]$ is asymmetric. If $[g]$ is symmetric, it suffices to claim that $U_{[g]}$ contains an indecomposable $\mathbf{k}_F\Psi_{E/F}$-component of odd degree. When $[g]=[\sigma^k ]$, $k\neq 0$ or $e/2$, the indecomposable $\mathbf{k}_F\Psi_{E/F}$-component is $\mathbf{k}_F[\zeta_e^k]$ as a field extension of $\mathbf{k}_F$. Hence the claim is true by Lemma \ref{order 2 trick}. When $[g]=[\sigma^k\phi^{f/2} ]$, we can show that $U_{[g]}=\mathbf{k}_E[\zeta_e^k\zeta_{\phi^{f/2}}]$ is the indecomposable $\mathbf{k}_F\Psi_{E/F}$-component. Hence the claim is clearly true. Therefore the generalization of (\ref{jacobi symbol depend on type only}) implies that the discriminant of $Q_{[g]}$ mod $\mathbf{k}_F^{\times2}$ is determined by the underlying $\mathbf{k}_F\Psi_{E/F}$-module structure. Moreover, using (\ref{gauss sum reduction}), we have
\begin{equation*}
\mathfrak{n}(Q_{[g]},\psi_F)=\left(\frac{\det Q_{[g]}}{\mathbf{k}_F^\times}\right)\mathfrak{n}(\psi_F)^{\dim_{\mathbf{k}_F} \mathbf{W}_{[g]}}=\pm \left(\frac{-1}{q}\right)^{\dim_{\mathbf{k}_F} \mathbf{W}_{[g]}/2}\left(\frac{-1}{q}\right)^{\dim_{\mathbf{k}_F} \mathbf{W}_{[g]}/2}=\pm1,
\end{equation*}
where the sign is determined by the symmetry of ${[g]}$ as in (\ref{jacobi symbol depend on type only}).
\end{proof}

Given a symplectic submodule $V$ of $U$ without trivial components, let $W$ be the \emph{complementary module} of $V$, in the sense that $$W\oplus V=U'=\bigoplus_{[g]\in (W_E\backslash W_{F} / W_E)'}U_{[g]}.$$
Suppose that $Q$ is a non-degenerate $\Psi_{E/F}$-invariant quadratic form on $W$ such that each $Q_{[g]}$ is also non-degenerate on $\mathbf{W}_{[g]}$. We define our extra t-factor to be $$t(\mathbf{W}_{[g]})=\mathfrak{n}(Q_{[g]},\psi_F).$$ In particular, we have the following.
\begin{enumerate}[(i)]
  \item If $W_{[g]}$ is trivial , then $t(W_{[g]})=1$.
  \item If $[g]$ is asymmetric, then $t(\mathbf{W}_{[g]})=1$.
        \item If $[g]\neq[\sigma^{e/2}]$ is symmetric and $W_{[g]}$ is non-trivial, then $t(W_{[g]})=-1$.
\end{enumerate}
Again if $[g]\neq[\sigma^{e/2}]$, the t-factor $t(\mathbf{W}_{[g]})$ is independent of the quadratic form $Q$ above.

\chapter{Essentially tame supercuspidal representations}\label{chapter Tame supercuspidal representations}

In section \ref{section admissible characters revisited}, we recall the admissible characters defined in section \ref{section admissible characters} and extract more structural information from them. With such information we can describe, in section \ref{Supercuspidal Representations}, the bijection in Proposition \ref{prop bijection of P and A} on extending from an admissible character in several steps to an essentially tame supercuspidal representation of $G(F)$. In between the steps we construct a finite symplectic $\mathbf{k}_F \Psi_{E/F}$-module, of which we can apply the decomposition in Proposition \ref{decomp of standard module} and attach t-factors on the isotypic components. Finally, in section \ref{section explicit rectifiers}, we restate the values of the rectifiers defined by Bushnell-Henniart in terms of t-factors.

\section{Admissible characters revisited}\label{section admissible characters revisited}

We recall more information about admissible characters, a notion defined in section \ref{section admissible characters}. The details can be found in \cite{Moy} and also section 1 of \cite{BH-ET2}. We further assume that our additive character $\psi_F$ of $F$ is of level 0, which means that $\psi_F$ is trivial on $\mathfrak{p}_F$ but not on $\mathfrak{o}_F$. For any tamely ramified extension $E/F$, we write $\psi_E=\psi_F\circ\mathrm{tr}_{E/F}$. Given a character $\xi$ of $E^\times$, the $E$-level $r_E(\xi)$ of $\xi$ is the smallest integer $r$ such that $\xi|_{U_E^{r+1}}\equiv 1$. Notice that $\xi$ is {tamely ramified} if $r_E(\xi)=0$. Suppose that $\xi$ is admissible over $F$, then it admits a \emph{Howe-factorization} (see Lemma 2.2.4 of \cite{Moy}) of the form
\begin{equation}\label{Howe factorization}
\xi=  \xi_{-1}(\xi_0 \circ N_{E/E_0})\cdots(\xi_d \circ N_{E/E_d})(\xi_{d+1}\circ N_{E/F}).
\end{equation}
We need to specify the notations in (\ref{Howe factorization}). \begin{enumerate}[(i)]
\item  We have a decreasing sequence of fields
\begin{equation}\label{decreasing subfield}
E=E_{-1}\supseteq E_0 \supsetneq E_1 \cdots \supsetneq E_d \supsetneq E_{d+1}=F.
 \end{equation}
 Each $\xi_i$ is a character of $E_i^\times$ and $\xi_{d+1}$ is a character of $F^\times$.
 \item Let $r_i$ be the $E$-level of $\xi_i$, defined as the $E$-level of $\xi_i \circ N_{E/E_i}$, and $r_{d+1}$ be the $E$-level of $\xi$. We assume that $\xi_{d+1}$ is trivial if $r_{d+1}=r_d$.  We call the $E$-levels
$r_0<\dots<r_d$
the \emph{jumps} of $\xi$.
 \item If $E_0=E$, then we replace $(\xi_0 \circ N_{E/E_0})\xi_{-1}$ by $\xi_0$. If $E_0\subsetneq E$, then we assume that $\xi_{-1}$ is tame and $E/E_0$ is unramified.
\end{enumerate}

We define the wild component of $\xi$ to be $\Xi_0 \circ N_{E/E_0}$ where $$\Xi_0 = (\xi_{d+1} \circ N_{E_0/F}) \cdots (\xi_1 \circ N_{E_0/E_1}) \xi_0.$$ We also define the tame component of $\xi$ to be $\xi_{-1} = \xi (\Xi_0 \circ N_{E/E_0})^{-1}$.

\begin{prop}\label{jump proposition}
For $i = 0, \dots, d+1$, we have the following.
\begin{enumerate}[(i)]
\item If $s_i$ is the $E_i$-level of $\xi_i$, then $s_i e(E/E_i) = r_i$. \label{jump is multiple}
\item There is a unique $\alpha_i \in \langle \varpi_{E_i}\rangle \times \mu_{E_i}$ such that $v_{E_i}(\alpha_i) = -s_i$ and $$\xi_i|_{U^{s_i}_{E_i}}(1+x) = \psi_{E_i}(\alpha_i x).$$ \label{additive char}
\newcounter{enumi_saved}
\setcounter{enumi_saved}{\value{enumi}}
 \end{enumerate}For $i = 0, \dots, d$, we have the following.
 \begin{enumerate}[(i)]
 \setcounter{enumi}{\value{enumi_saved}}
\item Each character $\xi_i$ is generic over $E_{i+1}$, in the sense that $E_{i+1}[\alpha_i] = E_i$. \label{generic}
\item We have the relation $\gcd(r_i, e(E/E_{i+1})) = e(E/E_i)$. \label{gcd-level-ramification}
\end{enumerate}
\end{prop}
\begin{proof}
 (\ref{jump is multiple}) comes from an elementary calculation of the image of $N_{E/E_i}$. (\ref{additive char}) and (\ref{generic}) can be found in section 2.2 of \cite{Moy}. For (\ref{gcd-level-ramification}), that $E_{i+1}[\alpha_i] = E_i$ in (\ref{generic}) implies that $$\gcd(v_{E_i}(\alpha_i), e(E_i/E_{i+1})) = 1.$$ Then (\ref{additive char}) and (\ref{jump is multiple}) imply the desired result.
\end{proof}

We define the \emph{jump data} of an admissible character $\xi$ as a collection consisting of the sequence of subfields $E\supseteq E_0 \supsetneq E_1 \supsetneq\cdots \supsetneq E_d \supsetneq F$ as above and the increasing sequence $r_0<\dots<r_d$ of positive integers, with $r_0\geq 1$, as the jumps of $\xi$. We can define a $W_F$-action on the sequence of subfields as $$g:\{E\supseteq E_0 \supsetneq E_1 \supsetneq\cdots \supsetneq E_d \supsetneq F\}\mapsto \{{}^gE\supseteq {}^gE_0 \supsetneq {}^gE_1 \supsetneq\cdots \supsetneq {}^gE_d \supsetneq F\}$$ for all $g\in W_F$. The jumps of ${}^g\xi$ are clearly the same as those of $\xi$. Hence we can define the jump data of the equivalence class $(E/F,\xi)\in P_n(F)$ in the obvious sense.

For each $\xi_i$, where $i=0,\dots,d+1$, there is $\beta_i\in E_i\cap \mathfrak{p}_E^{-r_i}$ such that
\begin{equation*}
\phi_i\circ N_{E/E_i}(1+x)= \psi_F(\text{tr}_{E/F}(\beta_ix))\text{ for all }x\in \mathfrak{p}_E^{\lfloor{r_i}/{2}\rfloor+1}.
\end{equation*}
Such $\beta_i$ can be chosen mod $\mathfrak{p}_E^{1-\left(\lfloor{r_i}/{2}\rfloor+1\right)}$. If we write
\begin{equation*}
\beta_i=\zeta_i\varpi_E^{-r_i}u_i\text{ for some }\zeta_i\in \mu_E\text{ and }u_i\in U^1_E,
\end{equation*}
then the `first term' $\zeta_i\varpi_E^{-r_i}$ of $\beta_i$ is equal to $\alpha_i$ defined in Proposition \ref{jump proposition}.(\ref{additive char}). We write
\begin{equation}\label{b is sum of b_i}
\beta=\beta(\xi)=\beta_{d+1}+\cdots+\beta_0.
\end{equation}

\section{Essentially tame supercuspidal representations}\label{Supercuspidal Representations}

We briefly describe how to parameterize essentially tame supercuspidal representations by admissible characters through the bijection $\Pi_n:P_n(F)\rightarrow\mathcal{A}^\mathrm{et}_n(F)$ defined in Proposition \ref{prop bijection of P and A}. The details can be found in \cite{Moy}, \cite{BK}, and are also summarized in \cite{BH-ET1}. We first recall certain subgroups of $G(F)$. Let $E\supseteq E_0 \supsetneq E_1 \cdots \supsetneq E_d \supsetneq F$ be a decreasing sequence of fields. Write $B_i = \mathrm{End}_{E_i}(E)$ and define
\begin{equation}\label{hereditary orders}
\begin{split}
&\mathfrak{B}_i = \{ x \in B_i | x\mathfrak{p}_E^k \subseteq \mathfrak{p}_E^k \text{ for all } k\in \mathbb{Z} \}\text{ and}
\\
&\mathfrak{P}_{\mathfrak{B}_i}=\{ x \in B_i | x \mathfrak{p}_E^k \subseteq \mathfrak{p}_E^{k+1} \text{ for all } k \in \mathbb{Z}\},
\end{split}
\end{equation}
namely the hereditary order in $B_i$ and its radical as those defined in (\ref{hereditary order A}) and (\ref{hereditary order PA}). We then define subgroups of $B_i^\times$
\begin{equation}\label{unit groups}
\begin{split}
U_{\mathfrak{B}_i}=\{ x \in B_i^\times | x \mathfrak{p}_E^k = \mathfrak{p}_E^k \text{ for all } k\in \mathbb{Z}\}\text{ and } U_{\mathfrak{B}_i}^j = 1 + \mathfrak{P}_{\mathfrak{B}_i}^j\text{ for }j>0.
\end{split}
\end{equation}
If $E/E_0$ is unramified, then we can replace all $\mathfrak{p}_E$ by $\mathfrak{p}_{E_0}$ in (\ref{hereditary order A}), (\ref{hereditary order PA}), (\ref{hereditary orders}) and (\ref{unit groups}). If we write $A=\mathrm{End}_F(E)$, then we certainly have $U_\mathfrak{A}$ and $U^j_\mathfrak{A}$ defined as in (\ref{unit groups}) with $B_i$ replaced by $A$. The multiplication of $E$ identifies $E$ as a subspace of $A$. Choose an isomorphism of $A^\times\cong G(F)$ such that $E^\times$ embeds into $G(F)$ by an $F$-regular morphism. Then $\mathfrak{A}^\times$, $U_\mathfrak{A}$, $U^j_\mathfrak{A}$, $\mathfrak{B}_i^\times$, $U_{\mathfrak{B}_i}$ and $ U_{\mathfrak{B}_i}^j$ embed into $G(F)$ accordingly.

Suppose that the fields $E_i$, $i=0,\dots,d$, are defined by the Howe-factorization of $\xi\in P(E/F)$ as in (\ref{Howe factorization}) with jumps $\{r_0,\dots,r_d\}$. We define two numbers $j_i$ and $h_i$ by
\begin{equation}\label{j_i and h_i}
j_i=\lfloor\frac{r_i+1}{2}\rfloor\leq h_i=\lfloor\frac{r_i}{2}\rfloor+1.
\end{equation}
 We construct the subgroups
\begin{equation}\label{group H(xi) J(xi) and bold-J(xi)}
\begin{split}
&H^1(\xi)= U_{\mathfrak{B}_{0}}^1 U_{\mathfrak{B}_{1}}^{h_{0}} \cdots U_{\mathfrak{B}_{d}}^{h_{d-1}} U_{\mathfrak{A}}^{h_{d}},
\\
&J^1(\xi)=U_{\mathfrak{B}_{0}}^1 U_{\mathfrak{B}_{1}}^{j_{0}} \cdots U_{\mathfrak{B}_{d}}^{j_{d-1}} U_{\mathfrak{A}}^{j_{d}}\subseteq J(\xi)=U_{\mathfrak{B}_{0}} U_{\mathfrak{B}_{1}}^{j_{0}} \cdots U_{\mathfrak{B}_{d}}^{j_{d-1}} U_{\mathfrak{A}}^{j_{d}}\text{, and}
\\
&\mathbf{J}(\xi)=E^\times J(\xi)=E_0^\times J(\xi).
\end{split}
\end{equation}
We abbreviate these groups by $H^1,\, J^1,\, J$ and $\mathbf{J}$ if the admissible character $\xi$ is fixed. Notice that $H^1,\, J^1,J$ are compact subgroups and  $\mathbf{J}$ is a compact-mod-center subgroup.

Now we briefly describe the construction by Bushnell and Kutzko \cite{BK} and in \cite{BH-ET1} of the supercuspidal representation $\pi_\xi$ from an admissible character $\xi$ in five steps.
\begin{enumerate}[(i)]
\item From the Howe factorization (\ref{Howe factorization}) of $\xi$ we can define a character $\theta=\theta(\xi)$ on $H^1$. This character depends only on the wild component $\Xi_0\circ N_{E/E_0}$ of $\xi$. In fact, according to the definition of simple characters in (3.2) of \cite{BK}, there can be a number of such characters associated to $\xi$. There is a canonical one $\theta(\xi)$ constructed in section 3.3 of \cite{Moy}, called the simple character of $\xi$ in this thesis.
\item By classical theory of Heisenberg representation (see (5.1) of \cite{BK}), we can extend $\theta$ to a unique representation $\eta$ of $J^1$ using the symplectic structure of $V=J^1/H^1$ defined by $\theta$. \label{Heisenberg extension}
    \item There is a unique extension $\Lambda_0=\Lambda(\Xi_0\circ N_{E/E_0})$ of $\eta$ on $\mathbf{J}$ satisfying the conditions in Lemma 1 and 2 of section 2.3 of \cite{BH-ET1}. The restriction $\Lambda_0|_J$ is called a $\beta$-extension of $\eta$, in the sense of (5.2.1) of \cite{BK}.\label{beta-extension}
       \item  We still need another representation $\Lambda(\xi_{-1})$ on $\mathbf{J}$, defined by the tame component $\xi_{-1}$ of $\xi$. Suppose that $\xi_{-1}|_{U_E}$ is the inflation of a character $\bar{\xi}_{-1}$ of $\mathbf{k}_E$. We first apply Green's parametrization \cite{Green} to obtain a unique (up to isomorphism) irreducible cuspidal representation $\bar{\lambda}_{-1}$ of $\mathrm{GL}_{|E/E_0|}(\mathbf{k}_{E_0}) = J / J^1$, then inflate $\bar{\lambda}_{-1}$ to a representation $\lambda_{-1}$ on $J$, and finally multiply $\xi(\varpi_E)$ to obtain $\Lambda(\xi_{-1})$ on $\mathbf{J} = \left<\varpi_E\right>J$.
           \item We form the extended maximal type $\Lambda_\xi=\Lambda(\xi_{-1})\otimes\Lambda_0$. The supercuspidal is then defined by
$\pi_\xi=\mathrm{cInd}_\mathbf{J}^G\Lambda_\xi.$
\end{enumerate}

\begin{rmk}
The wild component $\xi_w$ and tame component $\xi_t$ of $\xi$ are defined alternatively in \cite{BH-ET1}. We briefly explain that they extend to the same representation $\Lambda_\xi$ on $\mathbf{J}$. By construction we have $\xi = (\Xi_0 \circ N_{E/E_0}) \xi_{-1} = \xi_w \xi_t$. On the wild part, since $\xi_w|_{U^1_E} = (\Xi_0 \circ N_{E/E_0})|_{U^1_E}$, they induce the same simple character $\theta = \theta(\xi)$. Therefore we have an isomorphism of the $\beta$-extensions $\Lambda(\xi_w)=\Lambda_0\otimes \alpha$, where $\alpha$ is a tamely ramified character on $E^\times U_{\mathfrak{B}_0} / U^1_{\mathfrak{B}_0}$ such that $$\alpha|_{U_{\mathfrak{B}_0}} = \Xi_0|_{\mu_{E_0}} \circ \mathrm{det}_{\mathbf{k}_{E_0}} \circ (\mathrm{proj}^J_{J/J^1})\text{ and }\alpha(\varpi_E) = \xi^{-1}_w(\varpi_E)(\Xi_0 \circ N_{E/E_0})(\varpi_{E}).$$ (Compare this to (5.2.2) of \cite{BK} concerning $\beta$-extensions.) Here $\mathrm{proj}^J_{J/J_1}$ is the natural projection $J \rightarrow J/J^1 \cong \mathrm{GL}_{|E/E_0|}(\mathbf{k}_{E_0})$. On the tame part, it can be checked that $\Lambda(\xi_t) = \Lambda(\xi_{-1})\otimes \alpha^{-1}$. Indeed, $\xi_w$ is trivial on $\mu_E$ by construction in \cite{BH-ET1}, which implies that the Green's representations $\bar{\lambda_t}$ and $\bar{\lambda}_{-1}\otimes(\Xi_0 \circ \det_{\mathbf{k}_{E_0}})$ on $\mathrm{GL}_{|E/E_0|}(\mathbf{k}_{E_0})$ are isomorphic. Therefore $\Lambda(\xi_t) \otimes \Lambda(\xi_w) = \Lambda(\xi_{-1}) \otimes \Lambda_0=\Lambda_\xi$ as desired. With the Howe factorization of $\xi$ in hand, it is more natural to define our wild and tame component of $\xi$ as in the five steps in the preceding paragraph.
\qed\end{rmk}

We analyze the group extensions in step (\ref{Heisenberg extension}) and (\ref{beta-extension}) of the five steps of constructing supercuspidal representations. Since the group $\mathbf{J}$ normalizes $J^1$ and $H^1$, it acts on the quotient group $V=V(\xi) = J^1/H^1$. Such action clearly commutes with the scalars in $\mathbb{F}_p$ if we regard $V$ as an $\mathbb{F}_p$-vector space. We have a direct sum
\begin{equation}\label{coarse decomp of V}
  V=V_{E_0/E_1} \oplus \cdots \oplus V_{E_d/F},
\end{equation} where
$V_{E_i/E_{i+1}} =U^{j_i}_{\mathfrak{B}_{i+1}} / U^{j_i}_{\mathfrak{B}_i} U^{h_i}_{\mathfrak{B}_{i+1}}$. By the definitions in (\ref{j_i and h_i}), the module $V_i$ is non-trivial if and only if the jump $r_i$ is even, in which case we have $V_{E_i/E_{i+1}}\cong \mathfrak{B}_{i+1} / \mathfrak{B}_i + \mathfrak{P}_{\mathfrak{B}_{i+1}}$. We call this sum (\ref{coarse decomp of V}) the \emph{coarse decomposition} of $V$.
\begin{prop}\label{some facts in BK}
Let $H^1, J^1, \mathbf{J}, V, V_{E_i/E_{i+1}}$, and $\theta$ be those previously described.
\begin{enumerate}[(i)]
\item The commutator subgroup $[J^1, J^1]$ lies in $H^1$. \label{some facts in BK about commutator}
\item The group $\mathbf{J}$ normalizes each component $V_{E_i/E_{i+1}}$ and the simple character $\theta$. \label{some facts in BK about J normalizing}
\item The simple character $\theta$ induces a non-degenerate alternating $\mathbb{F}_p$-bilinear form $h_\theta:V \times V \rightarrow \mathbb{C}$ such that the coarse decomposition is an orthogonal sum. \label{some facts in BK about alt bilinear form}
\end{enumerate}
\end{prop}
\begin{proof}
Some of the proofs can be found in \cite{BK} chapter 3, for example (\ref{some facts in BK about commutator}) is in (3.1.15) of \cite{BK}, the existence of the form in (\ref{some facts in BK about alt bilinear form}) is in (3.4) of \cite{BK}, and that $\mathbf{J}$ normalizes $\theta$ in (\ref{some facts in BK about J normalizing}) is from (3.2.3) of \cite{BK}. That $\mathbf{J}$ normalizes each $V_{E_i/E_{i+1}}$ in (\ref{some facts in BK about J normalizing}) is clear by definition. That the coarse decomposition is orthogonal in (\ref{some facts in BK about alt bilinear form}) is from 6.3 of \cite{BH-LTL3}.
\end{proof}

 What are we interested in is the conjugate action of $E^\times$ on $V$ restricted form $\mathbf{J}$. By Proposition \ref{some facts in BK}.(\ref{some facts in BK about J normalizing}), The action of $\mathbf{J}$, and hence that of $E^\times$, preserves the symplectic structure defined by $\theta$. By Proposition \ref{some facts in BK}.(\ref{some facts in BK about commutator}), the subgroup $J^1$ of $\mathbf{J}$ acts trivially on $V$, so the $E^\times$-action factors through $E^\times/F^\times(E^\times\cap  J^1) \cong \Psi_{E/F}$. Hence $V$ is moreover a finite symplectic $\mathbb{F}_p\Gamma$-module for each cyclic subgroup $\Gamma$ of $\Psi_{E/F}$, and the form $h_\theta$ in Proposition \ref{some facts in BK}.(\ref{some facts in BK about alt bilinear form}) is $\Psi_{E/F}$-invariant. By construction the $\mathbf{k}_F\Psi_{E/F}$-module $V$ is always a submodule of the standard one $U=\mathfrak{A/P_A}$. We denote the $U_{[g]}$-isotypic component in $V$ by $V_{[g]}$, and call the decomposition $$V=\bigoplus_{[g]\in (\Gamma_E\backslash \Gamma_F / \Gamma_E)'}V_{[g]}$$ the \emph{complete decomposition} of $V$.

\begin{prop}\label{complete decomp orthogonal}
The complete decomposition of $V$ is an orthogonal sum with respect to the alternating form $h_\theta$.
\end{prop}
\begin{proof}
Recall the bijection $\Gamma_F \backslash \Phi \rightarrow (\Gamma_E\backslash \Gamma_F / \Gamma_E)'$ in Proposition \ref{orbit of roots as double coset} and write $V_{[g]}$ as $V_{[\lambda]}$ for suitable $[\lambda]\in \Gamma_F \backslash \Phi$. For every $\lambda$ and $\mu\in \Phi$ such that $\lambda\neq \mu$ or $\mu^{-1}$, let $v\in V_{[\lambda]}$ and $w\in V_{[\mu]}$. We can assume that $v$ and $w$ are respectively contained in certain $\Psi_{E/F}$-eigenspaces of $ V_{[\lambda]}\otimes_{\mathbb{F}_p}\bar{\mathbb{F}}$ and $V_{[\mu]}\otimes_{\mathbb{F}_p}\bar{\mathbb{F}}$. There is a unique $\Psi_{E/F}$-invariant alternating bilinear form $\tilde{h}$ of $V\otimes_{\mathbb{F}_p}\bar{\mathbb{F}}$ extending $h=h_\theta$. Therefore $\tilde{h}(v,w)=\tilde{h}({}^tv,{}^tw)=\lambda(t)\mu(t)\tilde{h}(v,w)$
for all $t\in \Psi_{E/F}$. The fact that $\lambda\neq \mu$ or $\mu^{-1}$ implies that $h(v,w)=\tilde{h}(v,w)=0$.
\end{proof}
For our purpose it is not necessary to know the form $h_\theta$ exactly. Indeed by Proposition \ref{hyper and aniso and properties}.(\ref{symplectic determined by underlying module}) the symplectic structure of $V$ is determined by its underlying $\mathbb{F}_p\Psi_{E/F}$-module structure. We conclude this section by proving a promised fact, that not all components of $U$ appear in the symplectic module $V$.

\begin{prop}\label{trivialness of some V[g]}
Let $E/F$ be a tame extension and $\xi$ runs through all admissible characters in $P(E/F)$.
\begin{enumerate}[(i)]
\item If $[g]\in \Gamma_E\backslash \Gamma_{E_0}/\Gamma_E=\Gamma_{E_0}/\Gamma_E$, then $V_{[g]}$ is always trivial.
\item If $e=e(E/F)$ is even, then $V_{[
\sigma^{e/2}]}$ is always trivial.
\end{enumerate}
\end{prop}
\begin{proof}
 The first statement is clear from the definition of $J^1(\xi)$ and $H^1(\xi)$ in (\ref{group H(xi) J(xi) and bold-J(xi)}). We prove the second statement. Let $E_j\supsetneq E_{j+1}$ be the intermediate subfields in (\ref{decreasing subfield}) such that $e(E/E_{j+1})$ is even and $e(E/E_j)$ is odd. By Proposition \ref{jump proposition}.(\ref{gcd-level-ramification}), the jump $r_j$ must be odd, and so $V_{E_j/E_{j+1}}$ is trivial. Since $\sigma^{e/2} \in \Gamma_{E_{j+1}} - \Gamma_{E_j}$, the component $V_{[\sigma^{e/2}]}$ is contained in $V_{E_j/E_{j+1}}$ and hence is also trivial.
\end{proof}

\section{Explicit values of rectifiers}\label{section explicit rectifiers}

 For each admissible character $\xi$ we give the values of the rectifier ${}_F\mu_\xi$ following \cite{BH-ET1}, \cite{BH-ET2}, \cite{BH-ET3}. Recall that each rectifier ${}_F\mu_\xi$ admits a factorization as in (\ref{product of rectifier}) in terms of $\nu$-rectifiers. Since each factor is tamely ramified, it is enough to give its values on $\mu_E$ and at $\varpi_E$.

As in section \ref{section autom ind}, we would distinguish between the following cases:
\begin{enumerate}[(I)]
  \item $E/K_l$ is totally ramified of odd degree, \label{section supercuspidal three cases totally-ram odd}
  \item each $K_i/K_{i-1}$, $i=1,\dots,l$, is totally ramified quadratic, and \label{section supercuspidal three cases totally-ram quad}
  \item $K_0/F$ is unramified. \label{section supercuspidal three cases unram}
\end{enumerate}
In case (\ref{section supercuspidal three cases totally-ram odd}), it is easy to describe the rectifier. By Theorem 4.4 of \cite{BH-ET1} we have
\begin{equation}\label{explicit rectifier ram-odd case}
  {}_{E/K_l}\mu_\xi|_{\mu_E}\equiv 1 \text{ and } {}_{E/K_l}\mu_\xi(\varpi_E)=\left(\frac{q^f}{e(E/K_l)}\right).
\end{equation}
We consider case (\ref{section supercuspidal three cases totally-ram quad}). For any field extension $K$, we denote by $$\mathfrak{G}_K=\mathfrak{G}_K(\varpi_K,\psi_K),\,\mathfrak{K}_K=\mathfrak{K}_K(\alpha_0,\varpi_K,\psi_K)\text{ and }\mathfrak{K}_K^\kappa=\mathfrak{K}^\kappa_K(\alpha_0,\varpi_K,\psi_K)$$ those various Gauss sums and Kloostermann sums in chapter 8 and 9 of \cite{BH-ET2}. By Theorem 3.1, Theorem 3.2, and Theorem 6.6 of \cite{BH-ET2}, we have
\begin{equation}\label{explicit rectifier ram-even odd case}
  \begin{split}
   & {}_{K_l/K_{l-1}}\mu_\xi|_{\mu_E}=\left(\frac{}{\mu_E}\right) \text{ and } \\ & {}_{K_l/K_{l-1}}\mu_\xi(\varpi_E)=t_\varpi(V_{K_{l}/K_{l-1}})\mathrm{sgn}((\mathfrak{G}_{K_{l-1}}/\mathfrak{G}_{K_l})(\mathfrak{K}^\kappa_{K_{l-1}}/\mathfrak{K}_{K_{l-1}})),
  \end{split}
\end{equation}
and for $j=l-1,\dots,1$,
\begin{equation}
  \label{explicit rectifier ram-even even case}
\begin{split}
  & {}_{K_j/K_{j-1}}\mu_\xi|_{\mu_E}\equiv 1
\text{ and } \\
&{}_{K_j/K_{j-1}}\mu_\xi(\varpi_E)=t_\varpi(V_{K_{j}/K_{j-1}})\mathrm{sgn}((\mathfrak{G}_{K_{j-1}}/\mathfrak{G}_{K_j})(\mathfrak{K}^\kappa_{K_{j-1}}/\mathfrak{K}_{K_{j-1}})).
\end{split}
\end{equation}
Here $\mathrm{sgn}(x)=x/|x|^{-1}$ for all $x\in\mathbb{C}^\times$. We need more notations to express the explicit values of these signs. For convenience we assume that $E/F$ is totally ramified and $K/F$ is a quadratic sub-extension of $E$. Suppose that the character $\xi$ is admissible over $F$ with jump data $E= E_0 \supsetneq E_1 \supsetneq\cdots \supsetneq E_d \supsetneq F$ and $1\leq r_0<\dots<r_d$. We define the following indexes.
 \begin{enumerate}[(i)]
\item $S$ is the index such that $r_S$ is the largest odd jump.
\item $T$ is the minimal index such that $|E_{T+1}/F|$ is odd.
\end{enumerate}
\begin{lemma} \label{about jump S and T}
  \begin{enumerate}[(i)]
    \item $S\leq T$, and $S=T$ if and only if $r_T$ is odd.
    \item $|E_i/E_{i+1}|$ is odd for $i<S$, and both $|E_S/E_{S+1}|$ and $r_{S+1}$ are even.
    \item $|E_T/E_{T+1}|$ is even, and for $i>T$ we have $|E_i/E_{i+1}|$ is odd and $r_{i}$ is even.
  \end{enumerate}
\end{lemma}
\begin{proof}
We make use of the results in Proposition \ref{jump proposition}.
  \begin{enumerate}[(i)]
    \item If $S>T$, then the odd $r_S$ is divisible by $|E/E_S|$. In turn, $r_S$ is divisible by $|E/E_{T+1}|$, which is even by definition. Therefore $S\leq T$. It is clear that $r_T$ is odd if $S=T$. Conversely if $r_T$ is odd, then $|E/E_T|$ is also odd. This forces $S=T$ by definition.
    \item If $i<S$, then $|E_i/E_{i+1}|$ divides $r_S$ and hence is odd. That $r_{S+1}$ is even is by definition. If $|E_S/E_{S+1}|$ is odd, then, by applying Proposition \ref{jump proposition}.(\ref{gcd-level-ramification}) that $\text{gcd}(r_i, |E/E_{i+1}|) = |E/E_i|$ for $i \geq S+1$, we show that all $|E_i/E_{i+1}|$ are odd for $i \geq S$. Hence $|E_S/E_{S+1}|$ must be even.
    \item Again by definition and divisibility in Proposition \ref{jump proposition}, similar to above.
  \end{enumerate}
\end{proof}
In \cite{BH-ET2}, the symbols $r_S,\,r_T,\,|E_{S+1}/F|$, and $ \zeta_S$ are denoted by $i^+,\,i_+ ,\, d^+$, and $\zeta(\varpi)=\zeta(\varpi, \xi)$ respectively. We would follow their notations and state the values of the signs of those quotient-sums in (\ref{explicit rectifier ram-even odd case}) and (\ref{explicit rectifier ram-even even case}). (Notice that there is a shift of indexes from ours to those in \cite{BH-ET2}: our $r_0$ is denoted by $r_1$ in \cite{BH-ET2}.)
\begin{enumerate}[(i)]
 \item  If the first jump $r_0$ is $>1$, then $\mathfrak{K}^\kappa_F/\mathfrak{K}_F = 1$ by convention and
$$ \mathrm{sgn}(\mathfrak{G}_F/\mathfrak{G}_K )= \begin{cases}
\left( \frac{-1}{q}\right)^{\frac{i^+-1}{2}} \left( \frac{d^+}{q} \right) \left( \frac{\zeta(\varpi)}{q} \right) \mathfrak{n}(\psi_F)^{e/2d^+} & \text{if }e/2\text{ is odd,} \\
 \left( \frac{-1}{q}\right)^{ei_+/4}& \text{if }e/2\text{ is even,}
\end{cases}$$
by Corollary 8.3 and Proposition 8.4 of \cite{BH-ET2}.\label{explicit signs first jump >1}

\item When $r_0=1$, there are three possible cases, namely $ i_+ = i^+=r_0=1$, $i_+ > i^+ = r_0 = 1$, and $i_+ \geq i^+ >r_0 =1$.
  \begin{enumerate}
    \item When $ i_+ = i^+=r_0=1$, we have $\mathfrak{G}_F=\mathfrak{G}_K=1$, by Lemma 8.1.(1) of \cite{BH-ET2}, and
$$ \mathrm{sgn}(\mathfrak{K}^\kappa_F/\mathfrak{K}_F )= \begin{cases}
\left( \frac{d^+}{q} \right) \left( \frac{\zeta(\varpi)}{q} \right) \mathfrak{n}(\psi_F)^{e/2d^+}&\text{if } e/2\text{ is odd,} \\
\left( \frac{-1}{q}\right)^{e/4} &\text{if }e/2\text{ is even,}
\end{cases}$$
both by section 9.3 of \cite{BH-ET2}. \label{explicit signs first jump =1 first}
\item When $i_+ > i^+ = r_0 = 1$, we must have that $i_+$ and $e/2$ are even. Indeed that $i_+$ being even is from Lemma 1.2.(1) of \cite{BH-ET2}. Assume that $e/2$ is odd. The assumption $i^+=1$ means that the set of jumps is equal to $\{1,r_1,\dots,r_d\}$ such that $r_i$ are all even for $i>0$. The condition (\ref{gcd-level-ramification}) of Proposition \ref{jump proposition} implies that $e(E/E_i)$ and $e(E/E_{i+1})$ have the same parity for all $i$. Since $e(E/E_{d+1})=e(E/F)$ is even, we have that $e(E/E_1)$ is also even. Hence $e(E_1/F)$ divides $e/2$ and is odd. This implies that $i_+=1$, which is a contradiction. We have $$ \mathfrak{G}_F=\mathfrak{G}_K=1\text{ and }\mathfrak{K}^\kappa_F=\mathfrak{K}_F  $$ by Lemma 8.1.(1) and Proposition 8.1 of \cite{BH-ET2} respectively. \label{explicit signs first jump =1 second}
\item When $i_+ \geq i^+ >r_0 =1$, we have
$$ \mathrm{sgn}(\mathfrak{G}_F/\mathfrak{G}_K) = \begin{cases}
\left( \frac{-1}{q}\right)^{\frac{i^+-1}{2}} \left( \frac{d^+}{q} \right) \left( \frac{\zeta(\varpi)}{q} \right) \mathfrak{n}(\psi_F)^{e/2d^+} & \text{if }e/2\text{ is odd, } \\
 \left( \frac{-1}{q}\right)^{ei_+/4}& \text{if }e/2\text{ is even, }
\end{cases}$$
by Corollary 8.3 and Proposition 8.4 of \cite{BH-ET2}, and $\mathfrak{K}^\kappa_F/\mathfrak{K}_F  = 1$ by Lemma 8.1(3) of \cite{BH-ET2}.\label{explicit signs first jump =1 third}
  \end{enumerate}\label{explicit signs first jump =1}
\end{enumerate}

Therefore in all cases we have
\begin{prop}\label{explicit rectifier quotient of sums ram-even case}
  $$\mathrm{sgn}((\mathfrak{G}_F/\mathfrak{G}_K)(\mathfrak{K}^\kappa_F/\mathfrak{K}_F )) = \begin{cases}
\left( \frac{-1}{q}\right)^{\frac{i^+-1}{2}} \left( \frac{d^+}{q} \right) \left( \frac{\zeta(\varpi)}{q} \right) \mathfrak{n}(\psi_F)^{e/2d^+} & \text{if }e/2\text{ is odd,} \\
 \left( \frac{-1}{q}\right)^{ei_+/4}& \text{if }e/2\text{ is even.}
\end{cases}$$
\end{prop}

We consider case (\ref{section supercuspidal three cases unram}) when $K_0/F$ is unramified and $E/K_0$ is totally ramified. Let $V^\varpi$ be the subspace of fixed points of the subgroup $\varpi$ of $\Psi_{E/F}$. We have, by the Main Theorem 5.2 of \cite{BH-ET3},
\begin{equation}\label{explicit rectifier unram case}
\begin{split}
&{}_{K_0/F}\mu_\xi|_{\mu_E}\equiv t^1_\mu(V_{K_0/F})
\text{ and }\\
&{}_{K_0/F}\mu_\xi(\varpi_E)=(-1)^{e(f-1)}t^0_\mu(V/V^\varpi)t_\varpi(V_{K_0/F}).
\end{split}
\end{equation}
We refer the reader to chapter 7 and 8 of \cite{BH-ET3} for the explicit values of the above t-factors. We would comment below on the factor $t^0_\mu(V/V^\varpi)=t^0_\mu(V)t^0_\mu(V^\varpi)$, which is needed in section \ref{section case unram}.

We first analyze the sign $t^0_\mu(V)=t^0_\mu(V_{K_0/F})$ by checking those anisotropic $\mathbb{F}_p\mu$-submodule $V_{[g]}$ in $V_{K_0/F}$. They are those of the form $V_{[\sigma^k\phi^{f/2}]}$ such that the corresponding jump is even. In this case, by Proposition \ref{summary of t-factors}, we have $t_\mu^0(V_{[\sigma^k\phi^{f/2}]})=-1$. Notice that if the complementary $W_{[\sigma^k\phi^{f/2}]}$ is symmetric non-trivial, then $t(W_{[\sigma^k\phi^{f/2}]})=-1$. Since either $V_{[\sigma^k\phi^{f/2}]}$ or $W_{[\sigma^k\phi^{f/2}]}$ is non-trivial but not both, we have
\begin{equation}\label{sym-unram t(V) is -t(W)}
t(W_{[\sigma^k\phi^{f/2}]})=- t_\mu^0(V_{[\sigma^k\phi^{f/2}]})\text{ if }[\sigma^k\phi^{f/2}]\text{ is symmetric.}
\end{equation}

Indeed we have computed the sign $t^0_\mu(V)=t^0_\mu(V_{K_0/F})$ somewhere else.
\begin{prop}\label{the sign t0mu computed in ET3}
The sign $t^0_\mu(V_{K_0/F})$ is $-1$ if and only if there exists an even jump $r_R$ such that
\begin{equation}\label{parity of degrees in sign t0mu}
  f(E/E_R)\text{ is odd, }f(E/E_{R+1})\text{ is even and }e(E/E_{R+1})\text{ is odd.}
\end{equation}
\end{prop}
\begin{proof}
  We refer to Corollary 8 of 8.3 in \cite{BH-ET3}.
\end{proof}
We restate the result as follows. If $f$ is even, we denote by $R$ the index such that $f(E/E_R) $ is odd and $f(E/E_{R+1})$ is even. We write $f_0=f(E/E_0)=|E/E_0|$ and recall the jump $r_S$ defined right before Lemma \ref{about jump S and T}.
\begin{prop} \label{the sign t0mu computed here}
\begin{enumerate}[(i)]
  \item If $f_0$ is even, or if $f_0$ is odd, $e$ is even, and $S\leq R$, then $t^0_\mu (V_{K_0/F}) = 1$.
  \item If $f_0$ and $e$ are odd, or if $f_0$ is odd, $e$ is even, and $S> R$, then  $t^0_\mu (V_{K_0/F}) = (-1)^{r_R +1}$.
\end{enumerate}
\end{prop}
\begin{proof}
  If $f_0$ is even, then all $f(E/E_i)$ are even for $i\geq 0$, and so the parity condition (\ref{parity of degrees in sign t0mu}) cannot occur. If $f_0$ is odd and $e$ is even, then the condition $S\leq R$ implies that $e(E/E_{R+1})$ must be even by Proposition \ref{jump proposition}.(\ref{jump is multiple}), and so again (\ref{parity of degrees in sign t0mu}) cannot occur. In the remaining cases, whether (\ref{parity of degrees in sign t0mu}) can occur depends only on the parity of the jump $r_R$. The result is then direct.
\end{proof}

We then analyze the sign $t^0_\mu(V^\varpi)=t^0_\mu(V^\varpi_{K_0/F})$. We check those anisotropic $\mathbb{F}_p\mu$-submodule $V_{[\sigma^k\phi^{f/2}]}$ with even jump and on which $\varpi$ acts trivially. We give some equivalent conditions for such $\sigma^k\phi^{f/2}$ exists. Recall from (\ref{e-th power of prime is also prime}) that $ \varpi_E^e=\zeta_{E/F}\varpi_F$ and from Proposition \ref{decomp of standard module} that $\varpi_E$ acts on $V_{[\sigma^k\phi^{f/2}]}$ by multiplying $(\zeta_e^k\zeta_{\phi^{f/2}})^{-1}$.

\begin{lemma}\label{exists root fixing varpi}
  The following are equivalent.
  \begin{enumerate}[(i)]
    \item There exists $\sigma^k \phi^{f/2} \in W_{F[\varpi_E]}$ for some $k$.  \label{exists root fixing varpi exist}
    \item  $\zeta_{\phi^{f/2}} $ is an $e$th root of unity. \label{exists root fixing varpi root}
    \item  $\zeta_{E/F} \in K_+$, where $K_0/K_+$ is quadratic unramified. \label{exists root fixing varpi quad}
    \item $f_\varpi =|E/F[\varpi_E]|$ is even. \label{exists root fixing varpi even}
  \end{enumerate}
\end{lemma}

\begin{proof}(\ref{exists root fixing varpi exist}) is equivalent to (\ref{exists root fixing varpi root}) since ${}^{\sigma^k \phi^{f/2}}\varpi_E=\zeta^k_e \zeta_{\phi^{f/2}}\varpi_E$. To show that (\ref{exists root fixing varpi quad}) implies (\ref{exists root fixing varpi root}), we recall that $\zeta_{\phi^{f/2}}$ is an $e$th root of $\zeta^{q^{f/2}-1}_{E/F}$. If $\zeta_{E/F} \in K_+$, then $\zeta^{q^{f/2}-1}_{E/F}=1$ and $\zeta_{\phi^{f/2}}$ is an $e$th root of unity. The converse is similar. For the equivalence of (\ref{exists root fixing varpi quad}) and (\ref{exists root fixing varpi even}), we notice that $f(F[\varpi_E]/F) = f(F[\zeta_{E/F}]/F) = f/f_\varpi$. Hence that $F[\zeta_{E/F}] \subseteq K_+ $ is equivalent to that $ f_\varpi$ is even.
\end{proof}
\begin{rmk}
By (\ref{exists root fixing varpi root}), we know that such $[\sigma^k \phi^{f/2}]$ is unique if it exists. \qed
\end{rmk}

\chapter{Comparing character identities}\label{chapter comparing character}

We proof the first main result, Theorem \ref{auto-ind-char-HH-LS-equal}, that in the following cases:
\begin{enumerate}[(I)]
  \item $E=K$ and $K/F$ is cyclic totally ramified of odd degree \cite{BH-ET1},\label{AI-constant in totally-ram odd}
  \item $E/F$ is totally ramified and $K/F$ is quadratic \cite{BH-ET2}, and \label{AI-constant in totally-ram quad}
  \item $E/K$ is totally ramified and $K/F$ is unramified \cite{BH-ET3}, \label{AI-constant in unram}
\end{enumerate}
the automorphic induction identity and the spectral transfer relation identity are equal when they are both normalized under the same Whittaker data. We first explain the normalizations of both identities in section \ref{section whit normalization}. We then provide three lemmas and explain the constant that relates the two identities in section \ref{section 3 lemmas}. We finally compute the values of such constants in the cases (\ref{AI-constant in totally-ram odd})-(\ref{AI-constant in unram}) in section \ref{label first cases}-\ref{section case unram} respectively and deduce the main result.

\section{Whittaker normalizations}\label{section whit normalization}

Kottwitz and Shelstad defined a normalization of the transfer factor in section 5.3 of \cite{KS}, such that the normalized factor depends only on the Whittaker datum. We summarize this result as follows. We first make use of the $F$-Borel subgroup $\mathbf{B}$ in a chosen $F$-splitting $\mathbf{spl}_G$. Let $\mathbf{U}$ be the unipotent radical of $\mathbf{B}$, which is also defined over $F$. The simple roots $\{X_\alpha\}_{\alpha \in \Delta}$ in $\mathbf{spl}_G$ give rise to a morphism $$\mathbf{U} \rightarrow \bigoplus_{\alpha \in \Delta} \mathbb{G}_a$$ defined only over $\bar{F}$ in general. We compose this morphism with the summing-up morphism $$\bigoplus_{\alpha \in \Delta} \mathbb{G}_a \rightarrow \mathbb{G}_a,\,(c_\alpha)_{\alpha \in \Delta} \mapsto \sum_{\alpha \in \Delta} c_\alpha.$$ Then the composition $\mathbf{U}  \rightarrow \mathbb{G}_a$ is defined over $F$, and by restricting to $F$-points we get a character $\mathscr{U} := \mathbf{U}(F) \rightarrow F.$ We choose a non-trivial additive character $\psi_F:F \rightarrow \mathbb{C}^\times$ and get $$\psi: \mathscr{U} \rightarrow F \rightarrow \mathbb{C}^\times$$ by composing the morphisms. Here $\psi$ is non-degenerate, in the sense that it is non-trivial when restricted to the subgroup associated to the simple roots in $\Delta$. We call this pair $(\mathscr{U}, \psi)$ a \emph{Whittaker-datum} for $G(F)$.
 In the case when $\mathscr{U} = \mathscr{U}_0$ the upper triangular unipotent subgroup and $$\psi_0 :\left( \begin{smallmatrix} 1 &x_1 & &* \\ &1& \ddots & \\ & &\ddots & x_{n-1} \\ 0& & & 1 \end{smallmatrix}\right) \mapsto \sum_{i=1}^{n-1}x_i,$$ we call $(\mathscr{U}_0, \psi_0)$ the \emph{standard Whittaker datum} for $G(F)$.

Recall that $H=\mathrm{Res}_{K/F}\mathrm{GL}_m$ regarded as an endoscopic group of $G$. Let $\mathbf{T}_G$ be the split maximal torus contained in $G$ and $\mathbf{T}_H$ be the maximal torus $H$ that splits over $K$. We define $\epsilon_\mathrm{L}(V_{G/H})$ to be the local constant $\epsilon_\mathrm{L}(V_{G/H}, \psi_F)$ in (3.6) of \cite{Tate-NTB}, depending on the chosen additive character $\psi_F$ of $F$. In our case $V_{G/H}$ is the virtual $W_F$-module $$X^*(\mathbf{T}_G)-X^*(\mathbf{T}_H)\cong (1_{W_F})^{\oplus{n}}-(\mathrm{Ind}_{K/F}1_{W_K})^{\oplus{m}}$$ of degree 0. Hence
\begin{equation*}
\epsilon_\mathrm{L}(V_{G/H}, \psi_F) = \epsilon_\mathrm{L}(1_F, \psi_F)^n \epsilon_\mathrm{L}(\mathrm{Ind}_{K/F} 1_K, \psi_F)^{-m}.
\end{equation*}
Let $\lambda_{K/F} =\lambda_{K/F} (\psi_F)$ be the Langlands constant, defined in section 2.4 of \cite{Moy}. Again this constant depends on the chosen additive character $\psi_F$. Since $\epsilon_L$ is equal to 1 on trivial modules and satisfies the induction property $$ \epsilon_\mathrm{L}(\mathrm{Ind}_{K/F} \sigma, \psi_F)=\lambda_{K/F} \epsilon_\mathrm{L}(\sigma, \psi_K)$$ for every finite dimensional $W_K$-representation $\sigma$, we have
\begin{equation}\label{epsilon and langlands constant}
  \epsilon_\mathrm{L}(V_{G/H})=\lambda_{K/F}^{-m}.
\end{equation}
We need the values of $\lambda_{K/F}$ when $\psi_F$ is of level 0, which means when $\psi_F|{\mathfrak{p}_F}\equiv 1$ and $\psi_F|{\mathfrak{o}_F}\neq 1$. In cases (\ref{AI-constant in totally-ram odd})-(\ref{AI-constant in unram}), they are computed in section 2.5 of \cite{Moy}.
\begin{enumerate}[(I)]
  \item If $K/F$ is totally ramified of odd degree $e$ and $\#\mathbf{k}_F=q$, then $\lambda_{K/F}=\left(\frac{q}{e}\right)$.
  \item If $K/F$ is quadratic totally ramified, then $\lambda_{K/F}=\mathfrak{n}(\psi_F)$ the normalized Gauss sum defined in section \ref{section comp mod}.
  \item If $K/F$ is unramified of degree $f$, then $\lambda_{K/F}=(-1)^{f-1}$.
\end{enumerate}
\begin{prop}
  The product $\epsilon_\mathrm{L}(V_{G/H})\Delta_0$ depends only on the standard Whittaker datum $(\mathscr{U}_0,\psi_0)$ and is independent of the choices of $\psi_F$ and $\mathbf{spl}_G$ that giving rise to $(\mathscr{U}_0,\psi_0)$.
\end{prop}
\begin{proof}
  We refer the proof to section 5.3 of \cite{KS}.
\end{proof}
We call the product $\epsilon_\mathrm{L}(V_{G/H})\Delta_0$ \emph{Whittaker-normalized} by $(\mathscr{U}_0,\psi_0)$, or standard Whittaker-normalized.

We then recall the Whittaker normalization of automorphic induction from section 3 of \cite{HL2010}. Suppose that $\pi \in \mathcal{A}^\mathrm{et}_{n}(F)$ and $\mathbb{V}$ is a vector space realizing $\pi$. It is well-known that any supercuspidal $\pi$ is generic, which means that there exists a $G(F)$-morphism $$(\pi, \mathbb{V}) \rightarrow \text{Ind}^{G(F)}_{\mathscr{U}} \psi.$$ Such morphism is unique up to scalar. We call this morphism a Whittaker model of $\pi$. Equivalently, there exists a $\mathscr{U}$-linear functional $\lambda = \lambda_\psi$ of $\mathbb{V}$, in the sense that $$\lambda(\pi(u)v) = \psi(u)\lambda(v)\text{, for all }u \in \mathscr{U}, v \in \mathbb{V}.$$
Suppose that $(\pi, \mathbb{V})$ is automorphically induced from $\rho\in \mathcal{A}^\mathrm{et}_{n/d}(K)$. Let $\Psi : \kappa \pi \rightarrow \pi$ be the intertwining operator. We call $\Psi$ \emph{Whittaker-normalized}, or more precisely $(\mathscr{U}, \psi)$-normalized, if we have $$\lambda \circ \Psi = \lambda.$$

\begin{rmk}
  Since all Whittaker data are $G(F)$-conjugate, the genericity of $\pi$ is independent of the choice of Whittaker datum. If we realize the Whittaker space $\text{Ind}^{G(F)}_{\mathscr{U}} \psi$ consisting of smooth functions $f: G(F) \rightarrow \mathbb{C}$ satisfying $$f(ux) = \psi(u)f(x)\text{, for all }u \in \mathscr{U}, x \in G(F),$$ then the Whittaker spaces $\text{Ind}^{G(F)}_{\mathscr{U}} \psi$ and $\text{Ind}^{G(F)}_{\mathscr{U}^g} \psi^g$, for some $g \in G(F)$, are isomorphic by the left-translation $$\text{Ind}^{G(F)}_{\mathscr{U}} \psi \rightarrow \text{Ind}^{G(F)}_{\mathscr{U}^g} \psi^g, f \mapsto(x \mapsto f(g(x))).$$ \qed
  \end{rmk}

\begin{prop}\label{constant indep of repres}
Suppose that for each $\rho\in \mathcal{A}^\mathrm{et}_{n/d}(K)$ with automorphic induction $\pi$, the intertwining operator $\Psi$ is $(\mathscr{U}_0, \psi_0)$-normalized. Then the constant $c(\rho, \kappa, \Psi)$ is independent of $\rho$ and also $\Psi$.
\end{prop}

\begin{proof}
We can deduce the statement by 4.11 Th\'{e}or\`{e}me of \cite{HL2010}.
\end{proof}
We denote $c(\rho, \kappa, \Psi)$ by $c(\kappa, \psi_0)$ in Proposition \ref{constant indep of repres}. When $K/F$ is unramified of degree $d$, we can combine 3.10 Proposition and 4.11 Th\'{e}or\`{e}me in \cite{HL2010} and show the following facts.
\begin{enumerate}[(i)]
\item  When $\psi_F$ is unramified (i.e., $\psi_F|{\mathfrak{o}_F}\equiv 1$ and $\psi_F|{\mathfrak{p}^{-1}_F}\neq 1$), we have $c(\kappa, \psi_0) = 1$.
\item  When $\psi_F$ is of level 0 (i.e., $\psi_F|{\mathfrak{p}_F}\equiv 1$ and $\psi_F|{\mathfrak{o}_F}\neq 1$), we have $c(\kappa, \psi_0) = (-1)^{(d-1)n/d}$.
\end{enumerate}
Notice that the the word `niveau' in \cite{HL2010} is defined to be our `level' here plus 1.
\begin{prop}\label{constant equals epsilon proved in HH}
When $K/F$ is unramified of degree $d$ and $\psi_F$ is unramified or of level 0, we have
$
c(\kappa, \psi_0) = \epsilon_\mathrm{L}(V_{G/H}, \psi_F).
$
\end{prop}
\begin{proof}
When $\psi_F$ is of level 0, we have computed that $\lambda_{K/F}=(-1)^{d-1}$. If $\psi_F$ is unramified, then we can follow the computations in section 2.5 of \cite{Moy} and show that $\lambda_{K/F}=1$. We then deduce the assertion by formula (\ref{epsilon and langlands constant}).
\end{proof}
In this chapter we will prove the following
\begin{thm}\label{auto-ind-char-HH-LS-equal}
  In the cases (\ref{AI-constant in totally-ram odd})-(\ref{AI-constant in unram}), if we normalize the automorphic induction character (\ref{interlude autom ind formula}) by the standard Whittaker datum, then it is equal to the spectral transfer character (\ref{interlude spec transfer formula}), which is already standard Whittaker-normalized.
\end{thm}

 Using results in chapter \ref{endoscopy}, it suffices to show that
\begin{equation}\label{chi-c-delta=epsilon-delta}
  c(\kappa, \psi_0) \Delta^2(\gamma) = \epsilon_L(V_{G/H}) \Delta_\mathrm{II,III_2}(\gamma).
\end{equation} for all $\gamma$ lying in the elliptic torus $E^\times$ of $H(F)$. We would actually verify a variant (\ref{variant chi-constant-delta-epsilon-formula}) which express the difference of the constants $c(\kappa, \psi_0)$ and that $c_\theta$ in (\ref{interlude autom ind formula}). Since we already know that $\Delta^2 = \Delta_\mathrm{II,III_2}$ up to a constant (a sign), it is enough to verify this equality on a particular choice of $\gamma$. We will specify this $\gamma$ in the different cases (\ref{AI-constant in totally-ram odd})-(\ref{AI-constant in unram}) in the subsequent sections.

\section{Three Lemmas and the normalization constant}\label{section 3 lemmas}
We deduce a variant of (\ref{chi-c-delta=epsilon-delta}) by first establishing three lemmas. The lemmas concern about the relation between the internal structure of our supercuspidal $\pi$ and various Whittaker data. The idea is based on \cite{BH-ET3} and \cite{BH-WHIT}.

Let $(E/F, \xi) \in P_n(F)$ be the class of admissible character that corresponds to $\pi \in \mathcal{A}^\mathrm{et}_n(F)$ and $\mathbb{V}$ be a vector space realizing $\pi$. From section \ref{Supercuspidal Representations}, we know that there is a representation $(\mathbf{J}, \Lambda)$ compactly inducing $\pi$. Let $g \in G(F)$ be an element that intertwines the standard Whittaker datum $(\mathscr{U}_0, \psi_0)$ and $(\mathbf{J}, \Lambda)$. This is equivalent to say that $\Lambda|_{\mathbf{J} \cap {}^g\mathscr{U}_0}$ contains the character $\psi_0|_{\mathbf{J} \cap {}^g\mathscr{U}_0}$. In \cite{BH-ET3}, we call that $({}^g\mathscr{U}_0, {}^g\psi_0)$ is adopted to $(\mathbf{J}, \Lambda)$. By the uniqueness of Whittaker model and Frobenius reciprocity, there is a unique double coset in $\mathbf{J}\backslash G(F) /\mathscr{U}_0$ whose elements intertwine $(\mathscr{U}_0, \psi_0)$ and $(\mathbf{J}, \Lambda)$.

We now setup the first lemma. Suppose that $\kappa \pi \cong \pi$. Take a Whittaker datum $(\mathscr{U}, \psi)$ and a model $(\pi, \mathbb{V}) \hookrightarrow \mathrm{Ind}^{G(F)}_\mathscr{U} \psi$. Notice that the morphism is also a model for $(\kappa \pi, \mathbb{V})$. Define a bijective operator $\mathbf{\Psi}$ on $\mathrm{Ind}^{G(F)}_\mathscr{U} \psi$ by
\begin{equation*}
\mathbf{\Psi} f: g \mapsto \kappa(\mathrm{det}(g))f(g)\text{, for all }f \in \mathrm{Ind}^{G(F)}_\mathscr{U} \Psi\text{ and }g \in G(F).
\end{equation*}
 This operator satisfies
\begin{equation*}
\kappa(\det(g))(\mathbf{\Psi}\circ \rho(g)) = \rho(g) \circ \mathbf{\Psi}.
\end{equation*}

\begin{lemma}\label{lemma 1 of 3}
If the intertwining operator $\Psi: \kappa \pi \rightarrow \pi$ is $(\mathscr{U}, \psi)$-normalized, then the diagram
$$\xymatrix{
(\kappa \pi, \mathbb{V}) \ar[r]^\Psi \ar[d] &(\pi, \mathbb{V}) \ar[d]\\
{\mathrm{Ind}}^{G(F)}_\mathscr{U} \psi \ar[r]^{\mathbf{\Psi}} &{\mathrm{Ind}}^{G(F)}_\mathscr{U} \psi}
$$commutes.
\end{lemma}
\begin{proof}
This is straightforward, or we refer to section 1.5 of \cite{BH-ET3}.
\end{proof}

To set up the second lemma, we assume from now on that
\begin{equation*}
\mathbf{J} \subseteq \ker (\kappa \circ \mathrm{det}).
\end{equation*}
We can check easily that this condition is satisfied in our cases (\ref{AI-constant in totally-ram odd})-(\ref{AI-constant in unram}).
\begin{lemma}\label{lemma 2 of 3}
Suppose that $(\mathscr{U},\psi)$ is adopted to $(\mathbf{J}, \Lambda)$. If $\Psi: \kappa \pi \rightarrow \pi$ is $(\mathscr{U}, \psi)$-normalized, then $\Psi$ acts on the $\Lambda$-isotypic component $\pi^\Lambda$ of $\pi$ as identity.
\end{lemma}

\begin{proof}
Identify $\pi$ with its Whittaker model in $\mathrm{Ind}^{G(F)}_\mathscr{U} \psi$. If $(\mathscr{U}, \psi)$ is adopted to $(\mathbf{J}, \Lambda)$, then $\pi^\Lambda$ is isomorphic to the subspace of functions in $\mathrm{Ind}^{G(F)}_\mathscr{U} \psi$ whose supports are all lying in the identity double coset $\mathbf{J}\mathscr{U}$ of $\mathbf{J} \backslash G(F) / \mathscr{U}$. Since $\mathbf{J}\mathscr{U} \subseteq \ker \kappa$, the operator $\mathbf{\Psi}$ acts on these functions as identity. By Lemma \ref{lemma 1 of 3}, we have that $\Psi|_{\pi^\Lambda}$ is also identity. For more details, see Lemma 3.(2) of \cite{BH-ET3}.
\end{proof}

For the last lemma, let $\alpha \in E^\times$ be the element defined in Proposition 2.4 and Definition 2.7 of \cite{BH-WHIT}. Indeed this $\alpha$ is our $\beta=\beta(\xi)$ defined in (\ref{b is sum of b_i}) when $E=E_0$, and is $\beta(\xi)+\zeta$ for a chosen regular $\zeta\in \mu_E$ when $E\neq E_0$.
\begin{lemma}\label{lemma 3 of 3}
If $(\mathscr{U}, \psi)$ is adopted to $(\mathbf{J}, \Lambda)$, then $(\mathscr{U},\psi)$ can be identified with the standard Whittaker datum with respect to the ordered basis $\mathfrak{a} = \{1, \alpha, ..., \alpha^{n-1}\}$.
\end{lemma}
\begin{proof}
We refer the proof to Proposition 2.4 and Theorem 2.9 of \cite{BH-WHIT}.
\end{proof}
We can now state the main result of this section. Let $(\mathscr{U}_0, \psi_0)$ be the standard Whittaker datum with respect to the basis $\mathfrak{b}$ defined in (\ref{chosen basis}). Write $\mathfrak{b}=\{b_1,\dots,b_m\}$ in this order. If the `Galois set' $\Gamma_{E/F}=\{g_1,\dots,g_n\}$ is also ordered, then we denote by $\mathrm{Van}(\mathfrak{b})$ the Vandemonde determinant
$$
\det\left(
\begin{matrix}
{}^{g_1}b_1 &\cdots &{}^{g_1}b_n \\
\vdots &{ } &\vdots \\
{}^{g_n}b_1 &\cdots &{}^{g_n}b_n \\
\end{matrix} \right).
$$
Let $(\mathscr{U}, \psi)$ be a Whittaker-datum adopted to $(\mathbf{J}, \Lambda)$. Denote by $\Psi^{\psi_0}$ and $\Psi^\psi$ respectively the $(\mathscr{U}_0, \psi_0)$- and $(\mathscr{U}, \psi)$-normalized intertwining operators of $\kappa \pi \rightarrow \pi$.

\begin{prop}\label{main prop on constants}
Let $c(\kappa, \psi_0)$ and $c(\kappa, \psi)$ be the automorphic induction constants with respect to $\Psi^{\psi_0}$ and $\Psi^{\psi}$ respectively. Then $c(\kappa, \psi_0) = \kappa(\mathrm{Van}(\mathfrak{a})\mathrm{Van}(\mathfrak{b})^{-1})c(\kappa, \psi)$.
\end{prop}
\begin{proof}
By Lemma \ref{lemma 2 of 3}, $\Psi^\psi$ acts on $\pi^\Lambda$ by identity. If we identify $\pi$ to its model in $\mathrm{Ind}^{G(F)}_{\mathscr{U}_0} \psi_0$, then, by Lemma \ref{lemma 1 of 3}, $\Psi^{\psi_0}$ acts on $\pi^\Lambda$ by $\kappa(\mathrm{det}g)$, where $g \in G(F)$ such that the functions in $(\mathrm{Ind}^{G(F)}_{\mathscr{U}_0} \psi_0)^\Lambda$ are supported in $\mathbf{J}g\mathscr{U}_0$. By Lemma \ref{lemma 3 of 3}, this element defines the change-of-basis isomorphism from $\mathfrak{b}$ to $\mathfrak{a}$, and so $\mathrm{det}(g) = \mathrm{Van}(\mathfrak{a})\mathrm{Van}(\mathfrak{b})^{-1}$. Since $\Psi^\psi$ and $\Psi^{\psi_0}$ are defined up to a scalar, the result follows.
\end{proof}

\begin{rmk}
  The $(\mathscr{U}, \psi)$-normalized intertwining operator $\Psi^\psi$ is the one used to compare the automorphic induction character with the twisted Mackey induction character (\ref{twisted mackey induction}). We refer to Lemma 3 in section 1.6 of \cite{BH-ET3} for the proof. \qed
\end{rmk}

The element $\mathrm{Van}(\mathfrak{a})\mathrm{Van}(\mathfrak{b})^{-1}$ is the determinant of that $x\in G(F)$ introduced in the Interlude, which is denoted by $x_{\mathfrak{ab}}$ subsequently. In the setting (\ref{AI-constant in totally-ram odd})-(\ref{AI-constant in unram}) above, the constant $c(\kappa,\psi)$ is denoted by $c$ in \cite{BH-ET1}, \cite{BH-ET2} and by $c_\theta$ in \cite{BH-ET3} respectively, where we can explicitly compute their values. We would use the notation $c_\theta$ for this constant from now on. By Proposition \ref{main prop on constants}, we would verify the variant of the equality (\ref{chi-c-delta=epsilon-delta}),
\begin{equation}\label{variant chi-constant-delta-epsilon-formula}
  \kappa(x_{\mathfrak{ab}})c_\theta\Delta^2(\gamma)=\epsilon_L(V_{G/H})\Delta_{\mathrm{I,II,III}}(\gamma),
\end{equation}
for certain choices of $\gamma$.

Before we move on we explain how to expand the product $$\mathrm{Van}(\mathfrak{a})=\prod_{\left[\begin{smallmatrix}
g \\h \end{smallmatrix}\right]\in \Phi(G,T)_+}({}^g\alpha-{}^h\alpha).$$ Let $E_i$ be an intermediate subfield in the jump-data of $\xi$. If $g|_{E_{i+1}}=h|_{E_{i+1}}$, then $${}^g\alpha-{}^h\alpha={}^g(\beta_i+\cdots+\beta_0+\zeta)-{}^h(\beta_i+\cdots+\beta_0+\zeta).$$
If further $g|_{E_{i}}\neq h|_{E_{i}}$, then we have
$${}^g\alpha-{}^h\alpha=({}^g\alpha_i-{}^h\alpha_i)u\text{ for some }u\in U^1_E,$$
since, by Proposition \ref{jump proposition}.(\ref{generic}), the first term $\alpha_i=\zeta_{i}\varpi_E^{-r_i}$ of $\beta_i$ is regular. In case when $\mathfrak{b}$ is also a cyclic base, we write $\mathfrak{b}=\{1,y,\dots,y^{n-1}\}$ for some $y\in E^\times$ and get
\begin{equation}\label{equation chi V(a)V(b) simplified}
  \kappa(x_{\mathfrak{ab}})=\kappa\left(\prod_{i=0}^d\prod_{\begin{smallmatrix}
 g|_{E_{i+1}}=h|_{E_{i+1}} \\ g|_{E_{i}}\neq h|_{E_{i}}
\end{smallmatrix}}\frac{{}^g\alpha_i-{}^h\alpha_i}{{}^gy-{}^hy}\right).
\end{equation}
Notice that, for arbitrary bases $\mathfrak{a}$ and $\mathfrak{b}$, the constant $\kappa(x_{\mathfrak{ab}})$ is independent of the order of the Galois set $\Gamma_{E/F}$.

\section{Case (\ref{AI-constant in totally-ram odd})}\label{label first cases}
We first compute the factor $\kappa(x_{\mathfrak{ab}})$ in the formula (\ref{variant chi-constant-delta-epsilon-formula}) in case (\ref{AI-constant in totally-ram odd}). We claim the following.
\begin{prop}
  $\kappa(x_{\mathfrak{ab}})=1$.
\end{prop}
\begin{proof}
The factor $\kappa(x_{\mathfrak{ab}})$ is equal to
\begin{equation*}
\kappa\left( \prod_{0\leq i < j \leq e-1} \frac{\sigma^i \beta - \sigma^j \beta}{\sigma^i \varpi_E - \sigma^j \varpi_E}\right) = \kappa \left( \prod^{e-1}_{k=1} \prod^{e-k-1}_{j=0} \frac{\sigma^j \beta - \sigma^{j+k} \beta}{\sigma^j \varpi_E - \sigma^{j+k} \varpi_E} \right).
\end{equation*}
 By (\ref{equation chi V(a)V(b) simplified}), the above is equal to
\begin{equation*}
\kappa \left( \prod^{e-1}_{k=1} \prod^{e-k-1}_{j=0} \zeta^{a_{r_{i(k)}}} \varpi_E^{-(r_{i(k)}+1)} \frac{\zeta_e^{j(-r_{i(k)})} - \zeta_e^{(j+k)(-r_{i(k)})}}{\zeta^j_e - \zeta^{j+k}_e} \right).
\end{equation*}
Here $r_{i(k)}$ is the jump corresponding to $[\sigma^k]$. We consider the $k$th factor and the $(e-k)$th factor. Notice that $i(k) = i(e-k)$. We can rewrite each quotient in the $k$th term as
\begin{equation*}
\frac{\zeta_e^{j(-r_{i(k)})} - \zeta_e^{(j+k)(-r_{i(k)})}}{\zeta^j_e - \zeta^{j+k}_e} = \zeta_e^{-j(r_{i(k)}+1)} \frac{1-\zeta_e^{k(-r_{i(k)})}}{1-\zeta^k_e}\text{, for } j=0, \dots, e-k-1,
\end{equation*}
and in the $(e-k)$th term as $$\frac{\zeta_e^{j(-r_{i(k)})} - \zeta_e^{(j+e-k)(-r_{i(k)})}}{\zeta^j_e - \zeta^{j+e-k}_e} = \zeta_e^{-(j+e-k)(r_{i(k)}+1)} \frac{\zeta_e^{k(-r_{i(k)})}-1}{\zeta^k_e-1} \text{, for } j=0, \dots, k-1.
$$
Hence by pairing up, the above is equal to
\begin{equation*}
\begin{split}
&\kappa \left( \prod^{\frac{e-1}{2}}_{k=1} \prod^{e-1}_{j=0} \zeta^{a_{r_{i(k)}}} \varpi_E^{-(r_{i(k)}+1)} \zeta_e^{-j(r_{i(k)}+1)} \frac{1-\zeta_E^{k(-r_{i(k)})}}{1-\zeta^k_e} \right)
\\
= &\kappa \left( \prod^{\frac{e-1}{2}}_{k=1} (\zeta^{a_{r_{i(k)}}} )^e(\varpi_F^{-(r_{i(k)}+1)})( \zeta_e^{(e(e-1)/{2})(r_{i(k)}+1)} )\left( \frac{1-\zeta_e^{k(-r_{i(k)})}}{1-\zeta^k_e} \right)^e \right).
\end{split}
\end{equation*}
The first three factors in the product belong to $N_{K/F}(K^\times)$, so the above is equal to
\begin{equation}\label{case-i-pair-up}
  \kappa \left( \prod^{\frac{e-1}{2}}_{k=1} \left( \frac{1-\zeta_e^{k(-r_{i(k)})}}{1-\zeta^k_e} \right)^e \right).
\end{equation}
The product belongs to $\mathfrak{o}_F^\times$. To claim that (\ref{case-i-pair-up}) is equal to 1, it suffices to show that
\begin{equation}\label{case-i-product-in-F}
  \prod^{\frac{e-1}{2}}_{k=1}\left( \frac{1-\zeta_e^{k(-r_{i(k)})}}{1-\zeta^k_e} \right)^{e/d} \in F^\times .
\end{equation}
We first compute the denominator (\ref{case-i-product-in-F}), starting from the fact that
$$\prod^{e-1}_{k=1}(1-\zeta^k_e) = e.$$
By pairing up the indices, we have
\begin{equation*}
\left( \prod^{\frac{e-1}{2}}_{k=1} (1-\zeta^k_e) \right)^2 \left( \prod^{\frac{e-1}{2}}_{k=1} \zeta^{-k}_e\right) = e.
\end{equation*}
Therefore $$ \left( \prod^{\frac{e-1}{2}}_{k=1} (1-\zeta^k_e) \right)^{e/d} = \left( \pm e^{1/2} \zeta^{\frac{1}{2}(\frac{e-1}{2})(\frac{e+1}{2})}_{2e}\right)^{e/d} = e^{e/2d}\left( \pm \zeta^{\frac{1}{2}(\frac{e-1}{2})(\frac{e+1}{2})}_{2d} \right).$$
The last factor $\left( \pm \zeta^{\frac{1}{2}(\frac{e-1}{2})(\frac{e+1}{2})}_{2d} \right)$ is in $F^\times$. We then compute the product in the numerator of (\ref{case-i-product-in-F}) for a fixed jump $i$. The product is then written as
 \begin{equation*}
\prod_{\begin{smallmatrix}
  k=1, ..., (e-1)/2 \\ |E_{i+1}/F||k,\, |E_i/F| \nmid k
\end{smallmatrix}} \left(1-\zeta_e^{-kr_i}\right)^{e/d}.
\end{equation*}
To compute its value, we proceed as computing the denominator above. We first consider the `full product'
 \begin{equation*}
 \prod_{\begin{smallmatrix}
  k=1, \dots, e-1 \\ |E_{i+1}/F||k,\, |E_i/F| \nmid k
\end{smallmatrix}} \left(1-\zeta_e^{-kr_i}\right).
\end{equation*}
Using Proposition \ref{jump proposition}, we can write $kr_i$ as $\ell |E_{i+1}/F| \cdot s_i |E/E_i|$, where $\ell = 1, \dots, |E/E_{i+1}|-1$ but $|E_i / E_{i+1}| \nmid \ell$. The `full product' is then equal to
\begin{equation*}
\prod_{\begin{smallmatrix}
 \ell = 1, \dots, |E/E_{i+1}|-1 \\ |E_i/E_{i+1}| \nmid \ell
\end{smallmatrix}}\left(1-\zeta^{-\ell s_i}_{|E_i/E_{i+1}|}\right).
\end{equation*}
Since $\gcd(s_i,|E_i/E_{i+1}|)=1$, the above is equal to
\begin{equation*}
  \prod_{\begin{smallmatrix}
 \ell = 1, \dots, |E/E_{i+1}|-1 \\ |E_i/E_{i+1}| \nmid \ell
\end{smallmatrix}}\left(1-\zeta^{-\ell }_{|E_i/E_{i+1}|}\right)
=\prod^{|E_i/E_{i+1}|-1}_{\ell =1}\left(1-\zeta^{-\ell}_{|E_i/E_{i+1}|}\right)^{|E/E_i|}
= |E_i/E_{i+1}|^{|E/E_i|}.
\end{equation*}
Therefore, using similar method as in the denominator of (\ref{case-i-product-in-F}), we can compute that
$$\prod_{\begin{smallmatrix}
  k=1, \dots, (e-1)/2 \\ |E_{i+1}/F||k,\, |E_i/F| \nmid k
\end{smallmatrix}} \left(1-\zeta^{-kr_i}_e\right)^{e/d} = |E_i/E_{i+1}|^{e |E/E_i|/2d}\left(\pm \zeta^*_{2d}\right).$$
Here $ \zeta^*_{2d}$ stands for certain integral power of $\zeta_{2d}$, so that $\left(\pm \zeta^*_{2d}\right)$ lies in $ F^\times$. Hence the quotient (\ref{case-i-product-in-F}) mod $F^\times$ is
\begin{equation*}
e^{e/2d}\left( \prod^d_{i=0} |E_i/E_{i+1}|^{{e}|E/E_i|/{2d}} \right)^{-1}  = \prod^d_{i=0} |E_i/E_{i+1}|^{\frac{e}{d}\left( \frac{1-|E/E_i|}{2} \right)},
\end{equation*}
which also lies in $F^\times$. Hence we have checked that $\kappa(x_{\mathfrak{ab}}) = 1$.
\end{proof}

Now we take $\gamma=\varpi_E$, which is elliptic. We then check the values of the remaining factors in (\ref{variant chi-constant-delta-epsilon-formula}).
\begin{enumerate}[(i)]
      \item Using Theorem 4.3 of \cite{BH-ET1}, we have $$c_\theta ={}_{K/F}\mu_\xi(\varpi_E)t_\varpi(V_{K/F})(\Delta^2(\varpi_E)/\Delta^1(\varpi_E))^{-1}.$$
          It turns out that $_{K/F}\mu_\xi(\varpi_E)=1$ by Proposition 4.5 of \cite{BH-ET1} and $\Delta^1(\varpi_E)=1$ by easy calculation.
  \item $\epsilon_L(V_{G/H})=\lambda_{K/F}^{-e/d}=\left(\frac{q}{e(K/F)}\right)=1$.
  \item $\Delta_{\mathrm{I,II,III}}(\varpi_E)=1$ by Proposition \ref{transfer factor trivial when totally-ram}.
\end{enumerate}
Hence, putting all together, we can change (\ref{variant chi-constant-delta-epsilon-formula}) into $$(c_\theta\Delta^2(\varpi_E))^{-1}\epsilon_L(V_{G/H})\Delta_{\mathrm{I,II,III}}(\varpi_E)=t_\varpi(V_{K/F}).$$
Recall that $t_\varpi(V_{K/F})=t^0_\varpi(V_{K/F})t^1_\varpi(V_{K/F})(\varpi_E)$. Therefore it suffices to show the following assertion.
\begin{prop}
  $t_\varpi(V_{K/F})=1.$
\end{prop}
\begin{proof}
  Since $e$ is odd, $t_\varpi(U_{[g]})$ is 1 or $-1$ depending on whether $[g]$ is asymmetric or symmetric respectively. Recall that $$V_{K/F}=\bigoplus_{[g]\in W_E\backslash W_F/W_E- W_E\backslash W_K/W_E}V_{[g]}$$ such that $V_{[g]}$ is non-trivial if the corresponding jump $r_i$ for $[g]$ is even. It is enough to show that, for fixed $i$, the number of symmetric $[g]$ with corresponding jump $r_i$ is always even.

Recall the following decomposition,
$$W_F/W_E\xrightarrow{\sim}\mathbb{Z}/e\xrightarrow{\sim}\bigoplus_{a\mid e}(e/a)(\mathbb{Z}/a)^\times,$$
such that $W_K/W_E$ decomposes accordingly as
$$W_K/W_E\xrightarrow{\sim}d\mathbb{Z}/e\xrightarrow{\sim}\bigoplus_{a\mid m}(e/a)(\mathbb{Z}/a)^\times.$$
Here $m=e/d$. Hence we have
$${W_E\backslash W_F/W_E- W_E\backslash W_K/W_E}\xrightarrow{\sim}\bigoplus_{a\mid e,\,a\nmid m}(e/a)\left(q\backslash(\mathbb{Z}/a)^\times\right).$$ For fixed $a$, the orbits in $(e/a)\left(q\backslash(\mathbb{Z}/a)^\times\right)$ have the same jump $r_i$. Hence it is enough to show that the cardinality of $q\backslash(\mathbb{Z}/a)^\times$ is even. This number is $\phi(a)/\mathrm{ord}(q,a)$ defined in section \ref{section Galois groups}. Since ${a\mid e}$ and ${a\nmid m}$, we have $\gcd(a,d)\neq1$. Consider the surjective morphism $$(\mathbb{Z}/a)^\times\xrightarrow{P}(\mathbb{Z}/\gcd(a,d))^\times.$$ The subgroup $\left<q\right>$ in $(\mathbb{Z}/a)^\times$ lies in $\ker P$ since $d\mid q-1$. Hence $\phi(a)/\mathrm{ord}(q,a)$ is divisible by $\phi(a)/\#\ker P =\phi(\gcd(a,d))$, which is always even.
\end{proof}

We have verified (\ref{variant chi-constant-delta-epsilon-formula}) in case (\ref{AI-constant in totally-ram odd}).

\section{Case (\ref{AI-constant in totally-ram quad})}
We would verify (\ref{variant chi-constant-delta-epsilon-formula}) in case (\ref{AI-constant in totally-ram quad}). In this case, $e$ is even and $d=2$. We sometimes write $m=e/2$. We first calculate the value of $\kappa(x_{\mathfrak{ab}})$.
\begin{prop}\label{case-ii-(e/2)}
$  \kappa(x_{\mathfrak{ab}}) =$
\begin{equation*}
  \begin{cases}
\left( \frac{-1}{q}\right)^{\frac{d^++i^+}{2}}\left( \frac{d^+}{q}\right) \left( \frac{\zeta^{a_{r_S}}}{q}\right)  &\text{ if }e/2\text{ is odd,}
\\
\left( \frac{-1}{q}\right)^{\frac{e}{4}(i_+-1)} &\text{ if }e/2\text{ is even.}
\end{cases}
\end{equation*}
\end{prop}
We refer the notations above to case (\ref{AI-constant in totally-ram quad}) in section \ref{section explicit rectifiers}.
\begin{proof}(of Proposition \ref{case-ii-(e/2)})
We proceed as in case (\ref{AI-constant in totally-ram odd}) up to (\ref{case-i-pair-up}), in which step we have an extra factor $$\kappa \left( \prod^{e/2-1}_{j=0} \frac{\sigma^j \beta - \sigma^{j+e/2} \beta}{\sigma^j \varpi_E - \sigma^{j+e/2} \varpi_E}\right)$$ corresponding to $k=e/2$. We recall the jump $r_S=i^+$ defined right before Lemma \ref{about jump S and T}. The above extra factor is equal to
\begin{equation*}
\kappa \left( (\zeta^{a_{r_S}})^{e/2} \varpi^{-(r_S+1)e/2}_{E} \prod^{e/2-1}_{j=0} \zeta^{-j(r_S+1)}_{e} \frac{1-\zeta^{e/2(-r_S)}_{e}}{1-\zeta^{e/2}_e}\right).
\end{equation*}
Since $r_S$ is odd and $-\varpi_F\in N_{K/F}(K^\times)$, the above is equal to
\begin{equation}\label{case ii term e/2}
\kappa \left(\zeta^{a_{r_S}}(-1)^{(r_S+1)/2}\right)^{e/2}
=
\begin{cases}
\left( \frac{\zeta^{a_{r_S}}}{q} \right) \left( \frac{-1}{q}\right)^{(r_S+1)/2}  &\text{ if }e/2\text{ is odd,}
\\
1  &\text{ if }e/2\text{ is even.}
\end{cases}
\end{equation}
We then compute the remaining terms. If we proceed as in (\ref{case-i-pair-up}) to get rid of the factors lying in $ N_{K/F}(K^\times)$, then we get
\begin{equation*}
\kappa\left(\prod^{e/2-1}_{k=1}\left( \frac{1-\zeta^{k(-r_{i(k)})}_e}{1-\zeta^k_e}\right)^e\right).
\end{equation*}
We then separate the product in half and regroup
\begin{equation*}
\begin{split}
  &\kappa\left(\left(\prod^{e/2-1}_{k=1}\left( \frac{1-\zeta^{k(-r_{i(k)})}_e}{1-\zeta^k_e}\right)\left( \frac{1-\zeta^{-k(-r_{i(k)})}_e}{1-\zeta^{-k}_e}\right)\zeta^{k(r_{i(k)}+1)}\right)^{e/2}\right)
\\
=&\kappa\left(\left(\prod^{e-1}_{k=1}\frac{1-\zeta^{k(-r_{i(k)})}_e}{1-\zeta^k_e}\right)^{e/2}
\left(\prod^{e/2-1}_{k=1}(-1)^{k(r_{i(k)}+1)}\right)\right).
\end{split}
\end{equation*}
We compute the values of the two products above. For the first product, we proceed as in case (\ref{AI-constant in totally-ram odd}) and get $$ \kappa\left(\prod^{e-1}_{k=1}\frac{1-\zeta^{k(-r_{i(k)})}_e}{1-\zeta^k_e}\right)^{e/2}=\kappa\left( \prod^d_{i=0} |E_i/E_{i+1}|^{|E/E_i|-1} \right)^{e/2}.$$
 It is clearly trivial if $e/2$ is even. When $e/2$ is odd, we check the parity of each power ${|E/E_i|-1}$ above. We recall the jump $r_T=i_+$ defined right before Lemma \ref{about jump S and T} and know that $|E/E_i|$ is odd if and only if $i\leq T$. Therefore we have
 \begin{equation}\label{case ii first product in kappa}
 \kappa\left( \prod ^d_{i=0} |E_i/E_{i+1}|^{|E/E_i|-1} \right) ^{e/2}=\begin{cases}
\left( \frac{|E_{T+1}/F|}{q} \right) =\left( \frac{d^+}{q}\right)  &\text{ if }e/2\text{ is odd,}
\\
1  &\text{ if }e/2\text{ is even.}
\end{cases}
 \end{equation}
For the second product, unlike case (\ref{AI-constant in totally-ram odd}), we now have to take care of $\kappa(-1)=\left(\frac{-1}{q}\right)$. We separate the factors according to the jumps like what we have done to the first product. We then obtain
 \begin{equation}\label{case II second product}
 \kappa\left(\prod^{e/2-1}_{k=1}(-1)^{k(r_{i(k)}+1)}\right)= \kappa\left(\prod ^d_{i=0} \prod_{\begin{smallmatrix}
  k=1, \dots, e/2-1 \\|E_{i+1}/F||k,\, |E_i/F| \nmid k
\end{smallmatrix}}(-1)^{k(r_i+1)}\right).
\end{equation}
Notice that $k$ is odd if and only if the corresponding jump $i \geq T$. We now compute the above sign in different cases.
\begin{enumerate}[(i)]
 \item  If $e/2$ is odd, then $S=T$. When $i=T$, we have that $r_T$ is odd. When $i>T$, the cardinality of the set
\begin{equation}\label{case-ii-the-set}
  \{k=1, \dots, e/2-1 | k \text{ is odd, } |E_{i+1}/F||k,\,|E_i/F| \nmid k \}
\end{equation}
is $\frac{1}{4}(|E/E_{i+1}| - |E/E_i|)$. Therefore the sign (\ref{case II second product}) is the parity of
$$\sum^d_{i=T+1} \frac{1}{4}\left(|E/E_{i+1}| - |E/E_i|\right) = \frac{1}{4}|E/E_{T+1}|\left(|E_{T+1}/F| - 1\right).$$
Notice that we have denoted $|E_{T+1}/F|$ by $d^+$. Also $|E/E_{T+1}|$ is even but not divisible by 4. Hence (\ref{case II second product}) is equal to $\left(\frac{-1}{q}\right)^{(d^+-1)/2}$.
\item
 If $e/2$ is even and $T=S$, then again $r_T=i_+$ is odd. When $i>T$, the cardinality of the set (\ref{case-ii-the-set}) is then even. Hence (\ref{case II second product}) is equal to $1=\left(\frac{-1}{q}\right)^{\frac{e}{4}(i_+-1)}$.
\item
If $e/2$ is even and $T>S$, then $r_T=i_+$ is even. When $i=T$, the cardinality of the set (\ref{case-ii-the-set}) is equal to $\frac{1}{4}|E/E_{T+1}|$, since the last condition $|E_T/F| \nmid k$ is void if $k$ is odd. The cardinality has the same parity as that of $e/4$, and so (\ref{case II second product}) is $\left(\frac{-1}{q}\right)^{\frac{e}{4}}=\left(\frac{-1}{q}\right)^{\frac{e}{4}(i_+-1)}$.
\end{enumerate}
Therefore $$\kappa\left(\prod ^d_{i=0} \prod_{\begin{smallmatrix}
  k=1, \dots, e/2-1 \\|E_{i+1}/F||k,\, |E_i/F| \nmid k
\end{smallmatrix}}(-1)^{k(r_i+1)}\right)=\begin{cases}
\left(\frac{-1}{q}\right)^{(d^+-1)/2}  &\text{ if }e/2\text{ is odd,}
\\
\left(\frac{-1}{q}\right)^{e/4}  &\text{ if }e/2\text{ is even.}
\end{cases}$$
If we combine this with (\ref{case ii term e/2}) and (\ref{case ii first product in kappa}), then we can verify Proposition \ref{case-ii-(e/2)}.
\end{proof}

Again we take $\gamma=\varpi_E$ elliptic. The next step is to simplify the product $c_\theta \Delta^2(\varpi_E)$.
\begin{prop}\label{case ii simplify c-delta}
We have $c_\theta \Delta^2(\varpi_E) = \mathrm{sgn}\left( (\mathfrak{G}_F/\mathfrak{G}_K)(\mathfrak{K}^\kappa_F/\mathfrak{K}_F ) \right). $
\end{prop}
\begin{proof}
In this proof, we follow the notations in chapter 6 of \cite{BH-ET2}. From there we recall that $$c_\theta = (\mathfrak{G}_F/\mathfrak{G}_K)(\mathfrak{K}^\kappa_F/\mathfrak{K}_F)(\text{dim}\Lambda/\text{dim}\Lambda_K) \delta(1+ \varpi)^{-1}.$$
Since $$\text{dim}\Lambda / \text{dim}\Lambda_K = (J^1/H^1)^{1/2}/(J^1_K / H^1_K)^{1/2} = q^{(\text{dim}_{\mathbf{k}_F}V_{K/F})/2}$$ and
$$|(\mathfrak{G}_F/\mathfrak{G}_K)(\mathfrak{K}^\kappa_F/\mathfrak{K}_F)| = q^{(\text{dim}_{\mathbf{k}_F}W_{K/F})/2},$$ we have $$|(\mathfrak{G}_F/\mathfrak{G}_K)(\mathfrak{K}^\kappa_F/\mathfrak{K}_F)(\text{dim}\Lambda/\text{dim}\Lambda_K) | = q^{(\text{dim}_{\mathbf{k}_F}U_{K/F})/2} = q^{m(d-1)/2}.$$ Also
\begin{equation*}
  \begin{split}
    \delta(1+\varpi) &= \frac{\Delta^2(1+\varpi_E)}{\Delta^1(1+\varpi_E)} = \frac{\Delta^2(\varpi_E)}{|\tilde{\Delta}(\varpi_E)^2|^{1/2}_F |\det_G(1+\varpi_E)|^{m(d-1)/2}_F}
    \\
    &= \frac{\Delta^2(\varpi_E)}{|\varpi_E^{}|^{m^2d(d-1)/2}_F} = \Delta^2(\varpi_E) q^{m(d-1)/2}.
  \end{split}
\end{equation*}
Therefore $c_\theta \Delta^2(\varpi_E) = \mathrm{sgn}\left( (\mathfrak{G}_F/\mathfrak{G}_K)(\mathfrak{K}^\kappa_F/\mathfrak{K}_F ) \right).$
\end{proof}
Recall that we have computed $\mathrm{sgn} \left( (\mathfrak{G}_F/\mathfrak{G}_K)(\mathfrak{K}^\kappa_F/\mathfrak{K}_F)  \right)$ in Proposition \ref{explicit rectifier quotient of sums ram-even case}. Combining this with Propositions \ref{case-ii-(e/2)} and \ref{case ii simplify c-delta}, we can compute the left side of the equation (\ref{variant chi-constant-delta-epsilon-formula}), which is
\begin{equation}\label{simplified left side in case II}
  \kappa(x_{\mathfrak{ab}}) c_\theta \Delta^2(\varpi_E)=
  \begin{cases}
\left( \frac{-1}{q}\right)^{\frac{d^+-1}{2}}\mathfrak{n}(\psi_F)^{m/d^+}  &\text{ if }e/2\text{ is odd,}
\\
\left( \frac{-1}{q}\right)^{e/4} &\text{ if }e/2\text{ is even.}
\end{cases}
\end{equation}
Notice that when $e/2$ is odd, we have $$\left( \frac{-1}{q}\right)^{\frac{d^+-1}{2}}\mathfrak{n}(\psi_F)^{m/d^+}=\mathfrak{n}(\psi_F)^{d^+-1+m/d^+}=\mathfrak{n}(\psi_F)^{-m}.$$
 Therefore the right side of (\ref{simplified left side in case II}) in both cases is $$\mathfrak{n}(\psi_F)^{-m}=\lambda_{K/F}^{-m}=\epsilon_L(V_{G/H}).$$
Since $\Delta_{\mathrm{I,II,III}}(\varpi_E)=1$ by Proposition \ref{transfer factor trivial when totally-ram}, we have verified (\ref{variant chi-constant-delta-epsilon-formula}) in case (\ref{AI-constant in totally-ram quad}).

\section{Case (\ref{AI-constant in unram})}\label{section case unram}

We would verify (\ref{variant chi-constant-delta-epsilon-formula}) in case (\ref{AI-constant in unram}). We first compute the factor $ \kappa(x_{\mathfrak{ab}})$.
\begin{lemma}\label{chi-V-in case III}
 $ \kappa(x_{\mathfrak{ab}})= \zeta_f^{-\frac{1}{2}\left({f(e-1)}+ \sum^d_{i=0} {r_if}(|E/E_{i+1}| - |E/E_i|)\right)}.$
\end{lemma}
\begin{proof}
If $K/F$ is unramified, then $\kappa$ is unramified and $\kappa(\varpi_F) = \zeta_f$. Hence it is enough to check the $v_F$-order of $x_{\mathfrak{ab}}$. Recall that $\mathrm{Van}(\mathfrak{b}) = \det \mathbf{g}$ defined in section \ref{section Trivializing the splitting invariant} and is equal to $$\prod^{f-1}_{k=0} \prod_{0 \leq_\phi i <_\phi j < e} ({}^{\phi^k \sigma^i} \varpi_E - {}^{\phi^k\sigma^j}\varpi_E) = \prod^{f-1}_{k=0} \prod_{0 \leq_\phi i <_\phi j < e} (\zeta_{\phi^k}(\zeta^{i{q^k}}_e - \zeta^{j{q^k}}_e) \varpi_E).$$
Each factor has $v_E$-order 1. Hence $\mathrm{Van}(\mathfrak{b})$ has $v_E$-order $ef(e-1)/2$, or $v_F$-order $f(e-1)/2$. We then compute the $v_E$-order of $\mathrm{Van}(\mathfrak{a}) = \mathrm{Van}(\{ 1, \alpha, ..., \alpha^{n-1}\})$, which is half of the $v_E$-order of
\begin{equation*}
    \prod_{g \neq h}({}^{g} \alpha - {}^{h} \alpha)
= \prod^d_{i=0} \prod_{\begin{smallmatrix}
  g|_{E_{i+1}} = h|_{E_{i+1}} \\   g|_{E_{i}} \neq h|_{E_{i}}
\end{smallmatrix} } ({}^g \alpha - {}^h \alpha)
= \prod^d_{i=0}\prod_{\begin{smallmatrix}
  g|_{E_{i+1}} = h|_{E_{i+1}} \\   g|_{E_{i}} \neq h|_{E_{i}}
\end{smallmatrix} } \varpi_E^{-r_i}.
\end{equation*}
If we use that $\#\{\left[\begin{smallmatrix}  g \\ h \end{smallmatrix} \right]\in \Phi|\left[\begin{smallmatrix}  g \\ h \end{smallmatrix}\right] |_{E_i} \equiv 1\} = |E/F|(|E/E_i| -1),$ then the above product is equal to
\begin{equation*}
 \prod^d_{i=0} \varpi_E^{-r_i |E/F|(|E/E_{i+1}|-|E/E_i|)} = \prod^d_{i=0} \varpi_F^{-r_i f(|E/E_{i+1}|-|E/E_i|)}.
\end{equation*}
Therefore by putting all together we have the formula indicated.
\end{proof}

 We consider the remaining factors in (\ref{variant chi-constant-delta-epsilon-formula}). In this case we take $\gamma=\varpi_Eu$ for some $u\in U_E^1$ such that $\gamma$ is elliptic. Write $f_0=|E/E_0|$ and $f_\varpi=|E/F[\varpi_E]|$. The factors are
\begin{enumerate}[(i)]
  \item $c_\theta = (-1)^{f_0 - 1} t_\mu^0(V_{K/F})$ by Theorem 5.2.(3) of \cite{BH-ET3},
  \item $\Delta^2(\varpi_E u) = (-1)^{e(f-1)+f_\varpi - 1}$ by (6.5.6) of \cite{BH-ET3} (still true even when $E \neq E_0$),
  \item $\epsilon_L(V_{G/H}) = \lambda^{-e}_{K/F} = (-1)^{e(f-1)}$, and
  \item the product $$\Delta_\mathrm{II,III_2}(\varpi_E u) = \prod_{\alpha \in \mathcal{R}_\mathrm{sym}-\Phi(H,T)} \chi_g \left( \frac{\varpi_E u - {}^g(\varpi_E u)}{a_\alpha}\right),$$
where $\alpha = [\begin{smallmatrix}
  1\\g
\end{smallmatrix}]$.
\end{enumerate}
We now distinguish between the cases when $f$ is odd and when $f$ is even. If $f$ = odd, then
\begin{enumerate}[(i)]
  \item $c_\theta = t_\mu^0(V_{K/F})=1$ since $V_{K/F}$ has no anisotropic component,
  \item $\Delta^2(\varpi_E u) =1$,
  \item $\epsilon_L(V_{G/H}) = 1$, and
  \item $\Delta_\mathrm{II,III_2}(\varpi_E u) = 1$ since $\Phi_\mathrm{sym}-\Phi(H)=\emptyset$.
\end{enumerate}
It remains to show the following.
\begin{prop}
 If $f$ is odd, then $ \kappa(x_{\mathfrak{ab}})$ in (\ref{chi-V-in case III}) is equal to 1.
\end{prop}

\begin{proof}
 If $e$ is odd, then the right side of (\ref{chi-V-in case III}) is equal to 1 since all $|E/E_{i+1}| - |E/E_i|$ are even. If $e$ is even, then all $|E/E_{i+1}| - |E/E_i|$ are even except at the place $i=S$ when $|E/E_S|$ is odd and $|E/E_{S+1}|$ is even. Recall that $r_S$ is the largest odd jump. A simple calculation shows that the power of the right side of (\ref{chi-V-in case III}) is a multiple of $f$.
\end{proof}
We have verified (\ref{variant chi-constant-delta-epsilon-formula}) in the case when $f$ is odd. We now assume that $f$ is even. To compute $ \kappa(x_{\mathfrak{ab}})$, we recall the jump $r_R$ from section \ref{section explicit rectifiers} such that $f(E/E_{R})$ is odd and $f(E/E_{R+1})$ is even. We distinguish between the cases when $e$ is odd and when $e$ is even.
\begin{enumerate}[(i)]
\item First we consider when $e$ is odd. If $f_0$ is odd, then we have
\begin{equation*}
   \kappa(x_{\mathfrak{ab}})= \zeta_f^{r_Rf/2} = (-1)^{r_R},
\end{equation*}  because $|E/E_{i+1}| - |E/E_i|$ is odd if and only if $i=R$. If $f_0$ is even, then $|E/E_{i+1}| - |E/E_i|$ is always even and so $\kappa(x_{\mathfrak{ab}}) = 1$.
\item When $e$ is even, we have to make more analysis. If $f_0$ is even, then all $|E/E_{i+1}| - |E/E_i|$ are even. Hence $$ \kappa(x_{\mathfrak{ab}})=\zeta_f^{f(e-1)/2} = -1.$$ We then assume that $f$ and $e$ are even and $f_0$ is odd. The parity of $|E/E_{i+1}| - |E/E_i|$ is odd when
\begin{equation}\label{e-charge}
  e(E/E_i)\text{ is odd and }e(E/E_{i+1})\text{ is even}
\end{equation}
or
\begin{equation}\label{f-charge}
  f(E/E_i)\text{ is odd and }f(E/E_{i+1})\text{ is even}.
\end{equation}
We know that (\ref{e-charge}) happens when $i=S$ and (\ref{f-charge}) happens when $i=R$. Therefore $ \kappa(x_{\mathfrak{ab}})$ is equal to $(-1)^{r_R+1}$ if $S>R$ and is equal to $(-1)^{r_S+1} = 1$  if $S \leq R$.
\end{enumerate}

To compute $c_\theta$, we need the value of $t^0_\mu(V_{K/F})$. This is already given in Proposition \ref{the sign t0mu computed here}. For $\Delta^2(\varpi_Eu)$, the work is easy. Therefore we have computed the constants on the left side of (\ref{variant chi-constant-delta-epsilon-formula}). It would be better to tabulate the factors in different cases as the following.
\begin{center}{\renewcommand{\arraystretch}{1.4}
      \begin{tabular}{ | l | l | l | l |}
    \hline
   & $\kappa(x_{\mathfrak{ab}})$ &$c_\theta$ &$\Delta^2(\varpi_Eu)$  \\ \hline
$e$ is odd, $f_0$ is odd           & $(-1)^{r_R}$      &$(-1)^{r_R+1}$    & $(-1)^{f_\varpi}$ \\
$e$ is odd, $f_0$ is even          & 1                 & $-1$             & $(-1)^{f_\varpi}$\\
$e$ is even, $f_0$ is odd, $S>R$     & $(-1)^{r_R+1}$    & $(-1)^{r_R+1}$   & $(-1)^{f_\varpi+1}$\\
$e$ is even, $f_0$ is odd, $S\leq R$ & 1                 & 1                & $(-1)^{f_\varpi+1}$\\
$e$ is even, $f_0$ is even         & $-1$               & $-1$            & $(-1)^{f_\varpi+1}$\\  \hline
    \end{tabular}}
\end{center}
Using this table, we can show that the left side of (\ref{variant chi-constant-delta-epsilon-formula}) is always $(-1)^{f_\varpi+1}$. We now compute the factors on the right side of (\ref{variant chi-constant-delta-epsilon-formula}). We have
$$\epsilon_L(V_{G/H}) = (-1)^e. $$
It remains to show that
\begin{prop}
If $f$ is even, then $ \Delta_\mathrm{II,III_2} (\varpi_E u)=(-1)^{e+f_\varpi-1}.$
\end{prop}
\begin{proof}
Be definition,
\begin{equation*}
  \begin{split}
    \Delta_\mathrm{II,III_2} (\varpi_E u) &= \prod_{\sigma^k \phi^{f/2} \in\mathcal{D}_\text{sym}} \chi_{\sigma^k \phi^{f/2}}\left( \frac{\varpi_E u - {}^{\sigma^k \phi^{f/2}}(\varpi_E u)}{a_{\sigma^k \phi^{f/2}}}\right)
\\
&= \prod_{\sigma^k \phi^{f/2} \in\mathcal{D}_\text{sym}} \chi_{\sigma^k \phi^{f/2}}\left( \frac{\varpi_E(u - \zeta_{\phi^{f/2}} \zeta^k_e(^{\sigma^k \phi^{f/2}}u))}{\zeta - {}^{\sigma^k \phi^{f/2}}\zeta}\right).
  \end{split}
\end{equation*}
Recall that $E_{\sigma^k\phi^{f/2}}/E_{\pm\sigma^k\phi^{f/2}}$ is unramified and $$N_{E_{\sigma^k\phi^{f/2}}/ E_{\pm\sigma^k\phi^{f/2}}}:\varpi_E\mapsto \zeta_{\phi^{f/2}}\zeta_e^k\varpi_E^2.$$
Since $\sigma^k \phi^{f/2} \in W_{F[\varpi_E]}$ if and only if $\zeta_{\phi^{f/2}} \zeta^k_e = 1$, we have
\begin{equation*}
\begin{split}&\prod_{[\sigma^k \phi^{f/2}] \in\mathcal{D}_\text{sym}} \chi_{\sigma^k \phi^{f/2}} \left( \frac{\varpi_E u - {}^{\sigma^k \phi^{f/2}} \varpi_E u}{\zeta - {}^{\sigma^k \phi^{f/2}}\zeta}\right)
\\
=&\begin{cases}
  \chi_{\sigma^k \phi^{f/2}} \left(\varpi_E^2 (\text{unit}) \right) = 1 &\text{ if }\sigma^k \phi^{f/2} \in W_{F[\varpi_E]},
  \\
  \chi_{\sigma^k \phi^{f/2}} \left(\varpi_E (\text{unit}) \right) = -1 &\text{ otherwise.}
\end{cases}\end{split}
\end{equation*}
Hence
$$ \Delta_\mathrm{II,III_2} (\varpi_E u) = \prod_{\begin{smallmatrix}
  \sigma^k \phi^{f/2} \in\mathcal{D}_\text{sym}\\ \sigma^k\phi^{f/2} \notin W_{F[\varpi_E]}
\end{smallmatrix}}(-1).$$
Recall
\begin{enumerate}[(i)]
  \item by Lemma \ref{exists root fixing varpi} that there exists $\sigma^k \phi^{f/2} \in W_{F[\varpi_E]}$ if and only if $f_\varpi$ is even, and
      \item by Proposition \ref{parity of double coset with i=f/2} that the parity of the number of symmetric $[\sigma^k\phi^{f/2}]$ is the same as that of $e$.
\end{enumerate} Combining these two facts we have
$  \Delta_\mathrm{II,III_2} (\varpi_E u)=(-1)^{e+f_\varpi-1}.$
\end{proof}
Therefore we have verified (\ref{variant chi-constant-delta-epsilon-formula}) in case (\ref{AI-constant in unram}). Notice in this case we have also shown that $$ \kappa(x_{\mathfrak{ab}})c_\theta=\epsilon_L(V_{G/H})\text{ and }\Delta^2=\Delta_\mathrm{II,III_2}.$$ The first statement is the same as Proposition \ref{constant equals epsilon proved in HH}, a fact proved in \cite{HL2010}.

\chapter{Rectifiers and transfer factors}\label{chapter rectifier and transfer}

We prove the second main result Theorem \ref{chi-data factor of BH-rectifier}. It asserts that, for each admissible character $\xi\in P(E/F)$, the rectifier ${}_F\mu_\xi$ is a product of restrictions of canonically chosen $\chi$-data. Therefore we can interpret the essentially tame local Langlands correspondence in terms of the locally constructed bijection $\Pi_n$ and the admissible embeddings of Langlands-Shelstad. We also have a simple Corollary \ref{rectifier as transfer factor} that we can express the $\nu$-rectifier associated to a cyclic extension $K/F$ as the transfer factor $\Delta_\mathrm{III_2}$ associated to the group $G=\mathrm{GL}_n$ and its endoscopic group $H=\mathrm{Res}_{K/F}\mathrm{GL}_{n/|K/F|}$.

\section{Main results}\label{section main result}

Let us recall the notations in chapter \ref{chapter finite cymplectic module}. For each double coset $ [g]\in (W_E\backslash W_F/W_E)'$, we let $U_{[g]}$ be the standard $\mathbf{k}_F\Psi_{E/F}$-module defined in section \ref{section standard module}. For each admissible character $\xi\in P(E/F)$, let $V=V_\xi$ be the $\mathbf{k}_F\Psi_{E/F}$-module defined by $\xi$. Let $V_{[g]}$ be the $U_{[g]}$-isotypic part of $V$, so that $V_{[g]}$ is either trivial or isomorphic to $U_{[g]}$. Let $W_{[g]}$ be the complementary submodule of $V_{[g]}$ defined in section \ref{section comp mod}. We denote by $\mathbf{V}_{[g]}$ the module ${V}_{[g]}\oplus {V}_{[g^{-1}]}$ if $[g]$ is asymmetric and the module ${V}_{[g]}$ if $[g]$ is symmetric. We write $\mathbf{W}_{[g]}$ similarly. The t-factors $$t^i_\mu(\mathbf{V}_{[g]}),\,t^i_\varpi(\mathbf{V}_{[g]}),\,i=0,1\text{, and }t(\mathbf{W}_{[g]})$$ are computed in Proposition \ref{summary of t-factors} and Proposition \ref{n-factor of comp-mod}. We write $$t_\Gamma(\mathbf{V}_{[g]})=t^0_\Gamma(\mathbf{V}_{[g]})t^1_\Gamma(\mathbf{V}_{[g]})(\gamma)$$ for arbitrary generator $\gamma$ of the cyclic group $\Gamma=\mu$ or $\varpi$.

We also recall that $\chi$-data are certain collection of characters $\{\chi_{\lambda}\}_{\lambda\in \Phi}$ defined in section \ref{section langlands shelstad chi data}. We take $\mathcal{D}$ to be a subset in $\Gamma_F/\Gamma_E$ representing the non-trivial double cosets in $
\Gamma_E\backslash\Gamma_F/\Gamma_E$. Let $\mathcal{D}_{\mathrm{asym}/\pm}$ and $\mathcal{D}_{\mathrm{sym}}$ be the subsets of $\mathcal{D}$ containing the representatives of asymmetric double cosets modulo the relation $[g]\sim[g^{-1}]$ and symmetric double cosets respectively. Write $\mathcal{D}_{\pm}=\mathcal{D}_{\mathrm{asym}/\pm}\sqcup\mathcal{D}_{\mathrm{sym}}$. If we write $\chi_g=\chi_{\left[\begin{smallmatrix}
  1\\g
\end{smallmatrix}\right]}$, then the collection of $\chi$-data is determined by its sub-collection $\{\chi_{g}\}_{g
\in\mathcal{D}_\pm}$. The product $$\prod_{g\in \mathcal{D}}\chi_{g}|_{E^\times}$$ is shown, in Remark \ref{two remarks about product}.(\ref{remark indep of choices of double coset}), to be independent of the choices of the representatives in $\mathcal{D}$. Therefore it is legitimate to write the above product as $$\prod_{[g]\in (W_E \backslash W_{F}/W_E)'}\chi_{g}|_{E^\times}.$$

\begin{thm}\label{chi-data factor of BH-rectifier}
Let $\xi$ be an admissible character of $E^\times$ over $F$.
 \begin{enumerate}[(i)]
\item  Let $V=V_\xi$ be the $\mathbf{k}_F\Psi_{E/F}$-module defined by $\xi$. The following conditions define a collection of $\chi$-data $\{\chi_{g,\xi}\}_{g\in\mathcal{D_\pm}}$.
 \begin{enumerate}
 \item All $\chi_{g,\xi}$ are tamely ramified.
 \item If $[g]$ is asymmetric, then $$\chi_{g,\xi}|_{\mu_{E_{g}}}=\mathrm{sgn}_{\mu_{E_{g}}}(V_{[g]})\text{ and }\chi_{g,\xi}(\varpi_E)\text{ can be anything appropriate.}$$
 \item If $[g]$ is symmetric, then $$\chi_{g,\xi}|_{\mu_{E}}=
\begin{cases}
\left(\frac{}{\mu_E}\right) & \text{if }[g]=[\sigma^{e/2}],\\
t^1_\mu(V_{[g]}) & \text{otherwise},
\end{cases}$$ and $$\chi_{g,\xi}(\varpi_E)=
t^0_\mu(V^{\varpi}_{[g]})t_\varpi(V_{[g]})t(W_{[g]}).$$
Moreover, $\chi_{g,\xi}|_{E_{\pm g}^\times}=\delta_{E_g^\times/E_{\pm g}^\times}$ the quadratic character of the field extension $E_g^\times/E_{\pm g}^\times$. Finally, outside $E^\times$ and $E_{\pm g}^\times$, the values of $\chi_{g,\xi}$ can be anything appropriate. \label{theorem sym case}
\end{enumerate}\label{theorem character}
\item Let ${}_F\mu_\xi$ be the rectifier of $\xi\in P(E/F)$ and  $\{\chi_{g,\xi}\}_{g\in\mathcal{D}_\pm}$ be the $\chi$-data defined above, then $${}_F\mu_\xi=\prod_{[g]\in (W_E \backslash W_{F}/W_E)'}\chi_{g,\xi}|_{E^\times}.$$
\end{enumerate}
\end{thm}

\begin{rmk}
  \begin{enumerate}[(i)]
    \item In the case when $[g]$ is asymmetric, that $\chi_{g,\xi}(\varpi_E)$ is assigned to any appropriate value is a natural result. Since $\chi_{g,\xi}$ and $\chi_{g^{-1},\xi}$ must come together, by the condition (\ref{chi-data condition transition}) of $\chi$-data we have $\chi_{g,\xi}\chi_{g^{-1},\xi}(\varpi_E)=\chi_{g,\xi}(\varpi_E({}^g\varpi_E^{-1}))$ and $\varpi_E({}^g\varpi_E^{-1})\in \mu_{E_g}.$ Hence we only need the values of $\chi_{g,\xi}$ on units. This result is comparable to Lemma 3.3.A of \cite{LS}.
    \item In (\ref{theorem sym case}) of the theorem, the module $V^\varpi_{[g]}$ consists of the unique isotypic component $V_{[\sigma^k\phi^{f/2}]}$ such that $({}^{\sigma^k\phi^{f/2}}\varpi_E)\varpi_E^{-1}  =\zeta_e^k\zeta_{\phi^{f/2}}=1.$ This is the one we considered in Lemma \ref{exists root fixing varpi}. Therefore if $[g]$ is symmetric, then $\chi_{g,\xi}(\varpi_E)$ is usually $t_\varpi(V_{[g]})t(W_{[g]})$ except when $f$ is even, $\zeta_e^k\zeta_{\phi^{f/2}}=1$, and $V_{[\sigma^k\phi^{f/2}]}$ is non-trivial.
  \end{enumerate}\qed
\end{rmk}

The proof of Theorem \ref{chi-data factor of BH-rectifier} occupies the remaining of this chapter. The structure of the proof is as follows. We recall the sequence of subfields in (\ref{sequence of fields, ramified}),
\begin{equation*}
  F=K_{-1}\hookrightarrow K_0\hookrightarrow K_{1}\hookrightarrow \cdots\hookrightarrow K_{l-1}\hookrightarrow K_l\hookrightarrow K_{l+1}=E,
\end{equation*}
and the factorization (\ref{product of rectifier}) of the rectifier into $\nu$-rectifiers
$${}_{F}\mu_\xi=({}_{K_l}\mu_\xi)({}_{K_l/K_{l-1}}\mu_\xi)\cdots({}_{K_1/K_0}\mu_\xi)({}_{K_0/F}\mu_\xi).$$
Notice that the rectifier ${}_{K_l}\mu_\xi$ can be regarded as the $\nu$-rectifier ${}_{E/K_l}\mu_\xi$. Each factor above is explicitly expressed in \cite{BH-ET1}, \cite{BH-ET2}, \cite{BH-ET3} respectively and restated in section \ref{section explicit rectifiers}. We then split the set $\mathcal{D}$ into the following subsets,
\begin{equation*}
\begin{split}
\mathcal{D}_{l+1}&=\mathcal{D} \cap ( W_{K_l}/W_E-[1]),
\\
\mathcal{D}_{l}&=\mathcal{D} \cap  (W_{K_{l-1}}/W_E- W_{K_l}/W_E),
\\
&\vdots
\\
\mathcal{D}_{1}&=\mathcal{D} \cap (  W_{K_{0}}/W_E- W_{K_1}/W_E)\text{, and }
\\
\mathcal{D}_{0}&=\mathcal{D} \cap (  W_{F}/W_E- W_{K_0}/W_E).
\end{split}
\end{equation*}
Let us call $\mathcal{D}_{\mathrm{ram}}=\mathcal{D}_{l+l}\sqcup\cdots\sqcup \mathcal{D}_{1} $ the totally ramified part and $\mathcal{D}_{0} $ the unramified part of $\mathcal{D}$. Our plan is the following.
\begin{enumerate}[(i)]
\item For each $j=l+1,\dots,0$ and each $g\in \mathcal{D}_{j}$, , we prove that if we define $\chi_{g,\xi}$ by the assigned values in Theorem \ref{chi-data factor of BH-rectifier}, then it satisfies the conditions of $\chi$-data.
\item For each $j=l+1,\dots,0$, we show that the product the restrictions of the $\chi$-data, $$\chi_{\mathcal{D}_j,\xi}|_{E^\times}:=\prod_{g\in \mathcal{D}_j}\chi_{g,\xi}|_{E^\times},$$ is equal to to the $\nu$-rectifier ${}_{K_j/K_{j-1}}\mu_\xi.$
\end{enumerate}
 We shall start proving Theorem \ref{chi-data factor of BH-rectifier} in the case when char$(\mathbf{k}_F)=p\neq2$.

\section{The asymmetric case}

For each $g=\sigma^k\phi^i\in \mathcal{D}_{\mathrm{asym}/\pm}$, we are going to compute the product of the characters
$$\chi_{g,\xi}\chi_{g^{-1},\xi}=\chi_{g,\xi}\left(\chi_{g,\xi}^{g}\right)^{-1}= \chi_{g,\xi}\circ\left[\begin{smallmatrix}
1\\  g
\end{smallmatrix}\right]$$
when restricted to $E^\times$. If we assign $\chi_{g,\xi}|_{\mu_{E_{g}}}=\mathrm{sgn}_{\mu_{E_{g}}}(V_{[g]})$ as stated in the assertion, then we have
\begin{equation}\label{value on mu-E for asymmetric}
\left(\chi_{g,\xi}\circ\left[\begin{smallmatrix}
1\\  g
\end{smallmatrix}\right]\right)|_{\mu_E}=\mathrm{sgn}_{\mu_{E}}(V_{[g]}) =t^1_{\mu}(\mathbf{V}_{[g]})
\end{equation}
and
\begin{equation}\label{asymmetric varpi not depend}
\left(\chi_{g,\xi}\circ\left[\begin{smallmatrix}
1\\  g
\end{smallmatrix}\right]\right)(\varpi_E)  = \text{sgn}_{\left[\begin{smallmatrix}
1\\  g
\end{smallmatrix}\right](\varpi_E)}(V_{[g]})=t^1_{\varpi}(\mathbf{V}_{[g]})(\varpi_E)= t_{\varpi}(\mathbf{V}_{[g]})(\varpi_E).
\end{equation}
The relation in (\ref{asymmetric varpi not depend}) does not depend on the exact value of  $\chi_{g,\xi}(\varpi_E)$ and $\chi_{g^{-1},\xi}(\varpi_E)$, so we can assign them to anything appropriate. We will use the t-factors on both right sides of (\ref{value on mu-E for asymmetric}) and (\ref{asymmetric varpi not depend}) later when we compare the product of $\chi$-data with the $\nu$-rectifier.

\section{The symmetric ramified case}

We now consider those $\chi$-data $\chi_{g,\xi}$ when $g \in \mathcal{D}_{\text{sym,ram}}=\mathcal{D}_{\text{sym}}\cap \mathcal{D}_{\text{ram}} $, the subset of symmetric ramified double cosets. They are of the form $g=\sigma^k$ such that $k$ satisfies certain divisibility condition in Proposition \ref{properties of symmetric [g]}.

\subsection{The case $g\in  \mathcal{D}_{l+1,\mathrm{sym}}$}

We first consider those $\chi_{g,\xi}$ when $g =\sigma^k\in  \mathcal{D}_{l+1,\mathrm{sym}}= \mathcal{D}_{l+1}\cap \mathcal{D}_{\mathrm{sym}}$. We assume that $k$ has odd order in the additive group $\mathbb{Z}/e$. We assign
\begin{equation}\label{value on mu-E for sym in S-(s+1)}
\chi_{\sigma^k,\xi}|_{\mu_E}=t^1_{\mu}(V_{[\sigma^k]}),
\end{equation}
which is 1 since $\mu_E$ acts on $V_{[\sigma^k]}$ trivially. We also assign
\begin{equation}\label{value on varpi for sym in S-(s+1)}
\chi_{\sigma^k,\xi}(\varpi_E) = t_\varpi(V_{[\sigma^k]})t(W_{[\sigma^k]}).
\end{equation}
Using the values in Proposition \ref{summary of t-factors}, we can compute the value on the right side of (\ref{value on varpi for sym in S-(s+1)}). If $V_{[\sigma^k]}$ is trivial, so that $W_{[\sigma^k]}$ is non-trivial, then $$t_\varpi(V_{[\sigma^k]})=1\text{ and }t(W_{[\sigma^k]})=-1.$$ Let $s$ be the minimal positive integer such that $e|(1+p^s)k$. If $V_{[\sigma^k]}$ is non-trivial, then we have $$t^0_\varpi(V_{[\sigma^k]})=-1\text{ and }t^1_\varpi(V_{[\sigma^k]})(\varpi_E) = \left(\frac{\zeta^k_e}{\mu_{p^s+1}}\right).$$
The latter is equal to 1 since $\zeta_e^k$ has odd order and $p^s+1$ is even. Therefore the right side of (\ref{value on varpi for sym in S-(s+1)}) is always equal to $t_\varpi(U_{[\sigma^k]})=-1$.

We need to further require that our $\chi_{\sigma^k,\xi}$ satisfies the condition of $\chi$-data
\begin{equation}\label{condition of chi data for sym in S-(s+1)}
\chi_{\sigma^k,\xi}|_{{E_{\pm\sigma^k}}} = \delta_{E_{\sigma^k}/E_{\pm\sigma^k}}.
\end{equation}
We can make such condition more explicit. Since $E_{\sigma^k}/E_{\pm\sigma^k}$ is quadratic unramified, the norm group $N_{E_{\sigma^k}/E_{\pm\sigma^k}}({E^\times_{\sigma^k}})$ has a decomposition
$$\mu_{E_{\pm\sigma^k}}\times\left<\zeta_e^k\varpi_E^2\right>\times U^1_{E_{\pm\sigma^k}}.$$
Take a root of unity $\zeta_0\in \mu_{E_{\sigma^k}}$ such that
 $$\zeta_0\varpi_E\in E_{\pm\sigma^k}-N_{E_{\sigma^k}/E_{\pm\sigma^k}}({E^\times_{\sigma^k}}).$$
  Then the condition (\ref{condition of chi data for sym in S-(s+1)}) is equivalent to the following explicit conditions:
\begin{equation}\label{3 conditions to check chi in S-(s+1)}
\chi_{\sigma^k,\xi}|_{\mu_{E_{\pm\sigma^k}}} \equiv 1,\,\chi_{\sigma^k,\xi}(\zeta^k_e \varpi^2_E)= 1\text{, and }\chi_{\sigma^k,\xi}(\zeta_0\varpi_E) = -1.
\end{equation}
If we define $\chi_{\sigma^k,\xi}$ to be the unramified character $$\chi_{\sigma^k,\xi}|_{\mu_{E_{\sigma^k}}} \equiv 1\text{ and }\chi_{\sigma^k,\xi}(\varpi_E) = -1,$$ then, since $\mu_E\subseteq \mu_{E_{\pm\sigma^k}}$, we can check that $\chi_{\sigma^k,\xi}$ satisfies all the conditions in (\ref{value on mu-E for sym in S-(s+1)}), (\ref{value on varpi for sym in S-(s+1)}) and (\ref{3 conditions to check chi in S-(s+1)}).

Finally, we have to show that the product of the characters when restricted to $E^\times$,
\begin{equation*}
 \chi_{\mathcal{D}_{l+1},\xi}=\prod_{{\sigma^k}\in \mathcal{D}_{l+1}}\chi_{\sigma^k,\xi}|_{E^\times}=\left(\prod_{{\sigma^k}\in\mathcal{D}_{{l+1},{\mathrm{asym}/\pm}}}\chi_{\sigma^k,\xi}\circ\left[\begin{smallmatrix}
1\\  \sigma^k
\end{smallmatrix}\right]|_{E^\times}\right)\left(\prod_{{\sigma^k}\in \mathcal{D}_{l+1,\text{sym}}}\chi_{\sigma^k,\xi}|_{E^\times}\right),
  \end{equation*}
 is equal to ${}_{K_l/E}\mu_\xi$. We first compare the values of both characters when restricted to $\mu_E$. Those asymmetric ${\sigma^k}$ have the right side of (\ref{value on mu-E for asymmetric}) being 1 since $\mu_E$ acts trivially on $\mathbf{V}_{{[\sigma^k]}}$. Combining them with the symmetric ones in (\ref{value on mu-E for sym in S-(s+1)}) and using the explicit value in (\ref{explicit rectifier ram-odd case}), we have
\begin{equation*}
 \chi_{\mathcal{D}_{l+1},\xi}|_{\mu_E}=1={}_{K_l}\mu_\xi|_{\mu_E}.
 \end{equation*}
 We then compare the values of both characters at $\varpi_E$. Combining (\ref{asymmetric varpi not depend}) and (\ref{value on varpi for sym in S-(s+1)}), we have
 \begin{equation*}
 \chi_{\mathcal{D}_{l+1},\xi}(\varpi_E)=t_\varpi(U_{E/K_l}).
 \end{equation*}
 By Lemma 4.2.(ii) of \cite{Rei}, we have
\begin{equation*}
 t_{\varpi}(U_{E/K_l}) = \left(\frac{q^f}{e(E/K_l)}\right).
 \end{equation*}
 This is equal to ${}_{E/K_l}\mu_\xi(\varpi_E)$ as in (\ref{explicit rectifier ram-odd case}). We have proved Theorem \ref{chi-data factor of BH-rectifier} for the part when $g \in \mathcal{D}_{l+1}$.

Before proceeding to the next case, we provide a couple of facts about the roots of unity appearing in (\ref{3 conditions to check chi in S-(s+1)}).

\begin{lemma}\label{lemma about explicit chi-data condition}
Suppose that $g=\sigma^k\in \mathcal{D}_{\mathrm{sym,ram}}$, with $g\neq 1$ or $\sigma^{e/2}$.
\begin{enumerate}[(i)]
\item The element $\zeta_e^k$ lies in $\ker(N_{{E_g}/{E_{\pm g}}})|_{\mu_{E_g}}-\mu_{E_{\pm g}}$. \label{lemma about zeta_e^k and kernel}
\item Suppose that $|E_g/E|=2t_k$. The element $\zeta_0$ satisfies the relation $\zeta_0^{1-q^{ft_k}} = \zeta^k_e$. \label{lemma about zeta_0}
\end{enumerate}
\end{lemma}
\begin{proof}
Recall that $\sigma^k\phi^{ft_k}\in W_{E_{\pm g}}-W_{E_{g}}$ and $(\sigma^k\phi^{ft_k})^2\in W_{E_{g}}$. We then have $$N_{{E_g}/{E_{\pm g}}}(\zeta_e^k)=\zeta_e^k({}^{\sigma^k\phi^{ft_k}}\zeta_e^k)=\zeta_e^{k(1+q^{ft_k})}=1,$$ hence (\ref{lemma about zeta_e^k and kernel}). Since $\zeta_0\varpi_E\in E_{\pm g}$, we have $$\zeta_0\varpi_E={}^{\sigma^k\phi^{ft_k}}(\zeta_0\varpi_E)=\zeta_0^{q^{ft_k}}\zeta_e^k\varpi_E,$$ hence (\ref{lemma about zeta_0}).
\end{proof}

\subsection{The case $g\in  \mathcal{D}_{l,\mathrm{sym}}$}

We next study those $\chi_{g,\xi}$ when $g \in  \mathcal{D}_{l,\mathrm{sym}}$. We first consider the distinguished element $g=\sigma^{e/2}$. By Proposition \ref{trivialness of some V[g]}, the module $V_{[\sigma^{e/2}]}$ is always trivial. Recall that $E_{\sigma^{e/2}}=E$ and $E/E_{\pm\sigma^{e/2}}$ is quadratic totally ramified. We assign
\begin{equation}\label{value on mu-E for e/2}
\chi_{\sigma^{e/2},\xi}|_{\mu_E}:\zeta\mapsto\left(\frac{\zeta}{\mu_E}\right).
\end{equation}
We would assign $\chi_{\sigma^{e/2},\xi}(\varpi_E)$ such that $\chi_{\sigma^{e/2},\xi}$ satisfies the condition of $\chi$-data $$\chi_{\sigma^{e/2},\xi}|_{E^\times_{\pm\sigma^{e/2}}}=\delta_{E/E_{\pm\sigma^{e/2}}}.$$
Since $N_{E/E_{\pm\sigma^{e/2}}}(\varpi_E)=-\varpi_{E_{\pm\sigma^{e/2}}}=-\varpi_E^2$, we just require that
$$\chi_{\sigma^{e/2},\xi}(\varpi_E)^2=\chi_{\sigma^{e/2},\xi}(-1)=\left(\frac{-1}{\mu_E}\right).$$
If we assign
\begin{equation}\label{value on varpi for e/2}
 \chi_{\sigma^{e/2},\xi}(\varpi_E)=t(W_{[\sigma^{e/2}]})= \mathfrak{n}(\psi_{K_{l-1}}),
\end{equation}
 then the above is clearly satisfied. Therefore we have shown that $\chi_{\sigma^{e/2},\xi}$ is a character in $\chi$-data.

We then consider other $\sigma^k\in \mathcal{D}_{l,\mathrm{sym}}$, $k \neq e/2$. Since $\mu_E$ acts trivially on $V_{[\sigma^k]}$, we can assign
\begin{equation}\label{value on mu-E for sym in S-s}
\chi_{\sigma^k,\xi}|_{\mu_E}=t^1_{\mu}(V_{[\sigma^k]})=1.
\end{equation}
We look for the value of $\chi_{\sigma^k,\xi}(\varpi_E)$ such that $\chi_{\sigma^k,\xi}$ satisfies the condition
$$\chi_{\sigma^k,\xi}|_{{E_{\pm\sigma^k}}}=\delta_{E_{\sigma^k} / E_{\pm\sigma^k}}. $$ Similar to (\ref{3 conditions to check chi in S-(s+1)}), we can write down such $\chi$-data conditions explicitly as
\begin{equation}\label{3 conditions to check chi in S-s}
\chi_{\sigma^k,\xi}|_{\mu_{E_{\pm\sigma^k}}} \equiv 1,\,\chi_{\sigma^k,\xi}(\zeta^k_e \varpi^2_E)= 1,\text{ and }\chi_{\sigma^k,\xi}(\zeta_0\varpi_E) = -1.
\end{equation}
We will check, in Lemma \ref{lemma determine a character, sym-ram case} below, that there exists a character $\chi$ of $E^\times_g$ with the following values,
\begin{equation}\label{conditions for a well-defined charcater in sym-ram case}
\begin{split}
 & \chi|_{\mu_{E_{\pm\sigma^k}}}\equiv 1,\,\chi(\zeta_e^k)=1,\,\chi(\zeta_0)= -t_{\varpi}(V_{[\sigma^k]})t(W_{[\sigma^k]}),
  \text{ and }\\&\chi(\varpi_E) = t_{\varpi}(V_{[\sigma^k]})t(W_{[\sigma^k]}).
\end{split}
\end{equation}
Since $\mu_E\subseteq \mu_{E_{\pm\sigma^k}}$ and $\zeta_0\in \mu_{E_{\sigma^k}}-(\mu_{E_{\pm\sigma^k}}\cup\{\zeta_e^k\})$, such character satisfies the conditions in (\ref{value on mu-E for sym in S-s}) and (\ref{3 conditions to check chi in S-s}). Therefore any character with the above values is our desired character $\chi_{\sigma^k,\xi}$.

\begin{lemma}\label{lemma determine a character, sym-ram case}
  There exists a character $\chi$ with the values specified in (\ref{conditions for a well-defined charcater in sym-ram case})
\end{lemma}

\begin{proof}
  We can take $\chi$ to be trivial on $\mu_{E_{\pm\sigma^k}}$ and at $\zeta_e^k$ so that the first two conditions in (\ref{conditions for a well-defined charcater in sym-ram case}) are satisfied. The last two conditions in (\ref{conditions for a well-defined charcater in sym-ram case}) depends on each other by the last requirement in (\ref{3 conditions to check chi in S-s}). It remains to show that we can extend the trivial character to $\zeta_0$ subjected to the first two conditions. Let us write $\zeta^k_e = \zeta_N$ for some $N=2M$. Such $M$ must be odd. Indeed, if $e=2^ld$ for some odd number $d$ and $\sigma^k\in \mathcal{D}_l$, then  $k\in  2^{l-1}\mathbb{Z} / e - 2^l\mathbb{Z} / e$; and if we write $k=2^{l-1}k' $ for some odd number $k'$, then $\zeta^k_e = \zeta^{k'}_{2d} = \zeta_N$, where $N=2d/\text{gcd}(k',d)$. This means that the group generated by $\zeta_e^k$ and $\mu_{E_{\pm\sigma^k}}$ is $\mu_{\mathrm{lcm} (1-q^{ft_k},N)}$. Recall, by Lemma \ref{lemma about explicit chi-data condition}.(\ref{lemma about zeta_0}), that $\zeta_0^{1-q^{ft_k}} = \zeta^k_e = \zeta_N$. Hence $\zeta_0 $ is a primitive $N(q^{ft_k}-1)$th root of unity; in other words, it generates $\mu_{N(q^{ft_k}-1)}$. The index of the subgroup $\mu_{\mathrm{lcm}(q^{ft_k}-1,N)}$ in $\mu_{N(q^{ft_k}-1)}$ is $2\gcd(M,(q^{ft_k}-1)/2)$, which is even. Therefore we can assign $\chi(\zeta_0)$ to be either $\pm1$. \end{proof}

We are ready to compare $\chi_{\mathcal{D}_l,\xi}$ with  ${}_{K_{l}/K_{l-1}}\mu_\xi|_{\mu_E}$. On $\mu_E$, we combine (\ref{value on mu-E for e/2}), (\ref{value on mu-E for sym in S-s}) and (\ref{value on mu-E for asymmetric}) to obtain that
\begin{equation*}
\chi_{\mathcal{D}_l,\xi}|_{\mu_E}=\left(\frac{}{\mu_E}\right).
\end{equation*}
This is equal to ${}_{K_{l}/K_{l-1}}\mu_\xi|_{\mu_E}$ by (\ref{explicit rectifier ram-even odd case}). We then compare $\chi_{\mathcal{D}_l,\xi}(\varpi_E)$ with $ {}_{K_l / K_{l-1}} \mu_\xi(\varpi_E)$. We again recall from (\ref{explicit rectifier ram-even odd case}) that
$$ {}_{K_l / K_{l-1}} \mu_\xi(\varpi_E)= t_\varpi(V_{K_l/K_{l-1}})\mathrm{sgn}(({\mathfrak{G}_{K_{l-1}}}/{\mathfrak{G}_{K_l}})(\mathfrak{K}^\kappa_{K_{l-1}}/\mathfrak{K}_{K_{l-1}} )).$$
From (\ref{value on varpi for e/2}) and (\ref{conditions for a well-defined charcater in sym-ram case}), it suffices to show that $$t(W_{K_l/K_{l-1}})=\mathrm{sgn}(({\mathfrak{G}_{K_{l-1}}}/{\mathfrak{G}_{K_l}})(\mathfrak{K}^\kappa_{K_{l-1}}/\mathfrak{K}_{K_{l-1}} )).$$
By regarding $\xi$ as an admissible character over $K_{l-1}$, we let $E =E_0\supsetneq E_1 \cdots \supsetneq E_m \supsetneq E_{m+1}=K_{l-1}$ and $\{r_0,\dots,r_m\}$ be the jump data of $\xi$. There is a $\varpi$-invariant quadratic form on $W_{E/K_{l-1}}$ defined in section 5.7 of \cite{BH-ET2} such that on each jump component $W_j$, $j=0,\dots,m$, this form is denoted by $\zeta_j(\varpi_E)\mathbf{q}^j_{K_{l-1}}$. Let $\mathbf{q}^j_{K_l/K_{l-1}}$ be the restriction of $\mathbf{q}^j_{K_{l-1}}$ on $W_{K_l/K_{l-1}}$. By definition in section \ref{section comp mod}, we have
\begin{equation*}
t(W_{K_l/K_{l-1}})
=\prod^m_{j=0} \mathfrak{n}(\zeta_j(\varpi_E)\mathbf{q}^j_{K_l/K_{l-1}}, \psi_{K_{l-1}}).
\end{equation*}
By excluding the factor indexed by $j=0$ in the above product, we have
$$t\left(\bigoplus_{j=1}^mW_j\right)=\prod^m_{j=1} \mathfrak{n}(\zeta_j(\varpi_E)\mathbf{q}^j_{K_l/K_{l-1}}, \psi_{K_{l-1}}),$$
 which is equal to $\mathrm{sgn}(\mathfrak{G}_{K_{l-1}}/\mathfrak{G}_{K_l})$ by section 8.1 of \cite{BH-ET2}. Therefore we only have to show that
\begin{equation}\label{t0=sgn(k), D_l case}
 t(W_0)=\mathfrak{n}(\zeta_0(\varpi_E)\mathbf{q}^0_{K_l/K_{l-1}}, \psi_{K_{l-1}})=\mathrm{sgn}(\mathfrak{K}^\kappa_{K_{l-1}}/\mathfrak{K}_{K_{l-1}} ).
\end{equation}
We first distinguish between $r_0>1$ and $r_0=1$.
 \begin{enumerate}[(i)]
   \item If $r_0>1$, then we have $W_0=0$ and so $\mathfrak{n}(\zeta_0(\varpi_E)\mathbf{q}^0_{K_l/K_{l-1}}, \psi_{K_{l-1}})=1$. By convention, $\mathfrak{K}^\kappa_{K_{l-1}}/\mathfrak{K}_{K_{l-1}}=1$ in the case $r_0 > 1$. Therefore we have checked that (\ref{t0=sgn(k), D_l case}) is true.

   \item If $r_0=1$, then we recall, as in section \ref{section explicit rectifiers}, that we can separate into three possible subcases, namely $$i_+=i^+=r_0=1,\,i_+>i^+=r_0=1,\text{ and }i_+\geq i^+>r_0=1.$$
      When $|E/K_{l-1}|/2$ is odd, we have shown that the second case $i_+>i^+=r_0=1$ cannot exist.

   \begin{enumerate}
     \item If $i_+ = i^+ = r_0 =1$, then the jumps are of the form $\{ 1, 2{s_1}, ..., 2{s_m}\}$ and $d^+=|E_1 / K_{l-1}|$ is odd. This implies that
         \begin{equation}\label{module Y}W_0=(W_{{E/E_1}})_{\mathcal{D}_l} = \bigoplus_{\begin{smallmatrix}
{k=1,\dots,|E/E_1|,}\\   k \text{ is odd}
 \end{smallmatrix}}W_{[\sigma^{ke(E_1 /F)}]}.
 \end{equation}
 This module is denoted by $\mathcal{Y}$ in section 8.2 of \cite{BH-ET2}, whose degree $\dim_{\mathbf{k}_E} \mathcal{Y}=|E/E_1|/2$ is odd. By the proof of Proposition 8.3 of \cite{BH-ET2}, which still goes when $s=0$ verbatim, we have
 $$ \left(\frac{\mathbf{q}_{K_{l-1}}|_{\mathcal{Y}}}{q^f} \right) = \left(\frac{|E_1/K_{l-1}|}{q^f} \right) = \left(\frac{d^+}{q^f} \right).$$ Hence
\begin{equation*}
\begin{split}
t(\mathcal{Y})=&\mathfrak{n}(\zeta_0(\varpi_E)\mathbf{q}_{K_{l-1}}|_{\mathcal{Y}} ,\psi_{K_{
l-1}})\\
=&\left( \frac{\zeta_0(\varpi_E)}{q^f}\right)^{\dim\mathcal{Y}} \left(\frac{\mathbf{q}_{K_{l-1}}|_{\mathcal{Y}}}{q^f} \right) \mathfrak{n}(\psi_{K_{l-1}})^{\dim\mathcal{Y}}
\\
=&\left( \frac{\zeta_0(\varpi_E)}{q^f}\right)\left(\frac{d^+}{q^f} \right) \mathfrak{n}(\psi_{K_{l-1}})^{|E/E_1|/2}.
\end{split}
\end{equation*}
This is equal to $\mathfrak{K}^\kappa_{K_{l-1}}/\mathfrak{K}_{K_{l-1}}$, as computed in section \ref{section explicit rectifiers} or by 9.3 of \cite{BH-ET2}.

\item
If $i_+\geq i^+>r_0=1$, then $|E_1/K_{l-1}|$ is even. In particular, ${W_0}=(W_{E/E_1})_{\mathcal{D}_l}=0$, and so $t({W_0})=1$. This is equal to $\mathfrak{K}^\kappa_{K_{l-1}}/\mathfrak{K}_{K_{l-1}}$, as computed in section \ref{section explicit rectifiers} or by Lemma 8.1(3) of \cite{BH-ET2}.

   \end{enumerate}
   Therefore, (\ref{t0=sgn(k), D_l case}) is true when $r_0=1$.
 \end{enumerate}

Combining all cases, we have proved Theorem \ref{chi-data factor of BH-rectifier} for the part $ \mathcal{D}_{l}\sqcup  \mathcal{D}_{l-1}$.

\subsection{The case $g\in\mathcal{D}_{l-1,\mathrm{sym}}\sqcup\cdot\cdot\cdot\sqcup\mathcal{D}_{1,\mathrm{sym}}$}

We study those $\chi_{g,\xi}$ when $g \in  \mathcal{D}_{l-1,\mathrm{sym}}$. We first consider another distinguished element $g=\sigma^{e/4}$, in case if it is symmetric.
\begin{lemma}\label{symmetric sigma-e-4 equivalent}
The following are equivalent.
\begin{enumerate}[(i)]
\item The double coset $[\sigma^{e/4}]$ is symmetric. \label{symmetric sigma-e-4 equivalent symmetric}
\item $(1+q^{ft})/2$ is even for some $t$. \label{symmetric sigma-e-4 equivalent even t}
\item $ q^f\equiv 3\mod 4$. \label{symmetric sigma-e-4 equivalent 3 mod 4}
\item $ \left(\frac{-1}{q^f}\right)=-1$. \label{symmetric sigma-e-4 equivalent -1 is square}
\item The degree 4 totally ramified extension $ K_l/K_{l-2}$ is non-Galois. \label{symmetric sigma-e-4 equivalent non galois}
\end{enumerate}
\end{lemma}

\begin{proof}
By the definition of symmetry, (\ref{symmetric sigma-e-4 equivalent symmetric}) is equivalent to the statement that $e$ divides $(1+q^{ft})e/4$ for some $t$, which is clearly equivalent to (\ref{symmetric sigma-e-4 equivalent even t}). The equivalence of (\ref{symmetric sigma-e-4 equivalent even t}), (\ref{symmetric sigma-e-4 equivalent 3 mod 4}), and (\ref{symmetric sigma-e-4 equivalent -1 is square}) is clear. A totally ramified extension $E/F$ of degree 4 is non-Galois if and only if 4 does not divide $q^f-1$. Hence (\ref{symmetric sigma-e-4 equivalent non galois}) is equivalent to (\ref{symmetric sigma-e-4 equivalent 3 mod 4}).
\end{proof}

By regarding $\xi$ as an admissible character of $E^\times$ over $K_{l-2}$, we let $E =E_0\supsetneq E_1 \cdots \supsetneq E_m \supsetneq E_{m+1}=K_{l-2}$ and $\{r_0,\dots,r_m\}$ be the corresponding jump data. Recall that $i^+$ and $i_+$ are the jumps $r_T$ and $r_S$ defined in section \ref{section explicit rectifiers}, and that $i_+\geq i^+$ always. We can characterize $V_{[\sigma^{e/4}]}$ by the following.
\begin{lemma}
$V_{[\sigma^{e/4}]}$ is non-trivial if and only if $i_+>i^+$.
\end{lemma}
\begin{proof}
Let $r_j$ be the jump that $V_{[\sigma^{e/4}]}\subseteq V_j$. The sufficient condition in the statement is equivalent to that $r_j$ is even, and the necessary condition in the statement is equivalent to that $i_+$ is even. Since $4|e(E/E_j)$ and $e(E_j/K_{l-2})$ is odd, we have $r_j=i_+$. The statement then follows.
\end{proof}

We assign
\begin{equation}\label{value on mu-E for e/4}
\chi_{\sigma^{e/4},\xi}|_{\mu_E}=t^1_{\mu}(V_{[\sigma^{e/4}]})\equiv 1
\end{equation} and
\begin{equation*}
\chi_{\sigma^{e/4},\xi}(\varpi_E)=t_\varpi(V_{[\sigma^{e/4}]})t(W_{[\sigma^{e/4}]})=
\begin{cases}
t_\varpi(V_{[\sigma^{e/4}]}) & \text{if }V_{[\sigma^{e/4}]}\text{ is non-trivial,} \\
t(W_{[\sigma^{e/4}]})=-1 & \text{if }V_{[\sigma^{e/4}]}\text{ is trivial.}
\end{cases}
\end{equation*}
The verification of $\chi_{\sigma^{e/4},\xi}$ being a $\chi$-data, no matter $V_{[\sigma^{e/4}]}$ is trivial or not, is similar to the statement following (\ref{3 conditions to check chi in S-s}).

\begin{rmk}
We can compute the value of $t_\varpi(V_{[\sigma^{e/4}]})$ if $V_{[\sigma^{e/4}]}$ is non-trivial. In this case, we have that $V_{[\sigma^{e/4}]} = \mathbf{k}_E(\zeta_4)$ as $\mathbf{k}_E$-vector spaces. By Lemma \ref{symmetric sigma-e-4 equivalent}, we have $q^f \equiv 3 \mod 4$, which implies also that $p \equiv 3 \mod4$. We have that $|\mathbb{F}_p(\zeta_4)/ \mathbb{F}_p| = 2$ and that $r = |\mathbf{k}_E(\zeta_4) / \mathbb{F}_p(\zeta_4)| = | \mathbf{k}_E/ \mathbb{F}_p|$ must be odd by Lemma \ref{order 2 trick}. Therefore
\begin{equation*}
t_\varpi(V_{[\sigma^{e/4}]}) =  \left(-\left(\frac{\zeta_4}{\mu_{p+1}}\right)\right)^r
=-\left(\frac{\zeta_4}{\mu_{p+1}}\right)
=
\begin{cases}
1  &\text{ if }  p\equiv 3 \mod 8 ,\\
-1   &\text{ if }  p\equiv 7 \mod 8.
\end{cases}
\end{equation*}
\qed\end{rmk}

For other symmetric $g=\sigma^{k}$, $k\neq e/4$, we proceed as in the case when $\sigma^k\in \mathcal{D}_l$, $k\neq e/2$, and obtain
\begin{equation}\label{value on mu-E and varpi for sym in S-(s-1)}
\chi_{\sigma^k,\xi}|_{\mu_E}=t^0_{\mu}(V_{[\sigma^k]})=1\text{ and }\chi_{\sigma^k,\xi}(\varpi_E) = t_{\varpi}(V_{[\sigma^k]})t(W_{[\sigma^k]}).
\end{equation}
Again we can use statements similar to Lemma \ref{lemma determine a character, sym-ram case} to verify that $\chi_{\sigma^k,\xi}$ is a character in $\chi$-data.

We then check that $\chi_{\mathcal{D}_{l-1},\xi}|_{E^\times}={}_{K_{l-1}/K_{l-2}}\mu_\xi$. We combine (\ref{value on mu-E for e/4}), (\ref{value on mu-E and varpi for sym in S-(s-1)}), and (\ref{value on mu-E for asymmetric}) to obtain that $\chi_{\mathcal{D}_{l-1},\xi}|_{\mu_E}\equiv1$, which is equal to ${}_{K_{l-1}/K_{l-2}}\mu_\xi|_{\mu_E}$ by (\ref{explicit rectifier ram-even even case}). To compare $\chi_{\mathcal{D}_{l-1},\xi}(\varpi_E)$ with  ${}_{K_{l-1}/K_{l-2}}\mu_\xi(\varpi_E)$, we proceed as in the case when $g\in  \mathcal{D}_{l,\mathrm{sym}}$. We consider the components $W_j$, $j=0,\dots,m$. As before, it suffices to show that $ t(W_0)=\mathrm{sgn}(\mathfrak{K}^\kappa_{K_{l-1}}/\mathfrak{K}_{K_{l-1}} ).$ When $r_0>1$, the situation is similar to the case when $g\in  \mathcal{D}_{l,\mathrm{sym}}$. When $r_0=1$, we compute the module $W_0$ by distinguishing between the cases $i_+=i^+=r_0=1$, $i_+>i^+=r_0=1$, and $i^+>r_0=1$.

\begin{enumerate}[(i)]
  \item If $i_+=i^+=r_0=1$, then $ |E_1/K_{l-2}|$ is odd. Similar to the case in (\ref{module Y}), we have
  $$W_0=(W_{E/E_1})_{\mathcal{D}_{l-1}}=\bigoplus_{\begin{smallmatrix}
    k=1,\dots,|E/E_1|,\\ k\text{ is odd}
  \end{smallmatrix}}W_{[\sigma^{ke(E_1/F)}]}.$$
However, now $\dim\mathcal{Y}=|E/E_1|/2$ is even. As in the proof of Proposition 8.4 of \cite{BH-ET2}, which is still true for $r_0=1$, we have $\left( \frac{\det \mathbf{q}_{K_{l-1}}|_\mathcal{Y}}{q^f}\right)=1$. Hence
$$t(\mathcal{Y})=\left( \frac{\zeta_0(\varpi_E)}{q^f}\right)^{\dim\mathcal{Y}}\left( \frac{\det \mathbf{q}_{K_{l-1}}|_\mathcal{Y}}{q^f}\right)\mathfrak{g}(\psi_{K_{l-2}})^{\dim\mathcal{Y}}=\left( \frac{-1}{q^f}\right)^{|E/E_1|/4}.$$ This is equal to $\mathfrak{K}^\kappa_{K_{l-2}}/\mathfrak{K}_{K_{l-2}} $ by the calculation in section \ref{section explicit rectifiers} or by 9.3 of \cite{BH-ET2}.

  \item In the cases $i_+>i^+=r_0=1$ or $i^+>r_0=1$, we must have that $ |E_1/K_{l-2}|$ is even, since otherwise $i_+=1$. Therefore $W_0$ is trivial and $t(W_0)=1$. This is equal to $\mathfrak{K}^\kappa_{K_{l-2}}/\mathfrak{K}_{K_{l-2}}$, as computed in section \ref{section explicit rectifiers} or by Proposition 8.1 of \cite{BH-ET2}.
\end{enumerate}
We have proved that $\chi_{\mathcal{D}_{l-1},\xi}|_{E^\times}={}_{K_{l-1}/K_{l-2}}\mu_\xi$. Therefore we have proved Theorem \ref{chi-data factor of BH-rectifier} for the part $\mathcal{D}_{l+1}\sqcup \mathcal{D}_{l}\sqcup \mathcal{D}_{l-1}$.

When $g=\sigma^k\in \mathcal{D}_{l-2}\sqcup\cdots\sqcup \mathcal{D}_1$, the proofs that $\chi_{\sigma^k,\xi}$ defines a character in $\chi$-data and that $\chi_{\mathcal{D}_{j},\xi}|_{E^\times}={}_{K_{j}/K_{{j-1}}}\mu_\xi$, for $j={l-2}, \dots,1,$ are very similar to the previous cases. Notice that $e/4$ is even in these cases. By the arguments preceding Proposition \ref{explicit rectifier quotient of sums ram-even case}, we can show that
$$\mathfrak{G}_{K_{l-2}}/\mathfrak{G}_{K_{l-1}}=\mathfrak{K}^\kappa_{K_{l-2}}/\mathfrak{K}_{K_{l-2}}=1.$$
We have proved Theorem \ref{chi-data factor of BH-rectifier} for the totally ramified part $\mathcal{D}_\mathrm{ram}$.

\section{The symmetric unramified case}
We now consider those $\chi_{g,\xi}$ when $g \in \mathcal{D}_{\text{sym,unram}} =\mathcal{D}_{\text{sym}}\cap \mathcal{D}_{0}$. These $g$ are of the form $g=\sigma^k\phi^{f/2}$ and satisfy certain divisibility condition in Proposition \ref{properties of symmetric [g]}. Their corresponding roots are of the form $\left[\begin{smallmatrix}
1\\  \sigma^k\phi^{f/2}
\end{smallmatrix}\right]:x\mapsto x({}^{\sigma^k\phi^{f/2}}x)^{-1}$, $x\in E^\times$, which maps $\varpi_E$ to $ \zeta_e^k\zeta_{\phi^{f/2}}$. We distinguish between the following three cases, $$\zeta_e^k\zeta_{\phi^{f/2}}=1,\,\zeta_e^k\zeta_{\phi^{f/2}}=-1\text{, and }\zeta_e^k\zeta_{\phi^{f/2}}\neq\pm1.$$

\begin{enumerate}[(i)]
  \item
When $\zeta_e^k\zeta_{\phi^{f/2}}=1$, we have $E_{g}=E$. Write $E_{\pm g}=E_\pm$. Similar to (\ref{3 conditions to check chi in S-(s+1)}), the $\chi$-data conditions are explicitly
\begin{equation}\label{3 conditions to check chi in phi-f/2}
\chi_{g,\xi}|_{\mu_{E_\pm}} \equiv 1,\,\chi_{g,\xi}(\varpi^2_E) = 1,\text{ and }\chi_{g,\xi}(\zeta_0\varpi_E) = -1,
\end{equation}
where $\zeta_0\in \mu_{E}$ such that $\zeta_0^{q^{f/2}-1}=1$ and so $\zeta_0\in \mu_{E_\pm}\subseteq \mu_{E}$. If $V_{[g]}$ is trivial, then we assign
\begin{equation*}
\chi_{g,\xi}|_{\mu_E} = t^1_{\mu}(V_{[g]})\equiv 1
\end{equation*}
and
\begin{equation*}
  \chi_{g,\xi}(\varpi_E) =t^0_\mu(V^{\varpi}_{[g]})t_\varpi(V_{[g]})t(W_{[g]})=(1)(1)(-1)= -1.
\end{equation*}
This character clearly satisfies the conditions in (\ref{3 conditions to check chi in phi-f/2}). If $V_{[g]}$ is non-trivial, then $V^\varpi=V_{[g]}$. We assign
\begin{equation*}
\chi_{g,\xi}|_{\mu_E} = t^1_{\mu}(V_{[g]}): \zeta \mapsto
\left(\frac{\zeta^{q^{f/2}-1}}{\mu_{q^{f/2}+1}}\right),
\end{equation*}
the unique quadratic character of $\mu_E$, and
\begin{equation*}
\chi_{g,\xi}(\varpi_E) =t^0_\mu(V^{\varpi}_{[g]})t_\varpi(V_{[g]})t(W_{[g]})=(-1)(1)(1)=-1.
\end{equation*}
Since ${\mu_{E_\pm}}=\mu_{q^{f/2}-1}$,  we have $\chi_{g,\xi}({\mu_{E_\pm}})=1$. The conditions in (\ref{3 conditions to check chi in phi-f/2}) are readily satisfied. Therefore $\chi_{g,\xi}$ is a character of $\chi$-data.

\item
When $\zeta_e^k\zeta_{\phi^{f/2}}=-1$, again $E_{g}=E$ and is quadratic unramified over $E_\pm=E_{\pm g}$. The $\chi$-data conditions are explicitly
\begin{equation}\label{3 conditions to check chi in sigma e/2 phi f/2}
\chi_{g,\xi}|_{\mu_{E_{\pm }}} \equiv 1,\,\chi_{g,\xi}(- \varpi^2_E)= 1,\text{ and }\chi_{g,\xi}(\zeta_0\varpi_E) = -1.
\end{equation}
Here $\zeta_0\in \mu_{E}$ satisfies $\zeta_0^{q^{f/2}-1}=-1$, so that $\zeta_0$ is a primitive root in $\mu_{2(q^{f/2}-1)}$. If $V_{[g]}$ is trivial, then we assign
\begin{equation*}
\chi_{g,\xi}|_{\mu_E} = t^1_{\mu}(V_{[g]})\equiv 1
\text{ and }
\chi_{g,\xi}(\varpi_E) =t_\varpi(V_{[g]})t(W_{[g]})=(-1)(1)= -1.
\end{equation*}
 By noticing that $\zeta_0\in \mu_{2(q^{f/2}-1)}\subset \mu_{q^f-1}=\mu_E$, the verification that $\chi_{g,\xi}$ is a character of  $\chi$-data is an easy analogue of the previous case. When $V_{[g]}$ is non-trivial, we assign
\begin{equation}\label{value on mu-E for sym in S-0 k=e/2}
\chi_{g,\xi}|_{\mu_E} = t^1_{\mu}(V_{[g]}) =\text{quadratic}
\end{equation}
and
\begin{equation}\label{value on varpi for sym in S-0 k=e/2}
\chi_{g,\xi}(\varpi_E) =t_\varpi(V_{[g]})t(W_{[g]})= (-1)^{\frac{1}{2}(q^{f/2}-1)}(1)=(-1)^{\frac{1}{2}(q^{f/2}-1)}.
\end{equation}
Since $\mu_{E_\pm}=\mu_{q^{f/2}-1}$, our assignment (\ref{value on mu-E for sym in S-0 k=e/2}) satisfies the first condition of (\ref{3 conditions to check chi in sigma e/2 phi f/2}). Also the second condition of (\ref{3 conditions to check chi in sigma e/2 phi f/2}) is satisfied if we assume $\chi_{g,\xi}(\varpi_E) =\pm1$. Since
\begin{equation*}
\chi_{g,\xi}({\zeta_0}) = t^1_{\mu}(V_{[g]})(\zeta_0)=\left(\frac{\zeta_0^{q^{f/2}-1}}{\mu_{q^{f/2}+1}}\right)=\left(\frac{-1}{\mu_{q^{f/2}+1}}\right)=(-1)^{\frac{1}{2}(q^{f/2}+1)}, \end{equation*}
 we have, using (\ref{value on varpi for sym in S-0 k=e/2}), that
\begin{equation*}
\chi_{g,\xi}({\zeta_0\varpi_E}) =(-1)^{\frac{1}{2}(q^{f/2}+1)}(-1)^{\frac{1}{2}(q^{f/2}-1)}=-1. \end{equation*}
Hence the third condition of (\ref{3 conditions to check chi in sigma e/2 phi f/2}) is satisfied, and $\chi_{g,\xi}$ is a character of $\chi$-data.

\item
For other symmetric $g=\sigma^{k}\phi^{f/2}$ with $\zeta_e^k\zeta_{\phi^{f/2}}\neq \pm1$, we assign
\begin{equation}\label{value on mu-E for sym in S-0 other k}
\chi_{g,\xi}|_{\mu_E} = t^1_{\mu}(V_{[g]})
\end{equation}
and
\begin{equation}\label{value on varpi for sym in S-0 other k}
\chi_{g,\xi}(\varpi_E) =t_\varpi(V_{[g]})t(W_{[g]}).
\end{equation}
We have to show that it satisfies the explicit conditions of $\chi$-data
\begin{equation}\label{3 conditions to check chi in S-0 other k}
\chi_{g,\xi}(\mu_{E_\pm}) = 1,\,\chi_{g,\xi}(\zeta_e^k\zeta_{\phi^{f/2}}\varpi^2_E) = 1,\text{ and }\chi_{g,\xi}(\zeta_0\varpi_E) = -1,
\end{equation}
where $\zeta_0\in \mu_{E_{g}}$ such that $\zeta_0\varpi_E\in E_{\pm g}$. If $V_{[g]}$ is trivial, then clearly there is a character with values in (\ref{value on mu-E for sym in S-0 other k}), (\ref{value on varpi for sym in S-0 other k}) and satisfying the conditions in (\ref{3 conditions to check chi in S-0 other k}). For example, we can choose the unramified quadratic character. When $V_{[g]}$ is non-trivial, similar to what we did in (\ref{conditions for a well-defined charcater in sym-ram case}), we will check in Lemma \ref{lemma determine a character, sym-unram case} below that there exists a character $\chi$ of $E^\times_g$ with the following values,
\begin{equation}\label{conditions for a well-defined charcater in sym-unram case}
\begin{split}
 & \chi|_{\left<\mu_E, \mu_{E_{\pm g}}\right>}= \text{quadratic},\,\chi(\zeta_e^k\zeta_{\phi^{f/2}})=1,
  \\& \chi(\zeta_0)= -t_{\varpi}(V_{[g]})t(W_{[g]})=t^1_{\varpi}(U_{[g]})(\varpi_E),
  \text{ and }\\&\chi(\varpi_E) = -t^1_{\varpi}(U_{[g]})(\varpi_E).
\end{split}
\end{equation}
We can then check that such character satisfies the conditions in (\ref{value on mu-E for sym in S-0 other k}), (\ref{value on varpi for sym in S-0 other k}) and (\ref{3 conditions to check chi in S-0 other k}). The proof is just similar to the case when $g\in \mathcal{D}_{\mathrm{sym,ram}}$, so we skip the details.
\end{enumerate}

\begin{lemma}\label{lemma about explicit chi-data condition, unram case}
Suppose that $g=\sigma^k\phi^{f/2}\in \mathcal{D}_{\mathrm{sym,unram}}$ such that $\zeta_e^k\zeta_{\phi^{f/2}}\neq \pm1$.
\begin{enumerate}[(i)]
\item The element $\zeta_e^k\zeta_{\phi^{f/2}}$ lies in $\ker(N_{{E_g}/{E_{\pm g}}})|_{\mu_{E_g}}-\mu_{E_{\pm g}}$. \label{lemma about zeta_e^k and kernel, unram case}
\item Suppose that $|E_g/E|=2t_k+1$. The element $\zeta_0$ satisfies the relation $\zeta_0^{1-q^{f(2t_k+1)/2}} = \zeta_e^k\zeta_{\phi^{f/2}}$. \label{lemma about zeta_0, unram case}
\end{enumerate}
\end{lemma}

\begin{proof}
  The proof is analogous to Lemma \ref{lemma about explicit chi-data condition}.
\end{proof}

\begin{lemma}\label{lemma determine a character, sym-unram case}
Let $g=\sigma^{k}\phi^{f/2}$ be symmetric, with $\zeta_e^k\zeta_{\phi^{f/2}}\neq \pm1$. Then there exists a character with specified values in (\ref{conditions for a well-defined charcater in sym-unram case}).
\end{lemma}

\begin{proof}
We extend the quadratic character $t^1_{\mu}(U_{[g]})$ of $\mu_E$ by several steps. Denote the group generated by the roots in $\mu_E$ and $\mu_{E_\pm}$ by $\left<\mu_E, \mu_{E_\pm}\right>$. Since $\mu_E \cap \mu_{E_\pm} =\mu_{q^f-1}\cap \mu_{q^{f(2t_k+1)/2}-1}= \mu_{q^{f/2}-1}$, which is mapped by the quadratic character into 1, we can extend $t^1_{\mu}(U_{[g]})$ to the quadratic character $\chi$ on $\left<\mu_E, \mu_{E_\pm}\right>$ such that $\chi|_{\mu_{E_{\pm g}}}\equiv 1$.

 Now by Lemma \ref{lemma about explicit chi-data condition, unram case}.(\ref{lemma about zeta_e^k and kernel, unram case}), we have
$\zeta_e^k\zeta_{\phi^{f/2}} \in \ker(N_{E_g/E_{\pm g}})=\mu_{1+q^{f(2t_k+1)/2}}$. We also have $\zeta_e^k\zeta_{\phi^{f/2}}\notin \left<\mu_E, \mu_{E_\pm}\right>$ because
$$\left<\mu_E, \mu_{E_\pm}\right>\cap \ker(N_{E_g/E_{\pm g}}) = \mu_{(q^{f/2}+1)(q^{{f(2t_k+1)/2}}-1)}\cap \mu_{1+q^{f(2t_k+1)/2}} = \mu_{q^{f/2}+1}\subseteq \mu_E$$ and it is impossible that $\zeta_e^k\zeta_{\phi^{f/2}}\in \mu_E$. Since $t^1_{\mu}(U_{[g]})|_{\mu_{q^{f/2}+1}} \equiv 1$, we can extend $\chi$ to assign the value 1 for $\chi(\zeta_e^k\zeta_{\phi^{f/2}})$.

 Now we would extend $\chi$ to have evaluation at $\zeta_0$. We first write $\zeta_N=\zeta_e^k\zeta_{\phi^{f/2}}$, which menas that $\zeta_e^k\zeta_{\phi^{f/2}}$ is a primitive $N$th root of unity. By Lemma \ref{lemma about explicit chi-data condition, unram case}.(\ref{lemma about zeta_0, unram case}), we know that $\zeta_0$ is a primitive $N(q^{{f(2t_k+1)/2}}-1)$th root of unity. The index of the subgroup $\left<\mu_E, \mu_{E_\pm}\right>\cap \left<\zeta_0\right>$ in $\left<\zeta_0\right>$ is $N/\gcd(N,q^{f/2}+1)$.
  Because
\begin{equation}\label{an odd divisilibity in sym-unram case}
  N\text{ divides }q^{f(2t_k+1)/2}+1=(q^{f/2}+1)(\text{an odd number}),
\end{equation} we can check that such index is an odd integer. Therefore we have to assign $$\chi(\zeta_0)=\chi\left(\zeta^{N/\gcd(N,q^{f/2}+1)}_0\right)=\left(\frac{\zeta^{N/\gcd(N,q^{f/2}+1)}_0}{\left<\mu_E, \mu_{E_\pm}\right>}\right)=\left(\frac{\zeta_{\gcd(N,q^{f/2}+1)}}{\mu_{q^{f/2}+1}}\right).$$
It remains to show that this value is $t^1_{\varpi}(U_{[g]})(\varpi_E)$ as in (\ref{conditions for a well-defined charcater in sym-unram case}). Recall from Proposition \ref{summary of t-factors} that $$t^1_{\varpi}(U_{[g]})(\varpi_E)= \left(\frac{\zeta_N}{\ker(N_{\mathbb{F}_p[\zeta_N]/\mathbb{F}_p[\zeta_N]_\pm})}\right).$$ Let $s$ be the degree of $\mathbb{F}_p[\zeta_N]/\mathbb{F}_p$, so that $\ker(N_{\mathbb{F}_p[\zeta_N]/\mathbb{F}_p[\zeta_N]_\pm})=\mu_{p^{s/2}+1}$. Let $r$ be the degree of the extension $U_{[g]}/\mathbb{F}_p[\zeta_N]$, which can be shown to be odd using the argument similar to that of Lemma \ref{order 2 trick}. Since \begin{equation*}
  q^{f(2t_k+1)/2}+1=(p^{rs/2}+1)=(p^{s/2}+1)(\text{an odd number}),
\end{equation*}
we can use (\ref{an odd divisilibity in sym-unram case}) and show that  $$\left(\frac{\zeta_N}{\mu_{p^{s/2}+1}}\right)=\left(\frac{\zeta_{\gcd(N,q^{f/2}+1)}}{\mu_{q^{f/2}+1}}\right)$$ as desired.
\end{proof}

We have checked that each $\chi_{g,\xi}$ is a character in $\chi$-data for all $g\in\mathcal{D}_{\text{sym,unram}}$. We then check whether
\begin{equation*}
\chi_{\mathcal{D}_0,\xi}|_{\mu_E}={}_{K_0/F}\mu_\xi|_{\mu_E}\text{ and }\chi_{\mathcal{D}_0,\xi}(\varpi_E)={}_{K_0/F}\mu_\xi(\varpi_E).
\end{equation*}
The first equality is clear by construction. Using the explicit value in (\ref{explicit rectifier unram case}) or by the Main Theorem 5.2 of \cite{BH-ET3}, to check the second equality is reduced to check whether
\begin{equation*}
t^0_\mu\left(\prod_{g\in\mathcal{D}_0}V^{\varpi}_{[g]}\right)t_\varpi\left(\prod_{g\in\mathcal{D}_0}V_{[g]}\right)t\left(\prod_{g\in\mathcal{D}_0}W_{[g]}\right)=(-1)^{e(f-1)}t^0_{\mu}(V_{K/F})t^0_{\mu}(V^\varpi_{K/F})t_\varpi(V_{K/F}).
\end{equation*}
By applying the fact (\ref{sym-unram t(V) is -t(W)}) that $t(W_{[g]})=- t_\mu^0(V_{[g]})$ if $[g]$ is symmetric unramified, all t-factors on both sides cancel out. We then have to check an equality on explicit signs. The sign on the left side is equal to the parity of $\#(W_E\backslash W_F/  W_E)_{\text{sym-unram}}$, while the sign on the right side is equal to $(-1)^{e(f-1)}$. They are equal just by Proposition \ref{parity of double coset with i=f/2}.

We have proved Theorem \ref{chi-data factor of BH-rectifier} for $g\in \mathcal{D}_0$. Hence we have finished the proof of Theorem \ref{chi-data factor of BH-rectifier} for the case when the residual characteristic $p$ is odd.

\section{Towards the end of the proof}
The case when the residual characteristic $p = 2$ is much simpler. When $E/F$ is tamely ramified, we have the sequence of subfields $F \subseteq  K \subseteq  E$ where $K/F$ is unramified and $E/K$ is totally ramified of odd degree. Since the order of $\Psi_{E/F}$ is odd, all sign characters and Jacobi symbols, and so all t-factors for $V_{[g]}$, are trivial. Hence Theorem \ref{chi-data factor of BH-rectifier} reduces to the following.
\begin{enumerate}[(i)]
\item If $g\in\mathcal{D}_{\text{asym}/\pm}$, then $$\chi_{g,\xi}|_{\mu_{E_{g}}} \equiv 1\text{ and }\chi_{g,\xi}(\varpi_E) = \text{anything appropriate}.$$
 \item If $g\in \mathcal{D}_{\text{sym}}$, then $$\chi_{g,\xi}|_{\mu_{E_{g}}} \equiv 1\text{ and }\chi_{g,\xi}(\varpi_E) =t^0_\varpi(V_{[g]})t(W_{[g]})= -1.$$
\end{enumerate}
The proof then proceed just as a simpler analogue of the odd residual characteristic case. We have completed the proof of Theorem \ref{chi-data factor of BH-rectifier}. Therefore we can express the essentially tame local Langlands correspondence by the inverse bijection of $\mathcal{L}_n$ as  $$\mathcal{L}_n^{-1}:\pi_\xi \rightarrow \chi_{\{\chi^{-1}_{\lambda,\xi}\}}\circ\tilde{\xi}$$ as stated in Theorem \ref{introduce main result ETLLC as adm-emb}.

\begin{rmk}\label{remark res rectifier is delta}
One of the conditions of the local Langlands correspondence, that $\omega_\pi=\det\sigma$ if $\mathcal{L}_n(\pi)=\sigma$ in Theorem \ref{langlands corresp}(\ref{ETLLC det is central char}), implies that ${}_F\mu_\xi|_{F^\times}=\delta_{E/F}$. This is a general fact about the restriction of the product of the characters in any $\chi$-data, as established in Proposition \ref{restriction of mu to F}.
\qed\end{rmk}

\begin{rmk}
If we examine the proof of Theorem \ref{chi-data factor of BH-rectifier} carefully, we see that the $\chi$-data  $\{\chi_{g,\xi}\}_{g\in\mathcal{D_\pm}}$ depend only on the $W_F$-equivalence of the jump data of the admissible character $\xi$ for almost all $g$, except when $e$ is even and $g=\sigma^{e/2}$. In this exceptional case, $\chi_{g,\xi}$ depends additionally on the additive character $\psi_F$ of $F$. \qed
\end{rmk}

Let $K$ be an intermediate subfield between $E/F$. We regard an $F$-admissible character $\xi$ as being admissible over $K$ and compute the corresponding rectifier ${}_K\mu_\xi$ as follows. Write $\mathcal{D}(K)_\pm={ \mathcal{D}_\pm}\cap W_K/W_E$.
\begin{cor}\label{factorization of rectifier over L}
Let $\{\chi_{g,\xi}\}_{g\in \mathcal{D}_\pm}$ be the collection of $\chi$-data corresponding to the admissible character $\xi$ of $E^\times$ over $F$. Let $K$ be a subfield between $E/F$. We then have
\begin{equation}\label{equation factorization of rectifier over L}
{}_{K}\mu_\xi=\prod_{[g]\in (W_E \backslash W_{K}/W_E)'}\chi_{g,\xi}|_{E^\times}.
\end{equation}
We see that the sub-collection $\{\chi_{g,\xi}\}_{g\in \mathcal{D}(K)_\pm}$ contains those $\chi$-data corresponding to $\xi$ as an $K$-admissible character.
\end{cor}
\begin{proof}
For each subfield $K$ between $E/F$, we write $$V^K=\bigoplus_{[g]\in (W_E\backslash W_K/W_E)'}V_{[g]}.$$ We first compute the character in (\ref{equation factorization of rectifier over L}) when restricted to $\mu_E$. If $e(E/K)$ is odd, then
\begin{equation*}
\prod_{[g]\in (W_E \backslash W_{K}/W_E)'}\chi_{g,\xi}|_{\mu_E}=\prod_{g\in \mathcal{D}(K)_\pm}t^1_{\mu}(\mathbf{V}_{[g]})=t^1_{\mu}(V^K),
\end{equation*}
which is equal to ${}_K\mu_\xi|_{\mu_E}$ in this case. If $e(E/K)$ is even, then
\begin{equation*}
\prod_{[g]\in (W_E \backslash W_{K}/W_E)'}\chi_{g,\xi}|_{\mu_E}=\left(\prod_{g\in \mathcal{D}(K)_\pm-\{\sigma^{e/2}\}}t^1_{\mu}(\mathbf{V}_{[g]})\right)\left(\frac{}{\mu_E}\right)=t^1_{\mu}(V^K)\left(\frac{}{\mu_E}\right),
\end{equation*}
which is also equal ${}_K\mu_\xi|_{\mu_E}$ in this case. We then compute (\ref{equation factorization of rectifier over L}) at $\varpi_E$. Let $L$ be the maximal unramified extension in $E/K$. We write
 $$V_{L/K}=\bigoplus_{[g]\in W_E\backslash W_K/W_E-W_E\backslash W_L/W_E}V_{[g]},$$
 such that $V^K=V^L\oplus V_{L/K}$. We have
\begin{equation*}
\begin{split}
&\prod_{[g]\in (W_E \backslash W_{K}/W_E)'}\chi_{g,\xi}({\varpi_E})
\\= & \left(\prod_{g\in\mathcal{D}(K)_{\text{asym}/\pm} } t_{\varpi_E}(\mathbf{V}_{[g]})t(\mathbf{W}_{[g]})\right)\left(\prod_{g\in\mathcal{D}(K)_{\text{sym}}} t^0_{\mu}(V^\varpi_{[g]})t_{\varpi}(V_{[g]})t(W_{[g]})\right).
\end{split}
\end{equation*}
We re-group the t-factors and obtain
\begin{equation*}
\left(t_{\varpi}(V^L)t(W^L)\right)
\left((-1)^{e(E/K)(f(E/K)-1)}t^0_{\mu}(V_{L/K})t^0_{\mu}(V_{L/K}^\varpi)t_{\varpi}(V_{L/K})\right).
\end{equation*}
The first factor is ${}_{L}\mu_\xi(\varpi_E)$ and the second factor is ${}_{L/K}\mu_\xi(\varpi_E)$. Therefore the product is ${}_{K}\mu_\xi(\varpi_E)$.
\end{proof}

\section{Rectifiers in the theory of endoscopy}

Let $K/F$ be a sub-extension of $E/F$ of degree $d$. Using Corollary \ref{factorization of rectifier over L}, we can extend the definition of $\nu$-rectifiers by defining $${}_{K/F}\mu_\xi= {}_F\mu_\xi {}_K\mu^{-1}_\xi.$$
 By Corollary \ref{factorization of rectifier over L}, this $\nu$-rectifier is equal to $$ \prod_{[g]\in W_E \backslash W_F / W_E - W_E \backslash W_K / W_E}  \chi_{g,\xi}|_{E^\times} .$$
When $K/F$ is cyclic, we let $\Delta_\mathrm{III_2}$ be the transfer factor associated to the groups $(G,H)=(\mathrm{GL}_n,\mathrm{Res}_{K/F}\mathrm{GL}_{n/d})$.
\begin{cor}\label{rectifier as transfer factor}
  The transfer factor $\Delta_\mathrm{III_2}$ is equal to the $\nu$-rectifier ${}_{K/F}\mu_\xi$. In particular, if $E/F$ is cyclic, then the rectifier ${}_F\mu_\xi$ is exactly $\Delta_\mathrm{III_2}$.
\end{cor}
\begin{proof}
The first assertion is just a consequence of Corollary \ref{definition of delta III2}. The second assertion is clear.
\end{proof}


\addcontentsline{toc}{chapter}{Bibliography}





\begin{thebibliography}{BH05b}

\bibitem[AC89]{AC}
James Arthur and Laurent Clozel.
\newblock {\em Simple algebras, base change, and the advanced theory of the
  trace formula}, volume 120 of {\em Annals of Mathematics Studies}.
\newblock Princeton University Press, Princeton, NJ, 1989.

\bibitem[BF83]{BF}
Colin~J. Bushnell and Albrecht Fr{\"o}hlich.
\newblock {\em Gauss sums and {$p$}-adic division algebras}, volume 987 of {\em
  Lecture Notes in Mathematics}.
\newblock Springer-Verlag, Berlin, 1983.

\bibitem[BH98]{BH-WHIT}
Colin~J. Bushnell and Guy Henniart.
\newblock Supercuspidal representations of {${\rm GL}\sb n$}: explicit
  {W}hittaker functions.
\newblock {\em J. Algebra}, 209(1):270--287, 1998.

\bibitem[BH03]{BH-LTL4}
Colin~J. Bushnell and Guy Henniart.
\newblock Local tame lifting for {${\rm GL}(n)$}. {IV}. {S}imple characters and
  base change.
\newblock {\em Proc. London Math. Soc. (3)}, 87(2):337--362, 2003.

\bibitem[BH05a]{BH-ET1}
Colin~J. Bushnell and Guy Henniart.
\newblock The essentially tame local {L}anglands correspondence. {I}.
\newblock {\em J. Amer. Math. Soc.}, 18(3):685--710 (electronic), 2005.

\bibitem[BH05b]{BH-ET2}
Colin~J. Bushnell and Guy Henniart.
\newblock The essentially tame local {L}anglands correspondence. {II}.
  {T}otally ramified representations.
\newblock {\em Compos. Math.}, 141(4):979--1011, 2005.

\bibitem[BH05c]{BH-LTL3}
Colin~J. Bushnell and Guy Henniart.
\newblock Local tame lifting for {${\rm GL}(n)$}. {III}. {E}xplicit base change
  and {J}acquet-{L}anglands correspondence.
\newblock {\em J. Reine Angew. Math.}, 580:39--100, 2005.

\bibitem[BH06]{BH-GL2}
C.J. Bushnell and G.~Henniart.
\newblock {\em {The local Langlands conjecture for GL(2)}}.
\newblock Grundlehren der mathematischen Wissenschaften. Springer, 2006.

\bibitem[BH10]{BH-ET3}
Colin~J. Bushnell and Guy Henniart.
\newblock The essentially tame local {L}anglands correspondence, {III}: the
  general case.
\newblock {\em Proc. Lond. Math. Soc. (3)}, 101(2):497--553, 2010.

\bibitem[BK93]{BK}
C.J. Bushnell and P.C. Kutzko.
\newblock {\em {The admissible dual of GL(N) via compact open subgroups}}.
\newblock Annals of mathematics studies. Princeton University Press, 1993.

\bibitem[Bor79]{Borel-autom-L}
A.~Borel.
\newblock Automorphic {$L$}-functions.
\newblock In {\em Automorphic forms, representations and {$L$}-functions
  ({P}roc. {S}ympos. {P}ure {M}ath., {O}regon {S}tate {U}niv., {C}orvallis,
  {O}re., 1977), {P}art 2}, Proc. Sympos. Pure Math., XXXIII, pages 27--61.
  Amer. Math. Soc., Providence, R.I., 1979.

\bibitem[Gre55]{Green}
J.~A. Green.
\newblock The characters of the finite general linear groups.
\newblock {\em Trans. Amer. Math. Soc.}, 80:402--447, 1955.

\bibitem[HC99]{HC-AID}
Harish-Chandra.
\newblock {\em Admissible invariant distributions on reductive {$p$}-adic
  groups}, volume~16 of {\em University Lecture Series}.
\newblock American Mathematical Society, Providence, RI, 1999.
\newblock Preface and notes by Stephen DeBacker and Paul J. Sally, Jr.

\bibitem[Hen00]{Hen-simple}
Guy Henniart.
\newblock Une preuve simple des conjectures de {L}anglands pour {${\rm GL}(n)$}
  sur un corps {$p$}-adique.
\newblock {\em Invent. Math.}, 139(2):439--455, 2000.

\bibitem[HH95]{Hen-Herb}
Guy Henniart and Rebecca Herb.
\newblock Automorphic induction for {${\rm GL}(n)$} (over local
  non-{A}rchimedean fields).
\newblock {\em Duke Math. J.}, 78(1):131--192, 1995.

\bibitem[HI10]{HI}
Kaoru Hiraga and Atsushi Ichino.
\newblock {O}n the {K}ottwitz-{S}helstad normalization of transfer factors for
  automorphic induction for {${\rm GL}\sb n$}.
\newblock 2010.

\bibitem[HL10]{HL2010}
Guy Henniart and Bertrand Lemaire.
\newblock Formules de caract\`eres pour l'induction automorphe.
\newblock {\em J. Reine Angew. Math.}, 645:41--84, 2010.

\bibitem[HL11]{HL2011}
Guy Henniart and Bertrand Lemaire.
\newblock Changement de base et induction automorphe pour {${\rm GL}\sb n$} en
  caract\'eristique non nulle.
\newblock {\em M\'em. Soc. Math. Fr. (N.S.)}, (124):vi+190, 2011.

\bibitem[How77]{Howe}
Roger~E. Howe.
\newblock Tamely ramified supercuspidal representations of {${\rm Gl}\sb{n}$}.
\newblock {\em Pacific J. Math.}, 73(2):437--460, 1977.

\bibitem[HT01]{HT}
Michael Harris and Richard Taylor.
\newblock {\em The geometry and cohomology of some simple {S}himura varieties},
  volume 151 of {\em Annals of Mathematics Studies}.
\newblock Princeton University Press, Princeton, NJ, 2001.
\newblock With an appendix by Vladimir G. Berkovich.

\bibitem[JL70]{Jac-Langlands-GL(2)}
H.~Jacquet and R.~P. Langlands.
\newblock {\em Automorphic forms on {${\rm GL}(2)$}}.
\newblock Lecture Notes in Mathematics, Vol. 114. Springer-Verlag, Berlin,
  1970.

\bibitem[KS99]{KS}
Robert~E. Kottwitz and Diana Shelstad.
\newblock Foundations of twisted endoscopy.
\newblock {\em Ast\'erisque}, (255):vi+190, 1999.

\bibitem[KS12]{KS2012}
Robert~E. Kottwitz and Diana Shelstad.
\newblock {O}n {S}plitting {I}nvariants and {S}ign {C}onventions in
  {E}ndoscopic {T}ransfer.
\newblock (preprint), \texttt{arXiv:1201.5658v1}, 2012.

\bibitem[Lan89]{Langlands-real-alg-gp}
R.~P. Langlands.
\newblock On the classification of irreducible representations of real
  algebraic groups.
\newblock In {\em Representation theory and harmonic analysis on semisimple
  {L}ie groups}, volume~31 of {\em Math. Surveys Monogr.}, pages 101--170.
  Amer. Math. Soc., Providence, RI, 1989.

\bibitem[Lan94]{Lang-ANT}
Serge Lang.
\newblock {\em Algebraic number theory}, volume 110 of {\em Graduate Texts in
  Mathematics}.
\newblock Springer-Verlag, New York, second edition, 1994.

\bibitem[LRS93]{LRS}
G.~Laumon, M.~Rapoport, and U.~Stuhler.
\newblock {${\mathcal{D}}$}-elliptic sheaves and the {L}anglands
  correspondence.
\newblock {\em Invent. Math.}, 113(2):217--338, 1993.

\bibitem[LS87]{LS}
R.~P. Langlands and D.~Shelstad.
\newblock On the definition of transfer factors.
\newblock {\em Math. Ann.}, 278(1-4):219--271, 1987.

\bibitem[Moy86]{Moy}
Allen Moy.
\newblock Local constants and the tame {L}anglands correspondence.
\newblock {\em Amer. J. Math.}, 108(4):863--930, 1986.

\bibitem[MP94]{MP1}
Allen Moy and Gopal Prasad.
\newblock Unrefined minimal {$K$}-types for {$p$}-adic groups.
\newblock {\em Invent. Math.}, 116(1-3):393--408, 1994.

\bibitem[Ng{\^o}10]{Ngo}
Bao~Ch{\^a}u Ng{\^o}.
\newblock Le lemme fondamental pour les alg\`ebres de {L}ie.
\newblock {\em Publ. Math. Inst. Hautes \'Etudes Sci.}, (111):1--169, 2010.

\bibitem[Rei91]{Rei}
Harry Reimann.
\newblock Representations of tamely ramified {$p$}-adic division and matrix
  algebras.
\newblock {\em J. Number Theory}, 38(1):58--105, 1991.

\bibitem[Rei03]{reiner-max-order}
I.~Reiner.
\newblock {\em Maximal orders}, volume~28 of {\em London Mathematical Society
  Monographs. New Series}.
\newblock The Clarendon Press Oxford University Press, Oxford, 2003.
\newblock Corrected reprint of the 1975 original, With a foreword by M. J.
  Taylor.

\bibitem[Ser77]{Serre-repres}
Jean-Pierre Serre.
\newblock {\em Linear representations of finite groups}.
\newblock Springer-Verlag, New York, 1977.
\newblock Translated from the second French edition by Leonard L. Scott,
  Graduate Texts in Mathematics, Vol. 42.

\bibitem[Ser79]{Serre-local-fields}
Jean-Pierre Serre.
\newblock {\em Local fields}, volume~67 of {\em Graduate Texts in Mathematics}.
\newblock Springer-Verlag, New York, 1979.
\newblock Translated from the French by Marvin Jay Greenberg.

\bibitem[She08]{Shelstad-temp-endo-for-real-group-1}
D.~Shelstad.
\newblock Tempered endoscopy for real groups. {I}. {G}eometric transfer with
  canonical factors.
\newblock In {\em Representation theory of real reductive {L}ie groups}, volume
  472 of {\em Contemp. Math.}, pages 215--246. Amer. Math. Soc., Providence,
  RI, 2008.

\bibitem[She10]{Shelstad-temp-endo-for-real-group-2}
D.~Shelstad.
\newblock Tempered endoscopy for real groups. {II}. {S}pectral transfer
  factors.
\newblock In {\em Automorphic forms and the {L}anglands program}, volume~9 of
  {\em Adv. Lect. Math. (ALM)}, pages 236--276. Int. Press, Somerville, MA,
  2010.

\bibitem[Tat79]{Tate-NTB}
J.~Tate.
\newblock Number theoretic background.
\newblock In {\em Automorphic forms, representations and {$L$}-functions
  ({P}roc. {S}ympos. {P}ure {M}ath., {O}regon {S}tate {U}niv., {C}orvallis,
  {O}re., 1977), {P}art 2}, Proc. Sympos. Pure Math., XXXIII, pages 3--26.
  Amer. Math. Soc., Providence, R.I., 1979.

\bibitem[Wal91]{Walds-GLn}
J.-L. Waldspurger.
\newblock Sur les int\'egrales orbitales tordues pour les groupes lin\'eaires:
  un lemme fondamental.
\newblock {\em Canad. J. Math.}, 43(4):852--896, 1991.

\bibitem[Yu01]{Yu}
Jiu-Kang Yu.
\newblock Construction of tame supercuspidal representations.
\newblock {\em J. Amer. Math. Soc.}, 14(3):579--622 (electronic), 2001.

\end{thebibliography}
\end{document}